\documentclass[a4paper,english 12pt]{amsart}

\usepackage[utf8]{inputenc}
\usepackage[T1]{fontenc}
\usepackage[english]{babel}
\usepackage{multicol}
\usepackage{enumitem}
\usepackage{pifont}
 \usepackage{pgf,tikz,pgfplots}
\pgfplotsset{compat=1.15}
\usepackage{mathrsfs}
\usetikzlibrary{arrows}
\usepackage{mathtools}
\usepackage{float}
\usepackage{appendix}
\usepackage{tikz}
\usepackage{xcolor}

\usepackage{hyperref}

\usepackage{amsmath,amssymb}
\usepackage{mathrsfs}
\usepackage{amsthm}
\usepackage{yhmath}
\usepackage{dsfont}
\numberwithin{equation}{section}

\usepackage{graphicx}
\usepackage{subcaption}

\usepackage[left=3cm,right=3cm,top=2.5cm,bottom=3cm]{geometry}

\DeclareMathOperator{\im}{Im}
\DeclareMathOperator{\re}{Re}

\DeclareMathOperator{\Id}{Id}

\DeclareMathOperator{\I}{I}

\DeclareMathOperator{\tr}{tr}

\newcommand{\R}{\mathbb{R}}

\newcommand{\N}{\mathbb{N}}
\newcommand{\Z}{\mathbb{Z}}
\newcommand{\C}{\mathbb{C}}

\DeclareMathOperator{\cinf}{\emph{C}^\infty}
\DeclareMathOperator{\cinfc}{\emph{C}_c^\infty}
\DeclareMathOperator{\supp}{supp}

\DeclareMathOperator{\gr}{Gr}

\DeclareMathOperator{\op}{Op_{\textit{h}}}
\DeclareMathOperator{\Gr}{Gr}
\DeclareMathOperator{\WF}{WF_{\textit{h}}}
\newcommand{\hinf}{O(h^\infty)}

\newtheorem{thm}{Theorem}
\newtheorem*{thm*}{Theorem}

\theoremstyle{definition}
\newtheorem*{ex}{Example}
\newtheorem{defi}{Definition}[section]
\newtheorem{prop}{Proposition}[section]
\newtheorem*{nota}{Notations}
\newtheorem{lem}{Lemma}[section]
\newtheorem{cor}{Corollary}[section]
\newtheorem*{rem}{Remark}

\makeatletter
\def\paragraph{\vspace{0.4cm} \@startsection{paragraph}{4}%
  \z@\z@{-\fontdimen2\font}%
  {\normalfont\bfseries}}
\makeatother

\makeatletter
\def\subparagraph{\vspace{0.3cm} \@startsection{subparagraph}{4}%
  \z@\z@{-\fontdimen2\font}%
  {\normalfont\bfseries}}
\makeatother

\title{Improved fractal Weyl upper bound in obstacle scattering}
\author{Lucas Vacossin}
\address{Universit\'e Paris-Saclay, Laboratoire de mathématiques d'Orsay, UMR 8628 du CNRS, B\^atiment 307, 91405 Orsay Cedex,}
\email{lucas.vacossin@universite-paris-saclay.fr}

\makeindex
\begin{document}
\maketitle

\begin{abstract}
In this paper, we are interested in the problem of scattering by strictly convex obstacles in the plane.  We provide an upper bound for the number  $N(r,\gamma)$ of resonances in the box $\{ r \leq \re(\lambda)   \leq r +1 ; \im(\lambda) \geq - \gamma \}$. It was proved in the work of \cite{NSZ14} that $N(r,\gamma) = O_{\gamma} (r^{d_H})$ where $2d_H +1$ is the Hausdorff dimension of the trapped set of the billiard flow. In this article, we provide an improved upper bound in the band $0 \leq \gamma < \gamma_{cl}/2$, where $\gamma_{cl}$ is the classical decay rate of the flow. This improved Weyl upper bound is in the spirit of the ones of \cite{Naud12} and \cite{DyBW} in the case of convex co-compact surfaces, and of \cite{DJ17} in the case of open quantum baker's maps. 
\end{abstract}

\section{Introduction}
\subsection{An improved fractal upper bound. }
\paragraph{Scattering by convex obstacles. }

In this paper, we are interested in the problem of scattering by strictly convex obstacles in the plane. We assume that $$\mathcal{O} = \bigcup_{j=1}^J \mathcal{O}_j$$ where $\mathcal{O}_j$ are open, strictly convex obstacles in $\R^2$ having smooth boundary and satisfying the \emph{Ikawa condition} : for $i \neq j \neq k$, $\overline{\mathcal{O}_i}$ does not intersect the convex hull of $ \overline{ \mathcal{O}_j }\cup \overline{\mathcal{O}_k}$. Let $$
\Omega = \R^2  \setminus \overline{\mathcal{O}}.$$
\begin{figure}[h]
\begin{center}
\includegraphics[scale=0.3]{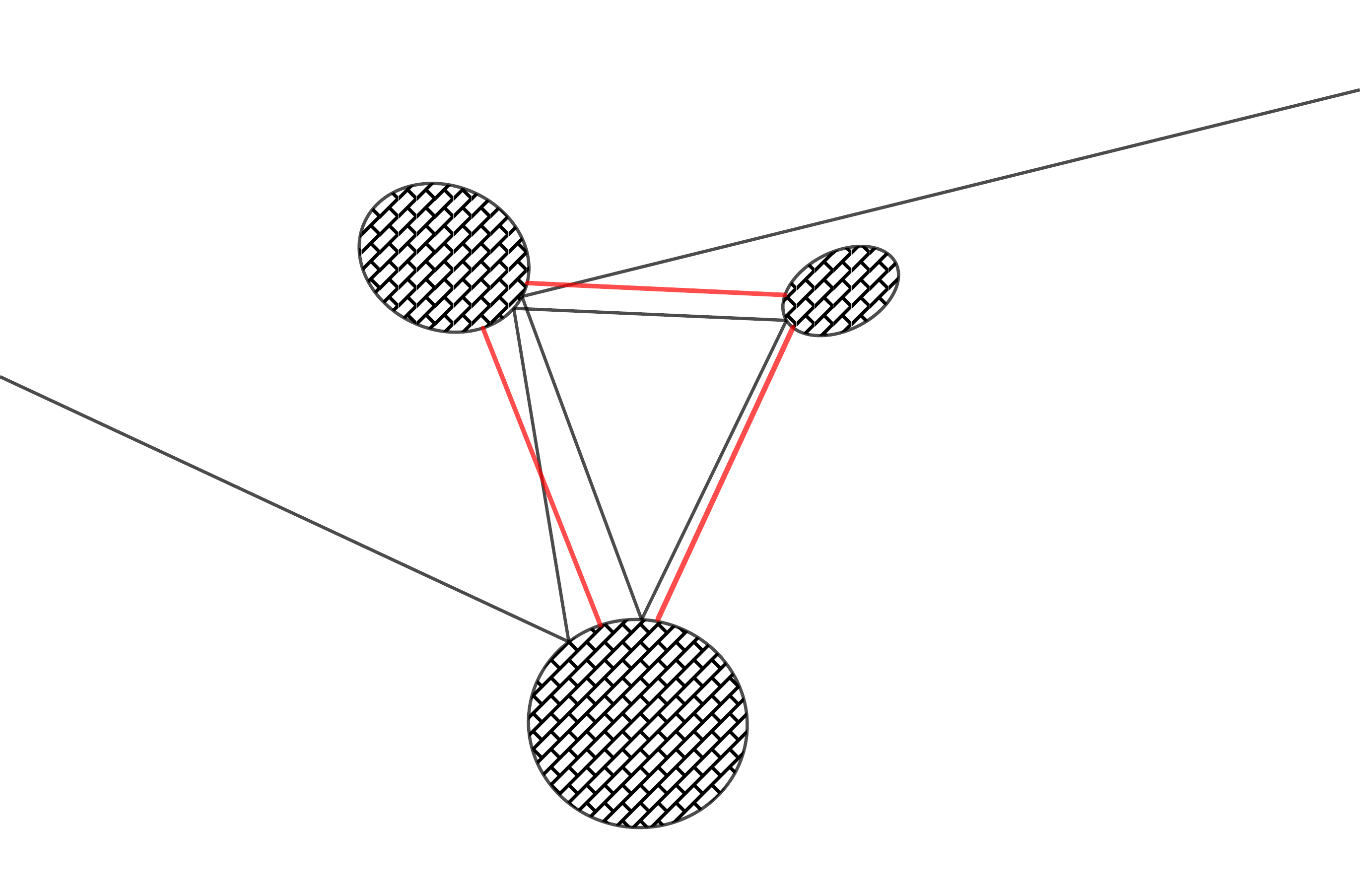}
\caption{Scattering by three obstacles in the plane}
\end{center}
\end{figure}
It is known that the resolvent of the Dirichlet Laplacian in $\Omega$ continues meromorphically to the logarithmic cover of $\C$ (see for instance \cite{DyZw}, Theorem 4.4 in Chapter 4).  More precisely, suppose that
$\chi \in \cinfc(\R^2)$ is equal to one in a neighborhood of $\overline{\mathcal{O}}$. 
$$R_\chi(\lambda) \coloneqq \chi (-\Delta - \lambda^2 )^{-1} \chi  : L^2 (\Omega) \to L^2(\Omega)$$ 
is holomorphic in the region $\{ \im \lambda >0 \}$ and it continues meromorphically to the logarithmic cover of $\C$. Its poles are the \emph{scattering resonances}. We are interested in the distribution of these scattering resonances in the first sheet of the logarithmic cover (i.e. $\C \setminus  i \R^-$), or more precisely, in a conic neighborhood of $\R$. 

The multiplicity of a (non-zero) resonance $\lambda_0$ is given by 
$$m(\lambda_0) = \text{ rank } \frac{1}{2i \pi} \int_{\gamma} R_\chi(\lambda) d\lambda \; , \; \gamma(t) = \lambda_0 + \epsilon^{2 i \pi t}  \; , \; 0 \leq t \leq 1  \; , \; 0 < \epsilon \ll 1 $$
We are interested in counting resonances in strips and in this paper we focus on an upper bound for the quantities 
$$ N(r, \gamma) = \sum_{ \substack{ r \leq \re \lambda \leq r+1 \\  \im \lambda \geq -\gamma}} m(\lambda) $$ 
The depth $\gamma$ of the strip being fixed, we are interested by upper bounds as $r \to + \infty$.

\paragraph{Fractal Weyl bounds. }
In this regime, it becomes a high-frequency problem and justifies the introduction of a small parameter $h = r^{-1}$. Under this rescaling, it becomes a semiclassical problem. 
In the semiclassical limit, that is $h \to 0$, the classical dynamics associated with this quantum problem is the billiard flow $\varphi^t$ in $\Omega \times \mathbb{S}^1$, that is to say, the free motion outside the obstacles with normal reflection on their boundaries. A relevant dynamical object is the trapped set $K$ corresponding to the points $(x, \xi) \in \Omega \times \mathbb{S}^1$ that do not escape to infinity in the backward and forward direction of the flow. In the case of two obstacles, it is a single closed geodesic. As soon as more obstacles are involved, the structure of the trapped set becomes complex and exhibits a fractal structure. This is a consequence of the hyperbolicity of the billiard flow. The structure of the trapped set plays a crucial role in the spectral properties of $-\Delta$. In particular, its Hausdorff dimension appears when estimating $N(r, \gamma)$. In \cite{NSZ14}, the authors proved a Fractal Weyl upper bound involving this fractal dimension.   

\begin{thm}\label{THM_previous} \textbf{Fractal Weyl upper bound \cite{NSZ14}}\\
Assume that the obstacles $\mathcal{O}_j \subset \R^2$ satisfy the conditions above. 
Assume that the trapped set of the billiard flow has Hausdorff dimension $2d_H +1$. Then, for every $\gamma >0$, there exists $C_\gamma>0$ such that for all $r \geq 1$, 
$$ N(r,\gamma) \leq C_\gamma r^{d_H}$$
\end{thm} 
\begin{rem}
Their result holds in any dimension, but in dimension $d >2$, one has to add an extra loss of $\varepsilon$ : for every $\varepsilon >0$, for every $\gamma >0$, there exists $C_{\varepsilon,\gamma}>0$ such that for all $r \geq 1$, 
$$ N(r,\gamma) \leq C_{\varepsilon,\gamma} r^{d_H+\varepsilon}$$
\end{rem}
\vspace*{0.5cm}
This bound is conjectured to be optimal for large values of $\gamma$ (see \cite{Zw17}, Conjecture 5). However, as soon as a spectral gap exists, the exponent $d_H$ cannot be optimal for any $\gamma$. It always exists in dimension 2, as proved in \cite{Vacossin} and it holds also in higher dimensions under some pressure condition (see \cite{Ik88}) on the billiard flow. Our Theorem \ref{THM_new} below gives a better bound in dimension 2 for 
$$ \gamma < \gamma_{cl}/2$$
where $\gamma_{cl}$ is the classical decay rate of the flow. $\gamma_{cl}$ is equal to $-P(-\varphi_u)$ where $\varphi_u$ is defined in (\ref{unstable_jacobian}) with the unstable Jacobian and $P$ is the topological pressure for the billiard map on the trapped set (see Definition \ref{def_pressure}). It is also given by the following formula (see \cite{BoRu}, Proposition 4.4) : 
\begin{equation}
- \gamma_{cl} = \lim_{\varepsilon \to 0}  \limsup_{T \to + \infty} \frac{1}{T} \log  \text{Leb} \left( \bigcup_{\rho \in K } B_\rho(\varepsilon, T ) \right) 
\end{equation} 
where 
$$  B_\rho (\varepsilon, T )  = \{ (x, \xi) \in \Omega \times \mathbb{S}^1,  \forall t \in [0,T],  d( \varphi^t (x,\xi) , \varphi^t(\rho) ) \leq \varepsilon\}$$
are Bowen balls. 

The theorem we prove in this article is 
 
\begin{thm}\label{THM_new}
Assume that the obstacles $\mathcal{O}_j \subset \R^2$ satisfy the conditions above. 
Then, there exists a non increasing function $\sigma : \R_+ \to \R_+$ satisfying 
\begin{itemize}
\item $\sigma(\gamma) > 0$ for $0 \leq \gamma < \gamma_{cl}/2$ ; 
\item $\sigma(\gamma) = 0$ for $\gamma \geq \gamma_{cl}/2 $
\end{itemize}
and such that 
for all $\gamma >0$ and for all $\varepsilon>0$ there exists $C_{\gamma, \varepsilon}>0$ such that 
$$\forall r \geq 1, N(r,\gamma) \leq C_{\gamma, \varepsilon} r^{d_H - (\sigma(\gamma) - \varepsilon)_+}$$ 
\end{thm}

\begin{rem}
A rather explicit value of $\sigma$ in term of topological pressure is given by the formula (\ref{value_of_sigma}).\\
Here, we can take $(\sigma(\gamma)- \varepsilon)_+= \max(\sigma(\gamma) - \varepsilon, 0)$ due to the result of \cite{NSZ14}. When $\gamma \geq \gamma_{cl}/2$, we always have $(\sigma(\gamma)- \varepsilon)_+ = 0$. When $\gamma < \gamma_{cl}/2$, we can find $\varepsilon>0$ such that the bound given by Theorem \ref{THM_new} improves the one of Theorem \ref{THM_previous}. 
\end{rem}

\paragraph{More on obstacle scattering. }
The problem of wave scattering by obstacles has a long history in the physics and mathematics literature. The case of two obstacles is particularly well-understood (see \cite{Ge88}, \cite{Ik82}), so is the diffraction by one convex obstacle (see for instance \cite{BLR}, \cite{HaLe}). As soon as 3 or more obstacles are involved, the underlying classical flow - in this case, the billiard flow - becomes highly chaotic. A particularly interesting model is the $n$-disk system, which has been intensively studied both numerically and experimentally (see for instance \cite{Gr},\cite{Phys3}) and the fractal upper bound has been successfully tested in \cite{PhysRevE1} or \cite{PhysRev2}. A recent result concerning a spectral gap has been proved in \cite{Vacossin}, improving the previous result of \cite{Ik88} (see also \cite{NZ09}).

\paragraph{Related results in open hyperbolic systems.}
The problem of scattering by obstacles falls into the wider class of spectral problems for open hyperbolic systems, that is scattering systems where most trajectories escape to infinity, so that the trapped set has Liouville measure zero, and supports a hyperbolic flow. We refer the reader to the article of review \cite{Nonnen11} for a survey on these open chaotic systems. Among the problems which widely interest mathematicians and physicists, resonance counting and spectral gaps are on the top of the list (see for instance \cite{Zw17} for results and open problems concerning resonances). An important example is given by the semiclassical scattering by a potential (see \ref{sub_section_application_potential}), with particular dynamical assumptions on the Hamiltonian flow. 

\vspace*{0.4cm}
\emph{Convex co-compact hyperbolic surfaces. } Another class of open hyperbolic systems exhibiting a fractal trapped set consists of the convex co-compact hyperbolic surfaces, which can be obtained as the quotient of the hyperbolic plane $\mathbb{H}^2$ by Schottky groups $\Gamma$. The spectral problem concerns the Laplacian on these surfaces and its classical counterpart is the geodesic flow on the cosphere bundle, which is known to be hyperbolic due to the negative curvature of these surfaces. In this context, it is common to write the energy variable $\lambda^2 = s(1-s)$ and study the meromorphic continuation of
$$s \in \C \mapsto \left(- \Delta - s(1-s) \right)^{-1} $$
The trapped set, and more particularly its dimension, influences the spectrum (see for instance \cite{Bo} for an introduction to this theory). 

\vspace*{0.4cm}
\emph{Weyl upper bounds. }
The first Fractal Weyl upper bound for the counting function in strips appeared in the work of Sjöstrand \cite{Sjo90} (see Section 5, Theorem 5.7) for Schrödinger operators $-h^2 \Delta +V$ in the analytic case. The author estimated the number of resonances in larger boxes $\{| \re z| \leq \delta,  - \gamma h \leq \im z \leq 0 \}$ in the limit $h \to 0$. More precise upper bounds $O(h^{-d_H})$ for smaller boxes $\{ |\re z | \leq Ch, - \gamma h \leq  \im z \leq 0 \}$, which correspond, under the rescaling $r=h^{-1}$ to the boxes we consider, were obtained in different smooth situations : for convex co-compact hyperbolic surfaces (\cite{Zwo99}), in scattering by a potential (\cite{SjZwFractal}), in obstacle scattering (\cite{NSZ14}), for asymptotically hyperbolic manifold (\cite{DaDy12}). 
It has been conjectured (see \cite{Zw17}, Conjecture 5) that the bound
$N(r,\gamma )= O(r^{d_H})$ is optimal when the strip is sufficiently large. However, numerical experiments (see for instance the appendix of \cite{DyBW} for the case of convex co-compact surfaces) show that it should be possible to improve this bound for strips of width smaller than some threshold. These numerical results lead \cite{Zw17} to conjecture that 
$$ \lim_{ r \to + \infty} N(r, \gamma) r^{-d_H}  = 0 \text{ when } 0 \leq \gamma < \frac{\gamma_{cl}}{2}$$
First results in this direction were obtained in the case of convex co-compact hyperbolic surfaces : 
\begin{itemize}
\item In \cite{Naud12}, the author showed a bound similar to the one in Theorem \ref{THM_new} (without the loss of $\varepsilon$), with a function $\sigma$ having the same properties ; 
\item In \cite{DyBW}, the author obtained the same result  with an explicit function $\sigma$ given by $\sigma(\gamma) = 1 - d_H - 2 \gamma$, which satisfies the same properties as the one in Theorem \ref{THM_new} (since in this context $\gamma_{cl} = 1- d_H$). His result can be generalized to higher dimensional convex co-compact hyperbolic manifold. 
\end{itemize} 
Theorem \ref{THM_new} gives a positive answer to this conjecture in obstacle scattering in dimension 2. There is also a stronger conjecture, due to Jakobson-Naud (\cite{JakobsonNaud}) in the case of convex co-compact surfaces, which states that for every $\gamma < \gamma_{cl}/2$, $N(r,\gamma) = 0$ for $r \gg 1$. Our work is still far from proving this conjecture. 

\vspace*{0.4cm}
\emph{Toy models and open quantum maps. } To test these conjectures, it is useful to work on toy models where numerical and theoretical computations are sometimes easier. A very appreciated toy model in the study of open hyperbolic systems is the open baker's map (see for instance \cite{Nonnen11}, section 6.1.1). The classical map is a piecewise affine open map $F_{a, \mathcal{A}}$ on the torus $\mathbb{T}^2$, associated with an alphabet $\mathcal{A} \subset \{0, \dots, a-1\}$ ($a$ is called the base) (see Figure \ref{baker}). 

\begin{figure}
\includegraphics[scale=0.6]{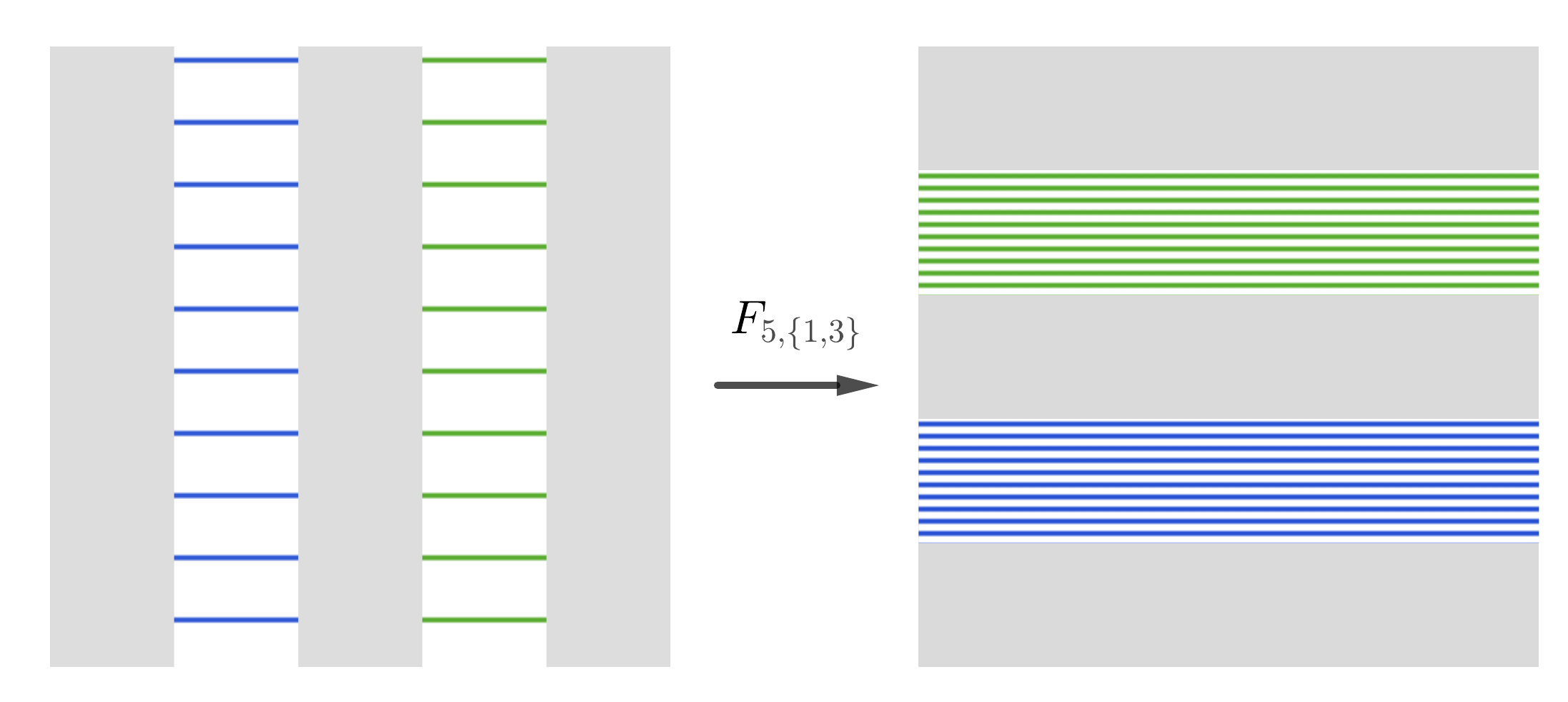}
\caption{Example of an open baker's map. }
\label{baker}
\end{figure}

It quantization is given by a matrix $M_N \in \mathcal{M}_N(\R)$ where $N$ plays the role of $(2 \pi h)^{-1}$.  In this context, one wants to count the number of eigenvalues of the matrix $M_N \in \mathcal{M}_N(\R)$ in the annulus $\{ |z| \geq \nu \}$ in the limit $N \to + \infty$.

These toy models are what we call \emph{open quantum maps}. 
There is a heuristic correspondence between open quantum maps and open quantum systems. These quantized open maps have to be interpreted as propagators at time $t = \log a$ of an open quantum system with constant unstable Jacobian $J^u=a$, so that, to a resonance $\lambda$ of the open quantum system, it corresponds an eigenvalue $e^{-i t \lambda }$ of the open quantum map, with $t = \log a$. 
In fact, \cite{NSZ11} and \cite{NSZ14} have shown that the study of the resonances in obstacle scattering can be reduced to the study of a family of more general open quantum maps. This reduction is the starting point of the proof of Theorem \ref{THM_new}. 

 Concerning the quantized open baker's map, there are convincing numerical and theoretical results. In \cite{Nonnenmacher_2007}, the authors gave numerical evidence of Weyl upper bounds of the type 
$$ \# \left( \text{Spec }(M_N) \cap \{ |z| \geq \nu \}\right) \leq C_\nu N^{d_H} $$ 
In \cite{DJ17}, the author proved an even more precise upper bound, when $N=a^k$ : 
$$ \# \left( \text{Spec }(M_{a^k}) \cap \{ |z| \geq \nu \}\right)  \leq C_\nu (a^k)^{d_H + \varepsilon - \Sigma(\nu) }  \, , \,  \forall k \in \N \quad;  \quad \Sigma(\nu) = \sigma\left( - \frac{\log \nu}{ \log J^u }\right)$$
where $J^u=a$ is the unstable Jacobian of the system and $\sigma( \gamma) = \max(1- d_H - 2 \gamma, 0)$. In particular, $\sigma$ shares the same properties as the one in Theorem \ref{THM_new}, since the classical decay rate of the baker's map is $1-d_H$. The link between $\Sigma(\nu)$ and $\sigma(\gamma)$ comes from the heuristic interpretation above.

\subsection{Statement of the main theorem}

Our proof of Theorem \ref{THM_new} relies on previous results of \cite{NSZ14}. Their Theorem 5 reduces the study of the scattering resonances $\lambda \in ]1/h -R, 1/h +R[ + i ]-R, R[$ to the study of the cancellation of
$$ z \in ]-R,R[ + i ]-R,R[ \mapsto \det( \I - M(z;h) ) $$ 
where 
\begin{equation}\label{M(z)}
M(z) = M(z;h) : L^2 (\partial \mathcal{O}) \to L^2 (\partial \mathcal{O})
\end{equation}  is a family of \emph{open quantum hyperbolic maps} (see below Section \ref{hyperbolic_open_quantum_map}). The family $z \mapsto M(z)$ depends holomorphically on $z \in ]-R,R[ + i ]-R,R[ $ for some arbitrary $R>0$ and is sometimes called a \emph{hyperbolic quantum monodromy operator}. The zeros $z$ and the resonances are related by the relation $ h \lambda = 1+z$. The notion of \emph{monodromy} comes from the fact that the outgoing solutions of the equation 
$- \Delta u = \lambda^2 u$ must satisfy the equation $M(z)u=u$, which dictates the behavior of $u$ on the boundary of the obstacles. We now introduce some definitions required to state the main theorem of this paper. We show how Theorem \ref{Theorem_main} implies Theorem \ref{THM_new} using the results of \cite{NSZ14} in \ref{subsection_billiard_map}.

\subsubsection{Open quantum hyperbolic maps and statement on the main theorem }
\label{hyperbolic_open_quantum_map}
The following long definition is based on the definitions in the works of Nonnenmacher, Sjöstrand and Zworski in \cite{NSZ11} and \cite{NSZ14} specialized to the 2-dimensional phase space. 
Consider open intervals $Y_1, \dots, Y_J$ of $J$ copies of $\R$ and set  : 
$$Y = \bigsqcup_{j=1}^J Y_j \subset \bigsqcup_{j=1}^J \R$$
and consider 
$$ U = \bigsqcup_{j=1}^J U_j \subset \bigsqcup_{j=1}^J T^*\R \quad ; \quad \text{where } U_j \Subset T^*Y_j \text{ are open sets }$$
The Hilbert space $L^2(Y)$ is the orthogonal sum $\bigoplus_{i=1}^J L^2(Y_i)$. 

For $j=1, \dots, J$, consider open disjoint subsets $\widetilde{D}_{i j } \Subset U_j$, $1 \leq i \leq J$, \emph{the departure sets}, and similarly, for $i= 1, \dots, J$ consider open disjoint subsets $\widetilde{A}_{i j } \Subset U_i$, $1 \leq j \leq J$, \emph{the arrival sets} (see Figure \ref{figure_example_1}). We assume that there exists smooth symplectomorphisms, with smooth inverse, 

\begin{figure}
\centering
\includegraphics[scale=0.5]{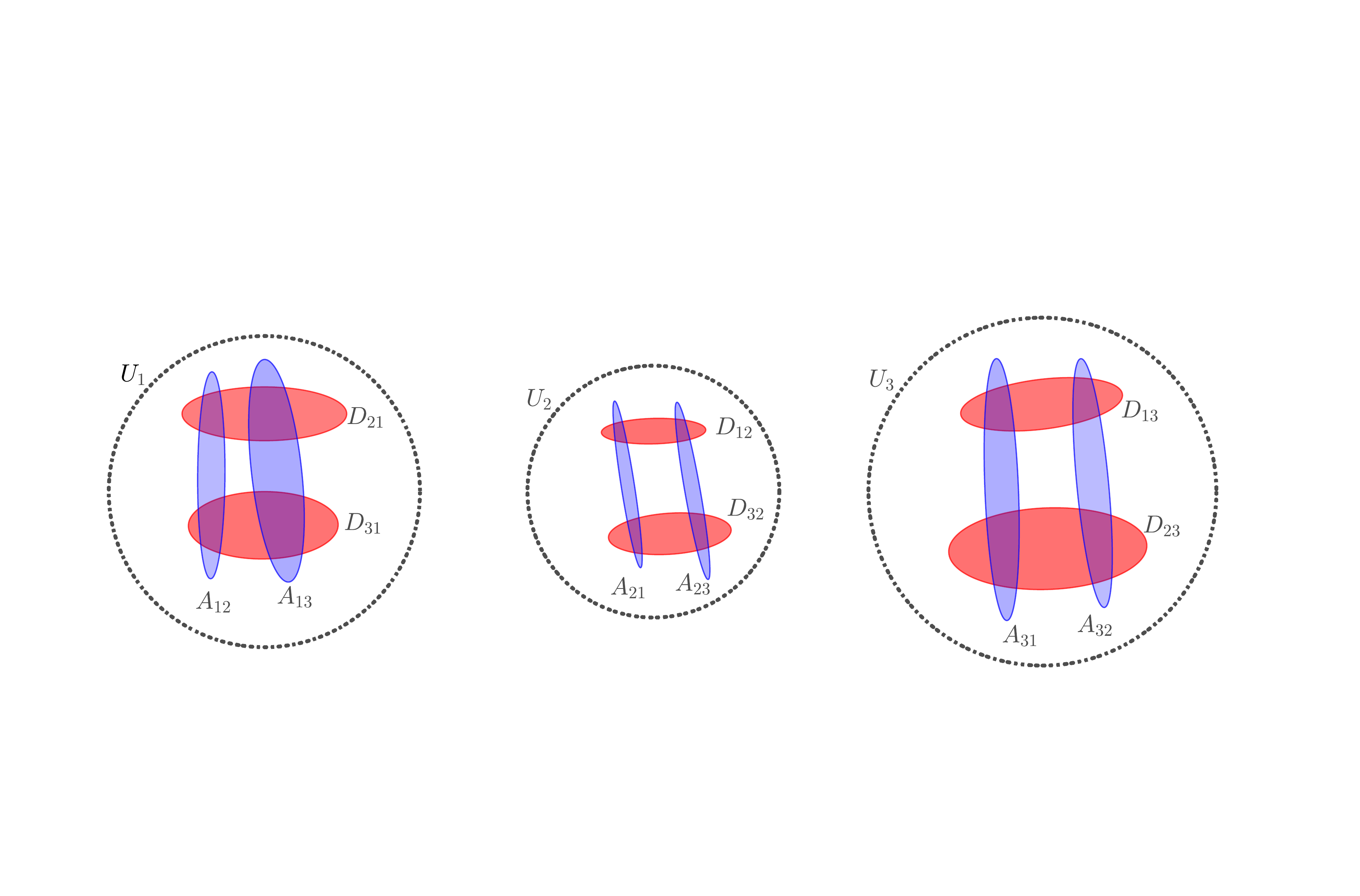}
\caption{A schematic example of open hyperbolic map with $J=3$  in a case where $D_{ii} = \emptyset$ for $i=1,2,3$.}
\label{figure_example_1}
\end{figure}

\begin{equation}
F_{i j } : \widetilde{D}_{i j } \to F_{ij } \left(\widetilde{D}_{i j } \right) = \widetilde{A}_{i j }
\end{equation}
We note $F$ for the global smooth map $F : \widetilde{D} \to \widetilde{A}$ where $\widetilde{A}$ and $\widetilde{D}$ are the full arrival and departure sets, defined as 
$$\widetilde{A} = \bigsqcup_{i=1}^J \bigsqcup_{j=1}^J \widetilde{A}_{i j } \subset \bigsqcup_{i=1}^J U_i $$ 
$$  \widetilde{D} = \bigsqcup_{j=1}^J \bigsqcup_{i=1}^J \widetilde{D}_{i j } \subset \bigsqcup_{j=1}^J U_j $$ 
We define the outgoing (resp. incoming) tail by $\mathcal{T}_+ \coloneqq \{ \rho \in U ; F^{-n}(\rho) \in U , \forall n \in \N \} $ (resp. $\mathcal{T}_- \coloneqq \{ \rho \in U ; F^{n}(\rho) \in U , \forall n \in \N \} $). We assume that they are closed subsets of $U$ and that the \emph{trapped set} 
\begin{equation}\label{trapped_set} 
\mathcal{T} = \mathcal{T}_+ \cap \mathcal{T}_-
\end{equation}
is compact. We also assume that 
\begin{center}
$\mathcal{T}$ is \emph{totally disconnected.}   
\end{center}

\begin{rem}
It is possible that for some values of $i$ and $j$, $\widetilde{D}_{ij} = \emptyset$. For instance, when dealing with the billiard map (see subsection \ref{subsection_billiard_map}), the sets $\widetilde{D}_{ii}$ are all empty. 
\end{rem}

We then make the following dynamical assumption. 

\begin{equation}\label{Hyperbolicity_assumption}
\mathcal{T} \text{ is a hyperbolic set for }F  
\end{equation}
Namely, for every $\rho \in \mathcal{T}$, we assume that there exist stable/unstable tangent spaces $E^{s}(\rho)$ and $E^{u}(\rho)$ such that : 
\begin{itemize}
\item $\dim E^{s}(\rho) = \dim E^{u}(\rho) = 1$
\item $T_\rho U = E^{s}(\rho) \oplus E^{u}(\rho)$
\item there exists $\lambda >0$, $C >0$  such that for every $v \in T_\rho U $ and any $n \in \N$, 
\begin{align}
v \in E^{s}(\rho)  \implies ||d_\rho F^n (v) || \leq C e^{-n \lambda} || v||  \label{hyp1}\\
v \in E^{u}(\rho) \implies ||d_\rho F^{- n} (v) || \leq C e^{-n \lambda} \label{hyp2} || v||
\end{align}where  $|| \cdot ||$ is a fixed Riemannian metric on $U$. 
\end{itemize}
The decomposition of $T_\rho U$ into stable and unstable spaces is assumed to be continuous. It allows to define stable and unstable Jacobians $J^u_n(\rho)$ and $J^s_n(\rho)$ (see Definition \ref{Def_jacobian} for the precise definition). We define the maximal Lyapounov exponent $\lambda_{\max}$ as 

\begin{equation}\label{def_lambda_max}
\lambda_{\max} = \sup_{\rho \in \mathcal{T}} \limsup_{ n \to + \infty} \frac{1}{n} \log J^u_n(\rho)
\end{equation}
We also note 
\begin{equation} \label{unstable_jacobian}
\varphi_u(\rho) = \log J^u_1(\rho)
\end{equation}

\begin{rem}\text{} \\
\begin{itemize}[label=-, nosep]
\item The definition is valid for any Riemannian metric on $U$ and we can of course suppose that is it the standard Euclidean metric. 
\item It is a standard fact (See \cite{Mather}) that there exists a smooth Riemannian metric on $U$, which is said to be adapted to the dynamic, such that (\ref{hyp1}) and (\ref{hyp2}) hold with $C=1$. 
\end{itemize}
\end{rem}

Here ends the description of the classical map. It encompasses the case of the billiard map, useful when dealing with obstacle scattering (see subsection \ref{subsection_billiard_map}). 
We then associate to $F$ \textit{open quantum hyperbolic maps}, which are its quantum counterpart. The definition of such operators is presented in detail in \ref{subsubsection_open_quantum_map}. An open quantum hyperbolic map $T : L^2(Y) \to L^2(Y)$ is an operator-valued matrices $(T_{ij})_{ij}$ where $T_{ij} : L^2(Y_j) \to L^2(Y_i)$ is a Fourier integral operator associated with $F_{ij}$ (see Definition \ref{def_FIO} for a precise definition).

We now come to the statement of the main theorem of this paper. 
\paragraph{Assumptions of Theorem \ref{Theorem_main}. } We consider a family $(M_h(z))_z = (M(z;h))_z$ of open hyperbolic quantum maps, associated with $F$, as defined in Definition \ref{def_FIO}, and depending holomorphically on a parameter $z \in \Omega=\Omega_R=]-R,R[ + i ]-R,R[ $ with $R$ fixed (but in practice, for applications, it can be chosen arbitrarily large). We suppose that there exists $L >0$ and $a \in \cinfc(T^*Y)$ such that $\supp (a)$ is contained in a compact neighborhood $\mathcal{W}$ of $\mathcal{T}$, $\mathcal{W} \subset \widetilde{D}$, $a = 1$ in a neighborhood of $\mathcal{T}$ and uniformly in $\Omega$, 
$$ M_h(z)(1 - \op(a) ) = O(h^{L})_{L^2 \to L^2}$$
Let's note $\alpha_h(z)$ the amplitude of $M_h(z)$ (as defined after definition \ref{def_FIO}). We make the following assumption on $\alpha_h$ : 
there exists a neighborhood $V \subset U$ of $\mathcal{T}$ and a smooth function $t_{ret} : V \to \R_+$\footnote{In the applications, $t_{ret}$ is a return time function.} such that $\inf_V t_{ret} >0$, $\sup_V t_{ret} < + \infty$ and for all $z \in \Omega$ and $\rho \in V$, 
$$ \alpha_h(z)(\rho) = \exp ( - i \im z t_{ret}(\rho)) + O(h^{1-})_{S_{0^+}}$$  
that is, there exists $\chi \in \cinfc(U)$ supported in a larger neighborhood of $\mathcal{T}$ with $\chi \equiv 1$ on $V$,  such that for every $\eta >0$, $\chi \left( \alpha_h(z) - \exp ( - i z t_{ret}) \right)$ is in $h^{1-\eta} S_\eta$ uniformly for $z \in \Omega$. The definition of the symbol class $S_\eta$ and $S_{0^+}$ are recalled in Section \ref{Section_preliminaries_semiclassical}.

\begin{rem}
In particular, the principal part $\alpha_z(\rho) = \exp ( - i z t_{ret}(\rho))$ of $\alpha_h$ is independent of $h$ in $V$. This assumption, which may look strong at first glance, is in fact satisfied in the two applications we consider (see \ref{subsection_billiard_map} and \ref{subsection_potential_scattering}). In fact, the works of \cite{NSZ11} and \cite{NSZ14} allow to work up to $\im z =C \log h$. For such $z$, $\alpha$ is clearly $h$ dependent and lives in the symbol class $S_{0^+}$. 
\end{rem}

We also assume that $M_h(z)$ is uniformly bounded for $z \in \Omega$ and for all $h$ small enough 
$$ ||M_h(z)||_{L^2 \to L^2} \leq C.$$
Let us now define the following quantity : 
\begin{equation}\label{definition_p_beta}
p(\beta) =  -\frac{1}{6 \lambda_{\max}} P(-\varphi_u + 2 \beta t_{ret} )
\end{equation}
where $P$ denotes the topological pressure of $ \varphi : \rho \in \mathcal{T} \mapsto -\varphi_u + 2 \beta t_{ret} $ with respect to the dynamics of $F$ on $\mathcal{T}$. 
It is defined as (see also \ref{def_pressure})
$$ P(\varphi) = \lim_{\epsilon \to 0 } \limsup_{ n \to + \infty} \frac{1}{n}\log  P_0(\varphi, n,\epsilon)$$
where 
$$ P_0(\varphi, n,\epsilon) = \sup \left\{ \sum_{x \in E} \exp \left( \sum_{i=0}^{n-1} \varphi (F^i(x) ) \right)  \; ; \; E \text{ is } (n,\epsilon) \text{ separated} \right\}$$
(a subset $E \subset \mathcal{T}$ is said to be $(n,\epsilon)$ separated if for every $x,y \in E, x\ \neq y$, there exits $0 \leq i \leq n-1$, $d( F^i(x), F^i(y) ) > \epsilon$). The quantity $\sum_{i=0}^{n-1} \varphi (F^i(x) )$ is called a Birkhoff sum. The map $\beta \mapsto p(\beta)$ is a non increasing function of $\beta$ and at $\beta=0$, we have 
$$ p(0) = -\frac{1}{6 \lambda_{max}} P(-\varphi_u) >0. $$

For $\Omega^\prime \subset \Omega$, we note 
$m_M(\Omega^\prime) = \sum_{z \in \Omega^\prime, f_h(z) = 0 } m(z)$ where $m(z)$ stands for the multiplicity of $z$ as a zero of $f_h(z)=\det(1- M_h(z) )$. Note that this determinant is well-defined since the operators $M_h(z)$ are constructed trace-class (see the \ref{subsubsection_open_quantum_map}). In this paper, we prove 

\begin{thm}\label{Theorem_main}

For every $\varepsilon >0$, $\gamma > 0$  and $0<R^\prime <R$, there exist $C = C_{\varepsilon, \gamma,R^\prime}>0$ and $h_0 >0$ such that  
$$ m_M \big( \{ | \re z | < R^\prime ,  \im z \in [-\gamma,0]  \} \big) \leq  C h^{-d_H + \max( p(\gamma+ \varepsilon)- \varepsilon,0) } \; , \; \forall 0 < h \leq h_0$$
where 
$2d_H$ is the Hausdorff dimension of $\mathcal{T}$. 
\end{thm}

\subsubsection{Application in semiclassical scattering by a potential}\label{sub_section_application_potential}
The reduction from an open quantum system to an open quantum hyperbolic map, proved in \cite{NSZ14} for the case of obstacle scattering, is also proved in the case of potential scattering in \cite{NSZ11}. As a consequence, we can prove a bound similar to the one given by Theorem \ref{THM_new} in potential scattering. The following theorem is proved in \ref{subsection_potential_scattering} using Theorem \ref{Theorem_main}.

\begin{thm}\label{Theorem_potential_scattering}
Let $V \in \cinfc(\R^2)$, $E_0>0$ and consider the semiclassical pseudodifferential operator $P_h = -h^2 \Delta + V - E_0$.
Let's note $p(x,\xi) = \xi^2 + V - E_0$ and assume that
$$ dp \neq 0 \text{ on } p^{-1}(0)$$ 
Let's note $H_p$ the Hamiltonian vector field associated with $p$ and $\Phi_t = \exp(tH_p)$ the corresponding Hamiltonian flow. Let's note $K_0$ the trapped set of $p$ at energy $0$ and let's assume that 
\begin{enumerate}
[label= (\roman*),nosep]
\item $\Phi_t$ is hyperbolic on $K_0$ ; 
\item $K_0$ is topologically one dimensional. 
\end{enumerate}
Let $\gamma_{cl}$ be the classical escape rate of the system at energy $0$ and $2d_H +1$ be the Hausdorff dimension of $K_0$. Let $N(R,\gamma ; h)$ be the number of resonances of $ P_h/h$ in $\{ |\re z | <R,   \im z \in [-\gamma ,0 ]  \}$, counted with multiplicity. Then, there exists a non increasing function $\sigma : \R_+ \to \R_+$ satisfying 
\begin{itemize}[nosep]
\item $\sigma(\gamma) > 0$ for $0 \leq \gamma < \gamma_{cl}/2$ ; 
\item $\sigma(\gamma) = 0$ for $\gamma \geq \gamma_{cl}/2 $
\end{itemize}
and such that 
for all $R,\gamma >0$ and for all $\varepsilon>0$ there exists $C_{R,\gamma, \varepsilon}>0$ and $h_0>0$ such that 
$$\forall 0 <h \leq h_0 \, ,\,N(R,\gamma ; h) \leq C_{R,\gamma, \varepsilon} h^{-d_H + \sigma(\gamma) - \varepsilon}.$$ 

\end{thm} 

\begin{rem}
We are interested in resonances of a Schrödinger operator $P_h = -h^2 \Delta + V - E_0$ in a neighborhood of $0$ of size $h$. To keep notations consistent with the spectral parameter $z$ appearing in Theorem \ref{Theorem_main}, we renormalize to study the resonances of $P_h/h$ in a fixed neighborhood of 0. 
\end{rem}

\begin{rem}
The theorem could be extended to a wider class a perturbations of the Laplacian in manifolds with Euclidean ends. We refer the reader to \cite{NSZ11} (Section 2.1) for more general assumptions
\end{rem}

 \subsubsection{Skecth of proof of Theorem \ref{Theorem_main}. }
In \cite{NSZ14}, to prove the Fractal Weyl upper bound, the author modify the monodromy operator $M(z;h)$ and replace it by 
$$M_{tG} (z;h)= e^{-t G} M(z;h) e^{tG}$$
where $G = \op(g)$ with $g$ an escape function in the critical symbol class $\tilde{S}_{1/2}$, constructed such that the Fourier integral operator $M_{tG}$ has a small amplitude outside a neighborhood $\mathcal{T}(h^{1/2- })$ of $\mathcal{T}$, and $t$ is a fixed parameter. Here, to avoid the critical symbol class $\tilde{S}_{1/2}$, we will work in the symbol class $S_\delta$ for some $\delta = 1/2 - \varepsilon$, so that the interesting neighborhood of $\mathcal{T}$ becomes $\mathcal{T}(h^\delta)$, which has a volume comparable to $h^{ \delta(2 - 2 d_H)}$. 

Since the zeros of $z \mapsto \det(1-M_{tG}(z;h))$ coincide (with multiplicity) with the zeros of $\det(1-M(z;h))$, we wish to count the zeros of $\det(1-M_{tG}(z;h))$. 
Jensen's formula and standard spectral inequalities on spectral determinants reduce the estimates on the zeros of $\det(1-M_{tG}(z;h))$ to a control on the Hilbert-Schmidt norm of $M_{tG}(z;h)$. In \cite{NSZ14}, the author show that $M_{tG}$ is close to an operator having a rank comparable to $h^{-d_H}$, which lead them to a bound of the form 
$$||M_{tG}(z;h)||_{HS}^2 \leq ||M_{tG}(z;h)||^2_{L^2 \to L^2} \times \text{ rank }  \leq Ch^{-d_H}$$

To improve the fractal upper bound of \cite{NSZ14} and prove Theorem \ref{Theorem_main}, we start with the simple observation that the zeros of $\det(1 - M_{tG}(z;h) ) $ are among the zeros of $\det(1- M_{tG}^n(z;h) ) $, for any $n \in \N^*$. We use this fact with an exponent $n=n(h)$ depending on $h$ : $n(h) \sim \nu \log 1/h$ for some $\nu >0$. \textit{A priori}, when $n(h)$ grows logarithmically, $M_{tG}^n$ becomes "nasty" (i.e. no more a Fourier Integral Operator in a suitable class; recall that essentially $g \in S_\delta$), and in particular, it becomes impossible to use Egorov's theorem as soon as $n \geq \varepsilon \log 1/h$, for some small $\varepsilon$ (essentially $\frac{1/2 - \delta}{\lambda_{\max}}$).  However, the action of the operator $M_{tG}(z)$ on coherent states $\varphi_{\rho}$ will remain under control for a sufficiently long logarithmic time. We will be able to obtain good estimates up to 
$$ n(h) \sim \frac{1}{6 \lambda_{\max}} \log 1/h$$ 
To use these estimates, we use the representation of the trace in terms of coherent states 
: 
\begin{equation}\label{integral}
|| M_{tG}(z;h)^n||_{HS}^2 = \frac{1}{2\pi h} \int_{U} ||M_{tG}(z;h)^n \varphi_\rho ||^2 d\rho 
\end{equation}
The main new ingredient in the present paper will consist in controlling precisely the evolved states $M_{tG}^n \varphi_{\rho}$ for such logarithmic times. The behavior of this state will depend on the initial point $\rho$ (see Proposition \ref{Prop_Key} for a precise and rigorous statement)
\begin{itemize}
\item If $\rho$ is not in an $h^\delta$ neighborhood of $\mathcal{T}$, we will show that for any $L>0$, we can find $t=t(L)$ such that the norm of $M_{tG}^n \varphi_{\rho}$ is $O(h^L)$. As a consequence, the mass in the integral in (\ref{integral}) is essentially contained in an $h^\delta$ neighborhood of $\mathcal{T}$. In particular, by simply estimating $||M_{tG}(z;h)^n||^2 \leq C$ in a $h^\delta$ neighborhood of $\mathcal{T}$, we find that 
$$  || M_{tG}(z;h)^n||_{HS}^2 \leq C h^{-1} h^{\delta(2- 2d_H)} \leq C h^{-d_H + O(\varepsilon)}$$
This gives the previous upper bound of \cite{NSZ14}
\item For states sufficiently close to $\mathcal{T}$, $\varphi_\rho$ will evolve into a squeezed coherent state, aligned along the unstable leaves of $\mathcal{T}_+$. This phenomenon can be understood as a delocalization of the coherent state. In the unstable direction, the components of this squeezed state far from $\mathcal{T}_-$ (that it at distance bigger that $h^{\delta}$) will experience a strong damping due to the escape function. For such state, we are able to control the squared $L^2$-norm $$w_z(\rho) \coloneqq || M_{tG}(z;h)^n \varphi_{\rho}||^2$$ by (again, see Proposition \ref{Prop_Key} for the rigorous statement)
\begin{equation}\label{crucial_bound}
 w_z(\rho) \leq C \left( \prod_{i=0}^{n-1} \alpha_{z}(F^i(\rho)) \right)^2 J^u_n(\rho)^{d_H -1}
\end{equation} 
where $\alpha_z(\rho) = \exp( - (\im z) t_{ret}(\rho) )$. 
This is the crucial estimate of this paper. 
\end{itemize}
Plugging this bound into the integral in (\ref{integral}), we are able to prove the following upper bound  $$ ||M_{tG}(z;h)^n||_{HS}^2 \leq C_\varepsilon h^{-d_H + \tilde{\sigma}(z) -O( \varepsilon)} \; ; \; \tilde{\sigma}(z) = -\frac{1}{6 \lambda_{\max}} P(- \varphi_u + 2 \log \alpha_z )$$ (see Proposition \ref{Prop_Key_2}). The link between the pressure and (\ref{crucial_bound}) appears when one writes 
$$w_z(\rho) \leq CJ^u_n(\rho)^{d_H} \exp \left( \sum_{i=0}^{n-1} ( -2 \im z t_{ret} - \varphi_u) \circ F^i (\rho) \right)$$
The factor $J^u_n(\rho)^{d_H}$ disappears after integrating (see the proof of Proposition \ref{Prop_Key_2}). 
It finally gives Theorem \ref{Theorem_main} (see Section \ref{Section_proof_main_theorem}). 

The crucial estimate (\ref{crucial_bound}) is the main novelty of this paper. It relies on propagation of coherent states and a subtle interaction of the evolved state with the escape function (see Figure \ref{figure_evolution_state}). The proof of (\ref{crucial_bound}) relies on the following ideas : 
\begin{itemize}
\item  The term \begin{equation}\label{definition_of_pi_alpha}
\pi_{\alpha,n}(\rho) \coloneqq  \prod_{i=0}^{n-1} \alpha_{z}(F^i(\rho)) 
\end{equation} comes from the repeated action of $M(z)$ on $\varphi_\rho$. 
\item The initial state $\varphi_\rho$ is a wavepacket of size $h^{1/2}$. $M(z)^n \varphi_\rho$ is a squeezed coherent state, microlocalized near $F^n(\rho)$. This is due to the fact that we will work with $n=n(h) \leq  \frac{1-\eta}{6 \lambda_{\max}} \log(1/h)$ for some $\eta>0$. Nevertheless, it is no more microlocalized in a $h^{1/2}$ neighborhood of this point. It will be more convenient to write it as a Lagrangian state, associated with a local unstable leaf $W_u(\rho_n)$, for some $\rho_n \in \mathcal{T}$ close to $F^n(\rho)$ : if $\psi_u(x)$ is a generating function for $W_u(\rho_n)$, that is, if we can write $W_u(\rho_n)= \{ (x, \psi_u^\prime(x)  \}$,  the state will be written
$$ a_h(x) e^{\frac{i}{h} \psi_u(x)}$$ 
The size of this Lagrangian state along the unstable manifold is controlled by the local Jacobian near $\rho$ and is $O(h^{1/2} J^u_n(\rho))$ : we will see that 
$$ |x| \gg h^{1/2} J^u_n(\rho) \implies a_h(x) = \hinf$$ 
\item Finally, we need to understand the interaction of the escape function with this evolved state. The action of the escape function damps the part of the state at distance larger that $h^{\delta}$ of $\mathcal{T}$. Since such a state is very close to an unstable manifold, the only relevant damping on this state comes from the components at distance larger that $h^{\delta}$ from $\mathcal{T}_{-}$. Roughly speaking, to obtain the bound we want, we prove that if  $d( (x, \psi_u^\prime(x)) , \mathcal{T}_-) \leq h^\delta$, then 
\begin{equation}\label{crucial_1}
 a_h(x) \leq C \pi_{\alpha,n}(\rho) \left( J_n^u(\rho)\right)^{-1/2} h^{-1/4}
\end{equation}
and we prove that we can neglect the remaining points $x$ such that $d( (x, \psi_u^\prime(x)) , \mathcal{T}_-) \geq h^\delta$ (see Proposition \ref{Prop_crucial_estimates}). It gives 
$$ ||M_{tG}(z)^n\varphi_\rho ||^2_{L^2} \leq C\pi_{\alpha,n}(\rho)^2  J_n^u(\rho)^{-1} h^{-1/2} \text{Len}(X_-(\rho,\rho_n) )$$
where $$ X_-(\rho,\rho_n) = \{ x \in \R, |x| \leq C J^u_n(\rho)h^{1/2} , d( (x, \psi_u^\prime(x)) , \mathcal{T}_-) \leq h^\delta \} $$ 
\item It remains to control the length of $X_-(\rho,\rho_n)$. We use the fact that $\mathcal{T}_- \cap W_u(\rho_n)$ has box dimension $d_H$. In fact, we are interested by a piece of $W_u(\rho_n)$ of size $h^{1/2}J^u_n(\rho)$ and we show that such a piece can be covered by $N_h$ balls of radius $h^{\delta}$ with (see Lemma \ref{Lemma_number_of_intervals}) $$N_h \leq C \left( \frac{h^\delta}{h^{1/2} J^u_n(\rho)}\right)^{-d_H}$$
so that 
$$ \text{Len}(X_-(\rho, \rho_n)) \leq C N_h h^\delta \leq C h^{1/2} J^u_n(\rho)^{d_H} h^{-O(\varepsilon)}$$
\item Putting the pieces together, we obtain (\ref{crucial_bound}). 
\end{itemize}

\begin{figure}
  \centering
     \begin{subfigure}[b]{0.4\textwidth}
         \centering
         \includegraphics[width=8cm]{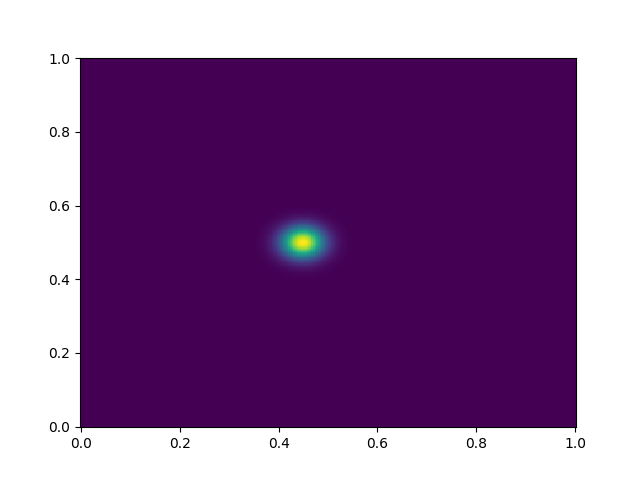}
         \caption{A coherent state of size $h^{1/2}$ ... \newline}
         \label{sub_1}
     \end{subfigure}
     \hfill
     \begin{subfigure}[b]{0.4\textwidth}
         \centering
         \includegraphics[width=8cm]{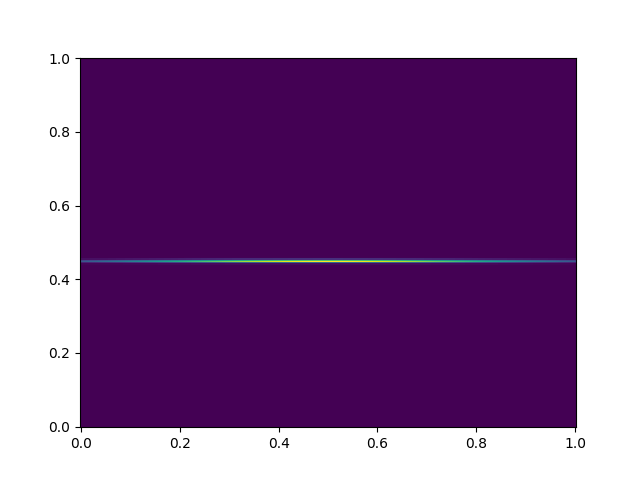}
         \caption{...evolved into a squeezed coherent state. }
         \label{sub_2}
     \end{subfigure}
      \begin{subfigure}[b]{0.4\textwidth}
         \centering
         \includegraphics[width=8cm]{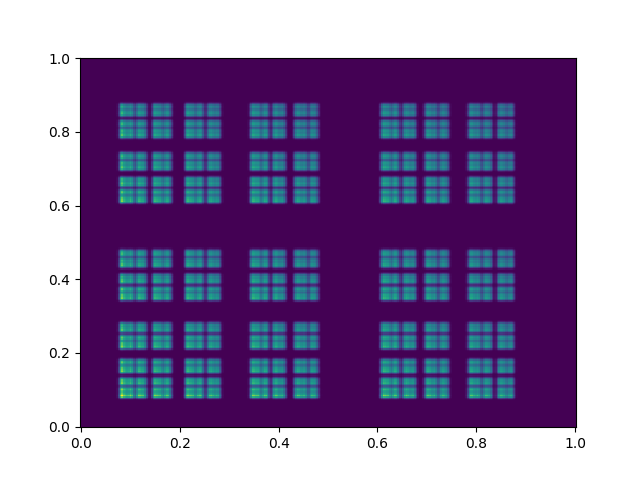}
         \caption{The escape function damps the region far from the trapped set...}
         \label{sub_3}
     \end{subfigure}
     \hfill
     \begin{subfigure}[b]{0.4\textwidth}
         \centering
         \includegraphics[width=8cm]{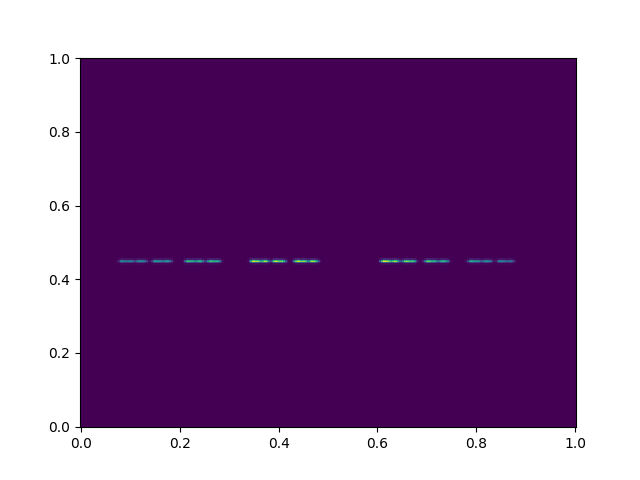}
         \caption{ ... and is responsible of a damping for the evolved coherent state.  }
         \label{sub_4}
     \end{subfigure}    
     \caption{We show the evolution in phase space of a coherent state in an open hyperbolic system, associated with an open baker's map. The color is related to the intensity of the Wigner distribution of the state. The damping due to the escape function is shown in (\ref{sub_3}). The initial coherent state is shown in (\ref{sub_1}), and the evolved state, without damping in (\ref{sub_2}). When we apply the damping, the evolved state loses part of its mass (\ref{sub_4})}. 
     \label{figure_evolution_state}
\end{figure}

\paragraph{Plan of the paper. } The paper is organized as follows : 
\begin{itemize}
\item We start with preliminaries in Section \ref{Section_preliminaries_semiclassical} and Section \ref{Section_preliminaires_dyna}. Section \ref{Section_preliminaries_semiclassical} is devoted to semiclassical results concerning pseudodifferential operators, Fourier integral operators, metaplectic operators and coherent states. Section \ref{Section_preliminaires_dyna} focuses on properties of hyperbolic dynamical systems. 
\item Section \ref{Section_proof_main_theorem} reduces the proof of Theorem \ref{Theorem_main} to the key proposition \ref{Prop_Key}, concerning the behavior of the propagated coherent states. 
\item Section \ref{Section_proof_key_prop} is devoted to the proof of Proposition \ref{Prop_Key}. 
\end{itemize}

\subparagraph{Notations.} 
Throughout the paper, we will use the same constant $C$ at different places, with different meaning. However, it will always have the same dependence on the dynamical system and the family of operators $M_h(z)$ we work with. That is, we write $f \leq Cg$ instead of : there exists $C>0$ depending on $F$ and $M_h(z)$ such that $f \leq C g$. At some point, we will fix a partition of unity of $U$, associated with local charts, depending on parameters $\varepsilon$ and $\varepsilon_0$. The constants $C$ will also depend on these objects. If the constant $C$ has other dependencies, we will make it precise it or write it in subscript if necessary.

  Finally, we write 
$ f \sim g$ to mean $C^{-1} f \leq g \leq C f$. 

\paragraph{Acknowledgment. } The author would like to thank Frédéric Naud and Stéphane Nonnenmacher to let him resume this project on an improved fractal upper bound, and Stéphane Nonnenmacher for useful discussions, a fruitful help and a careful reading of a first version of this article. 

\section{Semiclassical preliminaries} \label{Section_preliminaries_semiclassical}

\subsection{Pseudodifferential operators and Weyl quantization}
We recall some basic notions and properties of the Weyl quantization on $\R^n$. We refer the reader to \cite{ZW} for the proofs of the statements and further considerations on semiclassical analysis and quantizations.

\paragraph{Definitions.} We start by defining classes of $h$-dependent symbols. 
In the following definitions, $m$ is a positive functions defined on $T^*\R^n$ of the form $\langle \rho \rangle^N$, for some $\N \in \Z$, where $\langle \rho \rangle = \sqrt{1+ |\rho|^2}$ and $\rho = (x,\xi)$ is a point in phase space $T^* \R^n = \R^{2n}$. $m$ is called an order function (in the sense of \cite{ZW}, 4.4.1)

\begin{defi}
Let $0 \leq \delta < \frac{1}{2}$. We say that an $h$-dependent family $a \coloneqq \left( a (\cdot ; h) \right)_{0 < h \leqslant 1} \in \cinf(\R^{2n})$ is in the class $S_\delta(m)$ (and simply $S_\delta$ if $m=1$) if for every $\alpha \in \N^{2n}$, there exists $C_\alpha >0$ such that : 
$$\forall 0< h \leq 1 \, , \sup_{\rho \in \R^{2n}} | \partial^\alpha a (\rho ; h) | \leq C_\alpha h^{-\delta |\alpha|} m(\rho)$$
\end{defi}
We will use the notation $S_{0^+}(m) = \bigcap_{\delta >0} S_\delta(m)$. \\
We write $a =  O(h^N)_{S_\delta(m)}$ to mean that for every $\alpha \in \N^{2n}$, there exists $C_{\alpha,N}$ such that 
$$\forall 0< h \leq 1 \, , \sup_{\rho \in \R^{2n}}  h^{\delta |\alpha|}\partial^\alpha a (\rho ; h) | \leq  C_{\alpha,N} h^N m(\rho) $$ 
 If $a=O(h^N)_{S_\delta(m)}$ for all $N \in N$, we'll write $a= O(h^\infty)_{S_\delta(m)}$. 
For a given symbol $a \in S_\delta$, we say that $a$ has a compact essential support if there exists a compact set $K$ such that\footnote{Here $\mathcal{S}$ denotes the Schwartz space and the notation $\hinf_{\mathcal{S}(\R^{2n}) }$ means that every semi-norm is $\hinf$.} : 
$$ \forall \chi \in \cinfc(\R^{2n}), \supp \chi \cap K = \emptyset \implies \chi a =\hinf_{\mathcal{S}(\R^{2n}) }$$
We say that $a$ belongs to the class $S_\delta^{comp}$ and its essential support is then the intersection of all such compact $K$'s. We denote it $ \text{ess} \supp a \subset K$. In particular, the class $S_{\delta}^{comp}$ contains all the symbols supported in a $h$-independent compact set and these symbols correspond, modulo $\hinf_{\mathcal{S}(T^*\R) }$, to all symbols of $S_\delta^{comp}$. For this reason, we will adopt the following notation: if $V \Subset \R^{2n}$ is an open set, we say that $a \in S_\delta^{comp}(V)$ if $a \in S_{\delta}^{comp}(\R^{2n})$ and $\text{ess} \supp a \Subset V$.

For a symbol $a=a(\cdot ; h) \in S_\delta(m)$, we'll quantize it using Weyl's quantization procedure. It is informally written as : 
$$(\op(a)u)(x) = \frac{1}{(2\pi h)^n }\int_{\R^{2n}} a\left( \frac{x+y}{2}, \xi  ;h  \right)u(y) e^{i \frac{(x-y)\cdot \xi}{h}}dy d \xi$$

We will note $\Psi_\delta(m)$ the corresponding classes of pseudodifferential operators. By definition, the wavefront set of $A = \op(a)$ is $\WF(A) = \text{ess} \supp a $. 
\vspace*{0.5cm}

We say that a family $u=u(h) \in \mathcal{D}^\prime(\R^n)$ is $h$-tempered if for every $\chi \in \cinfc(\R^n)$, there exist $C >0$ and $N \in \N$ such that $|| \chi u ||_{H_n^{-N}} \leq Ch^{-N}$. For a $h$-tempered family $u$, we say that a point $\rho \in T^*\R^n$ does \emph{not} belong to the wavefront set of $u$ if there exists $a \in S^{comp}$ such that $a(\rho) \neq 0$ and $\op(a)u = \hinf_{\cinf}$. We note $\WF(u)$ the wavefront set of $u$. 

We say that a family of operators $B=B(h) : \cinfc(\R^{n_2}) \to \mathcal{D}^\prime(\R^{n_1})$ is $h$-tempered  if its Schwartz kernel $\mathcal{K}_B \in \mathcal{D}^\prime(\R^{n_1} \times \R^{n_2})$ is $h$-tempered. We define the twisted wavefront set of $B$ as 
$$ \WF^\prime(B) = \{ ( x,\xi, y , -\eta) \in T^*\R^{n_1}  \times T^* \R^{n_2} ,  (x,\xi, y , \eta)  \in \WF(\mathcal{K}_B) \}$$

\paragraph{Standard properties. }
Let us now recall standard results in semiclassical analysis concerning the $L^2$-boundedness of pseudodifferential operator and their composition. 
We'll use the following version of Calderon-Vaillancourt Theorem (\cite{ZW}, Theorem 4.23). 

\begin{prop}
There exists $ C_n>0$ such that the following holds. For every $0 \leq \delta < \frac{1}{2}$, and $a \in S_\delta$, $\op(a)$ is a bounded operator on $L^2$ and 
$$ || \op(a) ||_{L^2 (\R^n) \to L^2 (\R^n) } \leq C_n \sum_{ |\alpha| \leq 8n } h^{|\alpha|/2} || \partial^\alpha a ||_{L^\infty} $$ 
\end{prop}

As a consequence of the sharp Gärding inequality (see \cite{ZW}, Theorem 4.32), we also have the precise estimate of $L^2$ norms of pseudodifferential operator, 

\begin{prop}\label{Garding}
Assume that $a \in S_\delta(\R^{2n})$. Then, there exists $C_a$ depending on a finite number of semi-norms of $a$ such that : 
$$ || \op(a) ||_{L^2 \to L^2} \leq ||a||_\infty + C_ah^{\frac{1}{2} - \delta}$$
\end{prop}

We recall that the Weyl quantizations of real symbols are self-adjoint in $L^2$. If $m_1$ and $m_2$ are two order functions of the form $\langle \rho \rangle^{N_i}$, $i=1,2$, the composition of two pseudodifferential operators in $\Psi_\delta(m_1)$ and $\Psi_\delta(m_2)$ is a pseudodifferential operator in the class $\Psi_\delta(m_1 m_2)$. More precisely (see \cite{ZW}, Theorem 4.11 and 4.18), if $(a , b) \in S_\delta(m_1) \times S_\delta(m_2)$, $\op(a) \circ \op(b)$ is given by $\op(a \# b)$, where $a \# b$ is the Moyal product of $a$ and $b$. It is given by 
$$ a \# b (\rho) = e^{i h A(D) } (a \otimes b)|_{ \rho= \rho_1 = \rho_2} $$ 
where $a \otimes b (\rho_1, \rho_2) = a(\rho_1) b(\rho_2)$, $e^{i h A(D) }$ is a Fourier multiplier acting on functions on $\R^{4n}$ and, writing $\rho_i = (x_i, \xi_i)$,  $$A(D) = \frac{1}{2} \left( D_{\xi_1} \circ D_{x_2} - D_{x_1} \circ D_{\xi_2} \right) $$
We can estimate the Moyal product by a quadratic stationary phase and get the following expansion which holds in $S_\delta(m_1 m_2)$ for all $N \in \N$,  

\begin{equation*}
 a \# b (\rho) = \sum_{k=0}^{N-1} \frac{i^k h^k}{k!} A(D)^{k} (a \otimes b)|_{\rho= \rho_1= \rho_2} + r_N
\end{equation*}
where for all $\alpha \in \N^{2n}$, there exists $C_\alpha$, independent of $a$ and $b$, such that 
\begin{equation*}
||\partial^\alpha r_N ||_\infty \leq C_\alpha h^N ||a \otimes b ||_{ C^{2N+ 4n+1+ |\alpha| } }
\end{equation*}

\paragraph{Weighted Sobolev spaces. }
We can also define the weighted Sobolev spaces $H_h(m)$. In the case $m = \langle \rho \rangle^N$, we have 
$$ H_h\left( \langle \rho \rangle^N\right) = \op\left(  \langle \rho \rangle^{-N}\right) \big( L^2(\R^n)\big)  \subset \mathcal{S}^\prime (\R^n)$$ 
When $N \geq 0$, $H_h\left( \langle \rho \rangle^{N}\right)$ coincides with the space of functions $u \in \mathcal{S}^\prime(\R^n)$ such that 
$$\forall \alpha , \beta \in \N^n \text{ with } \alpha + \beta \leq N\; , \; x^\alpha (h\partial^\beta) u \in L^2(\R^n) $$ and we have the following equivalence of norms : 
$$ ||u||_{ H_h\left( \langle \rho \rangle^{N}\right)}^2 \sim \sup_{ |\alpha| + |\beta| \leq N }|| x^\alpha (h\partial^\beta) u||_{L^2}^2$$
 As a consequence of Calderon-Vaillancourt theorem,  we have for symbols $ a \in S_\delta(m)$ : 

\begin{prop}
Let $N \in \Z$. There exists $M \in \N$ and $C >0$ such that the following holds : 
For all $a \in S_\delta \left( \langle \rho \rangle^N\right)$,  $\op(a) : H_h\left( \langle \rho \rangle^{N}\right) \to L^2$ is uniformly bounded and  
$$ || \op(a) ||_{H_h\left( \langle \rho \rangle^{N}\right) \to L^2} \leq C \sup_{ |\alpha| \leq M} h^{ |\alpha|/2} || \langle\rho \rangle^{-N} \partial^\alpha a ||_\infty$$

\end{prop}

\subsection{Fourier Integral Operators}
We now review some aspects of the theory of Fourier integral operators. We follow \cite{ZW}, Chapter 11 and \cite{NSZ14}. We refer the reader to \cite{GuSt} for further details on Lagrangian distributions and Fourier integral operators. We also introduce the material needed to understand the definition \ref{def_FIO} of open hyperbolic quantum maps.  We also provide a quantitative version of Egorov's theorem. 

\subsubsection{Local symplectomorphisms and their quantization}
We momentarily work in dimension $n$. 
Let us note $\mathcal{K}$ the set of symplectomorphisms $\kappa : T^*\R^n \to T^* \R^n$ such that the following holds : there exist continuous and piecewise smooth (in $t$) families of smooth functions $(\kappa_t)_{t \in [0,1]} $, $(q_t)_{t \in [0,1]}$ such that : \begin{itemize}[nosep]
\item $\forall t \in [0,1]$, $\kappa_t : T^*\R^n \to T^*\R^n$ is a symplectomorphism ; 
\item $\kappa_0 = \Id_{T^* \R^n} , \kappa_1 = \kappa$ ; 
\item $\forall t \in [0,1], \kappa_t(0)=0 $ ; 
\item there exists $K \Subset T^*\R^n$ compact such that $\forall t \in [0,1], q_t : T^*\R^n \to \R$ and $ \supp q_t \subset K$ ; 
\item $\frac{d}{dt} \kappa_t = \left( \kappa_t \right)^* H_{q_t}$
\end{itemize}
 If $\kappa \in \mathcal{K}$, we note $G_\kappa = Gr^\prime(\kappa) = \{ (x, \xi, y , - \eta), (x,\xi) = \kappa(y,\eta) \}$ the twisted graph of $\kappa$, which is Lagrangian in $T^* \R^n$. 
We recall  \cite{ZW}, Lemma 11.4, which asserts that local symplectomorphisms fixing the origin can be seen as elements of $\mathcal{K}$, as soon as we have some geometric freedom. 

\begin{lem}\label{lemma_local_symp}
Let $U_0, U_1$ be open and precompact subsets of $T^*\R^n$. Assume that $\kappa : U_0 \to U_1$ is a local symplectomorphism fixing 0 and which extends to $V_0 \Supset U_0$ an open star-shaped neighborhood of 0. Then, there exists $\tilde{\kappa} \in \mathcal{K}$ such that $\tilde{\kappa}|_{U_0} = \kappa$. 
\end{lem}

 If $\kappa \in \mathcal{K}$ and  if $(q_t)$ denotes the family of smooth functions associated with $\kappa$ in its definition, we note $Q(t) = \op(q_t)$. It is a continuous and piecewise smooth family of operators. Then the Cauchy problem
 
 \begin{equation}\label{Cauchy_pb_Egorov}
\left\{ \begin{array}{c}
hD_t U(t) + U(t)Q(t) = 0 \\
U(0)=\Id
\end{array}
\right. 
\end{equation}
is globally well-posed.

Following \cite{NSZ14}, Definition 3.9, we adopt the definition (with$ G_\kappa= Gr^\prime(\kappa)$): 

\begin{defi}
Fix $\delta \in [0,1/2)$. We say that $T \in I_\delta(\R^n \times \R^n; G_\kappa )$ if there exist $a \in S_\delta(T^*\R^n)$ and a path $(\kappa_t)$ from $\Id$ to $\kappa$ satisfying the above assumptions such that $T = \op(a)U(1)$, where $t \mapsto U(t)$ is the solution of the Cauchy problem (\ref{Cauchy_pb_Egorov}). 

The class $I_{0^+}(\R^n \times \R^n, G_\kappa)$ is by definition $\bigcap_{\delta >0} I_\delta(\R^n \times \R^n, G_\kappa)$. 
\end{defi}

It is a standard result, known as Egorov's theorem (see \cite{ZW}, Theorem 11.1) that if $U(t)$ solves the Cauchy problem (\ref{Cauchy_pb_Egorov}) and if $b_0 \in S_\delta$, then $U(1)^{-1} \op(b_0) U(1)$ is a pseudodifferential operator in $\Psi_\delta$ and if $b_1= b_0 \circ \kappa$, then 
$U(t)^{-1} \op(b_0) U(t) - \op(b_1) \in h^{1- 2 \delta} \Psi_\delta$.

\begin{rem}\text{}
Applying Egorov's theorem and Beal's theorem, it is possible to show that if $(\kappa_t)$ is a closed path from $\Id$ to $\Id$, and $U(t)$ solves (\ref{Cauchy_pb_Egorov}), then $U(1) \in \Psi_0(\R^n)$. In other words, $I_\delta(\R \times \R , \Gr^\prime(\Id) ) \subset \Psi_\delta(\R^n)$. But the other inclusion is trivial. Hence, this in an equality : 
$$I_\delta(\R^n \times \R^n , \Gr^\prime(\Id) ) = \Psi_\delta(\R^n)$$
The notation $I(\R^n\times \R^n, G_\kappa)$ comes from the fact that the Schwartz kernels of such operators are Lagrangian distributions associated with $G_\kappa$, and in particular have wavefront sets included in $C$. As a consequence, if $T \in I_\delta(\R^n \times \R^n, G_\kappa)$, $\WF^\prime(T) \subset \Gr(\kappa)$. 
\end{rem}

We also recall that the composition of two Fourier integral operators is still a Fourier integral operator : if $\kappa_1, \kappa_2 \in \mathcal{K}$ and $U_1 \in I_\delta(\R^n \times \R^n, \Gr^\prime(\kappa_1) ) , U_2 \in I_\delta(\R^n \times \R^n, \Gr^\prime(\kappa_1) )$, then, 
$$ U_1 \circ U_2 \in  I_\delta(\R^n \times \R^n, \Gr^\prime(\kappa_1 \circ \kappa_2) )$$

\subsubsection{An important example}
Let us focus on a particular case of canonical transformations. 
Suppose that $\kappa: T^*\R^n \to T^*\R^n$ is a canonical transformation such that 
$$ (x,\xi, y, \eta) \in \Gr(\kappa) \mapsto (x, \eta) $$
is a local diffeomorphism near $(x_0, \xi_0 , y_0, \eta_0 )$. Then, there exists a phase function $\psi \in \cinf(\R^n \times \R^n)$, $\Omega_x, \Omega_\eta$ open sets of $\R^n$ and $\Omega$ a neighborhood of $(x_0, \xi_0 , y_0, -\eta_0 ) $,  such that 
$$ \Gr^\prime(\kappa)  \cap \Omega= \{ (x, \partial_x \psi(x, \eta) , \partial_\eta\psi(x,\eta), -\eta) , x \in \Omega_x, \eta \in \Omega_\eta \} $$ 
One says that $\psi$ generates $\Gr^\prime(\kappa)$. 
Suppose that that $a \in S_\delta^{comp} (\Omega_x \times \Omega_\eta)$. Then, the following operator $T$ is an element of $I_\delta ( \R^n \times \R^n , \Gr^\prime(\kappa) ) $ : 

\begin{equation}  \label{special_form_FIO} Tu(x) = \frac{1}{(2\pi h)^n} \int_{\R^{2n}} e^{ \frac{i}{h} (\psi(x,\eta) - y \cdot \eta) } a(x,\eta) u(y) dy d\eta 
\end{equation} 
and if $T^*T = \Id$ microlocally near $(y_0,\eta_0)$ - that is if $(y_0, \eta_0) \not \in \WF(T^* T- \Id)$-  then $|a(x, \eta)|^2 = | \det D^2_{ x \eta}  \psi (x, \eta) | + O(h^{1-2 \delta})_{S_\delta}$ near $(x_0, \xi_0, y_0, \eta_0)$. The converse statement holds : microlocally near $(x_0,\xi_0, y_0, \eta_0)$ and modulo a smoothing operator which is $\hinf$, the elements of $I_\delta(\R^n \times \R^n, \Gr^\prime(\kappa) ) $ can be written under this form.

\subsubsection{Open quantum hyperbolic maps}
\label{subsubsection_open_quantum_map}
The aim of this part is to provide the precise definition of open quantum hyperbolic maps in \ref{def_FIO}. 
Let us consider an open hyperbolic map $F$, as described by the formalism in \ref{hyperbolic_open_quantum_map}. We recall that this formalism relies on :
\begin{itemize}
\item open intervals $Y_1, \dots, Y_J$ of $\R$ and $Y = \bigsqcup_{j=1}^J Y_j \subset \bigsqcup_{j=1}^J \R$ ; 
\item $ U = \bigsqcup_{j=1}^J U_j \subset \bigsqcup_{j=1}^J T^*\R^d $ where $ U_j \Subset T^*Y_j $  are open sets; 
\item For $j=1, \dots, J$, open disjoint subsets $\widetilde{D}_{i j } \Subset U_j$, $1 \leq i \leq J$, \emph{the departure sets}, and for $i= 1, \dots, J$ open disjoint subsets $\widetilde{A}_{i j } \Subset U_i$, $1 \leq j \leq J$, \emph{the arrival sets} ; 
\item Smooth symplectomorphisms
$F_{i j } : \widetilde{D}_{i j } \to F_{ij } \left(\widetilde{D}_{i j } \right) = \widetilde{A}_{i j }$
\end{itemize}
Then, $F$ is the global smooth map $F : \widetilde{D} \to \widetilde{A}$ where $\widetilde{A}$ and $\widetilde{D}$ are the full arrival and departure sets, defined as 
$$\widetilde{A} = \bigsqcup_{i=1}^J \bigcup_{j=1}^J \widetilde{A}_{i j } \subset \bigsqcup_{i=1}^J U_i $$ 
$$  \widetilde{D} = \bigsqcup_{j=1}^J \bigcup_{i=1}^J \widetilde{D}_{i j } \subset \bigsqcup_{j=1}^J U_j $$ 
Finally, we recall that we note $\mathcal{T} \subset U$ the trapped set of $F$. 

Our aim is to define open quantum maps associated with $F$. We fix a compact set $W \subset \widetilde{A}$ containing some neighborhood of $\mathcal{T}$. Our definition will depend on $W$. 
Following \cite{NSZ14} (Section 3.4.2), we now focus on the definition of the elements of $I_\delta(Y \times Y; \Gr(F)^\prime)$. An element $T \in I_\delta(Y \times Y; \Gr(F)^\prime)$ is a matrix of operators  
$$T = (T_{ij})_{1 \leq i,j\leq J } : \bigoplus_{j=1}^J L^2(Y_j) \to \bigoplus_{i=1}^J L^2(Y_i)$$
Each $T_{ij}$ is an element of $I_\delta(Y_i \times Y_j , \Gr(F_{ij})^\prime )$. Let's now describe the recipe to construct elements of $I_\delta(Y_i \times Y_j, \Gr(F_{ij})^\prime )$.

We fix $i,j \in \{1, \dots , J \}$. 
\begin{itemize}
\item Fix some small $\varepsilon >0$ and two open covers of $U_j$, $U_j\subset \bigcup_{l=1}^L \Omega_l$, $\Omega_l \Subset \widetilde{\Omega}_l$, with $\widetilde{\Omega}_l$ star-shaped and having diameter smaller than $\varepsilon$. We note $\mathcal{L}$ the sets of indices $l$ such that $\widetilde{\Omega}_l \subset \widetilde{D}_{ij} \subset U_j$ and we require (this is possible if $\varepsilon$ is small enough)  
$$ F^{-1}(W) \cap U_j \subset \bigcup_{l \in \mathcal{L}} \Omega_l $$

\item Introduce a smooth partition of unity associated with the cover $(\Omega_l)$, $(\chi_l)_{1 \leq l \leq L} \in \cinfc(\Omega_l, [0,1])$, $\supp \chi_l \subset \Omega_l$, $\sum_l \chi_l = 1$ in a neighborhood of $\overline{U_j}$.
\item  For each $l \in \mathcal{L}$, we denote $F_l$ the restriction to $\widetilde{\Omega}_l$ of $F_{ij}$. By Lemma \ref{lemma_local_symp}, there exists $\kappa_l \in \mathcal{K}$ which coincides with $F_l$ on $\Omega_l$. 
\item  We consider $T_l=\op(\alpha_i) U_l(1)$ where $U_l(t)$ is the solution of the Cauchy problem (\ref{Cauchy_pb_Egorov}) associated with $\kappa_l$ and $\alpha_i \in S_\delta^{comp}(T^*\R)$. 
\item We set 

\begin{equation}\label{def_FIO_global}T^\R = \sum_{l \in \mathcal{L}} T_l \op(\chi_l) : L^2(\R) \to L^2(\R)
\end{equation} 
$T^\R$ is a globally defined Fourier integral operator. We will note $T^\R \in I_\delta(\R \times \R, \Gr(F_{ij})^\prime)$. Its wavefront set is included in $\widetilde{A}_{ij} \times \widetilde{D}_{ij}$. 
\item Finally, we fix cut-off functions $(\Psi_i, \Psi_j) \in \cinfc(Y_i, [0,1]) \times \cinfc(Y_j, [0,1])$ such that $\Psi_i \equiv 1$ on $\pi(U_i)$ and $\Psi_j \equiv 1$ on $\pi(U_j)$(here, $\pi : (x,\xi) \in T^*Y_{\cdot} \mapsto  x \in Y_{\cdot}$ is the natural projection) and we adopt the following definitions : 
\end{itemize}

\begin{defi}\label{Def_FIO_local}
We say that $T : \mathcal{D}^\prime(Y_j) \to \cinf(\overline{Y_i})$ is a Fourier integral operator in the class $I_\delta(Y_i \times Y_j, \Gr(F_{ij})^\prime)$  if there exists $T^\R \in I_\delta(\R \times \R, \Gr(F)^\prime)$ as constructed above such that
\begin{itemize}
\item $T - \Psi_i T \Psi_j = \hinf_{\mathcal{D}^\prime(Y_j) \to \cinf(\overline{Y_i})}$; 
\item $\Psi_i T \Psi_j =  \Psi_i T^\R \Psi_j $
\end{itemize}
\end{defi}

For $U^\prime_j \subset U_j$ and $U_i^\prime = F(U_j^\prime) \subset U_i$, we say that $T$ (or $T^\R$)  is microlocally unitary in $U_i^\prime \times U_j^\prime$ if $TT^* = \Id$ microlocally in $U_i^\prime$ and $T^*T = \Id$ microlocally in $U_j^\prime$. 

\begin{rem}
The definition of this class is not canonical since it depends in fact on the compact set $W$ through the partition of unity. 
\end{rem}

We can now state our definition for open quantum hyperbolic maps associated with $F$ :

\begin{defi}\label{def_FIO}
Fix $\delta \in [0,1/2[$. 
We say that $T=T(h)$ is an \emph{open quantum hyperbolic map} associated with $F$, and we note $T=T(h) \in I_{\delta} ( Y \times Y ,\Gr(F)^\prime )$ if : 
for each couple $(i,j) \in \{ 1, \dots , J \}^2$, there exists a semi-classical Fourier integral operator  $ T_{ i j }=T_{ i j } (h) \in I_{\delta} ( Y_j \times Y_i , \Gr(F_{ij})^\prime )$ associated with $F_{ i j}$ in the sense of definition \ref{Def_FIO_local}, such that
$$ T = ( T_{i j } )_{ 1 \leq i,j \leq J }  : \bigoplus_{i=1}^J L^2 (Y_i) \to \bigoplus_{i=1}^J L^2 (Y_i)$$ 
In particular $\WF^\prime (T) \subset \widetilde{A} \times \widetilde{D}$. 
We note $I_{0^+}(Y \times Y , \Gr(F)^\prime)= \bigcap_{ \delta >0}  I_{\delta} ( Y \times Y , \Gr(F)^\prime )$. 
\end{defi}

We will say that $T \in I_{0^+}(Y \times Y, \Gr(F)^\prime)$ is microlocally invertible near $\mathcal{T}$ if there exists a neighborhood $U^\prime \subset U$ of $\mathcal{T}$ and an operator $T^\prime \in I_{0^+}(Y \times Y , \Gr(F^{-1})^\prime)$  such that, for every $u =(u_1, \dots, u_J) \in L^2(Y) $  
$$ \forall j \in \{1, \dots, J \}, \WF(u_j) \subset U^\prime \cap U_j \implies TT^\prime u = u + \hinf ||u||_{L^2} , T^\prime T u = u +  \hinf ||u||_{L^2} $$

 Suppose that $T$ is microlocally invertible near $\mathcal{T}$ and recall that $T^* T \in \Psi_{0^+}(Y)$. Then, we can write
 $$T^* T = \op( a_h)$$ where $a_h$ is a smooth symbol in the class $S_{0^+}(U)$. We note $\alpha_h= \sqrt{|a_h|}$ and call it the \emph{amplitude} of $T$. Since $T$ is microlocally invertible near $\mathcal{T}$, $|a_h| >c^2$ near $\mathcal{T}$, for some $h$-independent constant $c>0$, showing that $\alpha_h$ is smooth and larger than $c$ in a neighborhood of $\mathcal{T}$. 
 
 \begin{rem}
 If $T$ has amplitude $\alpha$, at first approximation, $T$ transforms a wave packet $u_{\rho_0}$ of norm 1 centered at a point $\rho_0 $ lying in a small neighborhood of $\mathcal{T}$ into a wave packet of norm $\alpha(\rho_0)$ centered at the point $F(\rho_0)$. 
 \end{rem}

\subsubsection{A precise version of Egorov's theorem}

We will need a more quantitative version of Egorov's theorem, similar to the one in \cite{NDJ19} (Lemma A.7). The result does not show that $U(1)^{-1} \op(a) U(1)$ is a pseudodifferential operator (one would need Beal's theorem to say that) but it gives a precise estimate on the remainder, depending on the semi-norms of $a$. We specialize to the case of dimension 2. The statement is proved in \cite{Vacossin} (Proposition 3.3). 

\begin{prop}\label{prop_Egorov}
Consider $\kappa \in \mathcal{K}$ and note $U(t)$ the solution of (\ref{Cauchy_pb_Egorov}). There exists a family of differential operators $(D_{j})_{j \in \N}$ of order $j$ such that for all $a \in S_\delta$ and all $N \in \N$, 
\begin{equation}
U(1)^{-1} \op(a) U(1) = \op \left(a \circ \kappa +  \sum_{j=1}^{N-1} h^j (D_{j+1} a) \circ \kappa \right) + O_{\kappa} \left( h^N ||a||_{C^{2N +15}} \right)
\end{equation}
\end{prop}

By using local charts and composition results (see for instance \cite{DyZw}, Proposition E.10), it is possible to build local Fourier integral operators, which, combined with the last proposition, gives

\begin{prop}\label{Prop_def_FIO_loc}
Let $V \subset \R^2 = T^*\R$ an open set and $\kappa : V \to U \subset \R^2$ a symplectic map. Fix $\rho_V \in V$. There exists $W \subset V^\prime \subset V$,  neighborhoods of $\rho_V$ and a pair $(B,B^\prime)$ of Fourier integral operators in $I_0( \kappa(V^\prime) \times V^\prime , \Gr^\prime(\kappa) ) \times I_0(V^\prime \times \kappa(V^\prime) , \Gr^\prime(\kappa^{-1} ) ) $ which satisfy : there exists differential operator $(L_j)_{j \geq 1}$ of order $2j$ and supported in $V^\prime$ such that for all $a \in S_\delta(\R^2)$ with $\supp a \subset W$ and for all $N \in \N$, 
$$ B \op(a)B^\prime = \op( a \circ \kappa^{-1} ) + \sum_{ j=1}^{N-1} \op\left( L_j a) \circ \kappa^{-1} \right) + O\left( h^N ||a||_{C^{2N + M}} \right)$$
for some universal integer $M$. 
\end{prop}

\begin{proof}
It is enough to treat the case $\rho_V=0=\kappa(\rho_V)$. 
It suffices to consider sufficiently small neighborhoods of $0$ so that the restriction of $\kappa$ can be seen as the restriction of an element of $\mathcal{K}$. Then, one uses Proposition \ref{prop_Egorov}. 
\end{proof}

\subsection{Metaplectic operator}
Among the class of Fourier integral operators acting on $L^2(\R)$, metaplectic operators are the one quantizing the linear symplectic transformations of $T^*\R=\R^2$. The main advantage of metaplectic operators compared with general Fourier Integral operators is that the Egorov property is exact (see definition \ref{def_meta_egorov} below). We recall here a few standard facts on metaplectic operators. We refer the reader to \cite{ZW} (Section 11.3) and \cite{CoRo} (Chapter 3) for a more precise presentation and other references. 

\begin{defi}
For $\rho = (x_0, \xi_0) \in \R^{2}   = T^* \R$, the phase space translation operator $T_h(\rho)$ is defined as : 
$$T_h(\rho) u (x) = e^{-i\frac{x_0 \xi_0}{2h}} e^{ i \frac{x \xi_0}{h}} u(x-x_0)$$
It is a unitary on $L^2(\R)$ and $T_h(\rho)^* = T_h(-\rho)$. Moreover,
$T_h(\rho)^* \op(a) T_h(\rho) = \op(a ( \cdot - \rho) ) $ for any classical observable $a \in \mathcal{S}^\prime(\R)$. 
\end{defi}

\begin{prop}(\textbf{and Definition})\label{def_meta_egorov}
Let $\kappa : T^*\R \to T^*\R$ be a symplectic \emph{linear} map. There exists a unitary operator $\mathcal{M}_h(\kappa) : L^2(\R) \to L^2(\R)$ such that one of the two following equivalent conditions hold : 
\begin{enumerate}[label=(\roman*),nosep]
\item For every  $\rho \in T^*\R$, $ \mathcal{M}_h(\kappa) T_h(\rho) \mathcal{M}_h(\kappa)^* = T_h(\kappa(\rho) ) $ ; 
\item For all $a \in \mathcal{S}(\R)$, $ \mathcal{M}_h(\kappa) \op(a) \mathcal{M}_h(\kappa)^*  = \op( a \circ \kappa^{-1} ) $. 
\end{enumerate}
The operator $\mathcal{M}_h(\kappa)$ is unique up to multiplication by an element of $\mathbb{U}= \{ z \in \C, |z| =1 \}$. 
\end{prop}

Most of the time we won't precise that $T_h(\rho)$ and $\mathcal{M}_h(\kappa)$ depend on $h$ and we will simply write $T(\rho)$ and $\mathcal{M}(\kappa)$. We will write the index $h$ (or $h=1$) when needed. In fact, we can relate $\mathcal{M}_h(\kappa)$ and $\mathcal{M}_1(\kappa)$ by the relation : 

\begin{equation}
\mathcal{M}_h(\kappa) \Lambda_h = \Lambda_h \mathcal{M}_1(\kappa)
\end{equation}
where $\Lambda_h$ is the unitary scaling operator : 
\begin{equation}\label{Lambda_h}
\Lambda_h u(x) = h^{-1/4} u(h^{-1/2} x) 
\end{equation}

A way to obtain metaplectic operators is by solving the Schrödinger equation associated with quadratic Hamiltonians. 

\begin{prop}\label{Proposition_metaplectic_flow}
Let $S_2(\R)$ be the spaces of symmetric matrices of $\mathcal{M}_2(\R)$. 
Let $t \in [0,1] \mapsto S(t) \in S_2(\R)$ be $C^1$. 
We note  \begin{itemize}
\item  the quadratic time dependent Hamiltonian $H(t,\rho) = \frac{1}{2} (\rho, S(t) \rho )$ ; 
\item $t \in [0,1]\mapsto \kappa(t)$ the classical flow for the Hamiltonian $H(t)$, which solves the equation $$\dot{\kappa}(t)= JS(t) \kappa(t)$$ where $J= \left( \begin{matrix}
0 & 1 \\
-1 &  0 
\end{matrix} \right)$.  $\kappa(t)$ is a symplectic linear map for all $t \in [0,1]$
\item $U(t)$ the propagator of the Schrödinger equation $$\frac{h}{i} \frac{d}{dt} u(t) + \op(H(t) ) u = 0$$
 $U(t)$ is a unitary operator on $L^2(\R)$ for all $t \in [0,1]$. 
\end{itemize}
Then, for all $t \in [0,1]$, $U(t)$ is a metaplectic operator associated with the linear symplectic map $\kappa(t)$. 
\end{prop}

Note that for every $\kappa_1 \in \text{Sp}_2(\R)$, there always exists a (non unique) $C^1$ curve $\kappa : t \in [0,1] \to \text{Sp}_2(\R)$ such that $\kappa(0)=I_2$ and $\kappa(1)=\kappa_1$. (see for instance \cite{CoRo}, Proposition 31 in Chapter 3). So that we can construct $\mathcal{M}(\kappa_1)$ by use of the previous proposition. 

\begin{ex}
The unitary $h$-Fourier transform $\mathcal{F}_h$, where 
$$ \mathcal{F}_h u (\xi)= \frac{1}{(2\pi h)^{1/2}} \int_{\R} u(y) e^{- \frac{i y \xi}{h}}$$ is a metaplectic operator associated with $J$. 
\end{ex}

\begin{ex}
Suppose that $$\kappa = \left( \begin{matrix}
a & b \\
c & d 
\end{matrix}\right)$$
with $a \neq 0$. Then, the following operator is a metaplectic operator associated with $\kappa$ : 
\begin{equation} \label{meta_example_2}
 \mathcal{M}(\kappa) u (x) = \left( \frac{1}{2 \pi h |a| } \right)^{1/2} \int_\R e^{\frac{i}{2h}   (ca^{-1} x^2 + 2 a^{-1}x \xi - a^{-1}b\xi^2 )  } \mathcal{F}_h u(\xi) d \xi
\end{equation}
\end{ex}

\subsection{Coherent states}
\subsubsection{Definitions and notations}

In this subsection, we introduce the notations and definitions we will use for studying coherent states. We refer the reader to \cite{CoRo}. 
The semiclassical coherent state (or Gaussian state) centered at zero will be denoted by 
\begin{equation}
\varphi_0(x) = \frac{1}{(\pi h)^{1/4}} e^{- \frac{x^2}{2h}}
\end{equation}
and the coherent state centered at $\rho$ is simply 
\begin{equation}
\varphi_\rho \coloneqq T(\rho) \varphi_0
\end{equation}
We also write
$$ \varphi_0 = \Lambda_h \Psi_0$$
where $\Lambda_h$ is defined in (\ref{Lambda_h}) and $\Psi_0$ is the renormalized coherent state
\begin{equation}
\Psi_0(x) = \frac{1}{\pi^{1/4}} e^{- \frac{x^2}{2}}
\end{equation} 
We recall that $\varphi_0$ (resp. $\Psi_0$) is the ground sate of the harmonic oscillator $-h^2 \partial^2_x + x^2$ (resp. $- \partial^2_x + x^2$). The other eigenfunctions of this harmonic oscillator, called excited states, are obtained from $\varphi_0$ (resp. $\Psi_0$) by applying the creation operator $\mathbf{a} = \frac{1}{\sqrt{2h}}(- h\partial_x +x)$ (resp. $\Lambda_h^* \mathbf{a} \Lambda_h = \frac{1}{\sqrt{2}}(- \partial_x +x)$). For $n \in \N$,  we can note for instance 
$$ \varphi_{0,n} = \mathbf{a}^n \varphi_0 ; \Psi_n = \Lambda_h^* \mathbf{a}^n \Lambda_h \Psi_0$$
We recall that $\Psi_n  =h_n(x) \Psi_0$ where $h_n$ is a hermite polynomial of degree $n$. In particular, if $P \in \C[X]$, it is possible to decompose $P(x)\Psi_0(x)$ into a linear combination of excited states up to order $\deg (P)$. 

We can also define \emph{squeezed coherent states} : 

\begin{defi}
Let $\gamma \in \C$ with $\im \gamma >0$. The squeezed coherent state, deformed by $\gamma$ and centered at zero is 
$$ \varphi_0^{(\gamma)} (x) = (a_\gamma  \pi h)^{-1/4} e^{ i \gamma\frac{x^2}{2h}}$$
where $a_\gamma = \im(\gamma)^{-1}$ makes the norm of this state equal to one. 
We also define the squeezed coherent state centered at $\rho \in T^*\R$ by 
$$ \varphi_\rho^{(\gamma)}  = T(\rho)  \varphi_0^{(\gamma)} $$
and the squeezed renormalized coherent state at 0 
$$\Psi_0^{(\gamma)}(x) =  (a_\gamma  \pi )^{-1/4} e^{ i \gamma\frac{x^2}{2}}$$
\end{defi}
We conclude this section by recalling a useful formula - a resolution of the identity - which is the starting point of our analysis (see \cite{CoRo}, Proposition 6 in Section 1.2). I

\begin{lem}\label{trace_coherent_state}
Let $A : L^2(\R) \to L^2(\R)$ be a trace class operator. Then, 
$$ \tr(A) = \frac{1}{2\pi h} \int_{T^*\R} <A \varphi_\rho, \varphi_\rho> d\rho$$
where $d\rho$ denotes the Lebesgue measure of $\R^2$. 
\end{lem}

\subsubsection{Action of metaplectic operators on coherent states}

We recall here how metaplectic operators act on coherent states. We refer the reader to \cite{CoRo} (Section 3.2) for a complete proof and a general version in any dimension : 

\begin{prop}
Let $ \kappa = \left( \begin{matrix}
a & b \\
c& d 
\end{matrix} \right)
 $ be a symplectic linear map. Let $\mathcal{M}(\kappa)$ be a metaplectic operator associated with $\kappa$, constructed by use of Proposition \ref{Proposition_metaplectic_flow}, following a path $\kappa_t$ from $I_2$ to $\kappa$.  Then, we have : 
$$ \mathcal{M}(\kappa)\varphi_0(x) = (\pi h)^{-1/4} (a+ib)^{-1/2} e^{i \gamma_\kappa \frac{x^2}{2h}}$$
where $\gamma_\kappa = (c+id)(a+ib)^{-1}$. 
\end{prop}

\begin{rem}
The square root $(a+ib)^{1/2}$ is determined by the path $\kappa_t$ ($(a_t + i b_t)^{1/2}$ has to be continuous). \\
Since $\im \gamma_\kappa = |a+ib|^{-2}$, this proposition shows that for some $\lambda \in \mathbb{U}$, 
$$ \mathcal{M}(\kappa) \varphi_0 = \lambda \varphi^{(\gamma)}_0$$ 
Since the metaplectic operators are defined modulo $\mathbb{U}$,  in the rest of this article, we will sometimes omit to write the factor $\lambda$ and and by abuse, we could write $ \mathcal{M}(\kappa) \varphi_0 =  \varphi^{(\gamma)}_0$. 
It won't be specified anymore. Anyway, we are concerned by the norm of such states. 
\end{rem}

We also give the following formula concerning the action of metaplectic operators on excited coherent states (see \cite{Ha98}, Section 2) : 
\begin{prop}\label{Prop_meta_on_excited}
Let $ \kappa = \left( \begin{matrix}
a & b \\
c& d 
\end{matrix} \right)
 $ be a symplectic linear map. Let $\mathcal{M}(\kappa)$ be a metaplectic operator associated with $\kappa$, constructed by use of Proposition \ref{Proposition_metaplectic_flow}, following a path $\kappa_t$ from $I_2$ to $\kappa$. Then,
$$ \mathcal{M}(\kappa) \varphi_{0,n} =(|a+ib|^2 \pi h)^{-1/4} \left( \frac{(a-ib)}{(a+ib)}\right)^{n/2} h_n \left( \frac{x}{h^{1/2} |a+ib|} \right) e^{i\gamma_\kappa \frac{x^2}{2h}} $$
where $\gamma_\kappa = (c+id)(a+ib)^{-1}$. 
\end{prop}

In the sequel, we will need to estimate the $H_h(\langle\rho \rangle^N )$-norm of squeezed coherent states in terms of the squeezing parameter. Equivalently, we need to control this norm for a state of the form $\mathcal{M}(\kappa)\varphi_\rho$ in terms of $\kappa$. 
To do so, we start by fixing a norm $||\cdot||$ on $\mathcal{M}_2(\R)$ For convenience, let's assume that for all linear symplectic map, we have 
\begin{equation}\label{tricky_norm}
||\kappa|| \geq 1
\end{equation}
For instance, let's say that $||\kappa||=\sqrt{2} \max (|\kappa|_{11},|\kappa|_{12},|\kappa|_{21},|\kappa|_{22})$. It is not hard to check that this norm satisfies \ref{tricky_norm} since $\det(\kappa) =1$. The main interest of (\ref{tricky_norm}) is that $||\kappa||^a \leq ||\kappa||^b$ if $a \leq b$.

 We have : 

\begin{lem}\label{Control_mathcal_M_kappa} There exists a family of universal constants $(K_{N,k})_{(N,k) \in \N^2}$ such that the following holds : 
let $N \in \N$, $k \in \N$ and $\kappa$ be a symplectic linear map. Then, there exists for all $0 < h \leq 1$, 
$$ || \mathcal{M}(\kappa) (x^k \varphi_0) ||_{H_h(\langle\rho \rangle^N )} \leq K_{N+k} \sum_{l=0}^N h^{(l+k)/2} ||\kappa||^l $$
\end{lem}

\begin{proof}
Let's write $\kappa= \left( \begin{matrix}
a &b \\
c & d 
\end{matrix} \right)$. 
For a state $u \in H_h(\langle\rho \rangle^N )$, we have  $$ ||u ||^2_{H_h(\langle\rho \rangle^N )}   \sim \sum_{ \alpha + \beta \leq N } || \op(x^\alpha \xi^\beta) u ||_{L^2}^2$$
Let $\alpha, \beta \in N$ such that $\alpha + \beta \leq N$. We want to estimates $|| \op(x^\alpha \xi^\beta)  \mathcal{M}(\kappa)(x^k \varphi_0) ||_{L^2}^2$. We have 
$$ \op(x^\alpha \xi^\beta)  \mathcal{M}(\kappa)(x^k \varphi_0) =\mathcal{M}(\kappa) \op\left( (ax + b \xi)^\alpha (cx + d\xi)^\beta \right) ( x^k \varphi_0)$$
Since $\mathcal{M}(\kappa)$ is unitary on $L^2$, it is enough to estimates the $L^2$ norm of 
$$
\op\left( (ax + b \xi)^\alpha (cx + d\xi)^\beta \right)  (x^k \varphi_0) = \op \left( \sum_{ l=(l_1,l_2) \in \N^2, l_1 + l_2 = \alpha + \beta}  B_l(\kappa) x^{l_1} \xi^{l_2} \right) (x^k\varphi_0)
$$
where $B_{l}$ is some $l_1+ l_2$ multilinear form in $\kappa$, whose coefficients depend on $\alpha$ and $\beta$. In particular, $|B_l(\kappa)| \leq C_l ||\kappa||^{l_1 + l_2}$ for some universal $C_l$. Finally, we observe that $|| \op(x^{l_1} \xi^{l_2} ) (x^k \varphi_0) ||_{L^2}  \coloneqq h^{(l_1+l_2+k)/2} C_{(l_1,l_2,k)}$, for some $C_{(l_1,l_2,k)}$ depending only on $(l_1,l_2,k)$, and we find that
$$ ||  \op\left( (ax + b \xi)^\alpha (cx + d\xi)^\beta \right)  (x^k\varphi_0) || \leq  C_{(\alpha, \beta,k)} \sum_{p=0}^{\alpha+\beta} ||\kappa||^{p}h^{(p+k)/2} $$
 we find the required inequality with $K_{N,k}$ depending on the the $C_{\alpha,\beta,k}$ with $\alpha+\beta \leq N$. 
\end{proof}

As a corollary, by specializing at $h=1$, we obtain the following : 
\begin{cor}\label{cor_norm_H_1}
There exists a family of constants $K_{N,d}, d \in \N, N \in \N$ such that : for all $P \in \C[X]$, for all symplectic linear map $\kappa$ and for all $N \in \N$, 
$$|| \mathcal{M}_1(\kappa) (P \Psi_0)||_{ H_1(\langle \rho \rangle^N)} \leq K_{N,\deg P} N_\infty(P) ||\kappa||^N$$
where $N_\infty(P)$ is the sup norm of the coefficients of $P$.
\end{cor}

\subsubsection{Action of pseudodifferential operators on coherent states}

In this subsection, we give precise results for the actions of semiclassical pseudodifferential operators on coherent states, when the symbol of the pseudodifferential operator belong to the class $S_\delta$. 

\begin{lem}\label{Lemma_action_pseudo_coherent_state}
Suppose that $a \in S_\delta(T^*\R)$ with $0 \leq \delta < 1/2$. Assume that $\rho_0= (x_0, \xi_0) \in T^*\R$. 
Then, for every $N \in \N$, there exists $\rho_N(a,\rho_0) \in L^2$ such that 
$$ \op(a)\varphi_{\rho_0} = \sum_{k=0}^{N-1} h^{k/2} \psi_k(a, \rho_0) + h^{N/2} \rho_N(a,\rho_0)$$
where $$
 \psi_k(a, \rho_0) =  T(\rho_0) \Lambda_h  \text{Op}_1 \left( \sum_{\alpha + \beta = k} \frac{\partial^{\alpha}_x \partial^\beta_\xi a (\rho_0)}{\alpha! \beta ! } x^\alpha\xi^\beta \right)  \Psi_0$$
 and 
 $$||\rho_N(a,\rho_0) ||_{L^2} \leq C_N h^{-\delta N} \sup_{|\gamma| \leq  N + M} ||h^{ \delta |\gamma| }\partial^\gamma a ||_\infty$$
\end{lem}

\begin{rem}
\begin{itemize}[nosep]
\item $M$ is a universal constant. 
\item The first term of the expansions is $a(\rho_0) \varphi_{\rho_0}$. 
\item It is effectively an expansion in power of $h^{1/2-\delta}$ since $a \in S_\delta$.
\item We could also write $\text{Op}_1(x^\alpha \xi^\beta) \Psi_0$ in the form $P(x) \Psi_0$ where $P$ is a polynomial of degree $\alpha + \beta$, or equivalently, it is a linear combination of the first $|\alpha| + |\beta|$ excited states.
\end{itemize}
\end{rem}

\begin{proof}
Let's write $\varphi_{\rho_0} = T(\rho_0) \Lambda_h \Psi_0$. 
We have 
\begin{align*}
\op(a) \varphi_{\rho_0} &= \op(a) T(\rho_0) \Lambda_h \Psi_0 \\
&= T(\rho_0) \Lambda_h \text{Op}_1( b_h ) \Psi_0
\end{align*}
where $b_h(\rho) = a (\rho_0 + h^{1/2}\rho)$. 
Let's write the Taylor expansion of $a$ around $\rho_0$ : 
$$b_h(x,\xi) = \sum_{\alpha + \beta \leq N-1} h^{(\alpha+ \beta)/2} \frac{\partial^{\alpha}_x \partial^\beta_\xi a (\rho_0)}{\alpha! \beta ! } x^\alpha\xi^\beta + h^{N/2} R_N(x,\xi)$$ where 
$$ R_N(\rho) = \frac{1}{(N-1)!} \int_0^1 \frac{d^N}{dt^N} a (\rho_0 + t h^{1/2} \rho) (1-t)^{N-1} dt$$
Applying $\text{Op}_1$ to this expansion, we get the required asymptotic with 
$$ \rho_N(a,\rho_0) = T(\rho_0) \Lambda_h \text{Op}_1(R_N) \Psi_0$$
It remains to estimates the $L^2$ norm of $\rho_N$. Since $T(\rho_0)$ is unitary, it is enough to evaluate $$\Lambda_h \text{Op}_1(R_N) \Psi_0 = \op(\tilde{R}_N) \varphi_0$$
where $\tilde{R}_N(\rho) = R_N(h^{-1/2}\rho) =  \frac{1}{(N-1)!} \int_0^1 (1-t)^{N-1}  d^N a (\rho_0 + t \rho)(\rho^{\otimes N}) dt$ \footnote{Here, if $f \in C^N(\R^2, \R)$, we note $d^N f(\rho)(h^{\otimes N}) = \left. \frac{d^N}{dt^N}\right|_{t=0} f(\rho + th)$. }. Using that $a \in S_\delta$, it is not hard to see, after derivation under the integral that, for any $\gamma \in \N^2$ and $\rho \in T^*\R$, 
$$| \partial^{\gamma} \tilde{R}_N (\rho)| \leq C_N \sup_{\gamma_1 \leq  N +| \gamma|}|| \partial^{\gamma_1} a||_\infty \langle \rho\rangle^N \leq  h^{-\delta(N+ |\gamma|)} ||h^{ \delta |\gamma| }\partial^\gamma a ||_\infty\langle \rho\rangle^N $$
This shows that $\tilde{R}_N \in h^{-\delta N} S_\delta(\langle\rho\rangle^N)$ in the sense of \cite{ZW} (Definition 4.4.3). Then, we find that 
$$ h^{\delta N } \op(\tilde{R}_N) : H_h (\langle \rho\rangle^N) \to L^2(\R)$$ is a uniformly bounded family of operators, with norm depending on a finite number of semi-norms of $\tilde{R}_N$ in $S_\delta(\langle \rho\rangle^N)$. We conclude by noting that for any $N \in \N$, $\varphi_0$ is in $H_h (\langle \rho\rangle^N)$, with a norm bounded uniformly in $h$. Hence 
$$||\rho_N||_{L^2} \leq ||\op(\tilde{R}_N)||_{ H_h (\langle \rho\rangle^N) \to L^2(\R)} \times || \varphi_0||_{H_h (\langle \rho \rangle^N)} \leq h^{-\delta N} C_N \sup_{|\gamma| \leq  N + M} ||h^{ \delta |\gamma| }\partial^\gamma b ||_\infty$$
\end{proof}

As a simple corollary, we get : 
\begin{cor}
Assume that $a$ vanishes at order $k$ at $\rho_0$. Then, 
$$ \op(a) \varphi_{\rho_0} = O_{L^2} \left( h^{k(1/2- \delta)} \right) $$
\end{cor}
In particular, if $a$ vanishes in a neighborhood of $\rho_0$, we recover that $\op(a) \varphi_{\rho_0} = \hinf$. This is something well known since $\WF(\varphi_{\rho_0}) = \{ \rho_0 \}$.

\subsubsection{Action of Fourier integral operators on coherent states}

In \cite{CoRo} (Chapter 4), the authors study the quantum evolution of coherent states by the propagator of a Schrödinger equation with a time-dependent Hamiltonian. We refer the reader to their work, and in particular to Theorem 21 in this book for this very general version of the evolution of coherent states. Here, we will simply study the action of the particular type of Fourier integral operator of the form given in equation (\ref{special_form_FIO}) on states of the form $T(\rho_0) \mathcal{M} (\kappa) \varphi_0$. In other words, we want to study the action of a Fourier Integral Operator on these squeezed and translated states. More generally, we will consider also squeezed excited states of the form $T(\rho_0) \mathcal{M} (\kappa) \Lambda_h (P(x) \Psi_0(x) )$. We will give an asymptotic expansion of these evolved states with a controlled remainder. The dependence of this remainder on $\kappa$ will be crucial to use recursively the expansion. 

Let's describe the framework in which we want to work : we suppose that $\Omega_x, \Omega_\eta$ are open intervals of $\R$, $\psi \in \cinf(\Omega_x \times \Omega_\eta)$ is a phase function that generates the twisted graph of some symplectic map $F$ in some open set $\Omega_0 \subset \R^4$, that is 
$$ \Gr^\prime(F)  \cap \Omega_0= \big\{ (x, \partial_x \psi(x, \eta) , \partial_\eta\psi(x,\eta), -\eta) , x \in \Omega_x, \eta \in \Omega_\eta \big\} $$ 
We suppose that $a \in S_{0^+}^{comp} (\Omega_x \times \Omega_\eta)$  and we consider the Fourier integral operator : $$
Su(x) = \frac{1}{(2\pi h)} \int_{\R^2} e^{ \frac{i}{h} (\psi(x,\eta) - y \cdot \eta) } a(x,\eta) u(y) dy d\eta  $$
We do not necessarily assume that $S$ is microlocally unitary, but if it were the case, $a$ would satisfy $|a(x, \eta)|^2 = | \partial^2_{ x \eta}  \psi (x, \eta) | + O(h^{1-\varepsilon })$ for any $\varepsilon>0$. More generally, the amplitude $\alpha$ of $S$ as a Fourier integral operator is given, modulo $O(h^{1^-}) S_{0^+}$, by 
$$ \alpha_S(y,\eta) = \frac{a(x,\eta)}{ | \partial^2_{ x \eta}  \psi (x, \eta) |^{1/2}} \;  , \;  F(y,\eta) = (x,\xi)$$ 

\begin{prop}\label{Prop_propagation_coherent_states}
Assume that $S$ satisfies the above assumptions. Let $\kappa  \in \mathcal{M}_2(\R)$ be a symplectic linear map and $\rho_0 \in T^*\R$. 
Let's note $\rho_1 = F(\rho_0)$. Let $P \in \C[X]$. Then, there exists a family of polynomials $Q_k(P)_{k \in \N}$ such that
\begin{itemize}[nosep]
\item $Q_0(P) =\alpha_S(\rho_0) P$ (up to multiplication by an element of $\mathbb{U}$) ; 
\item $Q_k(P)$ is a polynomial of degree $\deg P + 3k$ and the map $P \mapsto Q_k(P)$ is linear, with coefficients depending on $\kappa
$ and the derivatives of $\psi$ and $a$ at $(x_1, \xi_0)$ up to the $3k$-th order, and we have $$N_\infty(Q_k(P)) \leq C_{3k}(\psi) ||a||_{C^k} ||\kappa||^{3k} N_\infty(P)$$
Moreover, if $(x_1, \xi_0) \not \in \supp a$, then $Q_k = 0$.
\item 
for every $N \in \N$, 
\begin{equation}
S \Big( T(\rho_0) \mathcal{M}(\kappa) \Lambda_h [ P \Psi_0] \Big) = T(\rho_1) \mathcal{M}( d_{\rho_0}F \circ \kappa) \Lambda_h  \left[ \sum_{k=0}^{N-1} h^{k/2} Q_k(P) \Psi_0 \right] + R_N
\end{equation}
with $$||R_N ||_{L^2} \leq  h^{N/2} C_{3N+M}(\psi) ||a||_{C^{N+M}} ||\kappa||^{3N} K_{N,\deg P} N_\infty(P) $$
\end{itemize} 
Here, \begin{itemize}[nosep]
\item $C_{k}(\psi)$ depends on the $C^k$ norm of $\psi$ 
\item $M$ is a universal constant ; 
\item $N_\infty(P)$ is the sup norm on the coefficients of $P$ ; 
\item $(K_{N,d})_{(N,d) \in \N^2}$ is a family of universal constants. 
\item For every $\varepsilon >0$ and $k \in \N$, there exists $C_{\varepsilon, k}$ such that $||a||_{C^k} \leq C_{\varepsilon, k} h^{- \varepsilon}$. 
\end{itemize}
\end{prop}

\begin{rem}
This proposition shows that a Fourier Integral operator transforms a wave packet centered at $\rho_0$ into a wave packet centered at $F(\rho_0)$. However, this transformation squeezes the wave packet according to the linearization of $F$ at $\rho_0$: this is the effect of $\mathcal{M}(d_{\rho_0}F)$. The control of the error is important if we want to iterate this formula and apply it to squeezed coherent states $\mathcal{M}(\kappa_h) \varphi_0$, with a symplectic linear map $\kappa_h$ potentially depending of $h$. As soon as 
$$ ||\kappa_h ||^3 \ll h^{-1/2}$$, the remainder stays smaller than the terms in the expansion. 
In particular, suppose that $\kappa_h= \kappa_{n(h)} \dots \kappa_0$ with $||\kappa_i || \leq e^\lambda$ and $n(h) \sim \nu \log(1/h)$. Then, the approximation is valid as soon as 
$$\nu \leq \frac{1 - \varepsilon}{6\lambda} $$
\end{rem}

\begin{proof}
The following computations are performed modulo multiplication by an element of $\mathbb{U}$.

Let's note $\rho_0 = (x_0, \xi_0)$ and $\rho_1 = (x_1, \xi_1)$. Recall that, by definition of $\psi$, 
\begin{equation}\label{formula_x_xi}
\xi_1 = \partial_x \psi(x_1, \xi_0) \quad ; \quad x_0 = \partial_\eta \psi (x_1, \xi_0) 
\end{equation} 
We have, for $u \in L^2(\R)$, 
\begin{align*}
&\left(\Lambda_h^* T(\rho_1)^* S T(\rho_0) \Lambda_h u \right) (x) = h^{1/4} e^{-\frac{i x \xi_1}{h^{1/2}}  } \left( S T(\rho_0)\Lambda_h u \right) (h^{1/2}x+x_1) \\ 
&=  e^{-\frac{i x \xi_1}{h^{1/2}}  }  \frac{1}{2 \pi h } \int_{\R^2} e^{ \frac{i}{h} (\psi(h^{1/2}x+x_1,\eta) - y \cdot \eta) } a(h^{1/2}x+x_1,\eta) e^{\frac{i y \xi_0}{h}  } u(h^{-1/2}y-x_0) dy d\eta \\
&= \frac{1}{2\pi} \int_{\R^2} e^{i \tilde{\psi}_h(x,\eta,y)} a(h^{1/2}x+x_1,h^{1/2}\eta + \xi_0)  u(y) dy d\eta
\end{align*}
after a change of variable, with 
\begin{equation}
\tilde{\psi}_h(x,\eta,y) = \frac{1}{h} \psi(h^{1/2}x+ x_1, h^{1/2}\eta + \xi_0) - y \eta - h^{-1/2} \left( x\xi_1  +  x_0 \eta \right) 
\end{equation}
Let us write the Taylor expansion of $ \psi(h^{1/2}x+ x_1, h^{1/2}\eta + \xi_0)$ at order $N+1 \in \N$ : 
\begin{multline} \label{taylor_psi}
  \psi(h^{1/2}x+ x_1, h^{1/2}\eta + \xi_0) = \psi(x_1, \xi_0) + h^{1/2} \left(x \partial_x \psi (x_1, \xi_0) + \eta \partial_{\eta}\psi(x_1,\xi_0) \right) \\ +  \frac{h}{2} (D^2 \psi( x_1, \xi_0) (x,\eta) , (x,\eta) ) 
  + \sum_{k=3}^{N+1}  h^{k/2}\psi_k(x,\eta)+ h^{(N+2)/2}r_{N+2}^\psi (x,\eta ;h)
\end{multline}
where $\psi_k$ is $k$-multilinear in $(x,\eta)$ with coefficients depending on the derivatives of $\psi$ of order $k$ at $(x_1, \xi_0)$ and for $\alpha \in \N^2$, 
$$h^{(N+2)/2}r_{N+2}^\psi(x,\eta ; h) = \frac{1}{(N+1)!} \int_0^1 (1-t)^{N+1} \frac{d^{N+2}}{dt^{N+2}} \big( \psi(x_1 + th^{1/2} x, \xi_0 + th^{1/2}\eta)\big) dt $$
In particular, we have the estimates
\begin{equation}\label{Taylor_remainder}
 |\partial^\alpha r_{N+2}^\psi(x,\eta ; h) | \leq C_N  \sup_{ N+2 \leq |\beta| \leq N+ 2+ |\alpha|} h^{(|\beta|-N-2)/2} ||  \partial^{ \beta} \psi||_\infty \langle (x,\eta) \rangle^{N+2} 
\end{equation}
Recalling (\ref{formula_x_xi}), we can write : 
\begin{multline}\label{computations_lambda_T_S_T}
\left(\Lambda_h^* T(\rho_1)^* S T(\rho_0) \Lambda_h u \right) (x)  = \\ \frac{1}{2\pi} \int_{T^*\R} e^{i \left(  \frac{1}{2} (D^2 \psi( x_1, \xi_0) (x,\eta) , (x,\eta) )  - y\eta\right)  } e^{i h^{1/2} r_3^\psi(x,\eta;h)} a(h^{1/2}x+x_1,h^{1/2}\eta + \xi_0)  u(y) dy d\eta
\end{multline}
Then, we write the Taylor expansion of $e^{i h^{1/2}r_3^\psi(x,\eta;h)}$ at order $N$:

\begin{equation}\label{devp_exp}
e^{i h^{1/2}r_3^\psi(x,\eta;h)}= \sum_{k=0}^{N-1} \frac{i^k h^{k/2}}{k!} r_3^\psi(x,\eta;h)^k + \underbrace{ \frac{i^N h^{N/2}}{(N-1)!} (r_3^\psi(x,\eta;h))^N \int_0^1 e^{i h^{1/2} s r_3^\psi(x,\eta;h)} (1-s)^{N-1} ds}_{\tilde{r}_N}
\end{equation}
Using (\ref{taylor_psi}), we write $r_3^\psi = \sum_{j=0}^{N-k-1} h^{j/2} \psi_{3+j} + h^{(N-k)/2}r_{3+N-k}^\psi$ and we can expand
$$\left(r_3^\psi\right)^k = \sum_{\alpha_1 + \dots + \alpha_k < N-k} h^{(\alpha_1+\dots + \alpha_k)/2}  \psi_{3+\alpha_1} \dots \psi_{3+\alpha_k} + h^{(N-k)/2} \textit{ Remainder } $$
where the remainder is a linear combination, with universal coefficients, of terms of the form 
\begin{equation}\label{form_sum}
r_{3+\alpha_1}^\psi\dots  r_{3+\alpha_j}^\psi \psi_{3 + \alpha_{j+1}} \dots \psi_{3 + \alpha_k} \quad ; \quad 0 \leq j \leq k  \, , \, \alpha_1 + \dots + \alpha_k = N-k
\end{equation}
Gathering all the terms of order $h^{k/2}$  for $k \leq N-1$, together and gathering all the terms of order $h^{N/2}$ in a single remainder term, we have 

\begin{align*}
e^{i h^{1/2}r_3^\psi(x,\eta;h)}=\sum_{k=0}^{N-1} h^{k/2} \tilde{P}_k(x,\eta ; \psi) + h^{N/2}r_{N,1} + \tilde{r}_N 
\end{align*}
where 
\begin{itemize}[nosep]
\item $\tilde{P}_k (\cdot ; \psi)$ is a polynomial of order $3k$ in $(x,\eta)$ with coefficients of the form $q\left(( \partial^{\alpha} \psi(x_1, \xi_0) )_{|\alpha| \leq 3 + k} \right)$, where $q$ is a universal polynomial of degree $k$; 
\item $r_{N,1}$ is a linear combination of terms of the form (\ref{form_sum}) with $0 \leq k \leq N-1$, $0 \leq j \leq k$ and $\alpha_1, \dots, \alpha_k$, with $\alpha_1 + \dots + \alpha_k = N-k$ ; 
\item $\tilde{r}_N$ is defined in  (\ref{devp_exp}). 
\end{itemize}
Similarly, we can Taylor expand $a(h^{1/2} x + x_1, h^{1/2}\eta + \xi_0 )$ to find that 
\begin{multline}\label{precise_expansion}
e^{i h^{1/2}r_3^\psi(x,\eta;h)}a(h^{1/2} x + x_1, h^{1/2}\eta + \xi_0 )= \sum_{k=0}^{N-1} h^{k/2} P_k(x,\eta ; \psi ,a) \\ 
+ \underbrace{h^{N/2} \sum_{k=0}^{N-1} \tilde{P}_k(x,\eta ; \psi) r_{N-k}^a (x,\eta ; h)
+ h^{N/2}r_{N,1}\times a(h^{1/2} x + x_1, h^{1/2}\eta + \xi_0 )}_{\text{first remainder term}} +\underbrace{ \tilde{r}_N \times a(h^{1/2} x + x_1, h^{1/2}\eta + \xi_0 )}_{\text{second remainder term} }
\end{multline}
where $P_k (\cdot ; \psi,a)$ is a polynomial of degree $3k$ in $(x,\eta)$, given by 
$$ P_k (x,\eta ; \psi, a) = \sum_{k_1 + k_2 = k} \tilde{P}_{k_1}(x,\eta ; \psi) \times \left(\frac{1}{k_2!} d^{k_2} a (x_1, \xi_0) ( (x,\eta)^{\otimes k})\right) $$
and for $p \in \N$, 
$$r_{p}^a(x,\eta ; h) = \frac{h^{-p/2}}{p!}\int^1_0 (1-t)^{p-1} \frac{d^p}{dt^p} a(x_1 + th^{1/2}x, \xi_0 + th^{1/2} \eta) dt$$ 
Plugging (\ref{precise_expansion}) in (\ref{computations_lambda_T_S_T}) with $u = \mathcal{M}_1(\kappa) (P\Psi_0)  $, we find an expansion in power of $h^{1/2}$ for $\Lambda_h^* T(\rho_1)^* S T(\rho_0) \Lambda_h u$. 

\paragraph{Identification of the first term.}
The first term of the expansion is 
$$ \frac{1}{2\pi} \int_{\R^2} e^{i \left(  \frac{1}{2} (D^2 \psi( x_1, \xi_0) (x,\eta) , (x,\eta) )  - y\eta\right)  } a(x_1, \xi_0) u(y) dy d\eta$$
Differentiating the relation 
$$ F \left( \partial_\eta \psi(x,\eta), \eta \right)  = ( x, \partial_x \psi(x,\eta) ) $$
it not hard to see that 
$$ d_{\rho_0}F =\frac{1}{ \partial^2_{x \eta} \psi(x_1, \xi_0)} \left( \begin{matrix}
1 & -  \partial^2_{\eta \eta} \psi(x_1, \xi_0) \\
 \partial^2_{x x} \psi(x_1, \xi_0) & \left(\partial^2_{x \eta} \psi(x_1, \xi_0) \right)^2-  \partial^2_{\eta \eta} \psi(x_1, \xi_0)  \partial^2_{x x} \psi(x_1, \xi_0) 
\end{matrix} \right)$$
As a consequence, comparing with (\ref{meta_example_2}), we observe that
$$v \mapsto \frac{1}{2 \pi} |\partial_{x\eta} \psi(x_1, \xi_0)|^{1/2} \int_{T^*\R} e^{i(\frac{1}{2}D^2 \psi(x_1, \xi_0) (x,\eta)- y \eta)} v(y) dy d\eta$$ is a metaplectic operator associated with $d_{\rho_0} F$, that we note $\mathcal{M}_1(d_{\rho_0} F)$, and hence, wee see that
$$ S T(\rho_0) \mathcal{M}(\kappa) \Lambda_h [P\Psi_0]  = T (\rho_1)  \mathcal{M}(d_{\rho_0}F \circ \kappa) \Lambda_h \left[ \frac{a(x_1, \xi_0)}{|\partial_{x\eta} \psi(x_1, \xi_0)|^{1/2}} P \Psi_0\right] + (\text{smaller terms} ) $$ 
This gives the required form for $Q_0(P)$. 

\paragraph{Identification of higher order terms.}
For the term of order $k$ in the expansion of (\ref{computations_lambda_T_S_T}), based on (\ref{precise_expansion}), we have to understand 
$$ \frac{1}{2\pi} \int_{\R^2} e^{i \left(  \frac{1}{2} (D^2 \psi( x_1, \xi_0) (x,\eta) , (x,\eta) )  - y\eta\right)  } P_k(x,\eta ; \psi ,a) u(y) dy d\eta$$
Hence, we focus on terms of the form 
$$ S_{l,m} (u) = \frac{1}{2\pi} \int_{\R^2} e^{i \left(  \frac{1}{2} (D^2 \psi( x_1, \xi_0) (x,\eta) , (x,\eta) )  - y\eta\right)  } x^l \eta^m u(y) dy d\eta$$
with $l+m \leq3 k$. The term $x^l$ can be put in front of the integral. Concerning, the $\eta$ term, repeated integrations by part (or equivalently, using the usual properties of the Fourier transform), we find that 
$$S_{l,m} (u) = x^l \mathcal{M}_1(d_{\rho_0}F) ((i \partial_y)^m u) ) $$
Now, combining this with the standard commutations properties of metapletic operators we write
\begin{align*}
S_{l,m} (\mathcal{M}_1(\kappa) [P \Psi_0] ) &=\text{Op}_1 (x^l)\mathcal{M}_1(d_{\rho_0}F) \text{Op}_1((-\xi)^m ) \mathcal{M}_1(\kappa) [P \Psi_0] ) \\
&=\mathcal{M}_1 (d_{\rho_0}F \circ \kappa) \text{Op}_1 \left( (x \circ (dF(\rho_0) \circ \kappa)^l \right)  \text{Op}_1\left( (-\xi \circ \kappa)^m \right) [P\Psi_0] ) 
\end{align*}
Finally, the action of $\text{Op}_1 \left( (x \circ (d_{\rho_0}F \circ \kappa)^l \right)  \text{Op}_1\left( (-\xi \circ \kappa)^m \right)$ transforms $P \Psi_0$ into another state of the form $Q \Psi_0$ where $Q$ is of degree $\deg P + l + m$, where the coefficients of $Q$ depend linearly on those of $P$, with coefficients in the linear combination depending on $\kappa$ and on $d_{\rho_0}F$. By developing the powers $\big(x \circ (d_{\rho_0}F \circ \kappa) \big)^l$ and $(-\xi \circ \kappa)^m$, we see that the coefficients of $Q$ are bounded by $C_{l,m} ||d_{\rho_0}F||^l  ||\kappa||^{l+m}$ for some constant $C_{l,m}$. 

As a consequence, we can write the entire term of order $k$ in the form : 
$T (\rho_1)  \mathcal{M}(d_{\rho_0}F \circ \kappa) \Lambda_h ( Q_k(P) \Psi_0))$ where $Q_k(P)$ is a polynomial of order $\deg P + 3k$, the map $P \mapsto Q_k(P)$ is linear and its coefficients depend on $\kappa$, the derivatives of $\psi$ and $a$ at $(x_1, \xi_0)$ up to the $3k$-th order. This gives the required polynomial. By putting the terms  $ ||d_{\rho_0}F||^l $ into $C_{3k}(\psi)$ and using the special form of $P_k$,  we obtain the required estimate 
$$ N_\infty(Q_k(P) ) \leq C_{3k}(\psi) ||a||_{C^k}  ||\kappa||^{3k} N_\infty(P).$$

\paragraph{Control of the remainders.} The last step of the proof consists in proving that the remainder term has the required bound. As already written with the underbrace in (\ref{precise_expansion}), this remainder can be decomposed in two terms: they have different properties. Let us start with the first term, and call it $\tilde{r}_{N,1}$. 

In the products of the form given by (\ref{form_sum}), gathering the factors $r^\psi_{3 + \alpha}$ into a single term and the polynomials $\psi_k$ into a single polynomial, we see that the term $r_{N,1}$, appearing in $\tilde{r}_{N,1}$, is a sum of terms of the form $Q_j^\psi(x,\eta) R_j^\psi(x,\eta ; h)$, for $0 \leq j \leq k$, where $Q_j^\psi$ is a polynomial of degree $3j$ and $R_j^\psi(x,\eta ; h)$ satisfies for $\alpha \in \N^2$, 
$$ | \partial^\alpha R_j^\psi(x,\eta ;h )| \leq C_{3N-3j + |\alpha|} (\psi) \langle (x,\eta) \rangle^{3N - 3j}$$ 
where $C_{3N-3j + |\alpha|}(\psi)$ depends on the derivatives of $\psi$ up to the order $3N-3j + |\alpha| $.\footnote{These estimates comes from (\ref{Taylor_remainder}) and in fact, we can take $$C_{3N-3j+|\alpha|}(\psi) = \sup_{ 3N-3j \leq |\beta| \leq 3N-3j + |\alpha|} h^{(|\beta|-3N+3j)/2} ||\partial^{\beta} \psi||_\infty$$} Using the same kind of estimates for $r^a_{N-k}(x,\eta ; a;h)$, we see that $\tilde{r}_{N,1}$ satisfies : 

\begin{equation}\label{tilde_r_N_1}
\forall \alpha \in \N^2, (x,\eta) \in \R^2 \,,\, |\partial^\alpha \tilde{r}_{N,1}(x,\eta) | \leq C_{3N+ |\alpha|}(\psi) ||a||_{C^{N + |\alpha|}} \langle (x,\eta) \rangle^{3N}
\end{equation}
We are now interested in controlling 
$$ \tilde{R}_{N,1}u(x) \coloneqq \frac{1}{2\pi} \int_{T^*\R} e^{i \left(  \frac{1}{2} (D^2 \psi( x_1, \xi_0) (x,\eta) , (x,\eta) )  - y\eta\right)  } \tilde{r}_{N,1}(x,\eta ) u(y) dy d\eta$$
We will use the following lemma, proved in the appendix \ref{appendix_lemma_stationnary_phase} : 
\begin{lem}\label{Lemma_1_stationnary_phase}
Let $\tilde{b}$ be a symbol in $S(\langle \rho \rangle^N)$. Then, there exists a symbol $b \in S(\langle \rho \rangle^N)$ such that for all $0 < h \leq 1$, 
$$\frac{1}{2\pi h} \int_{T^*\R} e^{\frac{i}{h} \left(  \frac{1}{2} (D^2 \psi( x_1, \xi_0) (x,\eta) , (x,\eta) )  - y\eta\right)  } \tilde{b}(x,\eta ) u(y) dy d\eta = \mathcal{M}(d_{\rho_0}F ) \op(b)u(x)$$
Moreover, there exists a universal integer $M^\prime \in \N$ such that 
$b$ satisfies : for all $\alpha \in \N^2$, 
$$\langle \rho \rangle^N  |\partial^\alpha b(\rho)| \leq C_\alpha \sup_{ |\beta| \leq  |\alpha| +M^\prime} \sup_{\rho \in T^*\R} \left(  |\partial^\beta \tilde{b}(\rho)| \langle \rho \rangle^N \right)  $$
where $C_\alpha$ depends on $d_{\rho_0} F$. 
\end{lem}
\vspace{0.5cm}
By applying lemma \ref{Lemma_1_stationnary_phase} (in the case $h=1$ in the lemma), we can find a symbol $r_{N,h}$ such that $$\tilde{R}_{N,1}  = \mathcal{M}_1(d_{\rho_0}F) \text{Op}_1(r_{N,h})$$
To conclude the treatment of the first part of the remainder, we compute : 
\begin{align*}
|| \tilde{R}_{N,1}  \mathcal{M}_1 (\kappa)[P \Psi_0]  ||_{L^2} &= || \mathcal{M}_1 (d_{\rho_0}F) \text{Op}_1(r_{N,h})\mathcal{M}_1(\kappa) [P \Psi_0] ||_{L^2}  \\
& \leq ||\text{Op}_1 (r_{N,h})||_{H_1(\langle \rho\rangle^{3N}) \to L^2 }  \times || \mathcal{M}_1(\kappa)  [P \Psi_0] ||_{H_1(\langle \rho\rangle^{3N}) } \\
&\leq C_{M}(r_{N,h}) || \kappa||^{3N}  K_{N,deg P} N_{\infty}(P)
\end{align*}
by using Corollary \ref{cor_norm_H_1}, 
where $C_{M}(r_{N,h})$ depends on the first $M$ semi-norms of $r_{N,h}$ in $S(\langle \rho \rangle^{3N})$, which, in turn depends on the first $M + M^\prime$ semi-norms of $\tilde{r}_{N,h}$ in $S(\langle \rho \rangle^{3N})$ according to Lemma \ref{Lemma_1_stationnary_phase}. By (\ref{tilde_r_N_1}), this can be controlled by some constant $C_{3N + M + M^\prime}(\psi) ||a||_{C^{N + M + M^\prime}}$. \\

Let's turn to the second remainder in (\ref{precise_expansion}). We want to control 
$$ \frac{1}{2\pi} \int_{\R^2} e^{i \left(  \frac{1}{2} (D^2 \psi( x_1, \xi_0) (x,\eta) , (x,\eta) )  - y\eta\right)  } \tilde{r}_N(x,\eta ) u(y) dy d\eta$$ 
Recalling the precise description of $\tilde{r}_N$ in (\ref{devp_exp}), we set, for $s \in [0,1]$ : 
$$ \tilde{R}_s u (x) =  \frac{1}{2\pi} \int_{\R^2} e^{i \left(  \frac{1}{2} (D^2 \psi( x_1, \xi_0) (x,\eta) , (x,\eta) ) + i sh^{1/2} r_3^\psi(x,\eta;h) - y\eta\right)  } r_3^\psi(x,\eta;h )^N a(x_1 + h^{1/2}x, \xi_0 + h^{1/2} \eta) u(y) dy d\eta$$
and we want to estimate $|| \tilde{R}_s \mathcal{M}_1 (\kappa) [P \Psi_0] ||_{L^2}$ uniformly in $s \in [0,1]$. The symbol $$b_N(x,\eta)\coloneqq  r_3^\psi(x,\eta;h)^N a(x_1 + h^{1/2}x, \xi_0 + h^{1/2} \eta)$$ lies in the symbol class $S(\langle \rho \rangle^{3N})$, with a control on its semi-norms due to (\ref{Taylor_remainder}). 
Let's admit the following lemma, whose proof is also put in the appendix \ref{appendix_lemma_stationnary_phase}.

\begin{lem}\label{Lemma_2_stationnary_phase}
For every $s \in [0,1]$, there exists $B_s (\cdot ) \in S\left( \langle \rho \rangle^{6N} \right)$ such that : 
\begin{itemize}
\item $\tilde{R}_s^* \tilde{R}_s = \text{Op}_1(B_s)$ ; 
\item There exists a universal $M \in \N$ such that for all $\alpha \in \N^2$, for all $s \in [0,1]$, with some universal constants $C_\alpha$, 
$$ \sup_{\rho} |\partial^\alpha B_s(\rho)| \leq C_\alpha \left( \sup_{\rho , |\beta| \leq |\alpha| + M} d^{\beta} b_N(\rho) \langle \rho \rangle^{-3N} \right)^2 \langle \rho \rangle^{6N} $$ 
\end{itemize}
\end{lem}
\vspace{0.5cm} 
This lemma allows us to control 
\begin{align*}
||\tilde{R}_s||^2_{H_1(\langle \rho \rangle^{3N}) \to L^2 } & \leq ||\tilde{R}_s^* \tilde{R}_s||_{H_1\left( \langle \rho \rangle^{3N} \right) \to H_1\left( \langle \rho \rangle^{-3N}  \right)} \\
&\leq ||\text{Op}_1 (B_s)||_{H_1\left( \langle \rho \rangle^{3N}\right) \to H_1 \left( \langle \rho \rangle^{-3N} \right) } \\
&\leq C_N \sup_{|\alpha| \leq M} \sup_{\rho} |\left( \partial^\alpha B_s(\rho) \right) \langle \rho \rangle^{-6N}| \\
&\leq  C_N \left( \sup_{|\beta| \leq 2M} \sup_{\rho }d^\beta  b_N(\rho) \langle \rho \rangle^{-3N} \right)^2 \\
&\leq \left( C_{3N+ M^\prime} (\psi) ||a||_{C^{N + M^\prime}}\right)^2
\end{align*}
We finally conclude as before for $\tilde{R}_{N,1}$ by using Corollary \ref{cor_norm_H_1}. This concludes the proof of Proposition \ref{Prop_propagation_coherent_states}.  
\end{proof}

\section{Dynamical preliminaries}\label{Section_preliminaires_dyna}
\subsection{Hyperbolic dynamics} \label{Subsection_hyperbolic_dynamics}

We assumed that $F$ is hyperbolic on the trapped set $\mathcal{T}$. As already mentioned, we can fix an adapted Riemannian metric on $U$ such that the following stronger version of the hyperbolic estimates are satisfied for some $\lambda_0 >0$ : for every $\rho \in \mathcal{T}$, $n \in \N$, 
\begin{align}
v \in E_u( \rho) \implies ||d_\rho F^{-n}(v)|| \leq e^{-\lambda_0 n } ||v||\\
v \in E_s( \rho) \implies ||d_\rho F^{n}(v)|| \leq e^{-\lambda_0 n } ||v||
\end{align}

\begin{nota}
We now use the induced Riemannian distance on $U$ and denote it $d$. \\
\end{nota}

If $\rho \in \mathcal{T}$,  $n \in \Z$, we use this Riemannian metric to define the unstable Jacobian $J^u_n (\rho)$ and stable Jacobian $J^s_n (\rho)$ at $\rho$ by : 

\begin{align}\label{Def_jacobian}
v \in E_u( \rho) \implies ||d_\rho F^{n}(v)|| = J^u_n (\rho) ||v|| \\
v \in E_s( \rho) \implies ||d_\rho F^{n}(v)|| = J^s_n (\rho)||v||
\end{align}
These Jacobians quantify the local hyperbolicity of the map. Since $F$ is volume preserving, $J^u_n(\rho) J^s_n(\rho) \sim 1$. 

\begin{rem}
If we define unstable and stable Jacobian $\tilde{J}^u_n$ and $\tilde{J}^s_n$ using another Riemannian metric, then, for every $n \in \Z$ and $\rho \in \mathcal{T}$, 
\begin{equation*}
\tilde{J}^u_n(\rho) \sim J^u_n(\rho)\quad ; \quad  \tilde{J}^s_n(\rho) \sim J^s_n(\rho)
\end{equation*}
\end{rem}

From the compactness of $\mathcal{T}$, there exists $\lambda_1 \geq \lambda_0$ which satisfies 

\begin{align}
 \forall n \in \N , \forall \rho \in \mathcal{T} \; ; \; 
e^{n\lambda_0} \leq J^{u}_n (\rho) \leq  e^{n\lambda_1 }\; \text{ and }  \; e^{-n\lambda_1} \leq J^{s}_n (\rho) \leq  e^{-n\lambda_0}  \label{lamba0et1}
\end{align}
In particular, the following Lyapounov exponents are well-defined 
\begin{align*}
\lambda_{max} = \sup_{\rho \in \mathcal{T}} \limsup_n \frac{1}{n}\log J^u_n(\rho) \\
\lambda_{min} = \inf_{\rho \in \mathcal{T}} \liminf_n \frac{1}{n} \log J^u_n(\rho) >0
\end{align*}
We cite here standard facts about the stable and unstable manifolds (see for instance \cite{KH}, Chapter 6). 
\begin{lem} \label{classical_hyperbolic}
For any $\rho \in \mathcal{T}$, there exist local stable and unstable manifolds $W_s(\rho), W_u( \rho) \subset U$ satisfying, for some $\varepsilon_1 >0$ (only depending on $F$) : 
\begin{enumerate}[label = (\arabic*)]
\item $W_s(\rho), W_u( \rho)$ are $C^\infty$-embedded curves, with the $C^\infty$ norms of the embedding uniformly bounded in $\rho$. 
\item the boundaries of $W_u(\rho)$ and $W_s(\rho)$ do not intersect $\overline{B(\rho, \varepsilon_1)} $ \footnote{in other words, there exists a smooth curve $\gamma : [-\delta,\delta]  \to U$ such that $\gamma(0)=\rho$, $\text{ran}( \gamma )\subset W_{u/s}(\rho)$ and $\overline{B(\rho, \varepsilon_1)} \cap W_{i/s}(\rho) =  \gamma([-\delta/2, \delta/2])$ : it means that the size of the unstable and stable manifolds is bounded from below uniformly. } and $W_{u/s}(\rho) \subset B(\rho, 2 \varepsilon_1)$ (these are local unstable/stable manifolds). 
\item $W_s(\rho) \cap W_u(\rho) = \{ \rho \}$, $T_\rho W_{u/s} (\rho) = E_{u/s} (\rho) $ 

\item $F(W_s(\rho) ) \subset W_s \left(F(\rho) \right)$ and $F^{-}(W_u(\rho) ) \subset W_u \left(F^{-1}(\rho) \right)$
\item \begin{enumerate}
\item For each $\rho^\prime \in W_s(\rho), d( F^{ n } (\rho), F^{  n}(\rho^\prime) ) \to 0$. 
\item For each $\rho^\prime \in W_u(\rho), d( F^{- n } (\rho), F^{  -n}(\rho^\prime) ) \to 0$. 
\end{enumerate}
\item Let $\theta>0$ satisfying $ e^{-\lambda_0} < \theta < 1$. There exists $C>0$ (independent of $\varepsilon_1$) such that the following holds : 
\begin{enumerate} 
\item
If $\rho^\prime \in U$ satisfies $d(F^{ i }(\rho) , F^{i }( \rho^\prime) ) \leq \varepsilon_1$ for all $i = 0 , \dots, n$  then $d\left( \rho^\prime, W_s(\rho) \right)\leq C \theta^n \varepsilon_1$ and for $0 \leq i \leq n$, $d(F^{i}(\rho), F^{i}(\rho^\prime) ) \leq C\varepsilon_1 \theta^{\min(i,n-i)}$. 
\item If $\rho^\prime \in U$ satisfies $d(F^{ -i }(\rho) , F^{-i }( \rho^\prime) ) \leq \varepsilon_1$ for all $i = 0 , \dots, n$  then $d\left( \rho^\prime, W_u(\rho) \right)\leq C \theta^n \varepsilon_1$ and for $0 \leq i \leq n$, $d(F^{-i}(\rho), F^{-i}(\rho^\prime) ) \leq C\varepsilon_1 \theta^{\min(i,n-i)}$. 
\end{enumerate}
\item If $\rho, \rho^\prime \in \mathcal{T}$ satisfy $d(\rho, \rho^\prime) \leq \varepsilon_1$, then  $W_u (\rho) \cap W_s (\rho^\prime)$ consists of exactly one point of $\mathcal{T}$. 
\end{enumerate}
\end{lem}
Below, we will require that $C \varepsilon_1 < 1$. Up to making $\varepsilon_1$ smaller, we assume this holds. 
 
For our purpose, we will need a more precise version of these results. The following lemmas are an adaptation of Lemma 2.1 in \cite{NDJ19} to our setting, appearing also in \cite{Vacossin}, where they have been partially proved. 

\begin{lem}\label{Local_hyperbolic_1}
There exist constants $\varepsilon_1>0$ and $C >0$ depending only on $(U,F)$, such that for all $\rho, \rho^\prime \in U$, 
\begin{enumerate}[label = (\arabic*)]
\item if $\rho \in \mathcal{T}$ and $\rho^\prime \in W_s( \rho)$ satisfy $d(\rho, \rho^\prime) \leq \varepsilon_1$, then 
\begin{equation} 
C^{-1} J_n^s(\rho) d( \rho, \rho^\prime)  \leq d\left( F^{n}(\rho) , F^{n} (\rho^\prime) \right) \leq CJ_n^s(\rho) d( \rho, \rho^\prime) \quad , \quad \forall n \in \N 
\end{equation}
\item if $\rho \in \mathcal{T}$ and $\rho^\prime \in W_u( \rho)$ satisfy $d(\rho, \rho^\prime) \leq \varepsilon_1$, then 
\begin{equation}
C^{-1}J_{-n}^u(\rho) d( \rho, \rho^\prime) \leq d\left( F^{-n}(\rho) , F^{-n} (\rho^\prime) \right) \leq CJ_{-n}^u(\rho) d( \rho, \rho^\prime) \quad , \quad  \forall n \in \N 
\end{equation} 
\end{enumerate}
\end{lem}

\begin{proof}
We prove (1). (2) is proved in a similar way by inverting the time direction.  
Let $\rho \in \mathcal{T}, \rho^\prime \in W_s(\rho)$. Since $T_{\rho}(W_s(\rho) )  =E_s(\rho) $ and $d_\rho F \left(E_s(\rho)\right) = E_s(F(\rho))$, the Taylor development of $F$ along $W_s(\rho)$ gives : 
\begin{equation}
d(F(\rho), F(\rho^\prime)) =  J^s_1(\rho) d(\rho, \rho^\prime) + O\left(  d(\rho, \rho^\prime)^2 \right) = J^s_1(\rho) d(\rho, \rho^\prime)  \left( 1 + O\left( d(\rho, \rho^\prime)\right)  \right)
\end{equation}
since $J^s_1 \geq e^{- \lambda_1}$. Applying this equality with $F^k (\rho)$ and $F^k (\rho^\prime)$ instead of $\rho$ and $\rho^\prime$, and recalling that, by lemma \ref{classical_hyperbolic}, $d(F^k(\rho), F^k (\rho^\prime)) \leq C\theta^k d(\rho, \rho^\prime)$, 
we can write, 
\begin{equation}
 d(F^{k+1}(\rho), F^{k+1} (\rho^\prime)) = J^s_1(F^k(\rho)) d(F^k(\rho), F^k (\rho^\prime)) (1 + O(\theta^k \varepsilon_1) )
\end{equation}
By this last inequality and the chain rule, we have 
\begin{equation}\label{inequality_lemma_hyperbolic}
 J^s_n(\rho) d(\rho, \rho^\prime) \prod_{k=0}^{n-1} (1 - C \theta^k \varepsilon_1) \leq d(F^n(\rho), F^n (\rho^\prime)) \leq J^s_n(\rho) d(\rho, \rho^\prime) \prod_{k=0}^{n-1} (1 + C \theta^k \varepsilon_1) 
\end{equation}
We conclude by noting that 
$$  \prod_{k=0}^{n-1} (1 + C \theta^k \varepsilon_1)  \leq \prod_{k=0}^{+ \infty} (1 + C \theta^k \varepsilon_1) < + \infty \; ; \; \prod_{k=0}^{n-1} (1 - C \theta^k \varepsilon_1) \geq  \prod_{k=0}^{\infty} (1 - C \theta^k \varepsilon_1)  \geq C^{-1}$$
(note that in the last inequality and in (\ref{inequality_lemma_hyperbolic}) , we need to ensure that $\varepsilon_1 C<1$ so that the product is effectively non zero). 
\end{proof}

The following lemma gives a stronger version of (6) in Lemma \ref{classical_hyperbolic} (it has been proved in \cite{Vacossin} - Lemma 3.10-, as the following corollary - Corollary 3.11). 

\begin{lem}\label{Local_hyperbolic_2}
There exist $C >0$ and $\varepsilon_1 >0$, depending only on $(U,F)$, such that for all $\rho, \rho^\prime \in U$ and $n \in \N$ : 
If $\rho \in \mathcal{T}$ and $d\left( F^i(\rho) , F^i(\rho^\prime) \right) \leq \varepsilon_1$ for all $i \in \{0, \dots, n \}$ then 
\begin{equation} \label{close_to_leaves}
d\left( \rho^\prime , W_s (\rho) \right) \leq \frac{C}{J_n^u(\rho)} d\left( F^n(\rho^\prime), W_s(F^n(\rho)) \right)
\end{equation}
\begin{equation}
d\Big( F^n( \rho^\prime) , W_u (F^n(\rho)) \Big) \leq \frac{C}{J_n^u(\rho)} d\left( \rho^\prime, W_u(\rho) \right)
\end{equation}
and
\begin{equation}\label{control_of_jacobian}
|| d_{F^{n}(\rho^\prime)} F^{-n} || , ||d_{\rho^\prime}F^n|| \leq C J^u_n(\rho)
\end{equation}
\end{lem}

As an immediate consequence of this lemma, we get : 
\begin{cor}\label{cor_control_jacobian}
There exists $C >0$ and $\varepsilon_1 >0$ (depending only on $(U,F)$) such that for all $\rho, \rho^\prime \in \mathcal{T}$ and $n \in \N$ : 
\begin{enumerate}[label = (\arabic*)]
\item if $d\left( F^i(\rho) , F^i(\rho^\prime) \right) \leq \varepsilon_1$ for all $i \in \{0, \dots, n \}$, then 
\begin{equation}
C^{-1} J^u_n(\rho) \leq  J^u_n(\rho^\prime) \leq CJ^u_n(\rho)
\end{equation}
\item if $d\left( F^{-i}(\rho) , F^{-i}(\rho^\prime) \right) \leq \varepsilon_1$ for all $i \in \{0, \dots, n \}$, then 
\begin{equation} 
  C^{-1} J^s_{-n}(\rho) \leq J^s_{-n}(\rho^\prime) \leq CJ^s_{-n}(\rho)
\end{equation}
\end{enumerate}
\end{cor}

We also record the following fact (see for instance \cite{NDJ19} - Lemma 2.4). 
\begin{lem} \label{adapted_chart_0}
There exist $\varepsilon_1>0$ and $C>0$ such that the following holds :  For every $\rho \in \mathcal{T}$, there exists a symplectic coordinate chart $\kappa_\rho : V_\rho \to W_\rho \subset \R^2$ such that
\begin{itemize}[itemsep=0.2em]
\item $B(\rho, \varepsilon_1) \subset V_\rho$; 
\item $\kappa_\rho(\rho) = (0,0)$
\item $\kappa_\rho \left( W_u(\rho)  \cap V_\rho \right) = \{ (u,0) , u \in \R \} \cap W_\rho$
\item $d\kappa_\rho (E_s(\rho)) = \R(0,1)$
\item For any $N \in \N$, the $C^N$ norm of $\kappa_\rho$ is bounded uniformly in $\rho$. 
\end{itemize}

\end{lem}

Finally, we conclude by a lemma concerning the linearized dynamics. If $\rho \in \mathcal{T}$ and $\rho^\prime \in W_u(\rho)$, the tangent space $T_{\rho^\prime} W_u(\rho)$ will naturally be denoted $E_u(\rho^\prime)$ and if $v \in T_{\rho^\prime} U$, we note $d(v, E_u(\rho^\prime))$ the distance between $v$ and $E_u(\rho^\prime)$ using the Riemanniann metric on $T_{\rho^\prime} U$. 

\begin{lem}\label{Lemma_linearized_dynamics}
There exist $\varepsilon_1 >0$, $ \gamma \in (0,1)$ and $C>0$ such that the following holds. Assume that $\rho \in \mathcal{T}$, $\rho^\prime \in W_u(\rho)$, $v_0 \in T_{\rho^\prime} U$ and $n \in \N$ satisfy :  $\forall i \in \{0, \dots, n \}$, $d(F^i(\rho), F^i(\rho^\prime)) \leq \varepsilon_1$. Assume also that $||v_0|| =1$ and that $d(v_0, E_u(\rho^\prime) ) \leq \gamma$.  Let's note $$v_n = \frac{d_{\rho^\prime} F^n (v_0)}{ ||d_{\rho^\prime} F^n (v_0) || } \in T_{F^n(\rho^\prime)} U  $$
Then (see Figure \ref{figure_linearized_dynamics}), $$d\Big(v_n, E_u(F^n(\rho^\prime) )\Big) \leq C J^u_n(\rho)^{-2}  d( v_0, E_u(\rho^\prime) )$$
\end{lem}

\begin{figure}
\includegraphics[scale=0.4]{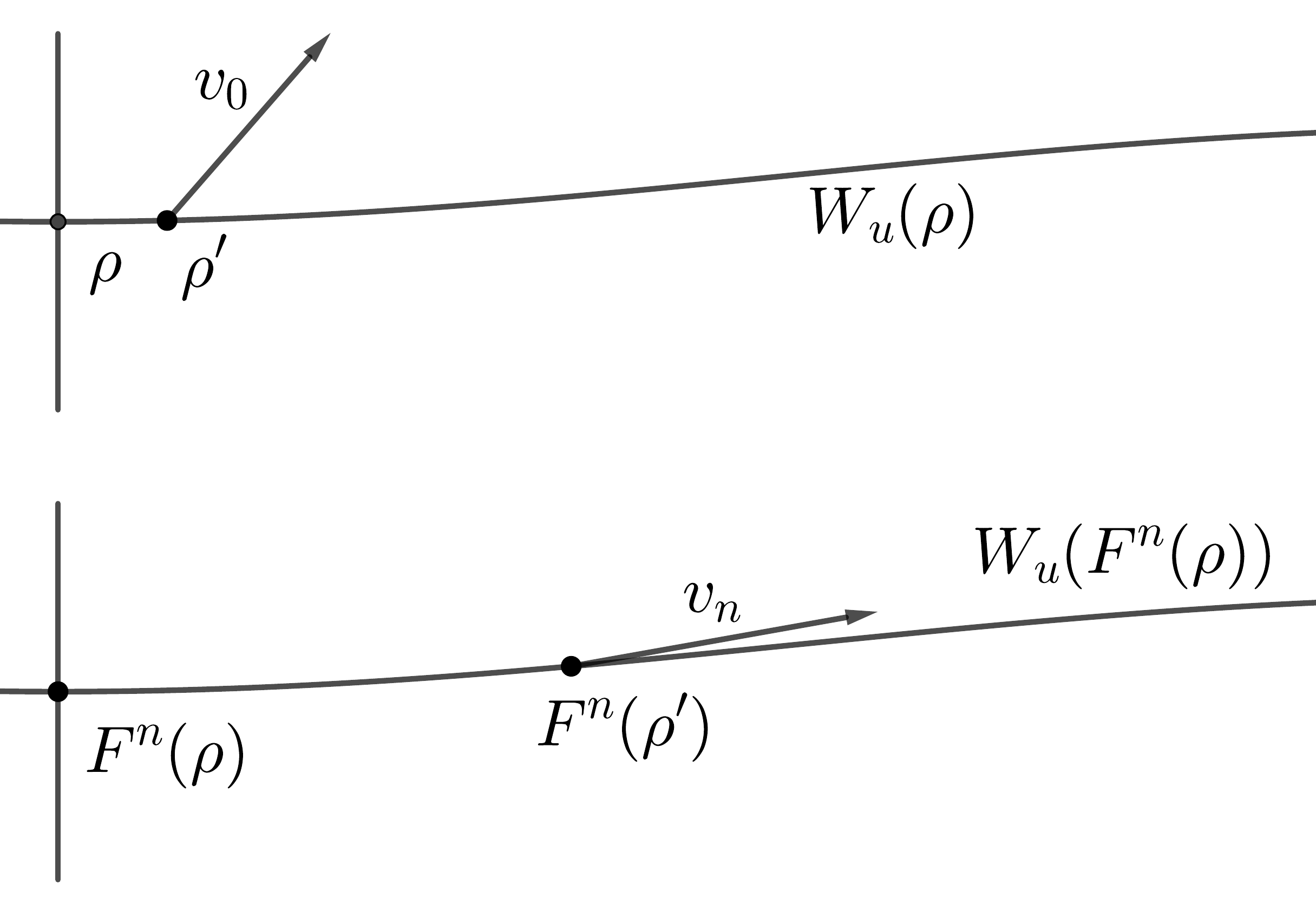}
\caption{The linearized dynamics makes the vector closer and closer to the tangent space of the unstable manifold. See Lemma \ref{Lemma_linearized_dynamics}. The vertical direction corresponds the to the stable direction, in which the dynamics contracts. }
\label{figure_linearized_dynamics}
\end{figure}

\begin{rem}
This is a form of inclination lemma : the tangent vectors are attracted toward the unstable direction upon the evolution. We provided a quantitative statement. 
The assumption $d(v_0, E_u(\rho^\prime) ) \leq \gamma$ is a transversality assumption : it says that $v_0$ has to be sufficiently transverse to the stable direction. 
\end{rem}

\begin{proof}
First note that due to the assumption on $\rho$ and $\rho^\prime$ and Lemma \ref{Local_hyperbolic_1}, $$d(F^i(\rho), F^i(\rho^\prime) \leq C\theta^{n-i} d(F^n(\rho), F^n(\rho^\prime)) \leq C \theta^{n-i}   \varepsilon_1$$
for some $0 < \theta < 1$ and for $0 \leq i \leq n$. 
 We use coordinates charts $\kappa_i$ centered at $F^i(\rho)$ (for $0 \leq i \leq n$), given by Lemma \ref{adapted_chart_0}. Let's note $(u^i, s^i)$ the coordinates in $\kappa_i$. Since $ \kappa_i (W_u(F^i(\rho) ) = \{ (u^i, 0) \}$, the map $F$ between the charts $\kappa_{i-1}$ and $\kappa_{i}$ is given by : 
$$ \kappa_i \circ F \circ \kappa_{i-1}^{-1} (u^{i-1}, s^{i-1} ) = (\nu_i u^{i-1}+ \alpha_i (u^{i-1}, s^{i-1} ) , \mu_i s^{i-1}+ \beta_i (u^{i-1}, s^{i-1} )) $$
with $\beta_i(u^{i-1}, 0) = 0$, $d\alpha_i(0,0) = 0$ and $d\beta_i(0,0) = 0$. 
Remark that $\nu_1 \dots \nu_i  \sim J^u_i(\rho) \sim (\mu_1 \dots \mu_i)^{-1}$ for $1 \leq i \leq n$. 

Let us note $w_0 = d_{\rho^\prime} \kappa_0 (v_0)$ and $\tilde{w}_n = d_{F^n(\rho^\prime)} \kappa_n (v_n)$. Hence, we want to show that $ d(\tilde{w}_n, \R e_u) \leq C J^u_n(\rho)^{-1}d(w_0,\R e_u )$ where $\R e_u = \{ (u,0) \}$. Here, to compute the distance, both between points and tangent vectors, we can simply use the usual euclidean distance in $\R^2$. 
Let us also introduce  $w_i = d_{\rho^\prime} (\kappa_i \circ F^i )(v_0)$ and write $w_i = (w_i^u, w_i^s)$. With these notations, we have $w_n = || d_{\rho^\prime} F^n(v_0)||  \tilde{w}_n $ and 
$$d(\tilde{w}_n, \R e_u) = \frac{w_n^s }{|| d_{\rho^\prime} F^n(v_0)||}$$ 
Since $ || d_{\rho^\prime} F^n(v_0) || \sim ||w_n||$, we are reduced to prove that 
\begin{equation}\label{eq_goal_0}
\frac{|w_n^s|}{||w_n||} \leq C J^u_n(\rho)^{-2} |w_0^s|
\end{equation}
If $\gamma$ is small enough, we can deduce from the transversality assumption on $v_0$ that $|w_0^u| \geq \frac{1}{2} |w_0^s|$. In particular, $||w_0||^2 \geq \frac{4}{3} |w_0^s|^2$. 
 Let us note $(u^i,0)$ the coordinates of $F^i(\rho^\prime)$ in the charts $\kappa_i$ and recall that $|u^i| \leq C \theta^{n-i} \varepsilon_1$. We have the relations
\begin{align*}
 w_i^u = \nu_i w_{i-1}^u + d \alpha_i(u^{i-1},0)  \cdot w_{i-1}\\
  w_i^s = \mu_i w_{i-1}^s +  d\beta_i (u^{i-1},0) \cdot w_{i-1}
\end{align*}
Since $\beta_i (u,0) = 0$, $ d\beta_i(u^{i-1},0) \cdot  w_{i-1} = \partial_{s^i} \beta_i(u^{i-1},0) w_{i-1}^s$. Moreover, $d\beta_i(0,0) = 0$, and hence, 
$|  \partial_{s^i} \beta_i(u^{i-1},0) w_{i-1}^s | \leq C |u^{i-1}| |w_{i-1}^s| \leq C \theta^{n+1-i} \varepsilon_1 |w_{i-1}^s|$. this gives, 
$$|w_i^s| \leq  (\mu_i + C\theta^{n+1-i} \varepsilon_1) \dots (\mu_1 + C \theta^n) |w_0^s|$$
$$ |w_i^s| \leq \mu_1 \dots \mu_i \prod_{k=1}^i \left( 1 + \frac{C\varepsilon_1 \theta^{n+1-k}}{\mu_k} \right)  |w_0^s|$$
Since $\mu_k \geq c$ for some $c >0$ and for all $1 \leq k \leq n$, we can estimate 
$$ \prod_{k=1}^i \left( 1 + \frac{C\varepsilon_1 \theta{n+1-k}}{\mu_k} \right) \leq \prod_{k=0}^{i-1} \left( 1 + C\varepsilon_1 c^{-1} \theta^k  \right) \leq \prod_{i=0}^\infty \left( 1 + C\varepsilon_1 c^{-1} \theta^i  \right)<  + \infty$$ 
As a consequence, $|w_i^s| \leq C J^u_i(\rho)^{-1} |w_0^s|$. We now turn to a lower bound for $||w_n||$. 
From 
$w_i^u = \nu_i w_{i-1}^u + O(|u^{i-1}| ||w_{i-1}||) $, we find that 
$$ |w_i^u| \geq \nu_i |w_{i-1}^u| - C |u^{i-1}| ||w_{i-1}|| \geq  \nu_i |w_{i-1}^u| - C \theta^{n+1-i} ||w_{i-1}||$$
We observe that $||w_{i-1} || \leq |w_{i-1}^u| + |w_{i-1}^s| \leq |w_{i-1}^u| + C J^u_{i-1}(\rho)^{-1}  |w_0^s|$, which gives that 
$$  |w_i^u| \geq |w_{i-1}^u | (\nu_i - C \theta^{n+1-i} ) -  C \theta^{n+1-i} J^u_{i-1}(\rho)^{-1} |w_0^s| $$
Recall that for $\theta_1 = e^{-\lambda_1}$, we have $J^u_i(\rho)^{-1} \leq  \theta_1^i$, so that for some $\theta_2 \in ]\max(\theta, \theta_1),1[$,  $\theta^{n+1-i} J^u_i(\rho)^{-1}  \leq C\theta_2^n$. Iterating this formula, we find that 
$$ |w_n^u|  \geq ( \nu_1 - C \theta^{n}) \dots (\nu_n - C \theta)|w_0^u| - \sum_{i=1}^{n-1} (\nu_n - C \theta) \dots (\nu_i - C \theta^{n+1-i} )\theta_2^n  |w_0^s|$$
By similar arguments as in the case of $|w_n^s|$, we can show that 
 $(\nu_n - C \theta) \dots (\nu_1 - C \theta^{n} ) \geq C^{-1} J^u_n(\rho)$ and $| (\nu_n - C \theta) \dots (\nu_1 - C \theta_{n} )| \leq \nu_1 \dots \nu_n \leq C J^u_n(\rho)$. 
 As a consequence, and using the fact that $|w_0^u| \geq  1/4 ||w_0||$ (by the transversality assumption), we find that 
 $$|w^u_n| \geq C^{-1}(1 - \theta_2^n) J^u_n(\rho)  ||w_0|| \geq C^{-1} J^u_n(\rho) ||w_0||.$$
 We conclude that $||w_n|| \geq |w_n^u| \geq  C^{-1} J^u_n(\rho)  ||w_0||$, which gives (\ref{eq_goal_0}). 
\end{proof}

\subsection{Topological pressure. Dimensions. }

\subsubsection{Topological pressure.}

We recall the definition and some formulas for the topological pressure associated with a continuous function $\varphi : \mathcal{T} \to \R$. The dynamical system we consider is the restriction of $F$ on $\mathcal{T}$. We consider a distance function $d$ on $\mathcal{T}$. For $n \in \N$ and $\epsilon >0$, we say that a subset $E \subset \mathcal{T}$ is $(n,\epsilon)$ separated if for every $x,y \in E, x\ \neq y$, there exists $0 \leq i \leq n-1$, $d( F^i(x), F^i(y) ) > \epsilon$. 

\begin{defi}\label{def_pressure}
If $\varphi$ is a continuous function on $\mathcal{T}$, the topological pressure associated with $\varphi \in C(\mathcal{T}, \R)$ is defined as 
$$ P(\varphi) = \lim_{\epsilon \to 0 } \limsup_{ n \to + \infty} \frac{1}{n}\log  P_0(\varphi, n,\epsilon)$$
where 
$$ P_0(\varphi, n,\epsilon) = \sup \left\{ \sum_{x \in E} \exp \left( \sum_{i=0}^{n-1} \varphi (f^i(x) ) \right)  \; ; \; E \text{ is } (n,\epsilon) \text{ separated} \right\}$$

\end{defi}

In this paper, we will use another formula for the pressure. To state it, let us introduce a few notations : if $\mathcal{Q}$ is a finite open cover of $\mathcal{T}$, we note $\text{diam} \mathcal{Q} = \sup_{A \in \mathcal{Q}} \text{diam} A$ and for $n \in \N$, $\mathcal{Q}^{\wedge n}$ is the open cover of $\mathcal{T}$ by the sets $ \bigcap_{i=0}^{n-1}f^{-i} (A_i)$ where $A_0, \dots, A_{n-1} \in \mathcal{Q}$.
For $\varphi : \mathcal{T} \to \R$ continuous, $n \in \N$ and an open cover $\mathcal{Q}$ of $\mathcal{T}$, we define
$$ P_1(\varphi, n, \mathcal{Q}) = \inf \left\{ \sum_{A \in \alpha} \sup_{x \in A} \exp \left( \sum_{i=0}^{n-1} \varphi(f^i (x) ) \right) \; ; \; \alpha \subset \mathcal{Q}^{\wedge n} , \mathcal{T} \subset \bigcup_{A \in \alpha} A \right\}$$
\begin{prop}\label{Prop_formula_pressure}\cite{Wal} (Theorem 1.6). 
The following formula holds: for any $\varphi \in C(\mathcal{T}, \R)$, 
\begin{equation}
P(\varphi) = \lim_{\text{diam}\mathcal{Q} \to 0 } \lim_{ n \to \infty} \frac{1}{n} \log P_1(\varphi, n , \mathcal{Q})
\end{equation}
\end{prop}
Note that in particular, it asserts that the limit in $n$ exists for all open cover $\mathcal{Q}$. 
 
\subsubsection{Dimensions}

Let us recall the definition of the upper box dimension of a compact metric space $(X,d)$. We denote by $N_X(\varepsilon)$ the minimal number of open balls of radius $\varepsilon$ needed to cover  $X$. Then, the upper box dimension of $X$ is defined as : 

\begin{equation}
\overline{\dim} X \coloneqq \limsup_{\varepsilon \to 0} \frac{\log N_X(\varepsilon)}{- \log \varepsilon}
\end{equation}

In particular, if $\delta > \overline{\dim}_X $, there exists $\varepsilon_0>0$ such that for every $\varepsilon \leq \varepsilon_0$, $N_X(\varepsilon) \leq \varepsilon^{-\delta}$. 

We recall the following well known result (see for instance \cite{Bar}, Theorem 4.3.2) : 

\begin{prop}
Let $s_0$ be the unique root of the equation $P(- s \varphi_u) =0 \; ; \;s \in \R$. Then, 
For every $\rho \in \mathcal{T}$, $\overline{\dim} \left(  \mathcal{T} \cap W_u(\rho) \right)  = \overline{\dim} \left( \mathcal{T} \cap W_s(\rho) \right)= s_0$. 
Moreover, $\overline{\dim} \mathcal{T}  = 2 s_0$. 
\end{prop}

\begin{rem}
In fact, this holds also for the Hausdorff dimension and the lower-box dimension but we will mainly use the upper-box dimension for practical and technical reasons. In the following, we note $s_0 = d_H$. 
\end{rem}

We will need the slightly more precise following result, which allows to control $N_{W_{u/s}(\rho) \cap \mathcal{T}}$ uniformly in $\rho$ : 
\begin{prop}\label{upper_box_dim}
There exists $\varepsilon_1 >0$ such that the following holds.  For every $\varepsilon>0$, there exists $C_\varepsilon>0$ such that for every $\rho \in \mathcal{T}$, if $X_\rho= W_{u/s}(\rho) \cap \mathcal{T} $, 
\begin{align*}
N_{X_\rho}(r) \leq C_\varepsilon r^{-(d_H + \varepsilon)} ; \; \forall r \leq \varepsilon_1
\end{align*}
\end{prop}

\begin{proof}
Obviously, this holds at every $\rho \in \mathcal{T}$ with $C_\varepsilon$ \textit{a priori} depending on $\rho$. 
The uniformity is a consequence of the fact that the holonomy maps are Lipschitz, with uniform Lipschitz norm due to the compactness of $\mathcal{T}$ (see for instance \cite{Vacossin}, Corollary 3.3). Then, due to the compactness of $\mathcal{T}$, one can estimate $N_{W_{u/s}(\rho) \cap \mathcal{T}}(\varepsilon)$ by considering only a finite number of (un)stable leaves as references and apply : 
\emph{Assume that $(X,d)$, $(Y,d^\prime)$ are compact metric spaces and $f : X \to Y$ is $C$-Lipschitz. Then, for every $\varepsilon >0$, 
$$N_{f(X)}(\varepsilon) \leq N_X(\varepsilon/C)$$}
\end{proof}

We finish by a lemma estimating the number of balls of size $\delta$ needed to cover $\mathcal{T} \cap W_u(\rho_0) \cap J$ where $J \subset W_u(\rho_0)$ is an interval of size $l$. The difference with Proposition \ref{upper_box_dim} is that the size of $J$ can be much small that $\varepsilon_1$.

\begin{lem}\label{Lemma_number_of_intervals}
Let $\rho_O \in \mathcal{T}$. Let $\kappa : U_0 \to V_0 \subset \R$ be a smooth chart such that the image of the local unstable manifold passing through $\rho_0$ is given by a graph 
$$ \kappa(W_u(\rho_O)  \cap U_0 )= \{ (x,g(x)) , x \in I  \} $$
for some open interval $I$. 
For $J \subset I$, let's note $$X(J) = \{ x \in J, \kappa^{-1} (x,g(x)) \in \mathcal{T} \} $$
Then, for every $\varepsilon >0$, there exists $C_\varepsilon>0$ depending only on $\varepsilon$, $F$ and $\kappa$ such that : 
for all $J \subset I$ interval of length $l$ and for all $\delta \in ]0,l]$, 
$$N_{X(J)}(\delta) \leq C_\varepsilon \left( \frac{\delta}{l} \right)^{-(d_H + \varepsilon)}.$$
\end{lem}

\begin{proof}
Let's note $N=N_{X(J)}(\delta)$. 
If $N= 0$, there is obviously nothing to prove. So let's assume that $N \neq 0$ and let's fix a reference point $x_0 \in X(J)$ : to $x_0$ corresponds a point $\rho_0 = \kappa^{-1}(x_0,g(x_0) )  \in \mathcal{T}$ and we are interested in a piece of unstable manifold of $\rho_0$ of size $l$. 

We know that the upper-box dimension of each $\mathcal{T} \cap W_u(\rho)$ is equal to $d_H$. However, since here we are interested by a small piece of an unstable manifold of size $l$, we will expand this piece to reach a size of order 1. 
We note $J_0 = \kappa^{-1} \big( \{  (x,g(x))  , x \in J  \} \big)$ and 
for $m \in \N$, we note $\rho_m = F^m(\rho_0)$ and introduce
$$ T \coloneqq \max \{ m \in \N, F^{m}(J) \subset W_u(\rho_m) \text{ and } \text{diam} F^m(J_0)  \leq \varepsilon_1 \}$$
In particular, the definition of $T$ implies that for all $\rho \in J_0$, $F^m(\rho)$ is well-defined for $0 \leq m \leq T$ and satisfies $d(F^m(\rho), F^m(\rho_0) ) \leq \varepsilon_1$.

\textbf{Claim : } We first claim that if $J^\prime \subset J$ is a subinterval with $X(J^\prime) \neq \emptyset$, then 
$$ \text{diam } J_T^\prime  \sim \frac{ \text{diam} J^\prime}{ J^u_{-T} (\rho_T) } $$
where $J_T^\prime = F^T(J_0^\prime)$ for $J_0^\prime =\kappa^{-1} \big( \{  (x,g(x))  , x \in J^\prime  \} \big)$. 
In particular, it holds for $J^\prime = J$. 

\textbf{Proof of the claim : }
Let's prove this claim and suppose that $J^\prime \subset J$ is an interval of length $l^\prime$ and consider $x^\prime \in X(J^\prime)$. Let's note $\rho^\prime = \kappa^{-1}(x^\prime, g(x^\prime) ) \in \mathcal{T}$. If $\hat{x} \in J^\prime$ and $\hat{\rho} = \kappa^{-1}(\hat{x}, g(\hat{x}) ) \in W_u(\rho^\prime)$, we have 
$$  d(F^T(\rho^\prime) ,F^T(\hat{\rho}) ) \sim \frac{d(\rho^\prime,\hat{\rho})}{J^u_{-T}(F^T(\rho^\prime) ) } \sim \frac{|x^\prime - \hat{x}|}{J^u_{-T}(F^T(\rho^\prime) )} $$
Since $d(F^m(\rho{^\prime}), F^m(\rho_0) ) \leq \varepsilon_1$ for $0 \leq m \leq T$, we have $$J^u_{-T}(F^T(\rho^{\prime}) ) \sim J^u_{-T}(\rho_T )$$
In particular, if we choose $\hat{x}$ such that $|x^\prime - \hat{x} | \geq \text{diam } J^\prime/3$, we have $$\text{ diam } J_T^\prime \geq C^{-1}  \frac{|x^\prime - \hat{x}|}{J^u_{-T}(\rho_T )} \geq C^{-1}  \frac{\text{diam } J^\prime}{J^u_{-T}(\rho_T )}$$
For the converse inequality, assume that $\rho_1, \rho_2 \in J_0^\prime$.
$$d(F^T(\rho_1), F^T(\rho_2 ) ) \leq d(F^T(\rho_1), F^T(\rho^\prime ) ) + d(F^T(\rho^\prime), F^T(\rho_2 ) ) \leq C \frac{|x_1 - x^\prime| + |x^\prime- x_2|}{J^u_{-T}(\rho_T)} \leq C \frac{\text{diam }J^\prime}{J^u_{-T}(\rho_T)}$$ 
which finally gives the required inequality by taking the supremum over $\rho_1$ and $\rho_2$. 

\textbf{End of proof. }
We have, $J^u_{-T}(\rho_T) \text{diam } J_T \sim \text{diam } J \sim l $. 
By definition of $T$, $$\text{diam } F^{T+1} (J_0) \geq \varepsilon_1$$ so that $$\text{diam } F^T(J_0) \geq C^{-1} \varepsilon_1$$ and hence, $J^u_{-T}(\rho_T) \leq C l $ (this $C$ also depends on $\varepsilon_1$, which is not a problem since $\varepsilon_1$ depends only on $F$). 
Let us fix $k>0$, to be determined later. By Proposition \ref{Prop_formula_pressure}, we can cover $J_T \cap \mathcal{T}$ by $N$ balls of diameter at most $k \delta$ with 
$N \leq C_\varepsilon (k \delta)^{- d_H - \varepsilon}$.
Let's choose $\rho_1, \dots, \rho_N \in \mathcal{T} \cap J_T$ such that $$\mathcal{T} \cap J_T \subset \bigcup_{i=1}^{N} B(\rho_i, k \delta).$$ We note $x_i$ the point in $J$ such that $\rho_i = F^T(\rho_i)$ with $\kappa(\rho_i) =(x_i,g(x_i) )$. If $x \in X(J)$, then $\rho \coloneqq F^T(\kappa^{-1}(x,g(x) ) ) \in \mathcal{T} \cap J_T$ and there exists $i \in \{1, \dots, N \}$ such that $d(\rho_i, \rho) \leq k \delta$. As a consequence, $|x- x_i| \leq C J^u_{-T}(\rho_T) d(\rho_i, \rho) \leq  C^\prime  l k\delta$ for some constant $C^\prime$ depending on $F$ and $\kappa$.
We now fix $k = (2 C^\prime l)^{-1}$, so that $X(J)$ can be covered by $N$ intervals of length $\delta$. As a consequence, 
$$N_{X(J)}(\delta) \leq N \leq C_\varepsilon \left( \frac{\delta}{2 C^\prime l}\right)^{- d_H - \varepsilon} = C_\varepsilon^\prime \left(\frac{\delta}{l} \right)^{-d_H -\varepsilon}$$ 

\end{proof}

\subsection{Escape function}
In this subsection, we record the construction of escape functions of \cite{NSZ14}, specialized to our open map $F : \widetilde{D} \subset U \to \widetilde{A} \subset U$. We do not give the proof, since it is entirely contained in \cite{NSZ14} (Lemmata 4.1 - 4.4). 

\begin{lem}
Assume that $\mathcal{V}_2$ is a small neighborhood of $\mathcal{T}$ in which $F$ is well defined. Then, there exists $C_0 >0$ and a neighborhood  $\mathcal{V}_1 \subset \mathcal{V}_2$ of $\mathcal{T}$ such that the following holds : 
For every $\epsilon>0$, there exist functions $\hat{\varphi}_{\pm} = \hat{\varphi}_{\pm,\epsilon} \in \cinf \left( \mathcal{V}_1 \cup F\left(\mathcal{V}_1\right) , [\varepsilon, + \infty [ \right)$ such that 
\begin{align*}
&\hat{\varphi}_\pm(\rho) \sim d\left(\rho , \mathcal{T}_\pm \right)^2 + \epsilon ;  \\
& \pm (\hat{\varphi}_{\pm}(\rho)  - \hat{\varphi}_{\pm} (F(\rho) ) + C_0 \epsilon \sim \hat{\varphi}_{\pm} (\rho);  \\
& \hat{\varphi}_{+} (\rho) + \hat{\varphi}_{-}(\rho)  \sim \epsilon ;\\
& \partial^\alpha \hat{\varphi}_{\pm} (\rho) = O\left( \hat{\varphi}_{\pm}(\rho)^{1- |\alpha|/2} \right).
\end{align*}
The constants in the $\sim$ and $O$ are independent of $\rho \in  \mathcal{V}_1 \cup F\left(\mathcal{V}_1\right)$ and $\epsilon$. 
\end{lem}

Armed with these two functions, we construct the following escape function
\begin{equation}\label{escape_function}
\hat{g}_\epsilon= \log( M\epsilon + \hat{\varphi}_- ) - \log (M \epsilon + \hat{\varphi}_+ ) 
\end{equation}
where $M \gg 1$ is a constant independent of $\epsilon$ and sufficiently large so that the following lemma holds : 

\begin{lem}
For$M$ large enough, there exists $C_1>0$ such that, uniformly with respect to $\epsilon$, 
$$ \rho \in \mathcal{V}_1 \cup F\left(\mathcal{V}_1\right) , d(\rho, \mathcal{T}) \geq C_1 \epsilon \implies \hat{g}_\epsilon(F(\rho) ) - \hat{g}_\epsilon \geq 1/C_1. $$
\end{lem}

Since we will be interested in the dynamics in a neighborhood of $\mathcal{T}$, we fix a smooth cut-off function $\hat{\chi} \in \cinfc( \mathcal{V}_1 \cup F(\mathcal{V}_1) ) $, independent of $\epsilon$, such that $\hat{\chi} = 1$ in a neighborhood of $\mathcal{T}$ and we set 

\begin{equation}\label{escape_function_2}
g_\epsilon = \hat{\chi} \hat{g}_\epsilon
\end{equation}

As a consequence of the construction of $\hat{\varphi}_\pm$, it is also possible  to check that 
\begin{lem}
$g_\epsilon$ satisfies the following estimates : there exist $C>0$, $\mu >0$ and a family of constants $C_\alpha>0$, $\alpha \in \N^2$, independent of $\epsilon$ such that for all $\rho, \zeta \in U$, 
\begin{align*}
& |g_\epsilon (\rho)| \leq C |\log \epsilon| \\
& |\partial^\alpha g_\epsilon(\rho)| \leq C_\alpha \left(\epsilon^{-|\alpha|/2} \right) \\
&\frac{\exp(g_\epsilon(\rho))}{\exp(g_\epsilon(\zeta))} \leq C\left\langle \frac{\rho - \zeta}{\sqrt{\epsilon}} \right\rangle^\mu
\end{align*}
\end{lem}
This last inequality makes $e^{g_\epsilon}$ an order function in the rescaled variable $\rho / \sqrt{\epsilon}$. 

We will specialize to $\epsilon = h^{2\delta}$ where $\delta  = \frac{1}{2} - \varepsilon$. For this reason, it is important that the constants do not depend on $\epsilon$.

\section{Proof of Theorem \ref{Theorem_main}} \label{Section_proof_main_theorem}
From now on, $M_h(z)=M(z;h)$ is an open hyperbolic quantum map satisfying the assumptions of Theorem \ref{Theorem_main}. Recall that we note $\alpha_h(z)$ the amplitude of $M_h(z)$. 
Our aim is to understand the zeros of the Fredholm determinants
$$ \det \left( \Id - M_h(z) \right) $$ 
Since the spectrum of $M_h(z)$ doesn't change by conjugation, we will instead study 
\begin{equation}
M_t (z;h) \coloneqq e^{-tG} M_h(z) e^{tG}
\end{equation}
where $t$ will be chosen below and $G = \op(g)$ where $g=g_{h^{2\delta}}$ is the escape function constructed (\ref{escape_function_2}), specialized in the case $\epsilon= h^{2\delta}$ where $\delta=1/2- \varepsilon$, for some fixed $\varepsilon$. To alleviate the notations, we now omit to write that $M_t(z)$ depends on $h$. The role of this conjugation is to damp the quantum map outside a small neighborhood of the trapped set so that it confers to the new operator nicer microlocal properties. 
To exploit the hyperbolicity of $F$ and the special structure of the trapped set, we note that the zeros (repeated with multiplicity) of $\det \left( \Id - M_t(z) \right) $ are among the zeros of 
$$ \det \left( \Id - M_t(z)^{2N} \right) $$
We will use this remark with an exponent $N(h)$ depending on $h$ in a controlled way and we will assume that $N(h) \leq C \log \frac{1}{h}$ for some $C >0$. A precise value of $N(h)$ will be given later. 

\subsection{Application of a Jensen formula. }

The proof of Theorem \ref{Theorem_main} relies on the following Proposition, whose proof will occupy the end of this section. Recall that $\Omega= ]-R, R[ + i ]-R, R[$ with $R$ fixed but large (in particular, $R \geq 4$). 

\begin{prop}\label{Prop_Key_2}
Let  $\varepsilon >0$. Let $g=g_{h^{2\delta}}$ be the escape function defined in (\ref{escape_function_2}) (with $1/2 - \delta = \varepsilon>0)$). Let's note $M_t(z) = e^{-t\op(g)} M_h(z) e^{t \op(g)}$.  Let us fix $\beta \in ]0,R[$.  
Then, there exist $t=t_\varepsilon >0$, $C = C_\varepsilon >0$ , $\nu_\varepsilon >0$, $\vartheta_\varepsilon >0$ and $N=N_{\varepsilon}(h) \in \N$ such that
\begin{itemize} 
\item When $\varepsilon \to 0$, 
$$ \nu_\varepsilon = d_H + O(\varepsilon) \; ; \; \vartheta_\varepsilon = \frac{1-O(\varepsilon)}{6 \lambda_{\max}} \;  $$ 
\item at fixed $\varepsilon$, when $h \to 0$, $N_{\varepsilon}(h) \sim \vartheta_\varepsilon \log(1/h) \;  $
\item for all $h$ sufficiently small and for all $z \in \Omega$ with $ \im(z) \in [-\beta, 4]$, 
\begin{equation}\label{key_equation}
 \tr\left( \left(M_t(z)^{N}\right)^* M_t(z)^N \right) \leq C h^{-\nu_\varepsilon} h^{ -\vartheta_\varepsilon P\left(-2\im z t_{ret} - \varphi_u \right)}  
\end{equation}
\end{itemize}
\end{prop}

\begin{rem}
Since $\im z \geq - \beta$ and since the function $ s \mapsto P(-2 s t_{ret} - \varphi_u)$ is non increasing, the right hand side can be estimated by $h^{-\nu_{\varepsilon}} h^{-\vartheta_\varepsilon P(2 \beta t_{ret} - \varphi_u) }= h^{-d_H + p(\beta) - O(\varepsilon)}$. This is where the function $p(\beta) =- \frac{1}{6 \lambda_{\max}} P(2 \beta t_{ret} - \varphi_u)$ appears. 
\end{rem}

Armed with this proposition, we can conclude the proof of Theorem \ref{Theorem_main} by using standard arguments of spectral theory and complex analysis (we mainly borrow the arguments from \cite{DaDy12}, \cite{DyBW}).

\vspace*{0.5cm}
\textit{Proof of Theorem \ref{Theorem_main}}. 
The exponent $d_H$ is known from \cite{NSZ14} in Theorem 4. We focus on the potential improvement given by $p(\gamma+ \varepsilon) - \varepsilon$. 

We fix $0<r<R$ and $\gamma >0$ and note $\Omega_0 = \{ |\re z| \leq r, \im(z) \in [-\gamma, 2] \} $. For $\eta>0$, we also note $\Omega_\eta = \{ | \re z | < R , \im z \in ]-\gamma - \eta, 4[ \}$. 
Since $\det(\Id- M(z;h) )=\det(\Id- M_t(z))$ and due to the relation: 
$\Id - A^{2N} = (\Id -A)(\Id + A + \dots A^{2N-1})$, we have (we note $m_T(\Omega_1)$ the numbers of zeros of $\det(\I - T)$ in $\Omega_1$, counted with multiplicity), 
$$m_M(\Omega_0) \leq m_{M_t} (\Omega_0) \leq m_{M_t^{2N}}(\Omega_0)$$ 
that is, it is enough to estimates the number of zeros of
$f(z) =  \det \left( \Id - M_t(z)^{2N} \right) $. 

We claim that if $H$ is some Hilbert space and if $A : H \to H$ is a trace-class operator, then $\log |\det ( I - A^2)| \leq ||A||_{HS}^2 = \tr( A^* A)$. Indeed, if we denote $\lambda_j(A)$ (resp. $\sigma_j(A)$) the eigenvalues (resp. singular values) of $A$ repeated with multiplicity, one has, 
\begin{align*}
\log | \det (I- A^2) | &= \sum_{j} \log (|1 - \lambda_j(A^2) |) = \sum_j \log |1- \lambda_j(A)^2 |  \leq \sum_j \log (1+ \lambda_j(A)^2) \leq \sum_j \lambda_j(A)^2 \\
\end{align*}
Weyl's inequalities imply that (see for instance \cite{DyZw}, Appendix B.5.1) $$\sum_j \lambda_j(A)^2  \leq \sum_j \sigma_j(A)^2 = ||A||_{HS}^2 = \tr(A^* A)$$
which gives the desired result. 
Hence, we have 
\begin{equation}
\log | \det \left( \Id - M_t(z)^{2N} \right) | \leq \tr \left( (M_t^{N}(z))^* M_t^N(z) \right) 
\end{equation}
which is known to be controlled by Proposition \ref{Prop_Key_2}. 
Let's note $z_0 = i \in \Omega_0$. 
By the Riemann mapping theorem, for any $\eta >0$, there exists a conformal map $c : \Omega_\eta \to \{ |z| < 1 \}$ such that $c(z_0) = 0$. $c(\Omega_0) \Subset c(\Omega_\eta)$, so that there exists $\delta >0$ such that 
$c(\Omega_0) \subset \{ |z| < 1 - \delta \}$.  We now apply Jensen's formula to the function $f \circ c$. 
Let $n(t)$ denote the number of zeros of $f \circ c $ (counted with multiplicities), in the disc of radius $t$ .
We have, by Jensen's formula, 
$$ \int_0^{1- \delta/2} \frac{n(t)}{t} dt = \frac{1}{2\pi} \int_{0}^{2\pi} \log |f\circ c ( (1-\delta/2) e^{i \theta})| d\theta  - \log| f(z_0)|$$
Therefore, 
\begin{align*}
m_{M} (\Omega_0) \leq m_{M_t^{2N}}(\Omega_0)& \leq n(1-\delta) \\
&\leq  \frac{2}{\delta(1- \delta) }  \int_{1-\delta}^{1 - \delta/2} \frac{n(t)}{t}dt \\ 
&\leq \frac{2}{\delta(1- \delta) } \int_0^{1-\delta/2} \frac{n(t)}{t}dt \\
&\leq  \frac{2}{\delta(1- \delta) } \left(  \frac{1}{2\pi} \int_{0}^{2\pi} \log |f\circ c ( (1-\delta/2) e^{i \theta})| d\theta  - \log| f(z_0)|\right)  \\
&\leq  \frac{2}{\delta(1- \delta) } \left( \sup_{z \in \Omega_\eta} \log |f(z)| - \log |f(z_0)| \right) 
\end{align*}
We apply Proposition \ref{Prop_Key_2} with a small parameter $\varepsilon^\prime$, depending on $\varepsilon$, giving exponents $ \nu_{\varepsilon^\prime}, \vartheta_{\varepsilon^\prime}$. 
Since $\nu_{\varepsilon^\prime} = d_H + O(\varepsilon^\prime)$ and $\vartheta_{\varepsilon^\prime}= \frac{1}{6 \lambda_{\max}} + O(\varepsilon^\prime)$,  we can choose $\varepsilon^\prime$ small enough so that $\nu_{\varepsilon^\prime} -  6 \lambda_{\max} \vartheta_{\varepsilon^\prime} p(\gamma + \eta) \leq d_H - p(\gamma+ \eta) + \varepsilon$. 
Hence, we have 
$$\sup_{z \in \Omega_\eta} \log |f(z)|\leq  \sup_{z \in \Omega_\eta} h^{- \nu_{\varepsilon^\prime}} h^{-\vartheta_{\varepsilon^\prime} P\left(-2 \im z t_{ret}  - \varphi_u \right) }  \leq  h^{ - d_H - \varepsilon + p(\gamma+ \eta)} $$
since the map $\beta \mapsto p(\beta)$ is non increasing (recall the definition of $p(\beta)$ in (\ref{definition_p_beta}). 
To handle the term $- \log |f(z_0)|$, since $\alpha_h(z_0) <1$ near $\mathcal{T}$, by choosing $t$ large enough, we may ensure that there exists $\rho \in [0,1[$ such that for $h$ small enough, $||M_t(z_0)|| \leq \rho$ (see the proof of Lemma 5.3 in \cite{NSZ14}). As a consequence, $||M_t^{2N(h)}(z_0) || \leq C_t \rho^{2N}$, so that for $h$ small enough, $||M_t^{2N}||< 1/2$. In particular, for such $h$, $\Id - M_t(z_0)^{2N} $ is invertible and  $$\left| \left| \left( \Id - M_t(z_0)^{2N} \right)^{-1} \right| \right| \leq 2.$$ 
As a consequence, one has 
\begin{align*}
- \log | \det \left( \Id - M_t(z_0)^{2N} \right) | &= \log   \left| \det \left( \Id - M_t(z_0)^{2N} \right)^{-1} \right| \\
&= \log   \left|  \det \left( \Id + M_t(z_0)^{2N} \left( \Id - M_t(z_0)^{2N} \right)^{-1} \right) \right|  \\
& \leq \left| \left| M_t(z_0)^{2N} \left( \Id - M_t(z_0)^{2N} \right)^{-1}  \right| \right|_{tr} \\
 &\leq || M_t(z_0)^{2N} ||_{tr} \left| \left| \left( \Id - M_t(z)^{2N} \right)^{-1} \right| \right|\\
 & \leq 2  ||M_t(z_0)^N||_{HS} \leq   Ch^{ - d_H - \varepsilon + p(\gamma+ \eta)}
\end{align*} 
This concludes the proof. 

\begin{flushright}
\qed
\end{flushright}

\subsection{Proof of Proposition \ref{Prop_Key_2} }
We start the proof of Proposition \ref{Prop_Key_2}. We fix some $\varepsilon >0$ and we froze the complex variable $z$ and note $M_h$ and $\alpha_h$ instead of $M_h(z)$ and $\alpha_h(z)$ : we momentarily forget this dependence but keep in mind that $\im(z) \in [-\beta,4]$ for some $\beta >0$. In particular, $\alpha_h(z) = e^{- \im z t_{ret}} + O\left( h^{-1}S_{0^+} \right)$ in a neighborhood of $\mathcal{T}$ and the constant in the estimates below can be chosen independent of $z$.

\paragraph{Reduction to FIO acting on $\R$.} We will note $\R_J = \bigsqcup_{j=1}^J \R$ and $L^2(\R_J) = \bigoplus_{j=1}^J L^2(\R)$. 
Recall that by construction (see \ref{subsubsection_open_quantum_map}), $M_h$ is an operator of the form $(M_{ij}(h))$ where $M_{ij}(h) : L^2(Y_j) \to L^2(Y_i)$. It will be more convenient for us to work on $L^2(\R)$. For this purpose, recall that, for all $i,j$, there exists $\tilde{M}_{ij}(h)  \in I_{0^+} \left( \R \times \R, \Gr(F_{ij})^\prime \right)$ and cut-off functions $\Psi_i, \Psi_j$ such that as operators $L^2(Y_j) \to L^2(Y_i)$
$$ M_{ij}(h) = \Psi_i \tilde{M}_{ij}(h) \Psi_j + \hinf$$ 
and as operator $L^2(\R) \to L^2(\R)$, 
$$ \tilde{M}_{ij}(h) = \Psi_i \tilde{M}_{ij}(h) \Psi_j + \hinf $$ 
where, in the two equalities above, the $\hinf$ hold for the trace norm. 
Let's note $M_\psi(h) = (\Psi_i M_{ij} (h)\Psi_j)_{ij}$. 
As soon as $N \leq C \log \frac{1}{h}$, $M(h)^N = M_\Psi(h)^N + \hinf$ as operators $L^2(Y) \to L^2(Y)$ and 
$M_\Psi(h)^N = \tilde{M}(h)^N + \hinf$ as operators $L^2(\R_J) \to L^2(\R_J)$. The same holds after conjugation by $e^{tG}$.  In particular, this sows that 
$$ \tr_{L^2(Y) } \left((M_t^{N})^* M_t^N  \right)  = \tr_{L^2(\R_J) } \left((\tilde{M}_t^{N})^* \tilde{M}_t^N  \right) + \hinf$$ 
Since the $\hinf$ will finally be adsorbed in our required inequality, it is enough to work with $\tilde{M}(h)$ instead of $M_h$.

From now on, we will write $M_h$ for the operator $\tilde{M}_h : L^2(\R_J) \to L^2(\R_J)$. There exists $\Psi_A, \Psi_D$ such that 
$$ \supp \Psi_A \Subset \widetilde{A} \quad ; \quad \supp \Psi_D \Subset \widetilde{D}$$ and 
$$\Psi_A M_h = M_h + \hinf ; \quad M_h \Psi_D  = M_h+ \hinf $$ 
Moreover, we will now omit the $h$-dependence of the semiclassical operators in the notations when this dependence is obvious. In particular, we will simply write $M$, $\alpha$ or $M_t$ instead of $M_h$, $\alpha_h$ and $M_t(h)$ respectively.

\begin{nota}
A function $a$ on $T^*\R_J = \bigsqcup_{j=1}^J T^*\R$ is a $J$-uple of functions $(a_1, \dots, a_J)$. The quantization $\op(a)$ is the diagonal matrix with diagonal entries $\op(a_j)$. The support of $a$ is the disjoint union of the supports of the $a_j$'s, so as the wavefront set of $\op(a)$. 
\end{nota}

\subsubsection{Refined quantum partition}

In virtue of Proposition \ref{Prop_formula_pressure}, applied with $\varphi = -2 \im z t_{ret} - \varphi_u$,  there exists $\eta >0$ such that for any open cover $\mathcal{Q}$ of $\mathcal{T}$ of diameter smaller than $\eta$, one has 
\begin{equation}\label{equa_pressure}
\left| \lim_{n \to + \infty} \frac{1}{n} \log P_1(\varphi, n , \mathcal{Q}) - P(\varphi)\right| \leq \varepsilon/3
\end{equation}

We consider some $\varepsilon_0 >0$, which is supposed to be small enough to satisfy all the assumptions which will appear in the following and which will follow us throughout the end of the chapter. In particular, we first impose $\varepsilon_0 < \eta$.  \\
Since $\mathcal{T}$ is totally disconnected, there exists an open cover of $\mathcal{T}$ by a finite number of \emph{disjoint} open sets (of $U$) of diameter smaller than $\varepsilon_0$ : 
$$ \mathcal{T} \subset \bigcup_{\mathfrak{A} \in \mathcal{Q}} \mathfrak{A} $$
We fix some $\rho_\mathfrak{A} \in \mathcal{T} \cap \mathfrak{A}$ and we assume that for all $ \mathfrak{A} \in \mathcal{Q}$, there exists $j_\mathfrak{A},l_\mathfrak{A},m_\mathfrak{A} \in \{1, \dots, J \}$ such that 
$$\mathfrak{A} \subset B(\rho_\mathfrak{A} , 2 \varepsilon_0) \subset \widetilde{A}_{ j_\mathfrak{A} l_\mathfrak{A} } \cap \widetilde{D}_{m_\mathfrak{A} j_\mathfrak{A}} \subset U_{j_\mathfrak{A}}$$
$\varepsilon_0$ is supposed to be small enough so that : 
\begin{itemize}
\item $e^{- \tau_m} \leq \alpha_h \leq e^{ \tau_M}$ in $B(\rho_\mathfrak{A}, 2\varepsilon_0)$ for some $\tau_m, \tau_M$, for all $h$ small enough. 
\item If $\varepsilon_1$ denotes the one appearing in Lemma \ref{adapted_chart_0},  $2 \varepsilon_0 \leq \varepsilon_1$, and then, there exists a chart $\kappa_\mathfrak{A} : B(\rho_\mathfrak{A}, 2 \varepsilon_0) \to W_\mathfrak{A} = \kappa_\mathfrak{A}(B(\rho_\mathfrak{A}, 2 \varepsilon_0) ) $, given by Lemma \ref{adapted_chart_0}, adapted to the dynamics, where $W_\mathfrak{A}$ is a subset of $T^*\R$ centered at $0$. . 
\item  There exist Fourier integral operators $B_\mathfrak{A}, B_\mathfrak{A}^\prime \in I_0(\R \times \R, \Gr^\prime(\kappa_\mathfrak{A})) \times I_0(\R \times \R, \Gr^\prime(\kappa_\mathfrak{A}^{-1}))$, 
quantizing $\kappa_\mathfrak{A}$ in a neighborhood of $\kappa_\mathfrak{A} \left( \overline{ \mathfrak{A}}  \right) \times \overline{\mathfrak{A}}$. 
\end{itemize}

\begin{nota}
We will still denote $B_\mathfrak{A}$ and $B_\mathfrak{A}^\prime$ the operators 
\begin{align*}
B_\mathfrak{A} =\text{Diag}( 0 ,\dots B_\mathfrak{A}, \dots ,0) : L^2 (\R_J) \to L^2 (\R_J) \quad ; \quad  
B_\mathfrak{A}^\prime = \text{Diag} (0 ,\dots ,B_\mathfrak{A}^\prime, \dots,0) : L^2 (\R_J) \to L^2(\R_J) 
\end{align*}
with the non zero entry in position $j_\mathfrak{A}$. When we say that $(B_\mathfrak{A},B_\mathfrak{A}^\prime)$ quantize $\kappa_\mathfrak{A}$ in a neighborhood of $\kappa_\mathfrak{A} \left( \overline{ \mathfrak{A}}  \right) \times \overline{\mathfrak{A}}$, we mean that $B_\mathfrak{A}^\prime B_\mathfrak{A} = I + \hinf $ microlocally in a neighborhood of $\overline{\mathfrak{A}}$ (in the sense that if $\supp (c)$  is included in this neighborhood of $\overline{\mathfrak{A}}$ and if $C = \op(c)$, then $ B_\mathfrak{A}^\prime B_\mathfrak{A} C = C + \hinf  \; ; \; CB_\mathfrak{A}^\prime B_\mathfrak{A}= C + \hinf$) and  $B_\mathfrak{A} B_\mathfrak{A}^\prime = I + \hinf$ microlocally in a neighborhood of $\kappa_\mathfrak{A}(\overline{\mathfrak{A}})$. 
\end{nota}
In virtue of the equation (\ref{equa_pressure}), there exists $n_0 \in \N$ such that 
$$  \left| \frac{1}{n_0} \log P_1(\varphi, n_0 , \mathcal{Q} ) - P(\varphi) \right| \leq 2\varepsilon/3$$
As a consequence, there exists a subpartition $(\mathcal{W}_q)_{q \in \mathcal{A}} \subset \mathcal{Q}^{n_0}$ such that $\mathcal{T} \subset \bigcup_{q \in \mathcal{A}} \mathcal{W}_q$ and 
\begin{equation}\label{nice_formula}
 \sum_{q \in \mathcal{A}} \sup_{ \rho \in \mathcal{W}_q \cap \mathcal{T} } \exp \left( \sum_{i=0}^{n-1} \varphi(F^i(\rho) ) \right)  \leq e^{n_0 (P(\varphi) + \varepsilon)}
\end{equation}
For $q \in \mathcal{A}$, we can find an open set $\mathcal{V}_q \Subset \mathcal{W}_q$ such that $\mathcal{T} \cap \mathcal{W}_q \subset \mathcal{V}_q$. $(\mathcal{V}_q)_{q \in \mathcal{A}}$ is still a cover of $\mathcal{T}$. 
We complete this cover with
\begin{equation}
\mathcal{V}_\infty =  \R_J \setminus \bigcup_{q \in \mathcal{A}} \mathcal{V}_q
\end{equation}
We note $\mathcal{A}_\infty = \mathcal{A} \cup \{ \infty \}$. Note also that for $q \in \mathcal{A}$, $\mathcal{W}_q$ is of the form $$\mathfrak{A}_0 \cap F^{-1}(\mathfrak{A}_1) \cap \dots \cap F^{-(n-1)}(\mathfrak{A}_{n-1})$$
and in particular $\mathcal{W}_q \subset \mathfrak{A}_0 $ : we note $j_q,l_q,m_q, \rho_q, \kappa_q, B_q, B^\prime_q$, $W_q$, instead of $j_{\mathfrak{A}_0},l_{\mathfrak{A}_0},m_{\mathfrak{A}_0}, \rho_{\mathfrak{A}_0}$, $\kappa_{\mathfrak{A}_0}, B_{\mathfrak{A}_0}, B^\prime_{\mathfrak{A}_0}$, $W_{\mathfrak{A}_0}$. 
Then,  for $q \in \mathcal{A}$, we consider a cut-off function $\chi_q \in \cinfc(\R_J, [0,1])$ such that $\supp(\chi_q) \subset \mathcal{W}_q$ and $\chi_q \equiv 1$ in a neighborhood of $\overline{\mathcal{V}_q}$. Finally, we note $\chi_\infty = 1 - \sum_{q \in \mathcal{A}} \chi_q$. We note that $\chi_q$ is supported in only one copy of $\R$ in $\R_J$ when $q \in \mathcal{A}$ and $\chi_\infty$ has non-zero components in all the copies of $\R$ in $\R_J$. Moreover, $\supp(\chi_\infty) \subset \mathcal{V}_\infty$. 

We then quantize the symbols $\chi_q$, $q \in \mathcal{A}_\infty$  : 
\begin{equation}
A_q = \op(\chi_q)
\end{equation}
Note that for $q \in \mathcal{A}$, $A_q$ is a diagonal matrix with a single non zero coefficient. 
The family $(A_q)_{ q \in \mathcal{A}_\infty}$ satisfies the following properties : 
\begin{equation}\label{properties_Aq_bis}
\sum_{ q \in \mathcal{A}_\infty} A_q = \Id \quad ; \quad \forall q \in \mathcal{A}_\infty , ||A_q|| \leq 1 + O(h )
\end{equation}
Since $M^{n_0} = \sum_{ q \in \mathcal{A}_\infty} M^{n_0}A_q $, we may write 
$$ M^{nn_0} = \sum_{ \mathbf{q} \in \mathcal{A}^n_\infty} M_{\mathbf{q} }$$ 
where for $\mathbf{q} = q_0 \dots q_{n-1}  \in \mathcal{A}^n_\infty$, 

\begin{equation}
M_{\mathbf{q} }  \coloneqq M^{n_0}A_{q_{n-1}} \dots M^{n_0}A_{q_0}
\end{equation} 
For $\mathbf{q} = q_0 \dots q_{n-1}  \in \mathcal{A}^n_\infty$, we also define a family of refined neighborhoods, forming a refined cover of $\mathcal{T}$, 

\begin{equation} 
 \mathcal{V}_{\textbf{q}}^{-} = \bigcap_{i=0}^{n-1} F^{-in_0} \left( \mathcal{V}_{q_i} \right) \quad ; \quad \mathcal{V}_{\textbf{q}}^+ = F^{nn_0} \left( \mathcal{V}_{\textbf{q}}^{-}\right) = \bigcap_{i=0}^{n-1} F^{(n-i)n_0}\left(\mathcal{V}_{q_i} \right)
  \end{equation}
  and we adopt the same definitions by changing $\mathcal{V}$ into $\mathcal{W}$. 
Roughly speaking, we expect that each operator $M_\mathbf{q}$ acts from $\mathcal{W}_\mathbf{q}^-$ to $\mathcal{W}_\mathbf{q}^+$ and is negligible elsewhere. Combining (\ref{properties_Aq_bis}), the fact that $\alpha_h \leq e^{\tau_M}$ in $B(\rho_\mathfrak{A})$ and the bound on $M$, the following bound is valid : 

\begin{equation}
||M_\mathbf{q}||_{L^2 \to L^2} \leq \left(e^{\tau_M}+ O(h^{1 -})\right)^{nn_0}
\end{equation}
As soon as $|n| \leq C_0 |\log h|$, we have $||M_\mathbf{q}||_{L^2 \to L^2} \leq Ce^{n n_0 \tau_M }$, for some $C$ depending on $C_0$ and a finite number of semi-norms of $\alpha_h$ and then
$$ ||M_\mathbf{q}|| \leq Ch^{-K}$$ for some $C,K>0$ depending on $C_0$ and $\alpha_h$.

\subsubsection{Local unstable Jacobian} 
We want to define unstable Jacobians associated with these refined partition. Let's fix a word $\mathbf{q} =q_0 \dots q_{n-1} \in \mathcal{A}^n$ and assume that $\mathcal{W}_\mathbf{q}^- \neq \emptyset$. 
Fix $\rho \in \mathcal{W}_\mathbf{q}^-$. By definition of $\mathcal{W}_{q_i}$, there exists $\mathfrak{A}_{0,i}, \dots, \mathfrak{A}_{n_0-1,i} \in \mathcal{Q}$ such that 
$$\mathcal{W}_{q_i} = \bigcap_{j=0}^{n_0-1} F^{-j}(\mathfrak{A}_{j,i})$$
Hence, for $0 \leq l \leq n^\prime = n \times n_0 - 1$, there exists $\rho_l \in \mathcal{T}$ such that  $d(\rho_l, F^l(\rho) ) \leq 2 \varepsilon_0$. Hence, 

$$d(F(\rho_l), \rho_{l+1}) \leq d( F(\rho_l) , F^{l+1}(\rho) ) + d (F^{l+1}(\rho), \rho_{l+1} ) \leq C \varepsilon_0$$
That is to say, $(\rho_0, \dots, \rho_{n^\prime})$ is a $C\varepsilon_0$ pseudo orbit. Assume that $\delta_0 >0$ is a small fixed parameter. In virtue of the shadowing lemma (see \cite{KH}, Section 18.1), if $\varepsilon_0$ is sufficiently small, $(\rho_0, \dots, \rho_{n^\prime})$ is $\delta_0$ shadowed by an orbit of $F$ : there exists $\rho^\prime \in \mathcal{T}$ such that for $i \in \{0, \dots, n^\prime \}$, $d(\rho_i, F^i(\rho^\prime) ) \leq \delta_0$. Consequently, 
$d(F^i (\rho) , F^i(\rho^\prime)) \leq \delta_0 + C\varepsilon_0$. If $\rho_2$ is another point in $\mathcal{W}_\mathbf{q}^-$, for $i= 0, \dots, n^\prime$, $d(F^i(\rho_2), F^i(\rho^\prime) ) \leq 2 \varepsilon_0 + C\varepsilon_0 + \delta_0$.  For convenience, set $\varepsilon_2 =2 \varepsilon_0+ \delta_0 + C\varepsilon_0$ and note that $\varepsilon_2$ can be arbitrarily small depending on $\varepsilon_0$. As a consequence, we have proven the following

\begin{lem}\label{help_def_jacobian_2_bis}
If $\mathcal{W}_\mathbf{q}^- \neq \emptyset$, there exists $\rho^\prime \in \mathcal{T}$ such that $\forall l \in \{0, \dots, nn_0-1 \}$ and for any $\rho \in \mathcal{W}_\mathbf{q}^-$, $d(F^{l}(\rho), F^l(\rho^\prime)) \leq \varepsilon_2$.
\end{lem}
We fix any $\rho^\prime$ satisfying the conclusions of this lemma and we arbitrarily set  (recall also the definition of $J^u_n(\rho)$ in (\ref{Def_jacobian}) for $\rho \in \mathcal{T}$ )
\begin{equation}\label{def_jacobian_2_bis}
J_\mathbf{q}^u \coloneqq J^u_{n n_0}(\rho^\prime) = \prod_{j=0}^{n-1} J^u_{n_0} (F^{j n_0} (\rho^\prime) )
\end{equation}
If $\rho^\prime_1$ is another point satisfying this conclusion, we have $d(F^i(\rho^\prime), F^i(\rho^\prime_1)) \leq 2 \varepsilon_2$ for $i \in \{0, \dots, n^\prime \}$ and in virtue of Corollary (\ref{cor_control_jacobian}), 
$$ J^u_{nn_0}(\rho^\prime) \sim J^u_{n n_0}(\rho^\prime_1)$$
Hence, up to global multiplicative constant, the definition of this unstable Jacobian is independent of the choice of $\rho^\prime$. Notice that if $\mathcal{W}_\mathbf{q}^-  \cap \mathcal{T} \neq \emptyset$, any $\rho^\prime \in \mathcal{T} \cap \mathcal{W}_\mathbf{q}^-$ satisfies the conclusions of Lemma \ref{help_def_jacobian_2_bis} and $J_\mathbf{q}^u \sim J^u_{nn_0}(\rho^\prime)$.

 We have the following facts concerning these local unstable Jacobian : 
\begin{lem}\label{Lemma_unstable_Jacobian}
If $\varepsilon_0$ is small enough, the following holds. 
There exists $C >0$ such that for all $\mathbf{q} \in \mathcal{A}^n$ and for all $\rho \in \mathcal{W}_\mathbf{q}^-$,  we have
\begin{itemize}
\item $||d_\rho F^{nn_0} || \leq C J^u_\mathbf{q}$
\item$ d(F^{nn_0-1}(\rho),  \mathcal{T}_+) \leq  C \left(J^u_{\mathbf{q}}\right)^{-1} d(\rho, \mathcal{T}_+) $
\item $ d(\rho, \mathcal{T}_-) \leq C \left(J^u_{\mathbf{q}}\right)^{-1} d(F^{nn_0-1}(\rho), \mathcal{T}_-) $
\end{itemize}
\end{lem}
\begin{proof}
The three points are consequences of Lemma \ref{Local_hyperbolic_2}. The first point is an easy one. Concerning the other two, first recall that $\mathcal{T}_+$ (resp. $\mathcal{T}_-$) is, in a neighborhood of $\mathcal{T}$, equal to the union of local unstable (resp. stable manifolds). Let's consider the second inequality. The proof of the third one is similar, by inverting the time direction. We fix $\zeta \in \mathcal{T}$ such that $d(\rho, \mathcal{T}_+) = d(\rho, W_u(\zeta) )$ and $d(\zeta, \rho) \leq  2 \varepsilon_0$. Recall that by Lemma \ref{help_def_jacobian_2_bis}, there exists $\rho^\prime$ such that $\forall i \in \{0, \dots, nn_0-1 \}$ and $d(F^{i}(\rho), F^i(\rho^\prime)) \leq \varepsilon_2$. We hence consider the unique point $\zeta^\prime \in W_u(\zeta) \cap W_s(\rho^\prime)$. Since $\zeta^\prime \in W_s(\rho^\prime)$, $d(F^i(\zeta^\prime) , F^i(\rho^\prime)  ) \leq C J_i^s(\rho) d(\rho^\prime, \zeta^\prime)$ for all $0 \leq i \leq nn_0-1$. If $\varepsilon_0$ is small enough, we may assume that $C J_i^s(\rho) d(\rho^\prime, \zeta^\prime) \leq \frac{1}{2}\varepsilon_1$ for $0 \leq i \leq n-1$ (where $\varepsilon_1$ appears in Lemma \ref{Local_hyperbolic_2}). As a consequence, $J^u_{nn_0}(\zeta^\prime) \sim J^u_{nn_0}(\rho^\prime) \sim J^u_\mathbf{q}$. Moreover, 
$d(F^i(\rho) , F^i(\zeta^\prime) ) \leq \frac{1}{2} \varepsilon_1+ \varepsilon_2$ for all $0 \leq i \leq nn_0-1$. Hence, if $\varepsilon_2 \leq \frac{1}{2} \varepsilon_1$, 
$$d(F^{nn_0-1}(\rho), \mathcal{T}_+) \leq d(F^{nn_0-1}(\rho), W_u(F^n(\zeta^\prime))) \leq C J_{nn_0}^u(\zeta^\prime) d(\rho, W_u(\zeta^\prime)) \leq C J_\mathbf{q}^u d(\rho, \mathcal{T}_+)$$
\end{proof}

\subsubsection{Numerology}

In this subsection, we introduce the parameters we will work with. Recall that $\varepsilon$ has been fixed. We set $\delta = 1/2 - \varepsilon$ : it is related to the regularity of the escape function $g$. For technical reasons, we also introduce 
\begin{equation}\label{exponent_delta}
\delta_0 = \frac{1}{2} -\frac{ \varepsilon}{2} ; \; \delta_1 = \delta - \frac{\varepsilon}{2} = \frac{1}{2}- \frac{3\varepsilon}{2}
\end{equation}
satisfying $\delta_1 < \delta < \delta_0 < 1/2$. 
Recall that $n_0$ has been chosen in (\ref{nice_formula}) and that 
$$ \lambda_{\max} = \sup_{ \rho \in \mathcal{T}} \limsup_{ n \to + \infty} \frac{1}{n}\log J^u_n(\rho) $$
We define precisely the parameter $\vartheta_\varepsilon$ appearing in Proposition \ref{Prop_Key_2} as \begin{equation}\label{definition_vartheta}
\vartheta_\varepsilon = \frac{1-4 \varepsilon}{6 \lambda_{\max}(1+\varepsilon)^2}
\end{equation}
The precise value of $\vartheta_\varepsilon$ will be used in the following : what is important is that $\vartheta_\varepsilon = \frac{1}{6 \lambda_{\max}} - O(\varepsilon) < 1/6 \lambda_{\max}$. Finally, we set 
$$ n=n(h) = \left\lfloor \frac{\vartheta_\varepsilon}{n_0} \log \frac{1}{h} \right\rfloor $$
which satisfies
$$ e^{\lambda_{\max}(1+ \varepsilon) nn_0} \leq h^{-\frac{1-4 \varepsilon}{6(1+\varepsilon)}}$$
In particular, we assume that $\varepsilon$ is small enough to ensure that 
$$ h^{\delta_0} h^{-\frac{1-4 \varepsilon}{6(1+\varepsilon)}}\leq h^{1/3}$$
This will constraints the width of the evolved coherent states. 

\subsubsection{Reduction to $L^2$-bounds of an evolved coherent state.} 
We can find a uniform $T_0 \in \N$ such that if $\rho \in \mathcal{V}_\infty$, there exists $k \in \{ - T_0, \dots, T_0 \}$ such that $F^{k}(\rho)$ "falls" in the hole - that is, either there exists $k \in \{1, \dots, T_0 \}$ such that $F^i(\rho) \in \widetilde{D}$ for $1 \leq i \leq k-1$ and $F^k(\rho) \in U \setminus \widetilde{D}$ or there exists $k \in \{1, \dots, T_0 \}$ such that $F^{-i}(\rho) \in \widetilde{A}$ for $1 \leq i \leq k-1$ and $F^{-k}(\rho) \in U \setminus \widetilde{A}$. By standard properties of the Fourier integral operators, each component $(M^{T_0})_{i j }$ of $M^{T_0}$ is a Fourier integral operator associated with the component $(F^{T_0})_{i j }$ of $F^{T_0}$. In particular, $\WF^\prime (M^{T_0}) \subset \gr( F^{T_0})$. 

Let us study $M^{2T_0 + nn_0}=M^{T_0}M^{nn_0}M^{T_0}$, and let's decompose $M^{nn_0} = \sum_{\mathbf{q} \in \mathcal{A}_\infty^n} M_{\mathbf{q}}$.  If $\mathbf{q}=q_0 \dots q_{n-1} \in \mathcal{A}^n_\infty$ and if there exists an index $i \in \{0, \dots, n-1 \} $ such that $q_i = \infty$, one can isolate this index $i$ and trap $A_{q_i}$ between two Fourier integral operators $M_1, M_2$, belonging to a finite family of FIO associated to $F^{T_0}$, so that we can write 

$$M^{T_0} M_\mathbf{q} M^{T_0}  = B_1 M_1 A_{\infty} M_2 B_2$$ 
where $B_1, B_2$ satisfy the $L^2$-bound : 
$$ ||B_1|| \times ||B_2|| \leq C(||\alpha_h||_\infty)^{nn_0-1} =O(h^{-K})$$
for some integer $K$, and we have $M_1 A_\infty M_2 = \hinf $, with constants that can be chosen independent of $\mathbf{q}$. Hence, the same is true for $M^{T_0} M_\mathbf{q} M^{T_0}$. So, we can write, keeping in mind that $|\mathcal{A}|^n = O(h^{-K})$ for some $K>0$ : 
\begin{align*}
M^{nn_0+ 2 T_0} = &\sum_{ \textbf{q} \in \mathcal{A}^n_\infty} M^{T_0} M_\mathbf{q} M^{T_0} \\
= &\sum_{ \textbf{q} \in \mathcal{A}^n} M^{T_0} M_\textbf{q} M^{T_0} + \hinf \\
= & M^{T_0} \left( \sum_{ \textbf{q} \in \mathcal{A}^n} M_\textbf{q} \right)M^{T_0} + \hinf \\
\end{align*}
Let us note 
\begin{equation}\mathfrak{M}= M^{n_0}(\Id - A_\infty) = M^{n_0} \sum_{q \in \mathcal{A}} A_q
\end{equation}
 We have shown the following lemma : 

\begin{lem}
There exists $T_0 \in \N$ such that
$$ M^{2T_0 + nn_0} = M^{T_0}\mathfrak{M}^n M^{T_0} + \hinf$$ 
\end{lem}
Let us now look at what this equation implies on the trace of $M^{2T_0+nn_0}$. In the following computations, we use : 
If $A$ is an Hilbert-Schmidt operator and $B$ bounded,\begin{enumerate}[label= (\roman*)]
\item $\tr(A^*A)  =||A||_{HS}^2$ ; 
\item $||AB||_{HS} \leq ||B|| \times ||A||_{HS}$  ; $||BA||_{HS} \leq ||B|| \times ||A||_{HS}$ 
\end{enumerate} 

\begin{align*}
\tr \left( \left(M_t^{2T_0+nn_0} \right)^* M_t^{2T_0 + nn_0} \right) &=||M_t^{2T_0 + nn_0}||_{HS}^2 \\
&=  \left| \left| M_t^{T_0} \mathfrak{M}^n_t M_t^{T_0} \right| \right|_{HS}^2 + \hinf \\
&\leq ||M_t^{T_0}||^4 ||\mathfrak{M}_t^n||_{HS}^2 + \hinf \\
&\leq  ||M_t^{T_0}||^4 \tr\left(  \left(\mathfrak{M}_t^n\right)^* \mathfrak{M}_t^n \right) + \hinf 
\end{align*}
Hence, is is enough to find the expected upper bound (\ref{key_equation}) for $\tr\left(  \left(\mathfrak{M}_t^n\right)^* \mathfrak{M}_t^n \right)$ to obtain the same kind of upper bounds for $\tr\left( \left(M_t(z)^{N}\right)^* M_t(z)^N \right)$. 

\paragraph{Evolution in local adapted charts. } 

We will be interested in the evolution of coherent states through the action of $\mathfrak{M}$. It will be more convenient to work in the charts $\kappa_q$ in which the action of $F$ is well adapted to the position-momentum coordinate $(x,\xi)$. For this purpose, we start by writing, 
 $$ \mathfrak{M}_t^n =e^{-tG} \mathfrak{M}^{n-1}\sum_{q \in \mathcal{A} } M^{n_0} A_q e^{tG}$$
 Recall that $B^\prime_q B_q = I+ \hinf$ microlocally near $\supp(a_q)$, hence,
 $$ \mathfrak{M}_t^n =e^{-tG} \mathfrak{M}^{n-1} \sum_{q \in \mathcal{A}}   M^{n_0} A_q B_q^\prime B_q e^{tG} B_q^\prime B_q+ \hinf $$
 Let's note 
 \begin{equation}\label{tilde_E_t}
 \tilde{E}_t = B_q e^{tG} B_q^\prime 
 \end{equation}
 We also fix $\widetilde{A}_q = \op(\tilde{a}_q)$ such that $\WF(\widetilde{A}_q) \subset \mathcal{W}_q$ and $\tilde{a}_q=1$ near $\supp(\chi_q)$. 
This gives : 
\begin{align*}
 \tr\left(  \left(\mathfrak{M}_t^n\right)^* \mathfrak{M}_t^n \right)&=\sum_{q,p \in \mathcal{A}} \tr \left(  \left( e^{-tG} \mathfrak{M}^{n-1} M^{n_0} A_p B_p^\prime \tilde{E}_t B_p \widetilde{A}_p \right)^* e^{-tG} \mathfrak{M}^{n-1} M^{n_0} A_q B_q^\prime \tilde{E}_t  B_q \widetilde{A}_q\right) + \hinf  \\
&=\sum_{q,p \in \mathcal{A}} \tr \left( B_p^* \left( e^{-tG} \mathfrak{M}^{n-1} M^{n_0} A_p B_p^\prime \tilde{E}_t  \right)^* e^{-tG}\mathfrak{M}^{n-1} M^{n_0} A_q B_q^\prime  \tilde{E}_t  B_q \widetilde{A}_q \widetilde{A}_p^* \right) + \hinf \\
&=\sum_{q \in \mathcal{A}} \tr \left( B_q^* \left(  e^{-tG}\mathfrak{M}^{n-1} M^{n_0} A_q B_q^\prime \tilde{E}_t \right)^* e^{-tG} \mathfrak{M}^{n-1} M^{n_0} A_q B_q^\prime \tilde{E}_t B_q \widetilde{A}_q \widetilde{A}_q^* \right) + \hinf \\
&\leq C Q \sup_{q \in \mathcal{A}} \tr \left( \left(  e^{-tG} \mathfrak{M}^{n-1} M^{n_0} A_q B_q^\prime \tilde{E}_t \right)^* e^{-tG}\mathfrak{M}^{n-1} M^{n_0} A_q B_q^\prime \tilde{E}_t \right) + \hinf 
\end{align*}
where $C$ is such that $C_0 ||B_q|| \times || B_q \widetilde{A}_q \widetilde{A}_q^* || \leq C$ for all $q \in \mathcal{A}$ (and $0 < h \leq 1$) and $Q= |\mathcal{A}|$. The passage from the second to the third line holds since $\widetilde{A}_q\widetilde{A}_p^* = \hinf$ when $q \neq p$, in virtue of the fact that $\mathcal{W}_p \cap \mathcal{W}_q \neq \emptyset$.  This computations show that it is enough to control, uniformly in $q$, the trace 

\begin{equation}
\tr \left( \left(  e^{-tG} \mathfrak{M}^{n-1} M^{n_0} A_q B_q^\prime \tilde{E}_t \right)^* e^{-tG} \mathfrak{M}^{n-1} M^{n_0} A_q B_q^\prime \tilde{E}_t \right)
\end{equation}
since we now have : 
\begin{equation}\label{equation_trace_1}
 \tr \left( \left(M_t^{2T_0+nn_0} \right)^* M_t^{2T_0 + nn_0} \right) \leq C Q \sup_{q \in \mathcal{A}} \tr \left( \left(  e^{-tG} \mathfrak{M}^{n-1} M^{n_0} A_q B_q^\prime \tilde{E}_t \right)^* e^{-tG}\mathfrak{M}^{n-1} M^{n_0} A_q B_q^\prime \tilde{E}_t \right) + \hinf 
\end{equation}

From now on, we will note $\rho, \zeta$, etc. points in $U$ and $\hat{\rho}, \hat{\zeta}$, etc. their images in the local charts $\kappa_q$. 
The resolution of identity of Lemma \ref{trace_coherent_state}, valid at the level of operators on $L^2(\R)$, extends to the case of matrix operator acting on $L^2(\R_J)$, in the following sense :

$$ \tr(A) =\sum_{j=1}^J  \frac{1}{2\pi h} \int_{T^*\R}  < A_{jj} \varphi_{\hat{\rho}}, \varphi_{\hat{\rho}}> d\hat{\rho}$$ 
Hence, if $K = e^{-tG} \mathfrak{M}^{n-1} M^{n_0} A_q B_q^\prime \tilde{E}_t $, we have 

\begin{align*}
\tr \left(K^*K \right) &= \sum_{j=1}^J  \frac{1}{2\pi h} \int_{T^*\R}  < (K^*K)_{jj} \varphi_{\hat{\rho}}, \varphi_{\hat{\rho}}> d\hat{\rho}\\
& = \sum_{1 \leq i,j \leq J} \frac{1}{2\pi h} \int_{T^*\R}  < K_{ij} \varphi_{\hat{\rho}}, K_{ij} \varphi_{\hat{\rho}}> d\hat{\rho}\\
\end{align*}
Since $A_q B^\prime_q$ is diagonal with only one non-zero diagonal entry in position $j_q$, $B_{ij}=0$ except when $j=j_q$. 
We can write : 
\begin{equation}\label{equation_trace_2}
\tr \left( \left(  e^{-tG} \mathfrak{M}^{n-1} M A_q B_q^\prime \right)^* e^{-tG} \mathfrak{M}^{n-1} M^{n_0} A_q B_q^\prime \tilde{E}_t \right) =\frac{1}{2\pi h} \int_{T^*\R}  \left| \left| e^{-tG}\mathfrak{M}^{n-1} M^{n_0} A_q B_q^\prime \tilde{E}_t \tilde{\varphi}_{\hat{\rho}} \right| \right|^2 d\hat{\rho}
\end{equation}
 where $ \tilde{\varphi}_{\hat{\rho}} $ is the column vector with only one non-zero entry equal to $\varphi_{\hat{\rho}}$ in position $j_q$.

\subsubsection{End of the proof. }

The main ingredient for the proof of the improved fractal Weyl law, which is also the main novelty of this article, is a good control for 
\begin{equation} w(\hat{\rho}) \coloneqq  \left| \left| e^{-tG}\mathfrak{M}^{n-1} M^{n_0} A_q B_q^\prime \tilde{E}_t \tilde{\varphi}_{\hat{\rho}} \right| \right|^2
\end{equation}
This weight $w$ depends on the parameter $t$ which governs the weight of the escape function. We omit to write this dependence explicitly : indeed, what is important is that once $t$ is fixed sufficiently large, $w$ will satisfy the expected decay in Proposition \ref{Prop_Key}.  
  To state this bound, let's introduce, for $\rho \in \mathcal{W}_\mathbf{q}^-$, 
$$ \Pi_{\alpha, \mathbf{q}}(\rho) =  \prod_{i=0}^{nn_0-1} \alpha\left( F^i(\rho) \right)$$
where 
\begin{equation}
\alpha(\rho) = \exp \left( -\im z t_{ret}(\rho) \right) \; ; \; \rho \in \bigcup_{q \in \mathcal{A}} \mathcal{W}_q
\end{equation}
so that, for $\rho \in \bigcup_{q \in \mathcal{A}}  \mathcal{W}_q$,  we have $\alpha_h(\rho) = \alpha(\rho) + O\left( h^{1^-} S_{0^+} \right)$. 
We also introduce the following neighborhood of $\mathcal{T}$ \begin{equation}
\mathcal{T}_{\delta,\delta_1} = \left\{ \rho \, ; \, d(\rho, \mathcal{T}_-) \leq h^{\delta}, d(\rho, \mathcal{T}_+) \leq h^{\delta_1} \right\} \subset U \subset T^* \R_J
\end{equation}

\begin{prop}\label{Prop_Key}
For any $L>0$, there exists $t= t(\varepsilon, L)$ such that the following holds. Let $\hat{\rho} \in \R^2$. If $\hat{\rho} \not \in \kappa_q(\mathcal{W}_q)$, then $w(\hat{\rho}) = O\left( \left( \frac{h}{\langle \hat{\rho} \rangle}\right)^\infty \right)$ with uniform constants.  Otherwise, assume that $\hat{\rho}= \kappa_q(\rho)  \in \kappa_q( \mathcal{W}_q)$. We have 
\begin{enumerate}
\item If, for all $\mathbf{q} \in \mathcal{A}^{n+1}$, $\rho \not \in \mathcal{W}_\mathbf{q}^-$, then $w(\hat{\rho}) = \hinf$ with uniform constants. 
\item Otherwise, there exists a unique $\mathbf{q} \in \mathcal{A}^{n+1}$ such that $\rho \in \mathcal{W}_\mathbf{q}^-$. In that case, for some uniform constants $C>0$ and $h_0 >0$, one has, for $0 < h \leq h_0$, 
\begin{enumerate}[label= (\roman*)]
\item If $\rho \not \in \mathcal{T}_{\delta, \delta_1}$, $w(\hat{\rho}) \leq  h^L$ ; 
\item If $\rho \in \mathcal{T}_{\delta, \delta_1}$, $$w(\hat{\rho}) \leq C  \left( \Pi_{\alpha, \mathbf{q}}(\rho)\right)^2  \left(J^u_{\mathbf{q}}\right)^{d_H - 1 + \varepsilon} h^{(\delta_0- \delta)(d_H + \varepsilon) + \delta -1/2}.$$
\end{enumerate}
\end{enumerate}
\end{prop}

This key proposition is proved in Section \ref{Section_proof_key_prop}. 
We will also require the following proposition : 

\begin{prop}\label{Prop_volume_small_neigh}
Let $\mathbf{q} = q_0 \dots q_n \in \mathcal{A}^{n+1}$ with $n= n(h)$ and assume $\mathcal{W}_\mathbf{q}^-  \neq \emptyset$. Then, for some uniform constant $C>0$, and for $h$ small enough, the following estimate holds : 
$$\text{Vol}\Big(\mathcal{T}_{\delta,\delta_1} \cap \mathcal{W}_\mathbf{q}^- \Big) \leq C h^{2 \delta_1} h^{-(\delta + \delta_1)(d_H + \varepsilon)} \left(J_\mathbf{q}^u\right)^{-(d_H + \varepsilon)}.$$
\end{prop} 

\begin{proof}
We assume that $\mathcal{W}_\mathbf{q}^- \neq \emptyset$. According to Lemma \ref{Lemma_unstable_Jacobian}, there exists $\rho_- \in \mathcal{T}$ such that for all $\rho \in \mathcal{W}_\mathbf{q}^-$, 
\begin{equation}\label{eq_close_to_leave}
d(\rho, W_s(\rho_-)) \leq C \left( J^u_\mathbf{q} \right)^{-1} \varepsilon_0
\end{equation} 
We also consider $\rho_+ \in \mathcal{T}$ such that $d(\rho, \rho_+) \leq 2h^{\delta_1}$. In particular, $d(\rho_-, \rho_+) \ll \varepsilon_0$ and we may consider a point $\rho_O \in W_s(\rho_-) \cap W_u(\rho_+)$ and we decide to work in an adapted chart $\kappa$ centered at $\rho_O$. We want to estimate the volume of $\kappa \Big(\mathcal{T}_{\delta,\delta_1} \cap \mathcal{W}_\mathbf{q}^- \Big)$. We assume that $\mathcal{W}_q$ is included in the domain of this chart (and so is $\mathcal{W}_\mathbf{q}^-$) and we choose this chart such that the image of $W_s(\rho_-) = W_s(\rho_O)$ is given by $\{ (0, \xi), \xi \in V \}$ : this is possible in virtue of Lemma \ref{adapted_chart_0} (by considering $F^{-1}$ instead of $F$ to change the unstable manifold into the stable one) \footnote{In fact, without the assumption on $\kappa$ being symplectic, we may assume that both $W_s(\rho_O)$ and $W_u(\rho_O)$ are rectified.}. 
In virtue of (\ref{eq_close_to_leave}), we have for some uniform constant $C^\prime>0$, 
$$(x,\xi) \in \kappa_q( \mathcal{W}_\mathbf{q}^- ) \implies |x| \leq C^\prime \left( J^u_\mathbf{q} \right)^{-1} \varepsilon_0$$
Let's consider $\Xi(\mathcal{T}) = \{ \xi \in V, \kappa_q^{-1}(0, \xi) \in \mathcal{T} \}$ and let's cover it by $N_s$ intervals of size $2h^{\delta_1}$ centered at point $\xi_1 , \dots, \xi_{N_s} \in \Xi(\mathcal{T} )$.
Since $\overline{\text{dim}} \mathcal{T} \cap W_s(\rho_-)= d_H$ and in virtue of Proposition \ref{upper_box_dim}, we may choose $N_s$ such that
\begin{equation} \label{equation_N_s}
N_s \leq C h^{-\delta_1(d_H + \varepsilon)}
\end{equation} 
for some uniform constant $C>0$. 
 For $1 \leq i \leq N_s$, let's note $\sigma_i = \kappa^{-1}(0, \xi_i) $. The local unstable manifold  passing through $\sigma_i$ can be written, in the chart $\kappa$, as a graph
 $\{ (x,g_i(x)), x \in U_i \}$. We note $$X_i(\mathcal{T}) = \{ x \in U_i, \kappa^{-1}(x,g_i(x) ) \in \mathcal{T} , |x|\leq 2C^\prime  \left( J^u_\mathbf{q} \right)^{-1} \varepsilon_0\}$$ and we cover $X_i(\mathcal{T})$ by $N_{i,u}$ intervals of size $2h^{\delta}$, centered at points $x_{i,j}$, $1 \leq j \leq N_{i,u}$. Lemma \ref{Lemma_number_of_intervals} shows that we can take $N_{i,u}$ such that for all $1 \leq i \leq N_s$, 
\begin{equation} \label{equation_N_u}
N_{i,u} \leq C \left( h^{\delta} J^u_\mathbf{q} \right)^{- d_H - \varepsilon} 
\end{equation}
 for some uniform constant $C$. 
 
 For $1 \leq i \leq N_s$ and $1 \leq j \leq N_{i,u}$, let's also note $\xi_{i,j} = g_i(x_{i,j})$. We claim that there exists a uniform constant $C>0$ such that
 
\begin{equation}\label{Claim_small_neigh}
  \kappa \left(\mathcal{T}_{\delta,\delta_1} \cap \mathcal{W}_\mathbf{q}^-  \right) \subset \bigcup_{ i=1}^{N_s} \bigcup_{j=1}^{N_{i,u}} [x_{i,j} - Ch^{\delta_1}, x_{i,j} +C h^{\delta_1} ] \times [ \xi_{i,j} - Ch^{\delta_1} , \xi_{i,j} + Ch^{\delta_1} ]
\end{equation} 

This claim obviously implies the proposition, by combining it with the bounds on $N_s$ (\ref{equation_N_s}) and the $N_{i,u}$ (\ref{equation_N_u}). We now turn to the proof of this claim. \\
Let's consider $(x,\xi) = \kappa(\sigma) \in  \kappa \left(\mathcal{T}_{\delta,\delta_1} \cap \mathcal{W}_\mathbf{q}^-  \right) $. 
We introduce different points (and encourage the reader to use Figure \ref{Fig_small_neigh} to follow the different steps) : 
\begin{itemize}
\item Since $d(\sigma, \mathcal{T}_+) \leq h^{\delta_1}$, there exists $\sigma_+ \in \mathcal{T}$ such that $d(\sigma, W_u(\sigma_+) ) \leq h^{\delta_1}$. We can replace $\sigma_+$ by the unique point in the intersection $W_u(\sigma_+) \cap W_s(\rho_O)$ and we can note $\kappa(\sigma_+) = (0, \xi_+)$. 
\item Since $\xi_+ \in \Xi (\mathcal{T})$, there exists $i \in \{1, \dots, N_s \}$ such that $|\xi_i - \xi_+| \leq h^{\delta_1}$. In particular, $d(\sigma_i, \sigma_+) \leq C h^{\delta_1}$. 
\item Since $d(\sigma, \mathcal{T}_-) \leq h^{\delta}$, there exists $\sigma_- \in \mathcal{T}$ such that $d(\sigma, W_s(\sigma_-) ) \leq h^{\delta}$. We note $\sigma_O$ the unique point in $W_s(\sigma_-) \cap W_u(\sigma_+)$.
\item We also note $\sigma_i^\prime$ the unique point in $W_s(\sigma_-) \cap W_u(\sigma_i)$. Due to the Lipschitzness of the holonomy maps (with uniform Lipschitz constant), 
$$ d(\sigma_O, \sigma^\prime_i) \leq C d(\sigma_+, \sigma_i) \leq Ch^{\delta_1}$$
\item Due to the local product structure near $\sigma_O$, we have $d(\sigma, \sigma_O)^2 \sim d(\sigma, W_s(\sigma_O) )^2 + d(\sigma,W_u(\sigma_O) )^2 \sim h^{2 \delta_1} + h^{2 \delta}$. It gives 
$d(\sigma, \sigma_O) \leq C h^{\delta_1}$ and hence, $d(\sigma_i^\prime, \sigma) \leq Ch^{\delta_1}$. 
\item Let's note $\sigma_i^\prime =(x^\prime, g_i(x^\prime))$. Since $x^\prime \in X_i(\mathcal{T})$, there exists $j \in \{1 , \dots, N_{i,u} \}$ such that $|x_{i,j} - x^\prime| \leq h^\delta$. Then we have 
$$ d(\sigma_i^\prime, \kappa( (x_{i,j}, \xi_{i,j} ) ) \leq C |x^\prime_i - x_{i,j}| \leq C h^\delta$$
\item We conclude that $d(\sigma, \kappa( (x_{i,j}, \xi_{i,j} )) \leq C h^{\delta_1}$, which gives $|x- x_{i,j}| \leq Ch^{\delta_1}, |\xi- \xi_{i,j}| \leq Ch^{\delta_1}$. 
\end{itemize}

\begin{figure}
\centering
\includegraphics[scale=0.5]{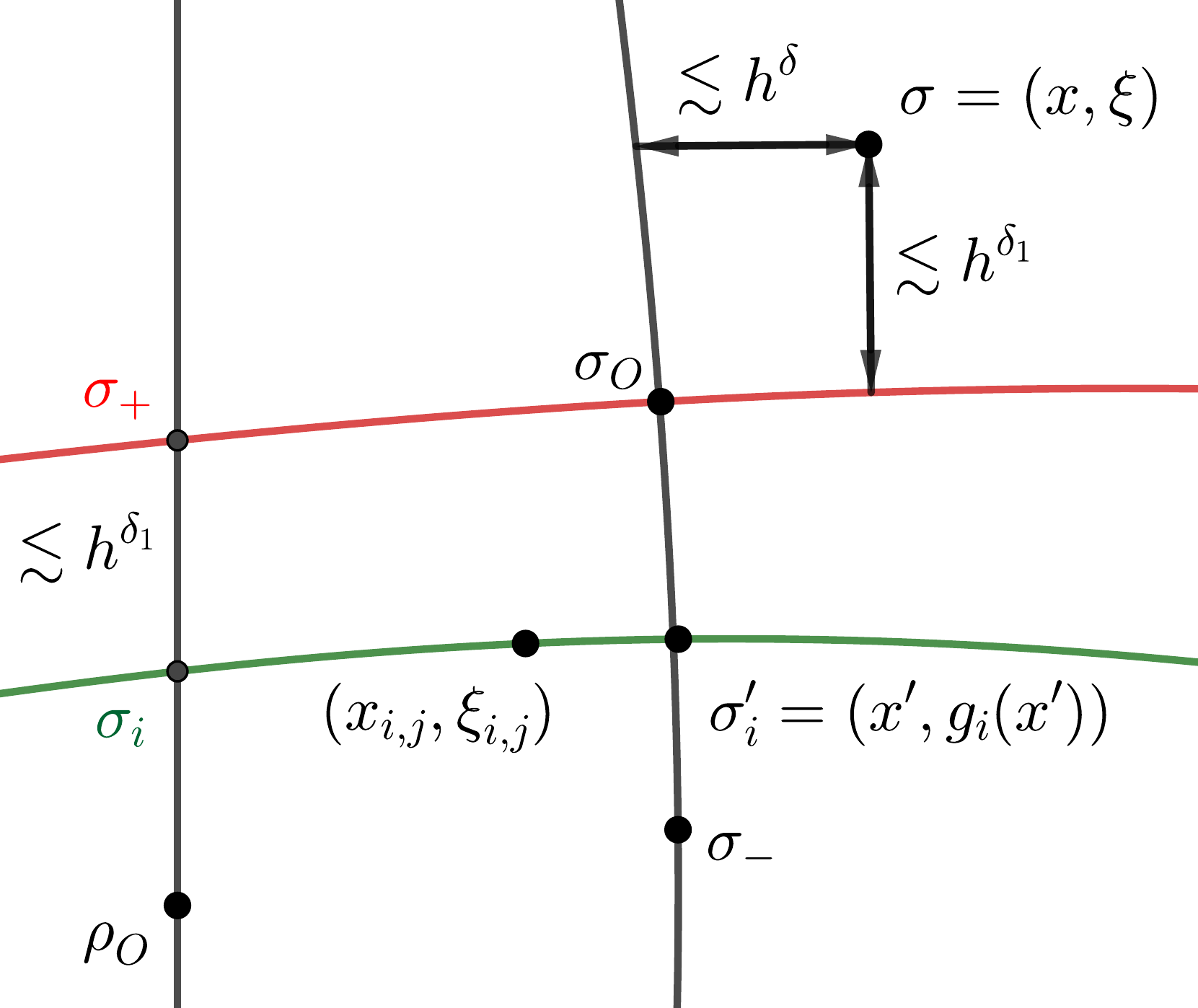}
\caption{The different points introduced in the proof of the claim \ref{Claim_small_neigh}. To alleviate the figure, we use the same notations for a point $\sigma$ and its image trough $\kappa$. }
\label{Fig_small_neigh}
\end{figure}

 \end{proof}

We can now conclude the proof of the main trace estimate. Set $N = 2T_0 + n(h)$. We want to plug the estimates of Proposition \ref{Prop_Key} into (\ref{equation_trace_1}) and (\ref{equation_trace_2}). For $q \in \mathcal{A}$, let's note $$O_q = \kappa_q \left( \mathcal{T}_{\delta, \delta_1} \cap \bigcup_{\mathbf{q} \in \mathcal{A}^{n+1}} \mathcal{W}_\mathbf{q}^- \right)$$ and let's write 
\begin{align*}
\tr( \left(M_t^N\right)^* M_t^N) &\leq C \sup_{q \in \mathcal{A}} \frac{1}{2 \pi h} \int_{\R^2} w(\hat{\rho}) d\hat{\rho}  \\
&\leq \frac{C}{h} \sup_{q \in \mathcal{A}} \left( \int_{O_q} w(\hat{\rho}) d\hat{\rho} + \int_{\R^2 \setminus O_q} w(\hat{\rho}) d\hat{\rho}  \right) \\
& \leq \sup_{q \in \mathcal{A}} \left( C h^{( \delta - \delta_0)(d_H + \varepsilon) + \delta -3/2  } \sum_{ \mathbf{q} \in \mathcal{A}^{n+1}} \int_{\kappa_q\Big(\mathcal{T}_{\delta, \delta_1} \cap \mathcal{W}_{\mathbf{q}}^- \Big)  } \left(\Pi_{\alpha, \mathbf{q}}(\rho)\right)^2 \left( J^u_{\mathbf{q}}\right)^{d_H-1 + \varepsilon}  d\hat{\rho} \right)+ O(h^L) \\
\end{align*}
The last inequality holds since the integral outside $O_q$ can be made $O(h^L)$ by choosing correctly$t$, to make $L$ arbitrarily large. Indeed, using Proposition \ref{Prop_Key}, the part outside $\kappa_q (\mathcal{W}_q)$ is $\hinf$ and the integral on $\kappa_q(\mathcal{W}_q) \setminus O_q$ is $O(h^L)$.

Let $\mathbf{q}=q_0 \dots q_{n-1} \in \mathcal{A}^{n}$. We write $\mathcal{W}_{q_i} = \bigcap_{j=0}^{n_0-1} F^{-1}(\mathfrak{A}_{i,j})$ with $\mathfrak{A}_{i,j} \in \mathcal{Q}$ and for $\rho \in \mathcal{W}_\mathbf{q}^-$.
Let's note $\varphi = -2 \im z t_{ret} - \varphi_u$ and recall that $\alpha = \exp( - \im z t_{ret} ) $. We have uniformly with respect to $\mathbf{q} \in \mathcal{A}^n$ and $\rho \in \mathcal{W}_\mathbf{q}^-$,
\begin{align*}
\left(\Pi_{\alpha, \mathbf{q}}(\rho)\right)^2 \left( J^u_{\mathbf{q}} \right)^{-1} &\leq C\left( \Pi_{\alpha, \mathbf{q}}(\rho)\right)^2 \left( J^u_{nn_0}(\rho)\right)^{-1} \\
&\leq C \prod_{i=0}^{n-1}  \left(  \exp \left(\sum_{j=0}^{n_0-1} \varphi ( F^{in_0 + j} (\rho)) \right) \right)  \\
& \leq C  \prod_{i=0}^{n-1}  \left( \sup_{\rho_i \in \mathcal{W}_{q_i}} \exp \left(\sum_{j=0}^{n_0-1} \varphi ( F^{ j} (\rho_i)) \right) \right) \\
& \leq  C  \prod_{i=0}^{n-1}  \left( C_0 \sup_{\rho_i \in \mathcal{W}_{q_i} \cap \mathcal{T}} \exp \left(\sum_{j=0}^{n_0-1} \varphi ( F^{ j} (\rho_i)) \right) \right) \\
\end{align*}
The last inequality holds for some $C_0 >0$ independent of $n_0$ (and $z$), since $\varphi$ is Hölder continuous (with constant uniform with respect to $z$). Indeed, if $\varepsilon_0$ is small enough, in virtue of Lemma \ref{Local_hyperbolic_1}, there exists $\theta \in [0,1)$ and $C >0$ such that if $\rho_1  \in \mathcal{W}_{q_i}$ and if $\rho_2  \in \mathcal{W}_{q_i} \cap \mathcal{T}$  then $d(F^j(\rho_1), F^j(\rho_2) ) \leq C \theta^{n-j}$. As a consequence, 
$|\varphi(F^j(\rho_1)) - \varphi(F^j(\rho_2)) | \leq C \theta_1^{n-j}$ (with $\theta_1 = \theta^\beta$ for some $0 < \beta \leq 1$). Since $\sum_{j=0}^{n-1} \theta_1^{n-j} \leq \sum_{j=0}^\infty \theta_1^j < + \infty$, we find that 
$$ \exp \left(\sum_{j=0}^{n_0-1} \varphi ( F^{ j} (\rho_1) \right)  \sim \exp \left(\sum_{j=0}^{n_0-1} \varphi ( F^{ j} (\rho_2)) \right)$$ 
For $q \in \mathcal{A}$, let's call $$p_{q} = \sup_{\rho \in \mathcal{W}_{q} \cap \mathcal{T}} \exp \left(\sum_{j=0}^{n_0-1} \varphi ( F^{ j} (\rho))  \right) $$
 and recall that, due to our special choice of the partition $(\mathcal{W}_q)_q$ (see (\ref{nice_formula})), we have $\sum_{q \in \mathcal{A}} p_q \leq e^{ n_0 (P(\varphi) + \varepsilon)}$. We may assume that $n_0$ is big enough so that $C_0 \leq e^{n_0 \varepsilon}$, and hence, $\sum_{q \in \mathcal{A}} C_0 p_q \leq e^{ n_0 (P(\varphi) + 2\varepsilon)}$. 
As a consequence, we find that 
\begin{align*}
\tr( \left(M_t^N\right)^* M_t^N)&\leq Ch^{( \delta - \delta_0)(d_H + \varepsilon) + \delta -3/2  } \sum_{ \mathbf{q} \in \mathcal{A}^{n}} \text{Vol}\Big(\mathcal{T}_{\delta,\delta_1} \cap \mathcal{W}_\mathbf{q}^- \Big) \left( J^u_{\mathbf{q}}\right)^{d_H + \varepsilon} \prod_{i=0}^{n-1} C_0 p_{q_i}\\
&\leq Ch^{( \delta - \delta_0)(d_H + \varepsilon) + \delta -3/2  } \sum_{ \mathbf{q} \in \mathcal{A}^{n}}  h^{2 \delta_1} h^{-(\delta + \delta_1)(d_H + \varepsilon)} \left(J_\mathbf{q}^u\right)^{-(d_H + \varepsilon)} \left( J^u_{\mathbf{q}}\right)^{d_H+ \varepsilon} \prod_{i=0}^{n-1} C_0 p_{q_i}\\ \\
&\leq C h^{-\nu_\varepsilon}  \sum_{ \mathbf{q} \in \mathcal{A}^{n}}  \prod_{i=0}^{n-1} C_0 p_{q_i} =  Ch^{-\nu_\varepsilon}  \left( \sum_{q \in \mathcal{A}} C_0 p_q \right)^n \leq  Ch^{-\nu_\varepsilon}e^{ n n_0 (P(\varphi) + 2\varepsilon)}
\end{align*}
where $$\nu_\varepsilon = d_H + (\delta_0 - \delta)(d_H + \varepsilon) + (1/2 - \delta) + (\delta_1 + \delta) \varepsilon + (2\delta_1   - 1)  - d_H(\delta + \delta_1- 1)= d_H + O(\varepsilon)
$$
(see the definitions of these exponents in (\ref{exponent_delta})). 
Recalling that $nn_0 \leq \vartheta_\varepsilon \log 1/h$, we find that 
$$ \tr( \left(M_t^N\right)^* M_t^N) \leq Ch^{-\nu_\varepsilon} h^{ -\vartheta_\varepsilon (P ( - 2 \im z t_{ret}-\varphi_u ) + 2\varepsilon)}$$
We can finally insert the term $2\vartheta_\varepsilon \varepsilon$ into the $\nu_\epsilon$ and we find that $$ \tr( \left(M_t^N\right)^* M_t^N) \leq Ch^{-\nu_\varepsilon } h^{- \vartheta_\varepsilon P ( - 2 \im z t_{ret}-\varphi_u )}$$ 
This concludes the proof of Proposition \ref{Prop_Key_2}.

\section{Proof of Proposition \ref{Prop_Key}.}\label{Section_proof_key_prop}
In this section we fix some $q=q_0$ and we aim at proving Proposition \ref{Prop_Key}. 
If $\hat{\rho} \not \in \kappa_q(\mathcal{W}_q)$, as we will explain, the estimate in $O \left(\left( \frac{h}{\langle \hat{\rho} \rangle} \right)^\infty \right)$ is nothing but a consequence of the fact that $\WF(B_q A_q B_q^\prime) \Subset \kappa_q \left( \mathcal{W}_q\right)$ and one can for instance apply Lemma 15 in \cite{CoRo}, Chapter 2, Section 3. 

The main part of the Proposition \ref{Prop_Key} concerns points $\hat{\rho}  = \kappa_q(\rho) \in \kappa_q(\mathcal{W}_q)$. 
To prove this proposition, we study separately the actions of the different operators in $e^{-tG} \mathfrak{M}^{n-1} M^{n_0} A_q B_q^\prime \tilde{E}_t$. 
\begin{itemize}
\item First, we analyze the action of $\tilde{E}_t$. We show that it is essentially given by the multiplication by $e^{t g(\rho)}$. 
\item We go on studying the propagation of Gaussian coherent state through the iterated actions of $\mathfrak{M}$. The hyperbolicity of the trajectories leads to a deformation of the Gaussian state. The results we obtain are related to the results of \cite{CoRo} with Hamiltonian flow. In particular, this is where we use the fact that $\vartheta_\varepsilon < 1/6$. The approximation we use fails for longer logarithmic times. 
\item Finally, we analyze the action of $e^{-tG}$ on the evolved coherent states. In a way, we treat this evolved state as a Lagrangian state with rapidly oscillating amplitude, of the form $a(x) e^{ i \frac{\phi(x)}{h}}$. The scale of oscillation of $a$ is larger than $h^\delta$, scale on which $g$ oscillates. We show that, at leading order, the action of $e^{-tG}$ is well approximated by the multiplication by $e^{- tg(x, \phi^\prime(x) ) }$. 
\end{itemize}

\begin{nota}
In the following, we will be lead to consider states $u \in L^2(\R_J)$ such that all the components of $u$ are $\hinf$, except one equal to some $v \in L^2(\R)$. By abuse, we will note $v$ instead of $u$ as soon as the component where $u$ is non zero is explicit in the context. For instance, we can simply note $\varphi_{\hat{\rho}}$ instead of $\tilde{\varphi}_{\hat{\rho}}$ as soon as we specify that $\hat{\rho} \in \kappa_q(\mathcal{W}_q)$. Another example : for any $u \in L^2(\R_J)$ and $q \in \mathcal{A}$, $B_q A_q u$ has only one non zero component at $j_q$ and we can use this component to write $u$. This will be widely used in the sequel since most of the time we will consider this type of elements. 
\end{nota}

\subsection{Preparatory work}
Due to standard properties of Fourier integral operators, we can consider a pseudodifferential operator $\Xi_q$ such that  $\WF(\Xi_q) \subset W_q$, $T^* \R \setminus \WF(1 - \Xi_q) \subset \kappa_q(\mathcal{W}_q)$ and $\Xi_q B_q^\prime B_q  = \Xi_q + \hinf$ (recall that $W_q = \kappa_q(B(\rho_q, 2 \varepsilon_0))$ and that $\mathcal{W}_q \Subset B(\rho_q, 2 \varepsilon_0)$ by construction). 
With these properties, we have in particular
$A_q B_q^\prime \tilde{E}_t = A_q B_q^\prime \tilde{E}_t   \Xi_q+ \hinf$ . This allows us to change harmlessly $\tilde{E}_t$ into $E_t \coloneqq \tilde{E}_ t\Xi_q$ in all the computations below. 
We first write 
\begin{align*}
e^{-tG}\mathfrak{M}^{n-1} M^{n_0} A_q B_q^\prime E_t  &= \sum_{(q_1, \dots, q_{n-1}) \in \mathcal{A}^{n-1} } e^{-t G} M^{n_0} A_{q_{n-1}} \dots M^{n_0} A_{q_1} M^{n_0} A_q B^\prime_q E_t \\
&= \sum_{(q_1, \dots, q_{n}) \in \mathcal{A}^n }  A_{q_n} e^{-t G} M^{n_0} A_{q_{n-1}} \dots M^{n_0} A_{q_1} M^{n_0} A_q B^\prime_q  E_t + \hinf \\
&= \sum_{(q_1, \dots, q_n) \in \mathcal{A}^{n} } A_{q_n} e^{-t G} B^\prime_{q_{n}} M_{q_n, q_{n-1} } \dots M_{q_2, q_1} M_{q_1,q} E_t+ \hinf
\end{align*} 
where 
\begin{equation}
M_{p,q} = B_p M^{n_0} A_q B_q^\prime
\end{equation}
We say that a pair $(p,q)$ is admissible if $F^{n_0}(\mathcal{W}_{q}) \cap \mathcal{W}_{p} \neq \emptyset$.  By standard properties of Fourier integral operators, if $(p,q)$ is not admissible, $M_{p,q} = \hinf$. 
We say that a word $(q_1, \dots, q_n ) \in \mathcal{A}^n$ is admissible if all the pairs $(q_i, q_{i-1})$ are admissible (with $q_0 = q)$. Hence, since $n =O\left( \log \frac{1}{h}\right)$, we can restrict the indices in the above sum to the admissible words.  

Suppose that $(p,q)$ is an admissible pair. By composition of Fourier integral operators, $M_{p,q}$ is a Fourier integral operator associated with the symplectic map $F_{p,q} \coloneqq \kappa_p \circ F^{n_0} \circ \kappa_q^{-1}$. Since $\text{diam}( \mathcal{W}_q) \leq \varepsilon_0$, by taking $\varepsilon_0$ sufficiently small, we can assume that $F^{n_0}(\mathcal{W}_q) $ is included in the domain of $\kappa_p$. Indeed, there exists $\tilde{\rho}_q \in \mathcal{W}_q$ such that $F^{n_0}(\tilde{\rho}_q) \in \mathcal{W}_p$ and hence if $\rho \in \mathcal{W}_q$, $$d(F^{n_0}(\rho), \rho_p) \leq d(F^{n_0}(\rho), F^{n_0}(\tilde{\rho}_q) ) +  d(F^{n_0}(\tilde{\rho}_q),\rho_p)  \leq C \varepsilon_0$$  
We note $(y,\eta)$ the variables in the charts and $(\partial_y, \partial_\eta)$ the canonical basis of $\R^2$  and we have 
\begin{itemize}
\item $F_{p,q}(0) = \kappa_p \circ F^{n_0} (\rho_q) = O (\varepsilon_0) $; 
\item $d_0 F_{p,q}  = d_{F(\rho_q)} \kappa_p  \circ d_{\rho_q}F^{n_0} \circ \left[ d_{\rho_q} \kappa_q  \right]^{-1} $ ; 
\item $d_{\rho_q} F^{n_0} (E_u(\rho_q)) = E_u(F^{n_0}(\rho_q))$ and $\rho \mapsto E_u(\rho)$ is Lipschitz. Hence, if we note $e_u(\rho_q) = \left( d_{\rho_q} \kappa_q  \right)^{-1} (\partial_y) \in E_u(\rho_q)$, due to the definitions of the adapted charts in Lemma \ref{adapted_chart_0}, there exists $\lambda_{p,q} \in \R^*$ such that 
$$ d_{\rho_q}(\kappa_p \circ F^{n_0}) (e_u(\rho_q)) =\lambda_{p,q} \partial_y + O(\varepsilon_0)$$
\item Similarly, $d_0 F_{p,q} (\partial_\eta) = \mu_{p,q} \partial_\eta + O(\varepsilon_0)$ for some $\mu_{p,q} \in \R^*$
\end{itemize}
Eventually, we use the fact that $F_{p,q} - F_{p,q}(0) - d_0 F_{p,q} =O(\varepsilon_0)_{C^1 ( W_q ) }$ and we get that 
\begin{equation}
F_{p,q}(y,\eta) = (\lambda_{p,q} y  + y_r(y, \eta) , \mu_{p,q} \eta + \eta_r(y,\eta) )  , (y,\eta) \in W_q
\end{equation}
where $y_r(y, \eta)$ and $\eta_r(y, \eta)$ are $O(\varepsilon_0)_{C^1}$.
In particular, if $\varepsilon_0$ is small enough, $(x,\xi, y, \eta) \in \Gr(F_{p,q}) \mapsto (x,\eta)$ is a local diffeomorphism near $(0,0,0,0)$. Then, there exists a phase function $\psi_{p,q}$ which generates $F_{p,q}$ in a neighborhood $\Omega$ of $(0,0,0,0)$. Assuming $\varepsilon_0$ small enough, we can assume that $ F_{p,q} (W_q) \times W_q \subset \Omega$. 

As a consequence (see for instance \cite{Al08}, \cite{ZW} Chapter 10), the Fourier integral operator $M_{p,q}$ can be written under the form (\ref{special_form_FIO}), up to $\hinf$, that is, 
\begin{equation}
M_{p,q} u (x) = \frac{1}{2 \pi h } \int_{\R^2} e^{\frac{i}{h} \left( \psi_{p,q}(x,\eta)- y\eta \right) } \alpha_{p,q}(x,\eta) u(y) dy d\eta
\end{equation}
where $\alpha_{p,q}$ is a symbol in $S_{0^+}(\R^2)$. It has an asymptotic expansion
\begin{equation}
\alpha_{p,q} \sim \sum_{j \geq 0} h^j \alpha_{p,q}^{(j)} 
\end{equation}
where $\alpha_{p,q}^{(j)} \in h^{0^-}S_{0^+}$, for all $j \geq 1$ (that is, $\alpha_{p,q}^{(j)}  \in \bigcap_{ \eta >0} h^{-\eta} S_{0^+}$) 
and we have 
\begin{equation}
| \alpha_{p,q}^{(0)} (x,\eta)| =  |\partial^2_{x,\eta} \psi (x,\eta)|^{1/2} \chi_q(\rho)  \times \left( \prod_{i=0}^{n_0-1}  \alpha \circ F^i(\rho)\right)  \quad ; \quad \rho = \kappa_q^{-1}(y,\eta) \; ; \; (x,\xi) = F_{p,q}(y,\eta) 
\end{equation}
Here, we use the fact that in $W_q$, $\alpha_h = \alpha +  O \left(h^{1^-} S_{0^+} \right)$ to put the $O \left(h^{1^-} S_{0^+} \right)$ in $\alpha_{p,q}^{(1)}$. 
Moreover, we have the following support properties : for $j \in \N$, 
\begin{equation}\label{Property_suppor}
(x,\eta) \in \supp (\alpha_{p,q}^{(j)}) \implies (y,\eta) \in \kappa_q \left( \supp (\chi_q) \right) \quad ; \quad (x,\xi) = F_{p,q}(y,\eta) 
\end{equation}
We now pick an admissible word $(q_1, \dots, q_n)$ and for $\hat{\rho} \in \R^2$, we aim at studying
$$ || A_{q_n} e^{-t G} B^\prime_{q_{n}} M_{q_n, q_{n-1} } \dots M_{q_2, q_1} M_{q_1,q} E_t \tilde{\varphi}_{\hat{\rho}} ||$$
We have $M_{q_1, q} E_t = B_{q_1} M^{n_0} A_q B_q^\prime B_q e^{tG} B_q^\prime$. Since $\WF\left( A_q B_q^\prime B_q e^{tG} \right) \subset \supp(\chi_q)$ and $B_q^\prime$ is a Fourier integral operator associated with $\kappa_q^{-1}$, we can find $\tilde{\chi}_q$ such that $\supp(\tilde{\chi}_q) \subset \kappa_q(\mathcal{W}_q) \subset W_q$ and 
$$ A B_q^\prime B_q e^{tG} B_q^\prime = A B_q^\prime B_q e^{tG} B_q^\prime \op(\tilde{\chi}_q) + \hinf$$

To prove the estimate in $O \left(\left( \frac{h}{\langle \hat{\rho} \rangle} \right)^\infty \right)$, we invoke Lemma 15 in \cite{CoRo}, Chapter 2, Section 3,  which allows us to say that, if $\hat{\rho} \not \in \kappa_q(\mathcal{W}_q)$, $$\op(\tilde{\chi}_q) \varphi_{\hat{\rho}} =O \left(\left( \frac{h}{\langle \hat{\rho} \rangle} \right)^\infty \right).$$ Since both $M_{q_n, q_{n-1}} \dots M_{q_1, q}$ and the number of terms in the sum are $O(h^{-K})$ for some $K>0$, we deduce the first part of Proposition \ref{Prop_Key} : 
\begin{lem}
Uniformly for $\hat{\rho} \not \in \kappa_q (\mathcal{W}_q) $, $w(\hat{\rho}) =O \left(\left( \frac{h}{\langle \hat{\rho} \rangle} \right)^\infty \right) $. 
\end{lem}

We now focus on $\hat{\rho} \in \kappa_q(\mathcal{W}_q) \Subset W_q$, for which $F_{p,q} (\hat{\rho})$ is well defined. 
We finish this preparatory subsection with an important computation. First note that the neighborhood $\bigcup_{q \in \mathcal{A}} \mathcal{W}_q$ has been fixed by dynamical considerations. We may assume that the cut-off function $\hat{\chi}$ used in (\ref{escape_function_2}) for the construction of $g$, is chosen such that $\supp \hat{\chi} \Subset \bigcup_{q \in \mathcal{A}} \mathcal{W}_q$. As a consequence, we can apply Proposition \ref{Prop_def_FIO_loc} and we have : 

\begin{lem}\label{Lemma_G_FIO}
For all $q \in \mathcal{A}$, there exists $(g_{j,q})_{j \geq 1} \in S_\delta$ such that for all $N \in \N$, the following holds: 
$$ B_q G B_q^\prime = \op \left( g \circ \kappa_q^{-1} \right) + \sum_{j=1}^{N-1} h^{j(1- 2 \delta)} \op(g_{j,q}) + R_N$$
where $||R_N||_{L^2 \to L^2} \leq C_{2N+M,g} h^{N(1-2 \delta)}$, for some constants $C_{2N+M,g}$ depending on semi-norms of $g$ in $S_\delta$ up to order $2N+M$. 
\end{lem}

\begin{rem}
Even if $g \not \in S_\delta$, it still satisfies $|\partial^\alpha g| \leq C_{|\alpha|,g}  h^{- \delta |\alpha|}$ as soon as $\alpha \neq 0$. This allows us to fairly define these semi-norms. 
\end{rem}

\begin{proof}
Let's note that $g \in \log \frac{1}{h} S_\delta$ so that we can apply Proposition \ref{Prop_def_FIO_loc} to $\left( \log \frac{1}{h} \right)^{-1} g$. We can find differential operators $L_{j,q}$ such that 
$$ B_q G B_q^\prime = \op( g \circ \kappa_q^{-1} ) + \sum_{j=1}^{N-1} h^j \op \left( (L_{j,q} g) \circ \kappa^{-1}_q \right) + O_{L^2 \to L^2} \left( h^N || g||_{C^{2N + M}} \right) $$
In fact, due to the properties of $\hat{g}$, $\partial^\alpha g \in h^{- \delta |\alpha|} S_\delta$ as soon as $\alpha \neq 0$ and the terms $g_{j,q} \coloneqq h^{2 \delta j } (L_{j,q} g) \circ \kappa^{-1}_q \in S_\delta$ for $j \geq 1$. Moreover, the $O$ is in fact an $O( C_{2N + M,g} h^{N(1- 2 \delta)})$, where $C_{2N + M,g}$ depends on semi-norms of $g$ in $S_\delta$ up to order $2N+M$. 
\end{proof}
 
 \subsection{Action of $E_t$} 
 We begin with the action of $E_t$. 
Recall the definition of $E_t = B_q e^{tG} B_q^\prime  \Xi_q$ in (\ref{tilde_E_t}). 

\begin{lem}\label{Lemma_Action_tilde_E_t}
For any $N \in \N$ and $\hat{\rho} = \kappa_q(\rho) \in \kappa_q(\mathcal{W}_q) $,  there exists $\psi_0, \dots, \psi_{2N-1}$ and $r_{N}$ such that 
$$ E_t \varphi_{\hat{\rho}} = \sum_{j=0}^{2N-1} h^{j(1/2 - \delta)} \psi_j + r_{N}$$ satisfying : 
\begin{itemize}
\item $\psi_0 = e^{ t g(\rho)} \varphi_{\hat{\rho}}$ ; 
\item For $1 \leq j \leq 2N- 1$, $\psi_j$ is of the form 
$$\psi_j =  e^{ t g( \rho)}T(\hat{\rho}) \Lambda_h \left( P^{(j)}_{t,h} \Psi_0 \right)$$
where $P^{(j)}_{t,h} $ is a polynomial of degree at most $2j$, with coefficients depending on $t$, $g$ (and hence, $h$) and $\hat{\rho}$. It satisfies, $N_\infty(P^{(j)}_{t,h} )\leq C_{j,t}$ with $h$-independent constants, depending on derivatives of $g$. 
\item $||r_{N}||_{L^2} \leq C_{N} (1+|t|)^{2N+1} h^{N(1- 2\delta)} h^{-K_0t}$ for some $K_0>0$ depending on $g$. 
\end{itemize}
\end{lem}

\begin{proof}
Let's fix $\hat{\rho} = \kappa_q(\rho) \in \kappa_q(\mathcal{W}_q)$, $N \in \N$ and set $\phi(t) =  E_t \varphi_{\hat{\rho}}$. $\phi$ solves the equation 
$$\phi^\prime(t) =  B_q G e^{tG} B_q^\prime\Xi_q \varphi_{\hat{\rho}}$$ Since $B_q B_q^\prime = I $ microlocally near $\WF(\Xi_q)$, we have $ e^{tG} B_q^\prime\Xi_q =  B_q^\prime B_q e^{tG} B_q^\prime\Xi_q+ \hinf$. Hence, up to $\hinf$ , $\psi(t)$ solves 
$\phi^\prime(t) = \tilde{G} \phi(t)$ with $\tilde{G} = B_q G B_q^\prime$. It is enough to find an expansion for the solution of this equation. By Lemma \ref{Lemma_G_FIO}, there exists $C_N$ (depending on $g$) such that, with $G_j = \op (g_{j,q}) $ and $G_0 = \op( g \circ \kappa_q^{-1} ) $, 
$$ \left| \left| \tilde{G} - \sum_{j=0}^{N-1} h^{j(1-2 \delta)} G_j \right| \right| \leq C_N h^{N(1- 2 \delta)}$$
Set $\psi(t) = T(\hat{\rho})^* \phi(t)$. It solves : 
$\psi^\prime(t) = T(\hat{\rho})^* \tilde{G} T(\hat{\rho})\psi(t)$.  
We also set $u(t) = \Lambda_h^* \psi(t)$, which solves $u^\prime(t) = A u(t)$ where $A = \Lambda_h^* T(\hat{\rho})^* \tilde{G} T(\hat{\rho}) \Lambda_h$. Let's also note $\tilde{A}_j = \Lambda_h^* T(\hat{\rho})^* G_j T(\hat{\rho}) \Lambda_h = \text{Op}_1 \left(a_j \right)$ where $a_j (\hat{\zeta} ) = g_{j,q} (\hat{\rho} + h^{1/2}\hat{\zeta})$.
We wish to apply the formalism of Appendix \ref{Appendix_exp} with $H = L^2(\R)$, the operator $A : H \to H$, $\mathcal{C}= \{ P \Psi_0, P \in \C[X] \}$ with initial state $u(0) = \Psi_0$. The parameter $h$ in Appendix \ref{Appendix_exp} is replaced by $\tilde{h}=h^{1/2-\delta}$. 
If $P \in \C[X]$, we approximate the action of $A$ by 
\begin{align*}
A (P \Psi_0) &= \sum_{j=0}^{N-1} \tilde{h}^{2j} \tilde{A}_j (P\Psi_0) + O_N\left( \tilde{h}^{2N} ||P\Psi_0|| \right) \\
&= \sum_{j=0}^{N-1}\tilde{h}^{2j} \left( \sum_{k=0}^{2N-1-2j} \tilde{h}^k A_{j,k}  (P\Psi_0) + R_{N,j}(P\Psi_0) \right) + O_N\left( \tilde{h}^{2N} ||P\Psi_0|| \right) \\
\end{align*}
where, according to Lemma \ref{Lemma_action_pseudo_coherent_state}, 
$$A_{j,k} = \text{Op}_1 \left( \sum_{ \alpha + \beta = k} h^{\delta k} \frac{\partial_x^\alpha \partial_\xi^\beta h_j}{\alpha! \beta!}(0)x^\alpha \xi^\beta \right) \quad ; \quad R_{N,j}(P \Psi_0) = O_{N,j, \deg P} \left( \tilde{h}^{2N-2j} N_\infty(P) \right) $$
where the constant in $O_{N,j, \deg P}$ depend on $g$ trough its semi-norms, but are $h$-independent. 
Gathering the term of same order together, we can write 
$$A (P \Psi_0) = \sum_{l=0}^{2N-1} \tilde{h}^l A_l( P \Psi_0) + O_{N,\deg P} \left( \tilde{h}^{2N} N_\infty(P) \right)$$
here 
$A_l= \sum_{2j+k =l} A_{j,k}$. It is not hard to see that $A_l(P\Psi_0) = P_l \Psi_0$ where $P \mapsto P_l$ is linear and $\deg P_l \leq \deg P + 2l$. 
Since $g_{j,q} \in S_\delta$ if $j \geq 1$ and since $ h^{|\gamma| \delta} \partial^\gamma g= O(1)$, we observe that as soon as $l \geq 1$, there exists $C_l$ depending on $g$ (trough a finite number of semi-norms), but independent of $h$, such that $$N_\infty(P_l) \leq C_l N_\infty(P)$$
Concerning $A_0$, it is clear that it is in fact  $g (\rho) \Id$. 
We now apply the formulas given in Appendix \ref{Appendix_exp} and use the notations introduced in this appendix, that is
$$R_{2N-1} (t) = e^{tA} \Psi_0 - e^{t A_0} \sum_{l=0}^{2N-1} \tilde{h}^l v_l(t)$$
with $v_l$ constructed inductively by (\ref{formula_v_k}) and $v_0 = \Psi_0$. 
Since $A_0$ is a multiplication, $A_k(s) = A_k$ for all $s \in \R$ and we see by induction that  $v_k(t)$ is of the form 
$$v_k(t) = \sum_{l=0}^k t^l P_{l,k} \Psi_0 = P_{k}(t) \Psi_0$$ where $P_{l,k} \in \C[X]$ has degree at most $2k$. In particular, $N_\infty(P_k(t)) \leq c_k (1+|t|)^k$ for some $h$-independent $c_k$ depending on $g$. 
Concerning the remainder, we have 
\begin{align*}
\left| \left| \tilde{r}_{k,2N}(t) \right|\right| &= e^{t g(\rho)} \left| \left| \left( A - \sum_{j=0}^{2N-k-1} \tilde{h}^j A_j \right) v_k(t) \right|\right|  \\& = O_{N,k} \left(  e^{t g(\rho)} \tilde{h}^{2N-k} N_\infty(P_k(t) \right) \\&\leq C_{N,k} (1+|t|)^k e^{t g(\rho)}  \tilde{h}^{2N-k}
\end{align*}
Finally, we recall that
$R_{2N-1}^\prime(t) = AR_{2N-1}(t)+ \sum_{j=0}^{2N-1} \tilde{h}^j \tilde{r}_{j,2N}(t)$. 
Hence, integrating this inequality, we find that 
$$|| R_{2N-1}(t) || \leq  \int_0^{t} ||A || \times  ||R_{2N-1}(s) ||ds  + C_N \tilde{h}^N e^{t g(\rho)} (1+ |t|)^{2N-1} $$
By  a version of Gronwall's lemma, we can find a constant $C_N$ such that
$$ ||R_{2N-1}(t) || \leq C_N \tilde{h}^{2N} e^{ t \max(|g(\rho)|, ||A||)} t^{2N+1}  $$
(where $C_N$ depends on finitely many semi-norms of $g$).  Since $g \in \log(1/h)S_\delta$, we can find $K_0 >0$ such that $\max(|g(\rho)|, ||A||) ) \leq K_0 \log (1/h)$. 
Going back to $\phi(t)$, we have proved the Lemma. 
\end{proof}

\begin{rem}
$t$ is supposed to be fixed, so that the only meaningful term involving $t$ is $h^{-Bt}$. The other mentions of $t$ can be put into the constants $C_N$. All the polynomials depend also on $h$, we will omit to mention it in the subscripts, but we keep in mind that in the following, all the polynomials potentially depend on $h$. Nevertheless, their $N_\infty$-norm can be controlled in an $h$-independent way. 
\end{rem}

\subsection{Repeated actions of $M_{q_i,q_{i-1}}$}
We fix some $\mathbf{q}=q q_1 \dots q_n \in \mathcal{A}^{n+1}$. 
Each term in the development of $E_t \varphi_{\hat{\rho}}$ is a sum of term of the form

$$  e^{ t g( \rho)} T(\hat{\rho}) \Lambda_h (P_0 \Psi_0)$$
with some $P_0 \in \C[X]$ depending on $h$. We now focus on the evolution of each of these terms under the repeated actions of $M_{q_i,q_{i-1}}$. 
We recall that this operator has the form 
\begin{equation}
M_{p,q} u (x) = \frac{1}{2 \pi h } \int_{\R^2} e^{\frac{i}{h} \left( \psi_{p,q}(x,\eta)- y\eta \right) } \alpha_{p,q}(x,\eta) u(y) dy d\eta
\end{equation}
with
\begin{equation}
\alpha_{p,q} \sim \sum_{j \geq 0} h^j \alpha_{p,q}^{(j)} 
\end{equation}
This will allow us to use Proposition \ref{Prop_propagation_coherent_states}, but we will have to deal with two different scales of asymptotic expansion : $h$ and $h^{1/2}$. To simplify the notations in this context, we note for $1 \leq i \leq n$, 
\begin{align*}
&M_{q_i,q_{i-1}} = M_i \\
&\psi_{q_i,q_{i-1}} = \psi_i \\
&F_{q_i,q_{i-1}} = F_i \\
&F^{(i)} = F_i \circ\dots \circ F_1= \kappa_{q_i} \circ F^{n_0 i} \circ \kappa_q^{-1}\\
&\alpha_{q_i,q_{i-1}}^{(j)} = \alpha_i^{(j)}  \\
\end{align*}
For $0 \leq i \leq n$, we also note $\hat{\rho}_i = F_i \circ \dots \circ F_1 (\hat{\rho})$ (with $\hat{\rho}_0 = \hat{\rho}$) and set $\hat{\rho}_i = (x_i, \xi_i)$. 

We fix a parameter $N$ and we start with an initial state 
\begin{equation}
u_0 = T(\hat{\rho}_0) \Lambda_h \big( P_0\Psi_0 \big)
\end{equation}
with $P_0$ a polynomial of degree $d_0$. Our aim is to show that we have an asymptotic expansion for $u_i = M_i \dots M_1 u_0$ of the form 

$$ u_i = \sum_{2j +k <2N }  h^j h^{k/2} u_i^{(j,k)} + r_i^{(N)}$$
where $u_i^{(j,k)} $ has the form 
$$T(\hat{\rho}_i) \mathcal{M}\left(d_{\hat{\rho}} F^{(i)} \right) \Lambda_h \left( P_i^{(j,k)} \Psi_0 \right) $$ with $P_i^{(j,k)}$ polynomial and 
with a good control on $r_i^{(N)}$. 
For $1 \leq i \leq n$ and $0 \leq j \leq N-1$, we apply Proposition \ref{Prop_propagation_coherent_states} to the operator 
$$(M_i^{(j)} u)(x) = \frac{1}{2\pi h} \int_{\R^2} e^{\frac{i}{h} (\psi_i(x,\eta) - y \eta)} \alpha_i^{(j)}(x,\eta) u(y) dy d\eta$$
and for a state of the form 
$$u = T(\hat{\rho}_{i-1}) \mathcal{M}\left( d_{\hat{\rho}}F^{(i-1)} \right) \Lambda_h \left( P \Psi_0 \right) $$
For each such polynomial $P$, we can find a family $Q_i^{(j,k)}(P)$ of polynomials such that 
\begin{itemize}[nosep]
\item $Q_i^{(j,0)}(P) =\frac{\alpha_i^{(j)}(x_i, \xi_{i-1})}{ | \partial^2_{ x \eta}  \psi_i (x_i, \xi_{i-1}) |^{1/2}}P$ (up to a multiplicative factor of norm 1 that we omit in the proof) ; 
\item $Q_i^{(j,k)}(P)$ is a polynomial of degree $\deg P + 3k$  and the map $P \mapsto Q_i^{(j,k)}$ is linear, with coefficients depending on $F^{(i)}
$ and the derivatives of $\psi_i$ and $\alpha_i^{(j)}$ at $(x_i, \xi_{i-1})$ up to the $3k$-th order and we have $$N_\infty(Q_i^{(j,k)}) \leq C_{3k}(\psi_i)||\alpha_i^{(j)}||_{C^k} ||d_{\hat{\rho}} F^{(i)}||^{3k} N_\infty(P)$$
Moreover, if $(x_i, \xi_{i-1}) \not \in \supp \alpha_i^{(j)}$, then $Q_i^{(j,k)} = 0$.
\item 
for every $N \in \N$, 
\begin{equation}\label{Formule_step_i_k}
M_i^{(j)} \Big( T(\hat{\rho}_{i-1}) \mathcal{M}(d_{\hat{\rho}} F^{(i-1)}) \Lambda_h [ P \Psi_0] \Big) = T(\hat{\rho}_i) \mathcal{M}( d_{\hat{\rho}} F^{(i)}) \Lambda_h  \left[ \sum_{k=0}^{N-1} h^{k/2} Q_i^{(j,k)}(P) \Psi_0 \right] + R_i^{(j,N)}(P)
\end{equation}
with $$||R_i^{(j,N)}(P) ||_{L^2} \leq  h^{N/2} C_{3N+M}(\psi_i)||\alpha^{(j)}_i||_{C^{N + M}}||d_{\hat{\rho}} F^{(i)}||^{3N} K_{N,\deg P} N_\infty(P) $$ 
\end{itemize}

\begin{rem}\label{Rem_condition_support}
In virtue of the properties of $\alpha_i^{(j)}$, the condition $(x_i, \xi_{i-1}) \in \supp \alpha_i^{(j)} \iff F^{i n_0}(\rho) \in \supp (\chi_q \alpha)$. 
\end{rem}

We also write the expansion of $M_i$ in the form, for every $N$, 
\begin{equation}\label{Formule_step_i_j}
M_i = \sum_{j=0}^{N-1} h^j M_i^{(j)} + \widetilde{S}_i^{(N)}
\end{equation}
with 
$$ ||\widetilde{S}_i^{(N)} || \leq \widetilde{C}_{i,N,\varepsilon} h^{N(1- \varepsilon)}$$
Since $M_i$ belongs to the finite family $(M_{p,q})$, we can replace $\widetilde{C}_{i,N,\varepsilon}$ by $\widetilde{C}_{N,\varepsilon} = \sup_{i}\widetilde{C}_{i,N,\varepsilon}$. 

\vspace{0.5cm} 
We now give the iteration formulas for the required expansion. 
We state $P_0^{(0,0)} = P_0$ and $P_0^{(j,k)} =0$ for the other values of $(j,k)$. For $2j + k < 2N$, we define inductively $P_i^{(j,k)}=P_{i,P_0}^{(j,k)}$ by the formula (to alleviate the notations, we will omit to specify the dependence in $P_0$ when this is not necessary) : 
\begin{equation}\label{Formule_iteration}
P_i^{(j,k)} = \sum_{j_1 + j_2 =j} \sum_{k_1 + k_2 = k} Q_i^{(j_2,k_2)} \left( P_{i-1}^{(j_1,k_1)} \right) 
 \end{equation}
Concerning the remainder term, we set 
\begin{multline}\label{Formule_iteration_reste}
r_i^{(N)}=r_{i,P_0}^{(N)} = M_i \left( r_{i-1}^{(N)} \right) + \sum_{2j + k <2N} h^{j + k/2} \widetilde{S}_i^{(N - j - \lceil k/2 \rceil)} \left( u_{i-1}^{(j,k)} \right) \\+ \sum_{2j_1 + 2j_2 + k_1 <2N} h^{N-j_1- j_2 - k_1/2}R_i^{(j_1, 2(N-j_1-j_2)-k_1)}\left( P_{i-1}^{(j_1,k_1)} \right)
\end{multline}

\begin{lem}\label{Lemma_Expansion_propagation}
With the above notations, we have for $1 \leq i \leq n$, 
$$  u_i = \sum_{2j +k <2N} h^j h^{k/2} u_{i,P_0}^{(j,k)} + r_{i,P_0}^{(N)} \quad ; \quad u_i^{(j,k)} = T(\hat{\rho}_i) \mathcal{M}\left(d_{\hat{\rho}} F^{(i)} \right) \Lambda_h \left( P_{i,P_0}^{(j,k)} \Psi_0 \right)$$
\end{lem}

We now analyze these formulas to understand more precisely these terms and obtain a good control of the remainder. In particular, concerning the polynomial $P_i^{(j,k)}$, we want to control their degree and the norms of their coefficients. 
\paragraph{Leading term. }
First note that the leading term (that is the term $(0,0)$) has a nice form. Indeed, up to a factor of norm 1, it is given by

\begin{align*}
P_i^{(0,0)} &= P_0 \times \prod_{l=1}^i \frac{\alpha_l^{(0)}(x_l, \xi_{l-1})}{| \partial^2_{ x \eta}  \psi_l(x_l, \xi_{l-1}) |^{1/2}} \\
&= P_0 \times  \prod_{l=0}^{i-1} \left[ \left( \chi_{q_l} \prod_{j=0}^{n_0-1} \alpha \circ F^j  \right) ( F^{l n_0} (\rho) )\right]
\end{align*}
The product on the right plays a crucial role in the analysis. Let's note $$p_{\alpha,l}(\rho) = \left( \chi_{q_l} \prod_{j=0}^{n_0-1} \alpha \circ F^j  \right) ( F^{l n_0} (\rho) )\quad ; \quad \pi_{\alpha,i}(\rho) = \prod_{l=0}^{i-1} p_{\alpha,l}(\rho)$$
  We remark that 
$$\pi_{\alpha,i}(\rho) \leq \Pi_{\alpha, q_0 \dots q_{i-1} } (\rho)$$
Recall that $ \Pi_{\alpha, \mathbf{q}}(\rho) =  \prod_{i=0}^{nn_0-1} \alpha\left( F^i(\rho) \right)$. 
To simplify the notations, let's note $\Pi_{\alpha, i} = \Pi_{\alpha, q_0 \dots q_{i-1} } $.

Moreover, combining the support property (\ref{Property_suppor}) of $\alpha_i^{(j)}$, Remark \ref{Rem_condition_support} and the properties of $Q_i^{(j,k)}$ given by Proposition \ref{Prop_propagation_coherent_states}, we see that for $\mathbf{q} = q_0 \dots q_{i-1}$, 

\begin{equation}\label{Property_support}
\rho \not \in \mathcal{W}_\mathbf{q}^- \implies \forall j,k \in \N , \quad  Q_i^{(j,k)} = 0
\end{equation}

\paragraph{Analysis of the polynomial $P_{i,P_0}^{(j,k)}$.}
According to (\ref{Property_support}), we  assume that $\rho = \kappa_q^{-1}(\hat{\rho}) \in \mathcal{W}_{q_0 \dots q_{i-1}}^-$. Otherwise, there is nothing more to say. 
We start by the easiest part of the analysis : 

\begin{lem}
For all $0 \leq i \leq n$ and all $(j,k)$ with $2j + k < 2N$, $P_{i,P_0}{(j,k)}$ is of degree at most $3k+ \deg P_0$. 
\end{lem}

\begin{proof}
We argue by induction on $i$. This is obvious for the case $i=0$. To pass from $i-1$ to $i$, we use (\ref{Formule_iteration}) which shows that
\begin{align*}
\deg P_i^{(j,k)}& \leq \max_{j_1 + j_2=j, k_1+k_2=k} \deg Q_i^{(j_2,k_2)}\left(P_{i-1}^{(j_1,k_1)} \right) \\
& \leq \max_{j_1 + j_2=j, k_1+k_2=k} 3k_2 + \deg P_{i-1}^{(j_1,k_1)} \\
& \leq \max_{j_1 + j_2=j, k_1+k_2=k} 3k_2 +3k_1 + d_0 \\
&\leq 3k + d_0
\end{align*}
\end{proof}

The analysis of $N_\infty \left( P_i^{(j,k)} \right)$ is a bit more tedious.

\begin{lem}\label{Lemme_estimates_Pijk}
For every $\varepsilon>0$, there exists a family of constants $C_{j,k,\varepsilon}$ depending on the dynamical system and on $M_h$ such that: 
For all $0 \leq i \leq n$ and all $(j,k)$ with $2j + k < 2N$, we have 
$$N_\infty \left( P_{i,P_0}^{(j,k)} \right) \leq C_{j,k,\varepsilon} h^{-k \varepsilon} i^{2j +k} \Pi_{\alpha,i} (\rho) \left(J^u_{q_0\dots q_{i-1}}\right)^{3k} N_\infty(P_0)$$ 

\end{lem}

\begin{rem}
The dependence on $i$ is of major importance. Here, $i \leq n =O( \log 1/h)$. Hence, the term $i^{2j+k}$ is essentially harmless compared to the second part $ \Pi_{\alpha,i} (\rho) \left(J^u_{q_0 \dots q_{i-1}}\right)^{3k}h^{-k \varepsilon}$. The factor $\Pi_{\alpha,i} (\rho) $ does not depend on $k$ and is common to all the terms. It can be put in front of the all expansion. On the contrary, the growth of $J^u_{q_0 \dots q_{i-1}}$ influences the precision and the validity of the expansion. So that the expansion holds, we need to require 
 $$\left(J^u_{q_0 \dots q_{i-1}}\right)^3 h^{- \varepsilon} \ll h^{-1/2}$$
 As a consequence, this is where the assumption $$\vartheta_\varepsilon < \frac{1-4 \varepsilon}{6 \lambda_{\max}}$$ (see its definition in (\ref{definition_vartheta})) is important and lead to a valid expansion.
 \end{rem}
 \begin{rem}
 The constant $C_{j,k,\varepsilon}$ depends on $M_h$ through its amplitude $\alpha_h$ as a Fourier integral operator in a class $ I_{\eta} (\R \times \R, \Gr(F)^\prime)$ (for some $\eta=\eta(\varepsilon)$) and it depends only a finite number $N_{j,k}$ of derivatives. 
\end{rem}

\begin{proof}
To alleviate the notations, we renormalize $P_0$ so that $N_\infty(P_0) =1$. 
We fix $(j,k)$ such that $2j + k < 2N$. By iterating (\ref{Formule_iteration}), we find that 
$$ P_i^{(j,k)} = \sum_{ \substack{ j_1 + \dots + j_i = j \\k_1 + \dots + k_i = k }} Q_i^{(j_i,k_i)} \circ \dots \circ Q_1^{(j_1,k_1)} (P_0)$$ 
We now use the simple following idea : when $i$ is large that is when $i \gg 2j + k$, and when $j_1 + \dots + j_i = j $ and $k_1 + \dots + k_i = k $, most of the couples $(j_l,k_l)$ are equal to $(0,0)$. From a more quantitative point of view, we have 
$$ \# \{ 1 \leq l \leq i , (j_l,k_l) \neq (0,0)  \} \leq 2j+k $$
Indeed, $$ 2j + k  = 2 (j_1 + \dots + j_i) + (k_1 + \dots + k_i) \geq \# \{ 1 \leq l \leq i , (j_l,k_l) \neq (0,0)  \} $$
Let's note $\mathcal{P}(i,2j+k)$ the set of subsets of $\{1, \dots, i \}$ of cardinals smaller than $2j+k$. For $\mathcal{L} \in\mathcal{P}(i,2j+k)$ we define the set of indices $\mathcal{I}_\mathcal{L} \subset \N^i \times \N^i$ by 
$$(\overrightarrow{j}, \overrightarrow{k})= \big((j_1, \dots, j_i) , (k_1, \dots, k_i) \big) \in \mathcal{I}_\mathcal{L} \iff \left\{ \begin{array}{l}
j_1 + \dots + j_i = j \\
k_1 + \dots + k_i = k \\
\forall 1 \leq l \leq i , (j_l, k_l) \neq (0,0) \iff l \in \mathcal{L}
\end{array} \right. $$
With these notations, we have
$$ P_i^{(j,k)} = \sum_{ \mathcal{L} \in \mathcal{P}(i,2j+k)} \sum_{(\overrightarrow{j}, \overrightarrow{k}) \in \mathcal{I}_\mathcal{L}} Q_i^{(j_i,k_i)} \circ \dots \circ Q_1^{(j_1,k_1)} (P_0)$$
Let's fix $\mathcal{L} \in  \mathcal{P}(i,2j+k)$ and $(\overrightarrow{j}, \overrightarrow{k}) \in \mathcal{I}_\mathcal{L}$. Let's write $\mathcal{L}= \{ l_1 < \dots < l_m \}$.  Since $Q_l^{(0,0)}$ is simply a multiplication by $p_{\alpha,l}(\rho)$, we have : 

$$ Q_i^{(j_i,k_i)} \circ \dots \circ Q_1^{(j_1,k_1)} (P_0)  =\left(\prod_{l \not \in \mathcal{L}} p_{\alpha,l}\right) \times  Q_{l_m}^{(j_{l_m},k_{l_m})} \circ \dots  \circ  Q_{l_1}^{(j_{l_1},k_{l_1})}(P_0)$$
Using Proposition \ref{Prop_propagation_coherent_states}, we can estimate 
$$ N_\infty \left( Q_{l_m}^{(j_{l_m},k_{l_m})} \circ \dots  \circ  Q_{l_1}^{(j_{l_1},k_{l_1})}(P_0)\right)  \leq  N_\infty(P_0) \times \prod_{p = 1}^m C_{3k_{l_p}}(\psi_{l_p}) || \alpha^{(j_{l_p})}_{l_p} ||_{C^{k_{l_p}}} || d_{\hat{\rho}} F^{(l_p)} ||^{3 k_{l_p}} $$
For $1 \leq l \leq i$, $\psi_l$ (resp. $\alpha_l^{(\cdot)})$ belongs to a finite family of functions (corresponding to the finite number of admissible transitions). Hence, recalling that $\alpha_l^{(j)} \in S_{0^+}$ if $j=0$ and $h^{0^-} S_{0^+}$ if $j \geq 1$, we can find a global uniform constant depending on the dynamical system, and on a certain number $N_{j,k}$ of derivatives of $\alpha$ such that for all $j^\prime \leq j$, $k^\prime \leq k$ and for all $l$, 
$$C_{3k^\prime}(\psi_{l}) || \alpha^{(j^\prime)}_{l} ||_{C^{k^\prime}}  \leq C_{j,k,\varepsilon}   \left\{ \begin{array}{l}
h^{- k^\prime \varepsilon/2}  \text{ if }j^\prime=0 \\
h^{- k^\prime\varepsilon/2} h^{-\eta_{k,j}} \text{ if } j^\prime \geq 1 
\end{array}  \right. .$$
where we artificially choose $\eta_{k,j} = \frac{k \varepsilon }{2j}$ and use the fact that $\alpha_l^{(j)} \in S_{\varepsilon/2}$ (resp. $h^{-\eta_{j,k}} S_{\varepsilon/2}$) if $j^\prime=0$ (resp. $j^\prime \geq 1$). 
As a consequence, since $j_{l_1} + \dots + j_{l_m}=j$ and $k_{l_1} + \dots + k_{l_p} = k$, we have 
$$ N_\infty \left( Q_{l_m}^{(j_{l_m},k_{l_m})} \circ \dots  \circ  Q_{l_1}^{(j_{l_1},k_{l_1})}(P_0)\right)  \leq    C_{j,k,\varepsilon}^m N_\infty(P_0)  h^{-  k \varepsilon/2}h^{-j \eta_{k,j}} \left( \sup_{ 1 \leq l \leq i } ||d_{\hat{\rho}} F^{(l)}|| \right)^{3k}$$
Since $m \leq 2j +k$, there exists a global constant, still denoted $C_{j,k,\varepsilon}$, such that, uniformly in $\mathcal{I}_\mathcal{L}$, 
 $$N_\infty \left( Q_{l_m}^{(j_{l_m},k_{l_m})} \circ \dots  \circ  Q_{l_1}^{(j_{l_1},k_{l_1})}(P_0)\right) \leq C_{j,k,\varepsilon} N_\infty(P_0)  h^{- k \varepsilon} \left( \sup_{ 1 \leq l \leq i } ||d_{\hat{\rho}} F^{(l)}|| \right)^{3k}$$
 We remark that for $1 \leq l \leq n$, $ || d_{\hat{\rho}} F^{(l)} || \leq C || d_{\hat{\rho}} F^{n_0 l}  || \leq C J^u_{q_0 \dots q_{l-1}}$.
 Finally, since $|\alpha| \geq e^{- \tau_m}$ in the neighborhood $\bigcup_{q \in \mathcal{A}} \mathcal{W}_q$ of $\mathcal{T}$, 
 we see that for every $(\overrightarrow{j}, \overrightarrow{k} )\in \mathcal{I}_\mathcal{L}$ we have
 $$ N_\infty\left( Q_i^{(j_i,k_i)} \circ \dots \circ Q_1^{(j_1,k_1)} (P_0) \right)  \leq C_{j,k,\varepsilon} N_\infty(P_0)  h^{-k \varepsilon} \Pi_{\alpha,i}(\rho) \left( J^u_{q_0 \dots q_{i-1}}\right)^{3k}$$
 We can now conclude the proof. Indeed, we have 
 $$N_\infty\left( P_i^{(j,k)}\right) \leq  \sum_{ \mathcal{L} \in \mathcal{P}(i,2j+k)}  \# \mathcal{I}_\mathcal{L}\times  \left( C_{j,k,\varepsilon} N_\infty(P_0)  h^{-k \varepsilon} \Pi_{\alpha,i}(\rho) \left( J^u_{q_0 \dots q_{i-1}}\right)^{3k} \right)$$
 If $\mathcal{L} \in \mathcal{P}(i,2j+k) $, we estimate (crudely) the cardinal of $\mathcal{I}_\mathcal{L}$ by 
 $$\#\mathcal{I}_\mathcal{L} \leq  (j+1)^{\#\mathcal{L}} (k+1)^{\#\mathcal{L}} \leq  (j+1)^{2j+k} (k+1)^{2j+k} $$
 Finally, 
 $$i \mapsto \sum_{ \mathcal{L} \in \mathcal{P}(i,2j+k)} 1 = \sum_{m=0}^{2j+k} { i \choose m } $$
 is a polynomial function of $i$, of degree $2j+k$. Hence, thee exists $C_{j,k}$ such that $$\left| \sum_{ \mathcal{L} \in \mathcal{P}(i,2j+k)} { i \choose \#\mathcal{L} } \right| \leq C_{j,k} i^{2j+k}$$
  This concludes the proof. 
\end{proof}

\paragraph{Control of the remainder.}  Armed with Lemma \ref{Lemme_estimates_Pijk} and the iterative formula (\ref{Formule_iteration_reste}), we can deduce a control for the remainder term. 
Let's consider $B \geq1$ such that $||M_{q,p} || \leq B$ for all admissible pair $(q,p)$ (it is possible to take $B \leq (1+ \varepsilon)||\alpha||_\infty$, or even with $\varepsilon$ going to $0$ as $h \to 0$, but the precise value of $B$ is not relevant for this term). For this reason, we will also get rid of the precise value of $\Pi_{\alpha, i}$ and assume that $||\alpha||_\infty \leq B$ so that $|| \Pi_{\alpha,i}|| \leq B^i$.

Plugging the previous estimates into (\ref{Formule_iteration_reste}), we get 
\begin{multline*}
||r_i^{(N)} || \leq B ||r_{i-1}^{(N)} || + \sum_{2j + k <2N} h^{j + k/2} \widetilde{C}_{N-j-\lceil k/2 \rceil, \varepsilon} h^{(N - j -\lceil k/2 \rceil) (1- \varepsilon)} ||u_{i-1}^{(j,k)}||  \\ 
+ \sum_{2j_1 + 2j_2 + k_1 < 2N} h^{N} C_{3(2N-2j_1-2j_2-k_1)+M}(\psi_{i-1})h^{-\varepsilon M} \left( h^{-\varepsilon} ||d_{\hat{\rho}} F^{(i)}||^3\right)^{2N-2j_1-2j_2-k_1} K_{3k_1,3k_1+\deg P_0 } N_\infty(P_{i-1}^{(j_1,k_1)}  ) 
\end{multline*}
Recall that $||P\Psi_0||_{L^2} \leq K_{\deg P} N_\infty(P) $ for some family of constants $K_n$ depending only on $n$. 
By the expression of $u_i^{(j,k)}$ , we have 
$$ ||u_i^{(j,k)} ||_{L^2} = ||P_i^{(j,k)} \Psi_0||_{L^2} \leq K_{3k+ \deg P_0} N_\infty\left( P_i^{(j,k)} \right)$$
We also recall that we can bound $||d_{\hat{\rho}} F^{(i)} ||$ by $CJ_{q_0 \dots q_{i-1}}$ for some global constant $C$.
\begin{multline*}
||r_i^{(N)} || \leq B ||r_{i-1}^{(N)} || + \sum_{2j + k <2N} h^{N(1-\varepsilon)} \widetilde{C}_{N-j-\lceil k/2 \rceil, \varepsilon}   K_{3k} C_{j,k,\varepsilon} B^i i^{2j+k} \left(J^u_{q_0 \dots q_{i-1}}\right)^{3k} \\ 
+ \sum_{2j_1 + 2j_2 + k_1 < 2N}  h^{N} C_{N,j,\varepsilon}h^{-M\varepsilon} K_{3k_1} C_{j_1,k_1,\varepsilon} B^i i^{2j_1+k_1} \left( h^{-\varepsilon} \left(J^u_{q_0 \dots q_{i-1}} \right)^3 \right)^{2(N-j_1-j_2)}
\end{multline*}
Finally, we plug the bound $J^u_{q_0\dots q_{i-1} } \leq C_\varepsilon e^{i \lambda_{\max}(1 + \varepsilon)}$ into the previous inequality. We can find a constant $C_{N,\deg P_0,\varepsilon}$ such that 
$$ ||r_i^{(N)}|| \leq B ||r_{i-1}^{(N)}|| + C_{N,\deg P_0, \varepsilon}B^i i^{2N} e^{6N i \lambda_{\max}(1+ \varepsilon) }  h^{N(1-2\varepsilon)}h^{-M \varepsilon}   $$  
This being valid for all $1 \leq i \leq n$, by induction on $i$, we find that 
\begin{align*}
||r_i^{(N)}|| &\leq \sum_{l=0}^i B^{i-l} \times C_{N,\deg P_0,\varepsilon} B^l l^{N} e^{6N l \lambda_{\max} (1+ \varepsilon)}h^{N(1-2\varepsilon)}h^{-M \varepsilon}   \\
&\leq  C_{N,\deg P_0,\varepsilon} B^i \sum_{l=0}^i l^{N} e^{ 6N l \lambda_{\max}(1 + \varepsilon)}h^{N(1-2\varepsilon)}h^{-M \varepsilon}  
\end{align*}
Let $c_{N,\varepsilon}>0$ be such that $\sum_{l=0}^i l^{N} e^{ 6N l \lambda_{\max}(1+ \varepsilon)}  \leq c_{N,\varepsilon} e^{6N i \lambda_{\max}(1+\varepsilon)^2}$ for all $i \in \N$. This gives, for a constant $C_{N,\deg P_0,\varepsilon}$, 
$$||r_i^{(N)}|| \leq C_{N,\deg P_0,\varepsilon} B^i e^{6N i \lambda_{\max}(1+\varepsilon)^2} h^{N(1-2\varepsilon)}h^{-M \varepsilon}  $$
To conclude, recall that 
$n(h) \leq \vartheta_\varepsilon \log 1/h$ with $\vartheta_\varepsilon = \frac{1- 4\varepsilon}{6 \lambda(1+\varepsilon)^2}$. Hence, as soon as $i \leq n(h)$, $ e^{6N i \lambda_{\max}(1+\varepsilon)^2} \leq h^{-N (1 -4 \varepsilon)}$ and this shows the following lemma 

\begin{lem}\label{Lemma_control_remainder}
There exists constants $C_{N,d,\varepsilon}$ such that for all $N \in \N$ and for all $P_0 \in \C[X]$, we have for all $1 \leq i \leq n(h)$, 
$$ ||r_{i,P_0}^{(N)}|| \leq C_{N,\deg P_0, \varepsilon} h^{2N \varepsilon} h^{-K} N_\infty(P_0)$$ 
with $K = \vartheta_\varepsilon \log B + M \varepsilon$.
\end{lem}

\paragraph{First consequences.}
Since $N$ can be taken arbitrarily large, we recover the known fact that a wave packet centered at $\hat{\rho}$ is changed after $n$ steps into an excited squeezed state centered at $F^{(n)}(\hat{\rho})$. The squeezing is governed by the unstable Jacobian along the orbit of $\rho$. In particular, we obtain the expected following corollary, which gives the first point in Proposition \ref{Prop_Key}. 

\begin{cor}
Let's note $\mathbf{q} = q_0 \dots q_n \in \mathcal{A}^{n+1}$. Let $\hat{\rho} \in \kappa_q(\mathcal{W}_q)$ and let us note $\rho = \kappa_q^{-1} (\hat{\rho})$. 
\begin{itemize}
\item If $\rho \not \in \mathcal{W}_{\mathbf{q}}^-$, then $A_{q_n} e^{-tG}  B^\prime_{q_n} M_{q_n,q_{n-1}} \dots M_{q_1,q} E_t \varphi_{\hat{\rho}} = \hinf$.
\item If $\rho \in \mathcal{W}_{\mathbf{q}}^-$, $$ e^{-tG} \mathfrak{M}^{n-1} M^{n_0} A_q B_q^\prime E_t \varphi_{\rho} = A_{q_n} e^{-tG}  B^\prime_{q_n} M_{q_n , q_{n-1}} \dots M_{q_1,q_0} E_t \varphi_{\rho} + \hinf$$
\end{itemize}
with constants independent of $\mathbf{q}$ and $\rho$. 
\end{cor}

\begin{proof}
This is a consequence of the previous results and the fact that $\WF(B_{q_n} A_{q_n} e^{-tG} B_{q_n}^\prime) \Subset \kappa_{q_n} (\mathcal{W}_{q_n})$. 
\end{proof}

Moreover, we can combine Lemma \ref{Lemma_Action_tilde_E_t} (the running index in the formula of Lemma \ref{Lemma_Action_tilde_E_t} was $j$, it becomes $l$ in the sum below) and Lemma \ref{Lemma_Expansion_propagation} to get :

\begin{prop}\label{Proposition_expansions_before_last_exp}
Assume that $\hat{\rho} = \kappa_q(\rho) \in \kappa_q(\mathcal{W}_q)$ with $\rho \in \mathcal{W}_{q_0 \dots q_n}^-$. Then, for any $N \in \N$, we have the following expansion (with $n=n(h)$)
$$ M_{q_n,q_{n-1}} \dots M_{q_1,q_0} E_t \varphi_{\hat{\rho}} = \sum_{2j + k +l < 2N} h^{j +k/2}h^{l(1/2- \delta)} u_n^{(j,k,l)} + R_n^{(N)}$$
where 
$$u_n^{(j,k,l)} = e^{ t g(\rho)} T(\hat{\rho}_n) \mathcal{M}(d_{\hat{\rho}} F^{(n)} ) \Lambda_h \left(P^{(j,k,l)}_n \Psi_0 \right)$$
$P_n^{(j,k,l)}$ is a polynomial of degree at most $3k + 2l$ and 
$$ N_\infty  \left(P^{(j,k,l)}_n \right) \leq C_{j,k,l, \varepsilon} n^{2j+k} \Pi_{\alpha,n}(\rho) \left( J^u_{q_0 \dots q_{n-1}} \right)^{3k} h^{-k \varepsilon}$$
Concerning the leading term, $P_n^{(0,0,0))} = \pi_{\alpha,n}(\rho)$. 
Concerning the remainder $R_n^{(N)}$ we have 
$$ ||R_n^{(N)} ||_{L^2} \leq C_{N,\varepsilon} h^{-(K+K_0t)}  h^{2N \varepsilon} $$ 

\end{prop}
\begin{proof}
We simply state $P_n^{(j,k,l)} = P_{n,P^{(l)}_{t,h}}^{(j,k)}$ which satisfies the required bound for the degree and $N_\infty$. Here, $P^{(l)}_{t,h}$ appears in the expansion of Lemma \ref{Lemma_Action_tilde_E_t}. Lemma \ref{Lemma_Action_tilde_E_t} and Lemma \ref{Lemma_Expansion_propagation} show that $$ M_{q_n,q_{n-1}} \dots M_{q_1,q_0} E_t \varphi_\rho = \sum_{2j + k +l < 2N} h^{j +k/2}h^{l(1/2- \delta)} u_n^{(j,k,l)}  + R_n^{(N)}$$
with $u_n^{(j,k,l)}$ given by the required formula and 
$$ R_n^{(N)} = M_{q_n,q_{n-1}} \dots M_{q_1,q_0} r_{2N}  + \sum_{l=0}^{2N-1} h^{l(1/2- \delta)} r_{n,P^{(l)}_{t,h}}^{(N-\lfloor l/2 \rfloor)} $$ 
We can use the bound $|| r_{2N}||  \leq C_N h^{N(1-2 \delta)}$ and the bound for $r_{n,P}^{(N)}$ in Lemma \ref{Lemma_control_remainder}. Since the degrees of the polynomial $P^{(l)}_{t,h}$ are bounded by $4N$, we can forget the depence in $\deg P$ in the estimates of Lemma \ref{Lemma_control_remainder}, so that we find 
$$ ||R_n^{(N)} || \leq C_{N,\varepsilon} h^{-K} h^{-K_0 t}  \sum_0^{2N} h^{l(1/2-\delta)} h^{2(N- l/2)\varepsilon} \leq C_{N,\varepsilon} h^{-(K+K_0t)} h^{2N \varepsilon} $$ 
where the last inequality follows from $\varepsilon = 1/2 -  \delta$. 
\end{proof}

\begin{rem}
This expansion mixes up different scales : \begin{itemize}
\item the scale $h^{1-2\delta}=h^{2\varepsilon}$ : it comes from the symbol class in which $g$ lives ; 
\item a second scale which is the scale $h^{1/2}$ when $n$ is independent of $h$. In our context, it is better to think this second scale to be $h^{1/2} (J^u_{q_0 \dots q_n})^{3}h^{-\varepsilon}$. This scale depends on the starting point $\rho$.  The definition of  $\vartheta_\varepsilon$ ensures that the higher order terms are smaller than the leading term, . 
\end{itemize}
\end{rem}

Since we can choose $N$ as large as we want, we can ensure that the remainder decays in $h$ and that the leading term controls the whole expansion.
Note also that the constants $C_{j,k,l,\varepsilon}$ and $C_{N,\varepsilon}$ depends on $g$ and $M=M_h(z)$ and they can be chosen uniform in $z \in \Omega(h) \cap \{ \im z \in [-\beta,4] \}$.

\subsection{Final action of $A_{q_n} e^{-tG} B_{q_n}^\prime$}

From now on, and until the end of the section, we assume that $\rho \in \mathcal{W}_{\mathbf{q}}^-$ and we prove the missing items of Proposition \ref{Prop_Key}.

We need to understand the action of $e^{-tG} B_{q_n}^\prime$ on the terms $u_n^{(j,k,l)}$ of the last expansion. Since all these terms have the same form, we consider a general polynomial $P$ of degree $d$ and want to understand 
$$ e^{-tG}  B_{q_n}^\prime \left( T(\hat{\rho}_n) \mathcal{M}\left( d_{\hat{\rho}} F^{(n)} \right) \Lambda_h (P \Psi_0) \right) $$
It is no more possible to reuse the strategy of Lemma \ref{Lemma_Action_tilde_E_t}. Indeed, if $g$ still oscillates on scale $h^{\delta}$,  $\mathcal{M}\left( d_{\hat{\rho}} F^{(n)} \right) \Lambda_h (P \Psi_0)$ is no more a wave packet in a box of size $h^{1/2}$. To see that in a model case, assume that $ d_{\hat{\rho}} F^{(n)} $ is given by the diagonal matrix 
$$ \left( \begin{matrix}
\lambda_h & 0 \\
0 &\lambda^{-1}_h
\end{matrix} \right) $$
with $\lambda_h \sim J^u_{q_0 \dots q_{n-1}} \sim h^{-\alpha}$ where
$$ \lambda_{\min} \vartheta_\varepsilon \leq \alpha \leq \lambda_{\max} \vartheta_{\varepsilon} = \frac{1-4 \varepsilon}{6(1+ \varepsilon)^2} $$
Then $\mathcal{M}\left( d_{\hat{\rho}} F^{(n)} \right)$ is nothing but $\Lambda_{\lambda_h^2}$ and hence, $\mathcal{M}\left( d_{\hat{\rho}} F^{(n)} \right) \Lambda_h (P \Psi_0) = \Lambda_{h \lambda_h^2 } (P \Psi_0)$. 
This states oscillate in the $x$-direction on a scale $h^{1/2 - \alpha} \gg h^{\delta}$.

\subsubsection{Precise description of $d_{\hat{\rho}} F^{(n)}$. } It is not possible to write $d_{\hat{\rho}} F^{(n)}$ exactly as a diagonal matrix in the standard position/momentum variable. However, the following lemma shows that $d_{\hat{\rho}} F^{(n)}$ stays close to a diagonal matrix : 

\begin{lem}\label{Lemma_desc_d_rho_F_n}
There exists $\varepsilon_2$ which can be made arbitrarily small depending on $\varepsilon_0$
such that the following holds. There exists $\lambda_{n,\mathbf{q}}, \mu_{n,\mathbf{q}} \in \R^+$  such that for all $n$, $\mathbf{q}=q_0 \dots q_n$ and $\hat{\rho} \in \kappa_q\left( \mathcal{W}_{\mathbf{q}}^-,\right)$, we have for some global constant $C>0$ : 
\begin{itemize}
\item  $C^{-1} J_{\mathbf{q}}^u \leq \lambda_{n,\mathbf{q}} \leq C J_{\mathbf{q}}^u $ ; 
\item $C^{-1} \leq \mu_{n,\mathbf{q}} \lambda_{n,\mathbf{q}} \leq C $  ; 
\item  $d_{\hat{\rho} F^{(n)}}$ is close to a diagonal matrix : $$
\left| \left| d_{\hat{\rho}} F^{(n)} - \left( \begin{matrix}
\lambda_{n,\mathbf{q}}& 0 \\
0 & \mu_{n,\mathbf{q}}
\end{matrix} \right)  \right| \right| \leq \varepsilon_2 J^u_{\mathbf{q}}
$$
\end{itemize}

\end{lem}

\begin{proof}
We note $\rho_i =F^{in_0} \left( \rho\right) = \kappa_{q_i}^{-1} \circ F^{(i)} (\hat{\rho})$. Recall also that $F^{(i)} = \kappa_{q_i} \circ F^{in_0} \circ \kappa_{q_0}^{-1}$.

\paragraph{Step 1 : Reduction to $\rho \in\mathcal{T}$. } 
By definition of $\mathcal{W}_{q_i}$, for $i \in \{0, \dots, n\}$, we have $d(\rho_{q_i}, \rho_i) \leq 2 \varepsilon_0$. Hence, 

$$d(F^{n_0}(\rho_{q_i}), \rho_{i+1}) \leq d( F^{n_0}(\rho_{q_i}), F^{n_0(i+1)}(\rho) ) + d (F^{n_0(i+1)}(\rho), \rho_{i+1} ) \leq C \varepsilon_0$$
for a constant $C$ only depending on $F$. That is to say, $(\rho_0, \dots, \rho_{n})$ is a $C\varepsilon_0$ pseudo orbit for $F^{n_0}$. Assume that $\delta_0 >0$ is a small fixed parameter. In virtue of the Shadowing Lemma (\cite{KH} , Section 18.1), if $\varepsilon_0$ is sufficiently small, $(\rho_0, \dots, \rho_{n})$ is $\delta_0$ shadowed by an orbit of $F^{n_0}$ i.e. there exists $\rho^\prime \in \mathcal{T}$ such that for $i \in \{0, \dots, n \}$, $d(\rho_i, F^{i n_0}(\rho^\prime) ) \leq \delta_0$. Consequently, 
$d(F^{i n_0} (\rho) , F^{in_0}(\rho^\prime)) \leq \delta_0 + C\varepsilon_0$. For convenience, set $\varepsilon_2 = \delta_0 + C\varepsilon_0$ and note that $\varepsilon_2$ can be arbitrarily small depending on $\varepsilon_0$. By Lemma \ref{Local_hyperbolic_2}, for $1 \leq i \leq n$, 
$$ ||d_{\rho} F^{i n_0} || \leq C J^u_{i n_0}(\rho^\prime) \quad ; \quad  C^{-1}J^u_{q_0 \dots q_{i-1}} \leq  J^u_{i n_0}(\rho^\prime) \leq C J^u_{q_0 \dots q_{i-1}} $$
Hence, using the relation 
$$ d_{\rho } F^{nn_0} - d_{\rho^\prime} F^{nn_0} =\sum_{k=0}^{n-1} d_{F^{n_0(k+1)}(\rho^\prime)} F^{n_0(n-k-1)} \circ \left( d_{F^{n_0 k} (\rho)} F^{n_0} - d_{F^{n_0 k}(\rho^\prime)} F^{n_0}\right) \circ d_{\rho}F^{n_0 k} $$
we find that

\begin{align*}
\left| \left|d_{\rho } F^{nn_0} - d_{\rho^\prime} F^{nn_0}   \right| \right| &\leq C \sum_{k=0}^{n-1} J^u_{n_0(n-k-1)}\left( F^{n_0(k+1)}(\rho^\prime) \right) \left| \left|d_{F^{n_0 k}(\rho)} F^{n_0} - d_{F^{n_0k }(\rho^\prime)}F^{n_0} \right| \right| J^u_{n_0 k}(\rho^\prime)\\
&\leq C  \sum_{k=0}^{n-1}  d\left( F^{n_0 k}(\rho) ,  F^{n_0 k} (\rho^\prime)\right)J^u_{n n_0}(\rho^\prime)\\
& \leq C J^u_{\mathbf{q}} \sum_{k=0}^{n-1}\theta^{\min(k,n-k )} \varepsilon_0 \\
& \leq  C J^u_{\mathbf{q}}  \varepsilon_0
\end{align*}
where we use the Lemma \ref{classical_hyperbolic} in the third equality and the last one follows from $\sum_{k=0}^{n-1}\theta^{\min(k,n-k )} \leq 2 \sum_{k=0, \lceil n/2 \rceil} \theta^k \leq 2 \sum_{k=0}^\infty \theta^k < +\infty$. 
It is not hard to deduce from this that 
$$|| d_{\hat{\rho}} F^{(n)}   - d_{\kappa_q(\rho^\prime)} F^{(n)} || \leq C J^u_{\mathbf{q}}  \varepsilon_0 $$
Hence, it is enough to prove the Lemma for $d_{\kappa_q(\rho^\prime)}F^{(n)}$. 

\paragraph{Step 2 : The case $ \rho \in \mathcal{T}$. } We assume that $\rho \in \mathcal{T}$.  The spaces $E_u(\rho), E_s(\rho), E_u(F^{nn_0}(\rho))$ and $E_s(F^{nn_0}(\rho))$ are well-defined.  For $q \in \mathcal{A}$ and $\bullet = s,u$, the maps $\zeta \in \mathcal{W}_q \cap \mathcal{T} \mapsto d_{\zeta} \kappa_q(E_\bullet(\zeta))$ are Lipschitz. Since $d_{\rho}\kappa_q (E_u(\rho_q) ) = \R\partial_y$, $d_{\rho}\kappa_q (E_s(\rho_q) ) = \R \partial_\eta$ and $d(\rho_{q_0}, \rho) \leq C\varepsilon_2$, $d(\rho_{q_n}, F^{nn_0}(\rho)) \leq C\varepsilon_2$, we can fix unit vectors 
$$e_0^u \in d_{\rho} \kappa_{q_0}\left( E_u(\rho )\right) \;, \;  e_0^s \in d_{\rho} \kappa_{q_0}\left( E_s(\rho \right)) $$ 
$$  e_n^u \in d_{F^{nn_0}(\rho)} \kappa_{q_n}\left( E_u(F^{nn_0}(\rho) )\right) \; \; e_n^s \in d_{F^{nn_0}(\rho)} \kappa_{q_n}\left( E_s(F^{nn_0}(\rho)) \right) $$ such that 
$e_0^u, e_n^u = \partial_y + O(\varepsilon_2)$ and $e_0^s, e_n^s = \partial_\eta+ O(\varepsilon_2)$. If we note $P_0$ (resp. $P_n$) the change-of-basis matrix from the natural basis of $\R^2$ to $(e_0^u,e_0^s)$ (resp. $(e_n^u,e_n^s)$), then $P_0,P_n = I_2 + O(\varepsilon_2)$ (with global constants in $O$ not depending on $n$). 
Moreover, since $d_{\hat{\rho}} F^{(n)} (e_0^u) \in \R e_n^u$ and $d_{\hat{\rho}} F^{(n)} (e_0^s) \in \R e_n^s$, the matrix $P_n^{-1} d_{\hat{\rho}} F^{(n)} P_0$ is diagonal. Let's write it $$\left(
\begin{matrix} \lambda_{n,\mathbf{q}} & 0 \\
0 & \mu_{n,\mathbf{q}} 
\end{matrix} \right) $$
$\lambda_{n,\mathbf{q}} $ (resp. $\mu_{n,\mathbf{q}} $) is nothing but \emph{an} unstable (resp. stable) Jacobian for $\rho$, and hence $\lambda_{n,\mathbf{q}}  \sim J^u_{\mathbf{q}}$. Since $\det d_{\rho} F^{nn_0} = 1$, $\lambda_{n,\mathbf{q}}  \mu_{n,\mathbf{q}}  = \det(P_0)^{-1} \det(P_n) = 1 + O(\varepsilon_2)$. Finally, 
$$ P_n^{-1} d_{\hat{\rho}} F^{(n)} P_0  = (I_2 + O(\varepsilon_2)) d_{\hat{\rho}} F^{(n)}(I_2 + O(\varepsilon_2) )= d_{\hat{\rho}} F^{(n)} + O \left( || d_{\hat{\rho}} F^{(n)}|| \varepsilon_2 \right) = d_{\hat{\rho}} F^{(n)} + O\left(\varepsilon_2 J^u_{\mathbf{q}}\right)  $$
This concludes the proof. 
\end{proof}

As a consequence of this lemma, in the standard position/momentum coordinates, we can write  

\begin{equation}\label{ecriture_d_rho_F_n}
d_{\hat{\rho}} F^{(n)} = \left( \begin{matrix}
a_n & b_n \\
c_n & d_n
\end{matrix}\right) \; ; \; a_n \sim J^u_{\mathbf{q}} \;; \; b_n, c_n, d_n = O \left( \varepsilon_2 J^u_{\mathbf{q}} \right) 
\end{equation} 
Here, $a_n,b_n,c_n,d_n$ depend on $\rho$, but we won't make this dependence precise since $\rho$ is fixed until the end of the section. 
Since we want to understand the action of $\mathcal{M}\left( d_{\hat{\rho}} F^{(n)}\right)$ on excited coherent states, we also introduce  
\begin{equation}\label{definition_alpha_beta}
\gamma_n = (c_n + id_n)(a_n + i b_n)^{-1}\; ; \; \beta_n = \re(\gamma_n) \; ; \; \alpha_n = \im(\gamma_n)^{-1} = |a_n+ i b_n|^2
\end{equation}
We've got the basic estimates 
\begin{equation}\label{equation_alpha_n}
\alpha_n \sim \left(J^u_{\mathbf{q}}\right)^2 ; \quad \beta_n= O(\varepsilon_2)
\end{equation}
Now assume that $P \in \C[X]$ and decompose $P$ into the basis of the renormalized hermite polynomials $(h_n)$ : $P = \sum_{k=0}^{\deg P}a_k(P) h_k$. By Proposition \ref{Prop_meta_on_excited}, 
$$ \mathcal{M}\left( d_{\hat{\rho}} F^{(n)} \right)  \Lambda_h \left( h_k \Psi_0\right)(x) = (\alpha_n \pi h)^{-1/4} \left(\frac{a_n-ib_n}{a_n+ib_n} \right)^{k/2} h_k\left( \frac{x}{(\alpha_n h)^{1/2}}\right) e^{i \gamma_n \frac{x^2}{2h}} = c_{n,k} \Lambda_{\alpha_n h} \left( h_k \Psi_0 \right)(x) e^{i \beta_n \frac{x^2}{2h} }$$ 
with $|c_{n,k}|=1$. As a consequence, there exist linear maps
$ \Phi_n : \C[X] \to \C[X]$ such that for all $n \in \N$ and $P \in \C[X]$, 
 \begin{itemize}
 \item $\deg \Phi_n (P) = \deg P$ for all $P \in \C[X]$ ; 
 \item $N_\infty(\Phi_n(P)) \leq K_{\deg P} N_\infty(P)$ where $K_{\deg P}$ depends only on $\deg P$; 
 \item  and the following relation holds 
 \begin{equation}\label{equation_transformation_M}
 \mathcal{M}\left( d_{\hat{\rho}} F^{(n)} \right)  \Lambda_h \left( P \Psi_0\right)=  \Lambda_{ \alpha_n h} \left( \Phi_n (P) \Psi_0 \right)  e^{i \beta_n \frac{x^2}{2h}}
\end{equation}  
 
 \end{itemize}
 
 \begin{rem}
 We can interpret this state as a (highly-oscillating) Lagrangian state associated with the Lagrangian manifold $\{ (x, \beta_n x) \}$, with amplitude $a(x) =  \Lambda_{ \alpha_n h} \left( \Phi_n (P) \Psi_0 \right) (x)$. Since $\alpha_n \sim \left(J^u_{\mathbf{q}}\right)^2$, $\alpha_n \sim h^{-\alpha}$ for some $\alpha \geq 2\lambda_{\min} \vartheta_\varepsilon$, the amplitude $a$ oscillates on a scale $h^{1/2-\alpha/2}$. Compared with the initial state $\varphi_0$, localized in position in an interval of size $h^{1/2}$, this expression shows a stretching in position. Moreover, this scale is larger than the scale $h^\delta$ on which the symbol $g$ oscillates. 
 \end{rem}

\subsubsection{ Asymptotic expansion for the exponential. }
We now aim at understanding the state $A_{q_n} e^{-tG} B_{q_n}^\prime u$ where $u$ is of the form 
$$ u(x) =T(\hat{\rho}_n) \left(  \Lambda_{\alpha_n h} f \right) (x) e^{ i \beta_n \frac{x^2}{2h}} $$
where $f = P\Psi$ for some $P \in \C[X]$.
We first claim that 

 $$A_{q_n} e^{-tG} B_{q_n}^\prime= A_{q_n} B_{q_n}^{\prime} e^{-t B_{q_n} G B_{q_n}^\prime} + \hinf$$

\begin{proof}
Set $A(t) = A_{q_n} e^{-tG} B_{q_n}^\prime e^{t B_{q_n} G B_{q_n}^\prime}$. At $t=0$, $A(0) = A_{q_n} B_{q_n}^\prime$. We differentiate: 
$$ \dot{A} (t) = A_{q_n} e^{-tG} \left[ B_{q_n}^\prime  B_{q_n} G B_{q_n}^\prime - G B_{q_n}^\prime \right] e^{t B_{q_n} G B_{q_n}^\prime}$$
The operator $A_{q_n} e^{-tG}$ is bounded on $L^2$ and has its semiclassical wavefront set included in $\supp \chi_{q_n}$. In particular, $A_{q_n} e^{-tG}\left(  B_{q_n}^\prime  B_{q_n}- \Id \right)  = \hinf$ (uniformly for $t$ in a bounded interval). This shows that $A^\prime(t) = \hinf$. We conclude that $A(t) = A_{q_n} B_{q_n}^\prime + \hinf$. 
\end{proof}
Hence we aim at understanding the action of $e^{-t B_{q_n} G B_{q_n}^\prime}$. 
We make use of Lemma \ref{Lemma_G_FIO} and we write for all $N \in \N$, 
$$G_{q_n} \coloneqq B_{q_n} G B_{q_n}^\prime = \op \left( g \circ \kappa_{q_n}^{-1} \right) + \sum_{j=1}^{N-1} h^{j(1- 2 \delta)} \op(g_{j,q_n}) + R_N$$
with  $||R_N|| \leq C_N h^{N(1-2 \delta)}$.  Let's write $g_{0,q_n} = g \circ \kappa_{q_n}^{-1} $. 
Similarly, we have $$A_{q_n} B_{q_n}^\prime e^{-tG_{q_n} } T(\hat{\rho}_n) = A_{q_n} B_{q_n}^\prime T(\hat{\rho}_n) e^{- t T(\hat{\rho}_n)^* G_{q_n} T(\hat{\rho}_n)  }$$  
and we recall that  $T(\hat{\rho}_n)^* \op(a) T(\hat{\rho}_n)= \op( a ( \cdot + \hat{\rho}_n) ) $ for any $a \in \mathcal{S}^\prime$. 
Let's note $h_j (\hat{\zeta}) = g_{j,q_n}(\hat{\rho}_n + \hat{\zeta})$, so that 
$$ A \coloneqq T(\hat{\rho}_n)^* G_{q_n} T(\hat{\rho}_n) = \sum_{j=0}^\infty h^{j(1-2\delta)} \op(h_j) + O_N( h^{N(1-2 \delta)}) $$
Recall that in virtue of Lemma  \ref{Lemma_G_FIO}, $h_0 \in \log(1/h) S_\delta$ and $h_j \in S_\delta$ for $j \geq 1$.

Finally, we need to understand the action of $e^{-t A}$ on states  $u(x) = \Lambda_{\alpha_n h} f (x) e^{ i \beta_n \frac{x^2}{2h}} $. 
We want to apply the formalism of Appendix \ref{Appendix_exp} with $H = L^2(\R)$ and $A$. The class of elements which will interest us is defined as follows : we say that a $h$-dependent family of states $u = u_h \in L^2(\R)$ belongs to the class $\mathcal{C}$ if $u$ has the form : 
$$ u(x) = a(x) e^{ i \beta_n \frac{x^2}{2h}} $$ where $a=a_h \in \cinf(\R)$ satisfies : for all $p \in \N$, there exists $C_p$ such that 
\begin{equation}\label{classe_condition}
 |a^{(p)}(x)| \leq C_p h^{- \delta p}  (\alpha_n h)^{-1/4}\left( 1 + \frac{x^2}{\alpha_n h} \right)^{-1} 
\end{equation}
This class depends on $h$ (and $n$, which himself depends on $h$). 
For such a state $u$, we define the natural semi-norms on $\mathcal{C}$ : 
\begin{equation}\label{Def_q_j}
 q_p(u) = \sup_{ k \leq p}\sup_{x \in \R} \left( |a^{(k)}(x)|  h^{ \delta k}  (\alpha_n h)^{1/4} \left( 1 + \frac{x^2}{\alpha_n h} \right) \right)
\end{equation} 
In particular, one has $||u|| \leq C q_0(u)$.

\begin{rem}
  In fact, the introduction of the semi-norms $q_j$ with the factor $(1+\frac{x^2}{\alpha_n h})^{-1}$ is purely technical : it allows to work in a symbol class depending on this order function (see the proof of Lemma \ref{Lemme_stationnary_phase_S_delta} in the appendix \ref{appendix_lemma_stationnary_phase_S_delta}). In the end, we will simply need to estimate the semi-norm $q_0$ of each term of the expansion of an evolved state $e^{-tA} u$, but this will require to control (a finite number of) semi-norms $q_j$ of the initial state $u$. This reason has motivated the introduction of the $q_j$'s. 
  We will mainly consider states $u$ with exponential decay and what is important is that $\partial^k \Psi_0 \leq C_{k,p} (1+x^2)^{-p/2}$ for all $k,p \in \N$. 
  \end{rem}
  
The following lemma ensures that the states we work with are indeed in $\mathcal{C}$, as soon as $h^{2\delta} \ll \alpha_n h$. Recall that $\alpha_n \geq Ch^{-\alpha_{\min}}$ where $\alpha_{\min}= 2\lambda_{\min} \vartheta_\varepsilon$. Then, it suffices to require $$ \varepsilon=1/2 - \delta \leq \alpha_{\min}/2.$$ This is clearly not a problem since we want to work with $\delta=1/2- \varepsilon$ very close to $1/2$ and we assume that this is true, that is, we assume that $$\varepsilon \leq \alpha_{\min}/2.$$
\begin{lem}
Assume that $u(x) = \Lambda_{\alpha_n h} (P \Psi_0) e^{ i \beta_n \frac{x^2}{2h}}$ where $P \in \C[X]$ has degree $d$. Then $u \in \mathcal{C}$ and for all $j \in \N$, there exists constants $C_{d,j}$ depending only on $d$ and $j$ such that 
$q_j(u) \leq C_{d,j} N_\infty(P)$
\end{lem}
\begin{proof}
\begin{align*}
 \left|\left( \Lambda_{\alpha_n h} (P \Psi_0)\right) ^{(j)}(x)  \right| &=  \left| (\alpha_n h)^{-j/2} \Lambda_{\alpha_n h} ( (P \Psi_0)^{(j)} )(x)  \right| \\
 & \leq h^{-\delta j} ( \pi h \alpha_n)^{-1/4} \left| (P_j  \Psi_0)( (\alpha_n h)^{-1/2} x)\right|
\end{align*} 
Here, we use that $\alpha_n h\gg h^{2 \delta}$. 
and $P_j$ is a polynomial which depends linearly on $P$, with $\deg P_j = \deg P +j $ and $N_\infty(P_j) \leq C_{d,j} N_\infty(P)$. Hence, we have 
\begin{align*}
q_k(u) &\leq \sup_{j \leq k} \sup_{x \in \R} \left| (P_j  \Psi_0)( (\alpha_n h)^{-1/2} x)\right| \left( 1 + \frac{x^2}{\alpha_n h} \right) \\
&\leq \sup_{j \leq k} \sup_{x \in \R} \left| (P_j  \Psi_0)( x)\right| \left( 1 + x^2\right)\\
&\leq \sup_{ j \leq k} C_{d,j} N_\infty(P_j) \leq C_{d,k} N_\infty(P) 
\end{align*}
\end{proof}

To apply the formalism of Appendix \ref{Appendix_exp}, we will require the following lemma. This a more or less direct application of the stationary phase theorem in the quadratic case.  We write its proof in appendix  \ref{appendix_lemma_stationnary_phase_S_delta}. This lemma explains how to compute $\op(m) u$ for $u \in \mathcal{C}$ and $m \in S_\delta$. 

\begin{lem}\label{Lemme_stationnary_phase_S_delta}
There exists $M>0$ such that the following holds. Assume that $m \in S_\delta$ or $m =h_0$. Then, for all $k \in \N$, there exists $A_k(m) : \mathcal{C} \to \mathcal{C}$ such that for $u \in \mathcal{C}$, written under the form $u(x)=a(x)e^{i \beta_n \frac{x^2} {2h}}$, we have
\begin{itemize}
\item $A_0(m) u(x) = m(x,\beta_n x) u(x)$; 
\item For $k \geq 1$, $A_k(m)$ is of the form 
\begin{equation}\label{eq_form_A_k} A_k (m)u (x)=  \sum_{l \leq 2k} c_l(x) \partial_x^{l}a(x) e^{ i \beta_n \frac{x^2}{2h}} 
\end{equation}
where $|c_l^{(p)}(x)| \leq C_{l,k,p} h^{(l-p) \delta}$.
\item For all $(j,k)\in \N^2 \setminus \{ (0,0) \}$, there exists $c_{j,k}>0$ such that for all $u  \in \mathcal{C}$, $q_j(A_k(m) u) \leq c_{j,k} q_{2k+j}(u)$; 
\item For all $N \in \N^*$ and for all $j \in \N$, there exists $C_{j,N}>0$ such that 
$$ q_j \left( \op(m) u - \sum_{k=0}^{N-1} h^{k(1-2 \delta)} A_k(m) u \right) \leq C_{j,N} q_{j+2N+M} (u) h^{N(1-2 \delta)}$$
\end{itemize}
\end{lem}

\begin{rem}
We need to distinguish the cases $m=h_0$ and $m \in S_\delta$ because $h_0$ is not in $S_\delta$ (recall that we only have $h_0 =O( \log(1/h))$). However, $h_0$ satisfies $|\partial^\alpha h_0| \leq C_\alpha h^{- |\alpha| \delta} $ as soon as $|\alpha| \geq 1$. This explains why we restrict on $(j,k) \neq (0,0)$ in the third item but in the case $m \in S_\delta$, the expression given in the first item shows that it also holds for $(j,k)=(0,0)$. 
\end{rem}

Gathering the terms of same order in the expansions of each $\op(h_k)$ given by Lemma \ref{Lemme_stationnary_phase_S_delta}, we can build the family of operators  $$A_k : \mathcal{C} \to \mathcal{C} \; ; \; A_k = \sum_{j+l=k}A_j(h_l).$$ 
Each $A_k$ has the same form as (\ref{eq_form_A_k}) and they satisfy, for all $u \in \mathcal{C}$, 
\begin{itemize}
\item $A_0 u(x) = h_0(x,\beta_n x) u(x)$. 
\item For all $(j,k) \in \N^2 \setminus \{ (0,0) \}$, there exists $c_{j,k}>0$ such that for all $u  \in \mathcal{C}$, $q_j(A_k u) \leq c_{j,k} q_{2k+j}(u)$; 
\item For all $N \in \N^*$ and for all $j \in \N$, there exists $C_{j,N}>0$ such that 
$$ q_j \left( A u - \sum_{k=0}^{N-1} h^{k(1-2 \delta)} A_k u \right) \leq C_{j,N} q_{j+2N+M} (u) h^{N(1-2 \delta)} $$
\end{itemize}

We now use the formulas and notations of Appendix \ref{Appendix_exp} to show : 

\begin{prop}\label{Prop_expansion_S_delta}
Assume that $P \in \C[X]$ is of degree $d$ and consider the state $u= \Lambda_{\alpha_n h} (P \Psi_0)e^{\frac{i \beta_n x^2}{2h}}$. Then, $t$ being fixed, there exists a family of functions $(f_k)$ and $K_1 >0$ such that,
 \begin{itemize}
 \item $v_0(x) = u(x)$ ; 
\item For all $N \in \N^*$, there exists $C_{N,d}$ such that 
$$ \left| \left| A_{q_n} \left( e^{-tG} B^\prime_{q_n}T(\hat{\rho}_n) u - \sum_{k=0}^{N-1} h^{k(1-2 \delta)} B_{q_n}^\prime T(\hat{\rho}_n) u_k \right) \right| \right|  \leq C_{N,d} h^{N(1-2 \delta)} h^{-tK_1} N_\infty(P)$$ 
where 
$$ u_k (x) =  \exp\left( -t h_0(x,\beta_n x)) \right) v_k(x)  \; ; \; v_k(x) = f_k(x) \left(\Lambda_{\alpha_n h} \Psi_0\right)(x) e^{\frac{i \beta_n x^2}{2 h} } $$
\item For all $k \in \N$, there exists $c_{k,d}>0$ such that for all $x \in \R$, $$|f_k(x)| \leq c_{k,d} \left( 1 + \frac{x^2}{\alpha_n h }\right)^{k/2} N_\infty(P)$$ 
\end{itemize}
\begin{rem}
In particular, these last estimates imply that $v_k \in \mathcal{C}$. 
\end{rem}

\end{prop}
\begin{proof}
We use the notations and formulas of Appendix \ref{Appendix_exp}, with parameter $\tilde{h}=h^{1-2\delta}$. 
We define a family $(v_k(t))$ by the iterative formula (\ref{formula_v_k}). 
The operator $A_0$ is nothing but the multiplication by $$a_0(x) = h_0(x,\beta_n x)$$ and hence, $e^{s A_0}$ is the multiplication by $\exp\left( s a_0 \right) $. 

Let us note$A_k(s)=e^{-s A_0} A_k e^{s A_0} $ and let us show that $A_k(s)u(x)$ has the same form as (\ref{eq_form_A_k}), with the functions $c_l(x)$ replaced by functions $\tilde{c}_l(s,x)$. We have 
\begin{align*}
e^{-s a_0} c_l(x) \partial_x^l \left( e^{sa_0} a(x) \right) &= c_l(x) \sum_{m=0}^l {l \choose m }  a_m(s,x) a^{(l-m)}(x) \\
&=\sum_{m=0}^l \tilde{c}_{l,m}(s,x) \partial_x^m a(x)
\end{align*}
where $a_m(s,x) =  e^{-sa_0} \partial_x^m (e^{sa_0})$ is a sum of terms of the form
$$  s^i \prod_{j=1}^i    a_0^{(k_j)} \text{ with } (k_1, \dots, k_i) \in (\N^*)^i \text{ and } k_1 + \dots + k_i = m$$ and $\tilde{c}_{l,m}(s,x) = {m \choose l}c_{l}(x) a_{l-m}(s,x)$. It is not hard to see that $|\partial^p_x a_m (s,x) | \leq C_{m,p}  (1+|s|)^m h^{-\delta (m+p)}$ so that we have 
\begin{align*}
\left| \partial_x^p \tilde{c}_{l,m}(x)  \right|=&{m \choose l} \left| \sum_{p_1 + p_2=p} { p \choose p_1 } c_l^{(p_1)}(x) a_{l-m}^{(p_2)}(s,x) \right| \\
&\leq C_{p,l,m} \sup_{p_1 + p_2 = p} h^{ \delta(l-p_1) \delta} h^{ - (l-m+p_2)\delta} \leq C_{p,l,m} h^{(m-p)\delta} 
\end{align*}
which shows that the term in front of $\partial_x^m$ has the correct behavior to be of the form (\ref{eq_form_A_k}) and we can set $\tilde{c}_m(s,x) = \sum_{l \leq 2k}  \tilde{c}_{l,m}(s,x)$ so that 
$$A_k(s)u(x) = \sum_{ m \leq 2k } \tilde{c}_m(s,x) \partial_x^m \left(  u e^{- i \beta_n x^2/2h} \right)  e^{i \beta_n x^2/2h} $$
 Let us now analyze the action of $A_k(s)$ on states of the form $c(x) \Lambda_{\alpha_n h} ( \Psi_0)(x)  e^{ i \beta_n \frac{x^2}{2h}} $.   We claim that we can write 
 \begin{equation}\label{eq_form_A_k_s}
 A_k(s) \left( c(x) \Lambda_{\alpha_n h} ( \Psi_0)(x)  e^{ i \beta_n \frac{x^2}{2h}} \right) = d_{k}(s,x) \Lambda_{\alpha_n h} ( \Psi_0)(x)  e^{ i \beta_n \frac{x^2}{2h}} 
 \end{equation}
where 
\begin{equation}|\partial_x^p d_{k}(s,x)| \leq C_{k,p} h^{- p \delta} (1+|s|)^k \sup_{y \in \R, m \leq 2k +p} |c^{(m)}(y)| \left(1 + \frac{x^2}{\alpha_n h} \right)^k .
\end{equation}
To see that, let us write 
\begin{align*}
 \tilde{c}_{m}(s,x) \partial_x^m \left( c(x) \Lambda_{\alpha_n h} ( \Psi_0)   \right)& = \tilde{c}_{m}(s,x)  \sum_{l=0}^m {m \choose l }c^{(m-l)}(x)      \left( \Lambda_{\alpha_n h} ( \Psi_0) \right)^{(l)}(x)) \\
 &=  \underbrace{\left(  \sum_{l=0}^m {m \choose l } \tilde{c}_{m}(s,x)  c^{(m-l)}(x)  \frac{Q_l( (\alpha_n h)^{-1/2} x) }{(\alpha_n h)^{l/2}  }  \right) }_{C_m(s,x)} \Lambda_{\alpha_n h} ( \Psi_0) (x) 
\end{align*}
where $Q_l \in \R[X]$ are some polynomials of degree $l$. We hence have, 
\begin{align*}
\left| \partial_x^{(p)} C_{m} (s,x)\right| &= \left|  \partial_x^p \left( \sum_{m_1+m_2=m} {m \choose m_1 } \tilde{c}_{m}(s,x)  c^{(m_1)}(x)  \frac{Q_{m_2}( (\alpha_n h)^{-1/2} x) }{(\alpha_n h)^{m_2/2}  } \right) \right| \\
&\leq C_{m,p} \sup_{\substack{m_1 + m_2 = m \\ p_1+ p_2 + p_3=p}} | \partial_x^{p_1} \tilde{c}_{m}(s,x) |  |c^{(m_1+p_2)}(x)| (\alpha_n h)^{-m_2 /2 -p_3/2} |Q_{m_2}^{(p_3)}((\alpha_n h)^{-1/2} x)|(1+|s|)^m \\
&\leq C_{m,p} \sup_{\substack{m_1 + m_2 = m \\ p_1+ p_2 + p_3=p}} h^{(m-p_1)\delta} |c^{(m_1+p_2)}(x)| (\alpha_n h)^{-m_2 /2 -p_3/2} (1+|s|)^m \left( 1 + \frac{x^2}{\alpha_n h} \right)^{(m_2-p_3)/2}\\
&\leq C_{m,p}  \sup_{y \in \R, l \leq m +p} |c^{(l)}(y)|   \sup_{\substack{m_1 + m_2 = m \\ p_1+ p_2 + p_3=p}}h^{(l-p_1)\delta}  h^{-(m_2 +p_3)\delta} (1+|s|)^m \left( 1 + \frac{x^2}{\alpha_n h} \right)^{m/2}\\
&\leq  C_{m,p}  \sup_{y \in \R, l \leq m +p} |c^{(l)}(y)| h^{-\delta p } (1+|s|)^m  \left( 1 + \frac{x^2}{\alpha_n h} \right)^{m/2}
\end{align*}
and the claim is proved, with $d_k(s,x) = \sum_{m \leq 2k} C_m(s,x)$.

We now analyze precisely the iteration formula (\ref{formula_v_k}) in Appendix \ref{Appendix_exp}. We use the notations of this appendix f(in particular, for the remainders $\tilde{r}_{j,N}$ and $R_N$). Let $K_0>0$ be such that $|h_0| \leq K_0 \log 1/h$ so that $|e^{t a_0}| \leq h^{-tK_0}$. For $j \geq 0$, we have 
$$ q_j(e^{ta_0}u ) \leq h^{-|t|K_0} c_{0,j} q_j(u) $$
This is obvious for $j=0$. For $j \geq 1$, it comes from the fact that the derivatives of $h_0$ satisfy $|\partial^\alpha h_0| \leq C_\alpha h^{- |\alpha| \delta } $ for $\alpha \neq 0$ and the definition of $q_j$ in (\ref{Def_q_j}). 

\subparagraph{Leading term. } For our leading term in the expansion we want in Proposition \ref{Prop_expansion_S_delta}, we simply have 
$$ u_0(t) = e^{ t a_0} u$$ 
As a consequence, 
$$ q_j\left( \tilde{r}_{0,N}(t) \right)  \leq C_{j,N} \tilde{h}^N q_{j+2N+M} \left( e^{ta_0} u \right) \leq C_{j,N} \tilde{h}^N  c_{0,j+2N+M} h^{-|t| K_0} q_{j+2N+M} (u) \leq C_{d,j,N} \tilde{h}^N h^{-|t| K_0} N_\infty(P)
$$

\subparagraph{Iteration.} 
By induction, using the formulas (\ref{formula_v_k}) and (\ref{eq_form_A_k_s}), we see that if the initial state is $u=\Lambda_{\alpha_n h} ( P\Psi_0)  e^{ i \beta_n \frac{x^2}{2h}}$ then 
$$v_k(t) = f_{k}(t,x) \Lambda_{\alpha_n h} ( \Psi_0)(x)  e^{ i \beta_n \frac{x^2}{2h}}$$ where $|\partial_x^p f_{k}(s,x)| \leq C_{k,p} h^{- p \delta} (1+|s|)^k \left( 1 + \frac{x^2}{\alpha_n h} \right)^{k/2} N_\infty(P)$. When $p=0$, it gives the required estimate for $|f_k(x)|$ in Proposition \ref{Prop_expansion_S_delta}. It follows that
\begin{align*}
 q_j(e^{tA_0} v_k(t) )  \leq h^{-|t| M_0} c_{0,j} q_j(v_k(t) ) 
\leq  h^{-|t| M_0} c_{0,j} (1+ |t|)^k  h^{-|t| M_0} \tilde{C}_{d,j,k} (1+ |t|)^k N_\infty(P)
\end{align*}
Moreover, we can estimate $$q_j\left( \tilde{r}_{k,N}(t) \right) \leq  C_{j,N-k } \tilde{h}^{N-k} q_{j+2(N-k)+M} \left( e^{tA_0} v_k(t) \right) \leq  C_{d,N,j,k}\tilde{h}^{N-k} h^{-|t| M_0} (1+ |t|)^k N_\infty(P) $$

\subparagraph{Conclusion} We find that for $j \in \N$, 
\begin{align*}
q_j\left(\sum_{k=0}^{N-1} \tilde{h}^k \tilde{r}_{k,N}(t) \right) &\leq \tilde{h}^N h^{-|t|M_0} \sum_{k=0}^{N-1} C_{d,N,j,k}  (1+ |t|)^k N_\infty(P)\\
& \leq C_{d,j,N} \tilde{h}^N  h^{-|t| M_0}  (1+ |t|)^{N-1} N_\infty(P) 
\end{align*}
Integrating (\ref{formule_reste_exact}), and recalling that $||\cdot || \leq C q_0$ in $\mathcal{C}$, we have 
$$||R_{N-1}(t) || \leq \int_0^{|t|} ||A|| ||R_{N-1}(s) || ds + C_{d,0,N} \tilde{h}^N  h^{-|t| M_0}  (1+ t^2)^{N-1} N_\infty(P) $$
By  a version of Gronwall's lemma, we can find a constant $C_{N,d}$ such that
$$ ||R_{N-1}(t) || \leq C_{N,d} \tilde{h}^N e^{|t| \max( K_0 |\log h|, ||A||) } (1+ t^2)^N N_\infty(P)$$
Since, $||A|| = O(\log h)$, there exists $K_1>0$  such that 
$e^{|t| \max( K_0 |\log h|, ||A||) } \leq h^{-|t|K_1}$ and it concludes the proof of Proposition \ref{Prop_expansion_S_delta}. 
\end{proof}

Combining Proposition \ref{Proposition_expansions_before_last_exp}, (\ref{equation_transformation_M}) and Proposition \ref{Prop_expansion_S_delta}, we deduce the following expansion :

\begin{cor}\label{Cor_expansion}
\begin{multline}
A_{q_n} e^{-tG}  \mathfrak{M}^{n-1} M^{n_0} A_q B_q^\prime \varphi_{\hat{\rho}} = A_{q_n} \sum_{2j+k+l+2m<2N} h^{j+k/2+ l\varepsilon + m\varepsilon } u_n^{(j,k,l,m)} \\
+ O \left(h^{-K - t(K_0+ K_1)} h^{2N \varepsilon}  \left( \log h \right)^{N} \right)
\end{multline}
where $$T(\hat{\rho}_n)^*u_n^{(j,k,l,m)}(x) =  e^{ tg(\rho)-t h_0(x,\beta_nx)  } f_n^{(j,k,l,m)}(x) \Lambda_{\alpha_n h} (\Psi_0)(x) e^{\frac{i \beta_n x^2}{2 h} }$$
where we have, for all $x \in \R$,  $$\left| f_n^{(j,k,l,m)} (x) \right| \leq C_{j,k,l,m} n^{2j+k} \Pi_{\alpha,n}(\rho) \left( J^u_{\mathbf{q}}\right)^{3k} h^{-k \varepsilon} \left( 1 + \frac{x^2}{\alpha_n h} \right)^{m/2} $$ 
Concerning the leading term, $f_n^{(0,0,0,0)}$ is constant equal to $\pi_{\alpha,n}(\rho)$. 
\end{cor}

\begin{proof}
In the expansion of Proposition \ref{Proposition_expansions_before_last_exp},  we transform the states $u_n^{(j,k,l)}$ using formula (\ref{equation_transformation_M}). Finally we use Proposition \ref{Prop_expansion_S_delta} on each such state.   For $u_n^{(j,k,l)}$, we keep the  $N_{j,k,l}$ first terms of the expansion, where $N_{j,k,l} = N - j - \lceil (k+l)/2 \rceil$. It gives a remainder term 
$r_n^{(j,k,l)}$ satisfying 
\begin{align*}
||r_n^{(j,k,l)} ||_{L^2} &\leq C_{N,j,k,l} h^{j + k/2+ l\varepsilon } h^{2\varepsilon  N_{j,k,l} } N_\infty\left( \Phi_n \left( P_n^{(j,k,l)}\right) \right) h^{-t(K_0 + K_1)} \\
&\leq C_{N,j,k,l} h^{j + k/2+ l\varepsilon } h^{2\varepsilon N_{j,k,l}} h^{-k \varepsilon} n^{2j+k} \Pi_{\alpha,n}(\rho) \left( J^u_{\mathbf{q}}\right)^{3k} h^{-t(K_0 + K_1)} \\
& \leq  C_{N,j,k,l} h^{j + k/2+ l\varepsilon } h^{2\varepsilon N_{j,k,l} } n^{2j+k} h^{-k \varepsilon} h^{-k \frac{1-4 \varepsilon}{2(1+\varepsilon)}} h^{-K-t(K_0 + K_1)} \\ 
&\leq C_{N,j,k,l } \left(\log 1/h\right)^{N}  h^{2 \delta j}  h^{2\varepsilon  N} h^{k \left(1/2 -2 \varepsilon-  \frac{1-4 \varepsilon}{2(1+\varepsilon)}\right)} h^{-K -t(K_0 + K_1)} 
\end{align*}
But we have $\frac{1}{2} - 2 \varepsilon - \frac{1-4 \varepsilon}{2(1+\varepsilon)} \geq 0$ (assuming that $\varepsilon \leq 1/4$, which is not a problem since we work with $\varepsilon$ small). Hence, $|| r_n^{(j,k,l)}|| \leq C_{N,j,k,l } \left(\log 1/h\right)^{N} h^{-K -t(K_0 + K_1)} $. As a consequence, gathering all the remainders $r_n^{(j,k,l)}$ together and adding them to $e^{-tG} B_{q_n}^\prime R_n^{(N)}$, we obtain a remainder term controlled by 
$ C_N  h^{-K - t(K_0+K_1)} h^{2N\varepsilon} \left(\log 1/h\right)^{N}$ as expected. 
\end{proof}

\subsection{Crucial estimates for the terms of the expansion.}

In the expansion of Corollary \ref{Cor_expansion}, the leading term is given by 
$$u_n^{ 0} \coloneqq  T^*(\hat{\rho_n}) u_n^{(0,0,0,0)} = \exp\left( tg(\rho)-t h_0 (x,\beta_n x)) \right) \Lambda_{\alpha_n h} (\Psi_0)(x) e^{\frac{i \beta_n x^2}{2 h} }\pi_{\alpha,n}(\rho) $$
As a consequence of the Corollary \ref{Cor_expansion}, the other terms have the form
$$T^*(\hat{\rho_n}) = u_n^{(j,k,l,m)}(x) = f_n^{(j,k,l,m)}(x) \frac{u_n^{0}(x)}{\pi_{\alpha,n}(\rho)}$$
with $$\left|f_n^{(j,k,l,m)}(x) \right| \leq C_{j,k,l,m} n^{2j + k} \left(J^u_{q_0 \dots q_n} \right)^{3k} \Pi_{\alpha,n}(\rho) h^{-k \varepsilon} \left( 1 + \frac{x^2}{\alpha_n h} \right)^{m/2}$$
so that, denoting 
$$v_n = \frac{u^0_n}{\pi_{\alpha,n}}=\exp\left( tg(\rho)-t h_0 (x,\beta_n x) \right) \Lambda_{\alpha_n h} (\Psi_0)(x) e^{\frac{i \beta_n x^2}{2 h} }$$ we have 
\begin{equation}\label{other_terms_control}
\frac{ \left| \left| h^{j+k/2+ l\varepsilon + m\varepsilon} u_n^{(j,k,l,m)} \right| \right|_{L^2}}{\Pi_{\alpha,n}(\rho)  \times \left| \left| \left( 1 + \frac{x^2}{\alpha_n h} \right)^{m/2}v_n \right| \right|_{L^2}} \leq C_{j,k,l,m} h^{j+k/2+ l\varepsilon + m\varepsilon} n^{2j + k} \left(J^u_{q_0 \dots q_n} \right)^{3k} h^{-k \varepsilon} 
\end{equation}
Recalling that $\left(J^u_{q_0 \dots q_n} \right)^{3k} h^{-k \varepsilon} \leq C_k e^{3kn \lambda_{\max}(1+\varepsilon)}h^{- k \varepsilon} \leq C_k h^{- k \frac{1-4\varepsilon}{2(1+ \varepsilon)}- \varepsilon} \ll C_k h^{-k/2}$, we see that the right hand side in the above inequality tends to $0$ when $h \to 0$. As a consequence, it is enough to control 
$$\left| \left| \left( 1 + \frac{x^2}{\alpha_n h} \right)^{m/2}v_n \right| \right|_{L^2}$$ This is what we do in the rest of this subsection.

\subsubsection{Reduction to a compact interval. }
Note that since $\WF(A_q)$ is compact, there exists $\chi \in \cinfc(\R)$ such that 
$$ A_q B_q^\prime =  A_{q} B_{q}^\prime\chi(x) + \hinf$$
and it is possible to choose a single $\chi$ for all the $A_q$.
Indeed, recall that $\WF(A_q) = \supp(\chi_q) \Subset \mathcal{W}_q \subset B(\rho_q, 2 \varepsilon_0)$ and that $\kappa_q$ is well-defined in a neighborhood of $\rho_q$ of fixed size $\varepsilon_1$ bigger than $ \varepsilon_0$. There exists a  $\Xi_q \in \Psi_0(\R^2)$ such that 
$$ A_q B_q^\prime \Xi_q = A_q B_q^\prime + \hinf \quad ; \quad \WF(\Xi_q) \Subset \kappa_q (\mathcal{W}_q)$$ 
In particular, $\text{diam } ( \WF(\Xi_q))= O(\varepsilon_0)$. 
Hence, it is enough to fix $\chi \in \cinfc(\R)$ such that  $\chi = 1$ in a neighborhood of $\pi_x ( \WF(\Xi_q) ) $ for all $q \in \mathcal{A}$ and such that $ \supp \chi \subset [-C\varepsilon_0, C\varepsilon_0]$ for some large constant $C$ independent of $\varepsilon_0$. 
As a consequence, we focus on $\chi v_n$.

We set 

\begin{equation}\boxed{ \zeta_n(x) = \kappa_{q_n}^{-1} (\hat{\rho}_n + (x, \beta_n x)) \in \mathcal{W}_{q_n}}
\end{equation}
It describes a curve, preimage by $\kappa_{q_n}$ of the line $\hat{\rho}_n + (x, \beta_n x)$. To ensure that $\zeta_n(x)$ is well defined, $\hat{\rho}_n + (x,\beta_n x)$ has to be at distance at most $\varepsilon_1 $ of $\kappa_{q_n}(\rho_{q_n})=0$. We claim that we may choose $\varepsilon_0$ small enough so that
\begin{equation}\label{Property_chi}
x \in \supp (\chi) \implies \zeta_n(x) \text{ is well defined } 
\end{equation} 
Indeed, $x \in \supp \chi \implies |x| \leq C \varepsilon_0 $, so that $d(\hat{\rho}_n,\kappa_{q_n}(\rho_{q_n})) =O( \varepsilon_0)$ and we choose $\varepsilon_0$, ensuring the good definition of $\zeta_n(x)$.

\subsubsection{Control of the norm of $v_n$.}
Our goal is to control the norm of $\chi v_n$, which allows to control the leading term. In fact, as already explained, to control the higher order terms, it is also necessary to control the norm of $\chi \left(1+ \frac{x^2}{\alpha_n h}\right)^{m/2} v_n$. Let us note $\widetilde{\Psi}_m (x) =\pi^{-1/4} (1+ x^2)^{m/2} e^{-x^2/2}$. 

 \begin{align*}
\left| \left|\left( 1 + \frac{x^2}{\alpha_n h} \right)^{m/2} \chi v_n \right| \right|_{L^2}^2 &=  \int_\R |\chi(x)|^2  e^{ 2tg(\rho)-2t g(\zeta_n(x))} |\Lambda_{\alpha_n h} (\widetilde{\Psi}_m)(x)|^2 dx \\
 &=\int_\R \left| \chi(  (\alpha_n h)^{1/2} x)  \right|^2  e^{ 2tg(\rho)-2t g(\zeta_n((\alpha_n h)^{1/2} x))} (1+x^2)^m e^{-x^2}  dx 
\end{align*}

We have 

\begin{lem}\label{Lemma_restriction_close_to_0}
$$\mathds{1}_{ |x| \geq  \alpha_n^{1/2} h^{\delta_0} }   \left(1+ \frac{x^2}{\alpha_n h}\right)^{m/2} \chi(x)  v_n(x) = \hinf_{L^2}$$
The constants in $\hinf$ depend on $m$ and $\varepsilon$, but neither on $\rho$ nor $n$ as soon as $n \sim \vartheta_\varepsilon \log 1/h$. 
\end{lem}

\begin{proof}
 Since $g = O(\log h)$, we have 
$$e^{2tg(\rho)-2t g (\zeta_n((\alpha_n h)^{1/2} x))} \leq \exp( O( \log(h))) = O (h^{- C})$$ for some $C$ depending on $t$ and $g$. We also have $|\chi| = O(1)$. Hence, after a change of variable $x=\alpha_n^{1/2} h^{1/2} y $, it suffices to estimate 
$$ \int_\R \mathds{1}_{  [  -\alpha_n^{1/2} h^{\delta_0}, \alpha_n^{1/2} h^{\delta_0}]} (\alpha_n^{1/2}h^{1/2} y)  (1+y^2)^m e^{-y^2} dy = \int_{ |y| \geq h^{\delta_0 - 1/2}} (1+y^2)^m e^{-y^2} dy$$
Since $\delta_0 <1/2$, we conclude by using the standard estimate
$$\int_{ |y| > \lambda} (1+y^2)^m e^{-y^2} dy = O_m (\lambda^{-\infty} )$$ 
\end{proof}

A very important consequence of this lemma is that we only need to focus on $\zeta_n(x)$ where $|x| \leq \alpha_n^{1/2} h^{\delta_0}$. In particular, 
$$ |x| \leq \alpha_n^{1/2} h^{\delta_0} \implies  d(\zeta_n(x) , \rho_n) \leq  C \alpha_n^{1/2} h^{\delta_0}  $$ 
Recall that $\alpha_n^{1/2} \leq Ce^{ n \lambda_{\max} (1+ \varepsilon)} \leq C h^{ - \frac{1-4 \varepsilon}{6(1+\varepsilon)}}$ and recall that $\delta_0$ is such that 

\begin{equation}
\delta_0-  \frac{1-4 \varepsilon}{6(1+\varepsilon)} \geq \frac{1}{3}
\end{equation}
ensuring that $\alpha_n^{1/2} h^{\delta_0} \leq C h^{1/3}$. It will turn out to be important.

\subsubsection{Control outside an $h^\delta$-neighborhood of $\mathcal{T}$}

We first treat the case where $\rho$ lies outside an $h^\delta$-neighborhood of $\mathcal{T}$ (in fact, we will be slightly less precise in the unstable direction, see the Lemmata below). The following estimate strongly relies on the structure of the escape function $g$. The escape property of $g$ has been used in \cite{NSZ14} to damp the symbol of the Fourier integral operator $M(h)$ and they prove that the norm of $M_t(h)$ outside an $h^\delta$-neighborhood of $\mathcal{T}$ can be smaller than any arbitrary $\varepsilon$ as soon as $t$ is well chosen. Here, we want to obtain strong polynomial decay of the form $h^L$ for some $L=L(t)$ as large as we want if $t$ is sufficiently large. This will be possible since we propagate on logarithmic times $n(h)$. 

We are interested in controlling the term 
$$d(x) \coloneqq \exp\left(  t g(\rho) - t g (\zeta_n(x)) \right)  $$ 
which controls the norm of $\chi v_n$. Indeed, since $||\widetilde{\Psi}_m||_\infty < + \infty$, we have 
\begin{align*}
 \left| \left|\left( 1 + \frac{x ^2}{\alpha_n h} \right)^{m/2} \chi v_n\right| \right|_{L^2}^2& \leq  \int_{|x| \leq \alpha_n^{1/2}h^{\delta_0}} d(x)^2 |\Lambda_{\alpha_n h} \widetilde{\Psi}_m(x)|^2 dx + \hinf \\
  &\leq C_m  (\alpha_n h)^{-1/2} \int_{|x| \leq \alpha_n^{1/2}h^{\delta_0}} d(x)^2 dx + \hinf \\
\end{align*}

In virtue of the construction of $g$ in (\ref{escape_function}), we have 
\begin{equation}
d(x) =\left( R_-(x)R_+(x)\right)^t \,, \, R_-(x) \coloneqq \frac{Mh^{2\delta} + \hat{\varphi}_-(\rho)}{Mh^{2\delta} + \hat{\varphi}_-(\zeta_n(x))} \, , \,R_+(x) \coloneqq \frac{Mh^{2\delta} + \hat{\varphi}_+(\zeta_n(x))}{Mh^{2\delta} + \hat{\varphi}_+(\rho)}
\end{equation}
(These terms depend on $\rho$ and $h$, but we voluntarily omit this dependence to alleviate the notations). Recall also that $\rho_n = F^{nn_0}(\rho) = \kappa_{q_n}^{-1}(\hat{\rho}_n) = \zeta_n(0)$. 

\begin{prop}\textbf{Estimates for $R_-$.} There exists a global constant $C>0$ (i.e. depending only $F$ and $\varepsilon$ trough the choice of the partition of unity, but independent of $\rho$, $h$ and $\mathbf{q}$) such that for all $x \in [- \alpha_n^{1/2} h^{\delta_0}, \alpha_n^{1/2} h^{\delta_0}]$, we have 
\begin{itemize}
\item If $d(\rho, \mathcal{T}_-) \geq h^\delta$, $ R_-(x) \leq C \left(J^u_\mathbf{q}\right)^{-2} $ ; 
\item If $d(\rho, \mathcal{T}_-) \leq h^\delta$, $R_-(x) \leq C$. 
\end{itemize}
\end{prop}

\begin{proof}
We pick $x \in [- \alpha_n^{1/2} h^{\delta_0}, \alpha_n^{1/2} h^{\delta_0}]$. \begin{itemize}
\item
We assume that $d(\rho, \mathcal{T}_-) \geq h^\delta$.  
By Lemma \ref{Lemma_unstable_Jacobian} and (\ref{equation_alpha_n}), $d( \rho_n, \mathcal{T}_-) \geq C^{-1} J^u_{\mathbf{q}} d(\rho, \mathcal{T}_-) \geq C^{-1} \alpha_n^{1/2} h^\delta$. Then, we have 
$$d(\zeta_n(x), \mathcal{T}_-) \geq d(\rho_n, \mathcal{T}_-) - d (\zeta_n(x), \rho_n ) \geq C^{-1} \alpha_n^{1/2} d(\rho, \mathcal{T}_-)- C \alpha_n^{1/2} h^{\delta_0}  \geq \tilde{C}^{-1} \alpha_n^{1/2} d(\rho,\mathcal{T}_-) $$
since $\delta < \delta_0$ and $h^{\delta} \leq d(\rho, \mathcal{T}_-)$. 
Recall that $\hat{\varphi}_-(\zeta) \sim h^{2\delta} + d(\zeta, \mathcal{T}_-)^2 $. Hence, 
$$R_-(x) \leq \frac{C_1 d(\rho, \mathcal{T}_-)^2}{C_2 \alpha_n d(\rho, \mathcal{T}_-)^2} \leq C \alpha_n^{-1}  \leq C \left(J^u_\mathbf{q} \right)^{-2}$$
\item The second point is much easier (and in fact very crude at this stage) : if $d(\rho, \mathcal{T}_-) \leq h^\delta$, the numerator $Mh^{2\delta} + \hat{\varphi}_-(\rho)$ is smaller that $Ch^{2\delta}$. Concerning the denominator, we simply use the fact that $\hat{\varphi}_- \geq 0$ to bound it from below by $Mh^{2\delta}$, and we deduce that $R_-(x) \leq C$. 
\end{itemize} 
\end{proof}

We now come to the case of $R_+$. Before, we need to understand more precisely the Lagrangian space $\{ (x,\beta_n x) \}$. We expect it to be a good first order approximation of an unstable manifold. This is the content of the following lemma :

\begin{lem}\label{Lemma_local_frame}
There exists a global constant $C>0$ such that the following holds : there exists $\zeta_n^{\mathbf{\star}} \in \mathcal{T}$ such that for all $x \in[ - \alpha_n^{1/2} h^{\delta_0}, \alpha_n^{1/2}h^{\delta_0} ]$, we have 
$$d(\zeta_n(x), \mathcal{T}_+) \leq d(\zeta_n(x), W_u(\zeta_n^{\mathbf{\star}} )) \leq  C \left(h^{1/3} + \left(J^u_\mathbf{q}\right)^{-1} \right) d(\rho, \mathcal{T}_+) + Ch^{\delta_0}$$
\end{lem}

\begin{rem} 
The different terms which compose the error above appear at different places in the proof. One of this term is due to the fact that it is a first order approximation of an unstable manifold : we need to control the error term in this approximation. It turns out that as soon as $|x| \leq \alpha_n^{1/2} h^{\delta_0} = O(h^{1/3})$, this error is $O(|x|^2) = O(h^{2/3})$. Depending on $\rho$ (and $\mathbf{q}$), the main term of the error can differ. As we will see, when $d(\rho, \mathcal{T}) \geq h^{\delta_1}$, the term $h^{\delta_0}$ is negligible. 
\end{rem}

\begin{proof}
\textbf{Step 1 : $\rho_n$ is close to a reference unstable manifold $W_u(\zeta_n^{\mathbf{\star}})$. } (See Figure \ref{Figure_Step_1})
\begin{figure}
\centering
\includegraphics[width=15cm]{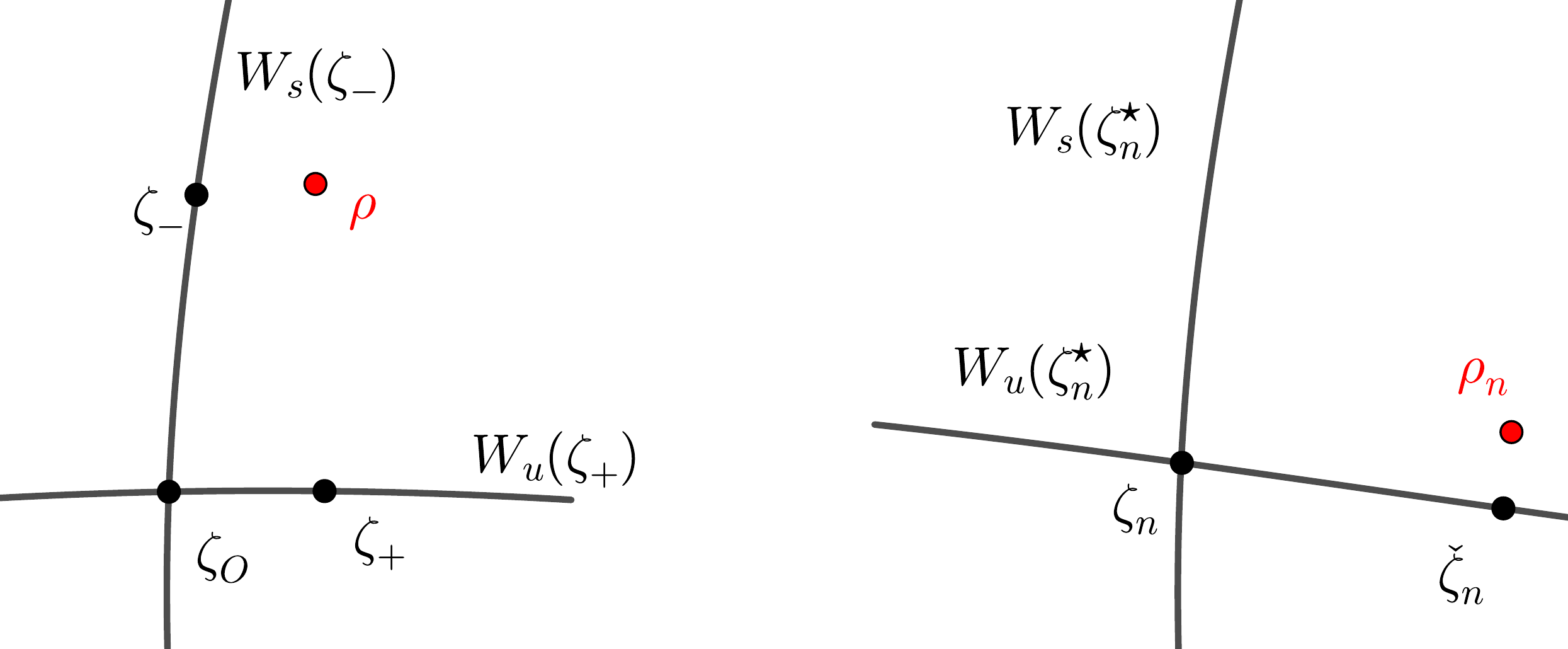}
\caption{The points introduced in Step 1 of the proof of Lemma \ref{Lemma_local_frame}.}
\label{Figure_Step_1}
\end{figure}

 In this first step, we want to show that $\rho_n$ is close to an unstable manifold $W_u(\zeta_n^{\mathbf{\star}})$. 
As in the proof of Lemma \ref{Lemma_unstable_Jacobian}, we consider a point $\zeta_+ \in \mathcal{T}$ such that $d(\rho, \mathcal{T}_+) = d(\rho, W_u(\zeta_+))$ and such that for all $0 \leq i \leq n$, $d(F^{in_0}(\zeta_+), F^{in_0}(\rho)) \leq \varepsilon_2$ for some small $\varepsilon_2$ depending on $\varepsilon_0$. Note also that, by the third point of Lemma \ref{Lemma_unstable_Jacobian}, $d(\rho, \mathcal{T}_-) \leq C_2 \left( J^u_\mathbf{q} \right)^{-1} \varepsilon_0$. Let's fix a point $\zeta_-$ such that $d(\rho, \mathcal{T}_-) = d(\rho, W_s(\zeta_-) )$ and let's consider $\zeta_O$ the unique point in $W_u(\zeta_+) \cap W_s(\zeta_-)$. Then, we still have $d(\rho, \mathcal{T}_+) = d(\rho, W_u(\zeta_O)) $ and we also have 
$$ d(\rho, \zeta_O)^2 \sim  d(\rho, \mathcal{T}_-)^2 + d(\rho, \mathcal{T}_+)^2 \leq C \left( \left(J^u_\mathbf{q}\right)^{-1} \varepsilon_0 \right)^2+ d(\rho, \mathcal{T}_+)^2$$
For $0 \leq i \leq n$, set $\zeta_i^{\mathbf{\star}}= F^{i n_0} (\zeta_O)$. 
We have 
$$ d(\rho_n, W_u(\zeta_n^{\mathbf{\star}}) ) \leq C \left( J^u_\mathbf{q} \right)^{-1} d(\rho, W_u(\zeta_O)) \leq  C \left( J^u_\mathbf{q} \right)^{-1}  d(\rho, \mathcal{T}_+) $$
and $$d(\rho_n, \zeta_n^{\mathbf{\star}})^2 \leq C\left(J^u_\mathbf{q}\right)^2 d(\rho, \mathcal{T}_-)^2 + C\left(J^u_\mathbf{q}\right)^{-2} d(\rho, \mathcal{T}_+)^2 \leq C \varepsilon_0$$
Let us fix $\check{\zeta}_n\in W_u(\zeta_n^{\mathbf{\star}})$ such that 
$$ d(\rho_n, W_u(\zeta_n^{\mathbf{\star}}) )  = d(\rho_n, \check{\zeta}_n )$$

\textbf{Step 2 : The curve $\zeta_n(x)$ is close to the (unstable) tangent space $E_u(\check{\zeta}_n)$. }

\emph{Step 2-a : First approximation. } (See Figure \ref{Fig_graph}). 

We now want to show that the curve is a rather good approximation of the tangent space of $W_u(\zeta_n^{\mathbf{\star}})$ at $\check{\zeta}_n$. 
To do so, we make the following observation (recall the notations of (\ref{ecriture_d_rho_F_n}) and the definition of $\beta_n$ in (\ref{definition_alpha_beta})). 
$$\mathbf{v}_n \coloneqq \left( \begin{matrix}
1 \\ \beta_n
\end{matrix} \right)  = \alpha_n^{-1/2} d_{\hat{\rho}} F^{(n)} (\mathbf{v}^\prime_n) \; ; \; \mathbf{v}_n^\prime = \alpha_n^{-1/2} \left( \begin{matrix}
a_n \\ b_n 
\end{matrix} \right) $$
and note that $||\mathbf{v}_n^\prime|| =1$ (since $\alpha_n^2 = a_n^2 + b_n^2)$.
We compare this vector $\mathbf{v}_n$ to $\mathbf{w}_n \coloneqq \alpha_n^{-1/2} d_{\hat{\zeta}} F^{(n)} (\mathbf{v}_n^\prime)$ where $\hat{\zeta} = \kappa_{q_0}(\check{\zeta}_0)$ with $\check{\zeta}_0=  F^{-nn_0} ( \check{\zeta}_n ) $. Arguing as in the proof of Lemma \ref{Lemma_desc_d_rho_F_n}, we can show that 
$$ || d_{\hat{\rho}} F^{(n)} - d_{\hat{\zeta}} F^{(n)} || \leq C J^u_{\mathbf{q}} d(\rho, \check{\zeta}_0) $$
By the triangular inequality, $d(\rho,\check{\zeta}_0)  \leq d(\rho,  \zeta_O) + d(\zeta_O, \check{\zeta}_0 ) $ where the first term is controlled by $C \left( \left(J^u_\mathbf{q}\right)^{-1}\varepsilon_0 + d(\rho, \mathcal{T}_+)\right) $. For the second term, we use the fact that $\check{\zeta}_0 \in W_u(\zeta_O)$ and $d(\check{\zeta}_n, F^{nn_0} (\zeta_O)) \leq C \varepsilon_0$, this gives $d(\zeta_O, \check{\zeta}_0) \leq C\left(J^u_\mathbf{q}\right)^{-1} \varepsilon_0$. 
As a consequence, we find that $$||d_{\hat{\rho}} F^{(n)} - d_{\hat{\zeta}} F^{(n)} || \leq C J^u_{\mathbf{q}} d(\rho, \mathcal{T}_+) + C \varepsilon_0$$  
Finally, recalling that $\alpha_n^{1/2} \sim J^u_{\mathbf{q}}$ - so that $\mathbf{v}_n$ and $\mathbf{w}_n$ are close from being normalized -, we get that $||\mathbf{v}_n - \mathbf{w}_n || \leq  Cd(\rho, \mathcal{T}_+) + C \varepsilon_0 \left(J^u_\mathbf{q}\right)^{-1}$.\\
Let's now define $\widetilde{\zeta}_n(x)$ by 
$$ \kappa_{q_n} \left( \widetilde{\zeta}_n(x) \right)= \kappa_{q_n} \left( \check{\zeta}_n \right) + x \mathbf{w}_n  $$
We have (recall that $\zeta_n(x) = \kappa_{q_n}^{-1} (\hat{\rho}_n + x \mathbf{v}_n)$)
\begin{align*}
 d( \widetilde{\zeta}_n(x) , \zeta_n(x)) )& \leq  C d( \kappa_{q_n} \left( \check{\zeta}_n\right) + x \mathbf{w}_n , \hat{\rho}_n + x \mathbf{v}_n)\\
 & \leq C  d( \check{\zeta}_n, \rho_n) + |x| ||\mathbf{w}_n - \mathbf{v}_n||  \\
 &\leq C \left( J^u_\mathbf{q} \right)^{-1} d(\rho, \mathcal{T}_+) + |x| \left(C d(\rho, \mathcal{T}_+) +  \varepsilon_0 \left(J^u_\mathbf{q}\right)^{-1} \right)  \\
 & \leq  C \left( J^u_\mathbf{q} \right)^{-1} d(\rho, \mathcal{T}_+) + C J^u_\mathbf{q} h^{\delta_0} \left( d(\rho, \mathcal{T}_+) + C \varepsilon_0 \left(J^u_\mathbf{q}\right)^{-1} \right) \text{ if } |x| \leq \alpha_n^{1/2} h^{\delta_0} 
\\
&\leq  C \Big( \left( J^u_\mathbf{q} \right)^{-1} + h^{1/3} \Big) d(\rho, \mathcal{T}_+ ) + Ch^{\delta_0}
\end{align*}
where we use the fact that $J^u_\mathbf{q} h^{\delta_0} \leq Ch^{1/3}$. 
We will now control the distance of $\widetilde{\zeta}_n(x)$ to $\mathcal{T}_+$. 

\vspace{0.3cm}

\emph{Step 2-b : Comparison with the tangent space. }(See Figure \ref{Fig_graph}).

\begin{figure}
\includegraphics[scale=0.5]{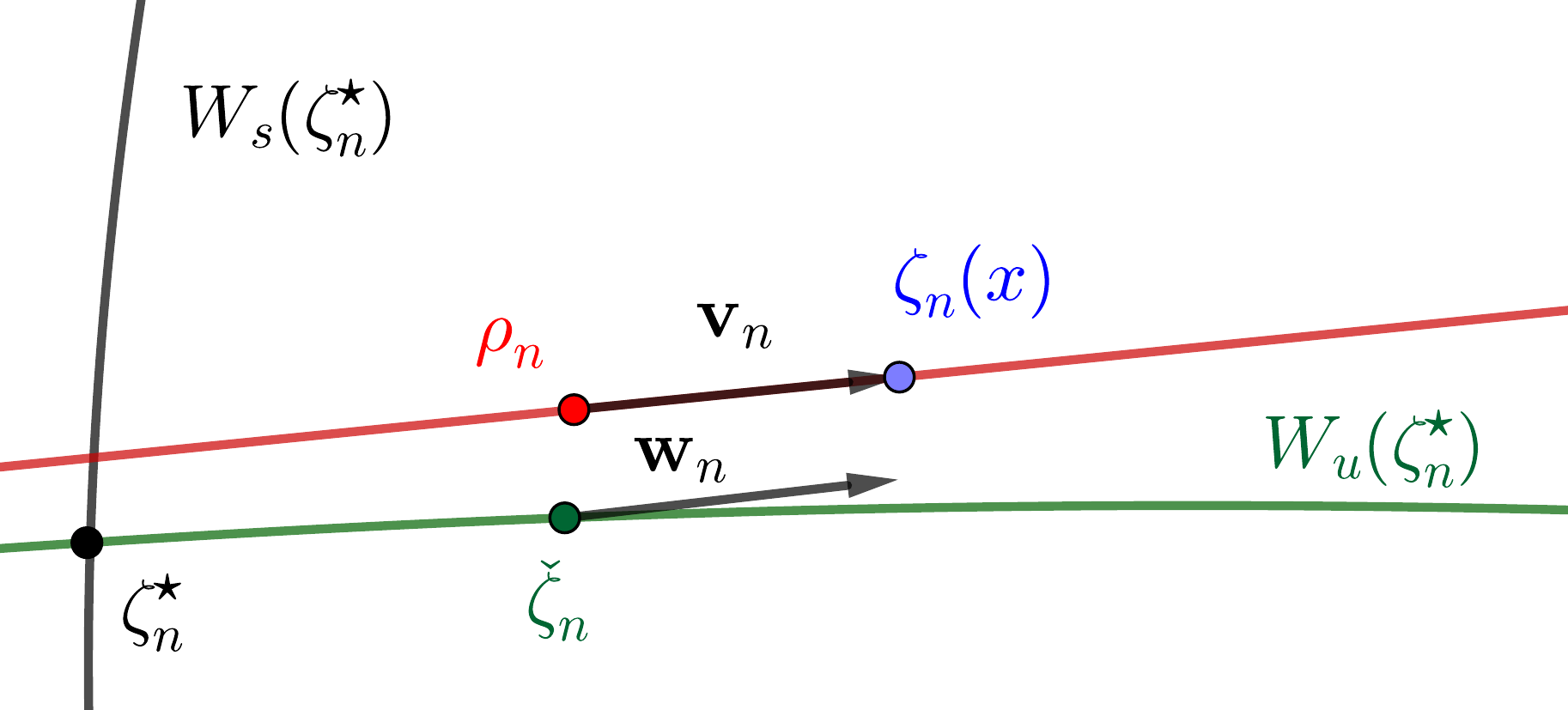}
\caption{The curve $\zeta_n(x)$ (in red) passing through $\rho_n$ is close to an unstable manifold $W_u(\zeta_n^{\mathbf{\star}})$ (in green). $W_u(\zeta_n^{\mathbf{\star}})$ is well approximated, near $\check{\zeta}_n$, by its tangent space at $\check{\zeta}_n$, spanned by a vector close to $\mathbf{w}_n$.}
\label{Fig_graph}
\end{figure}

In this step, we want to show that  $\mathbf{w}_n$ is close to a vector spanning $d\kappa_{q_n} (T_{\check{\zeta}_n} W_u(\zeta_n^{\mathbf{\star}} ))$. To do so, we use Lemma \ref{Lemma_linearized_dynamics}. If $\varepsilon_0$ is small enough (depending on the parameter $\varepsilon_1$ appearing in Lemma \ref{Lemma_linearized_dynamics}), we can ensure that the vector $\mathbf{v}_n$ is suffienctly close to $\R \times \{0\}$ and hence,  $\left( d_{\check{\zeta}_0}\kappa_0 \right) ^{-1} \mathbf{v}_n^\prime $ is sufficiently close to $T_{\check{\zeta}_0} W_u (\zeta_O)$,  so that we can apply this lemma with initial vector $\left( d_{\check{\zeta}_0}\kappa_0 \right) ^{-1} \mathbf{v}_n^\prime $. To alleviate the notations, let's note  $\mathcal{L}= \kappa_{q_n} (W_u(\zeta_n^{\mathbf{\star}}))$, $ m =\kappa_{q_n}(\check{\zeta}_n)$. By applying Lemma \ref{Lemma_linearized_dynamics} and sending the result in the chart $\kappa_{q_n}$, we obtain that  
$$ d\left( \frac{\mathbf{w}_n}{||\mathbf{w}_n||},T_{m} \mathcal{L}  \right) \leq C\left( J_\mathbf{q}^u\right)^{-2}  $$
since $||\mathbf{w}_n|| \leq C$, the same is true for $\mathbf{w}_n$. Let's pick $\mathbf{w}_n^\prime \in T_m \mathcal{L}$ such that $||\mathbf{w}_n - \mathbf{w}_n^\prime|| \leq  C\left( J_\mathbf{q}^u\right)^{-2} $.\\
We now define $Z_n(x)$ by the relation 
$$ \kappa_{q_n} \left(  Z_n(x)\right) = \kappa_{q_n} \left(  \widetilde{\zeta}_n(x)\right) + \mathbf{w}_n^\prime x$$
If $|x| \leq \alpha_n^{1/2} h^{\delta_0}$, it is clear that 
$$ d(Z_n(x) , \widetilde{\zeta}_n(x) ) \leq |x| ||\mathbf{w}_n - \mathbf{w}_n^\prime|| \leq  C \left( J^u_\mathbf{q} \right)^{-2} \alpha_n^{1/2} h^{\delta_0} \leq C \alpha_n^{-1/2} h^{1/2} \ll h^{\delta_0}$$
Gathering the steps 2-a and 2-b, we see that 
$$ d(\zeta_n(x), \mathcal{T}_+) \leq d( Z_n(x), \mathcal{T}_+) +  C \Big( \left( J^u_\mathbf{q} \right)^{-1} + h^{1/3} \Big) d(\rho, \mathcal{T}_+ ) + Ch^{\delta_0}$$

\textbf{Step 3 : The tangent space is a good approximation.}
The only remaining point is to control $d( Z_n(x), \mathcal{T}_+)$. We observe that $\mathbf{w}_n^\prime \in T_m \mathcal{L}$. Hence, by standard results of differential geometry, 
$ d(m + x \mathbf{w}_n^\prime, \mathcal{L}) \leq C x^2$ where $C$ depends on $||\mathbf{w}_n^\prime||$ and on the curvature of $\mathcal{L}$ - which can be controlled independently of the base point $\widetilde{\zeta}_n$ of this unstable manifold.  As a consequence, if $|x| \leq \alpha_n^{1/2}h^{\delta_0} \leq h^{1/3}$, 
$ d(m + x \mathbf{w}_n^\prime, \mathcal{L})  \leq C h^{2/3} \ll Ch^{\delta_0}$. This shows that 
$d( Z_n(x), \mathcal{T}_+) \leq C  d(m + x \mathbf{w}_n^\prime, \mathcal{L}) \leq C h^{\delta_0}$ and concludes the proof of the lemma. 
\end{proof}

This Lagrangian being well understood, we can now come to the estimates for $R_+$ : 
\begin{prop}\textbf{Estimates for $R_+$.}\label{Prop_estimates_R_+}
There exists a global constant $C>0$ such that for all $x \in [- \alpha_n^{1/2} h^{\delta_0}, \alpha_n^{1/2} h^{\delta_0}]$, we have 
\begin{itemize}
\item If $d(\rho, \mathcal{T}_+) \geq h^{\delta_1}$,  $R_+(x) \leq Ch^{2 \varepsilon}$ ; 
\item If $d(\rho, \mathcal{T}_+) \leq h^{\delta_1}$,  $R_+(x) \leq C$
\end{itemize}
 (for some constant $C>0$).
\end{prop}

\begin{proof}
Recall that $\delta_1 = \delta-\varepsilon$. 
We pick $x \in [- \alpha_n^{1/2} h^{\delta_0}, \alpha_n^{1/2} h^{\delta_0}]$. 
Here, we will use the inequality $d(\rho_n, \mathcal{T}_+) \leq C\left(J^u_{\mathbf{q} } \right)^{-1} d(\rho, \mathcal{T}_+)$ and the result of the previous lemma, namely, 
$$d(\zeta_n(x), \mathcal{T}_+) \leq  C \left(h^{1/3} + \left(J^u_\mathbf{q}\right)^{-1} \right) d(\rho, \mathcal{T}_+) + Ch^{\delta_0}$$
Recall that $J^u_\mathbf{q} \geq C_\varepsilon e^{ n(h) \lambda_{\min} (1- \varepsilon)} \geq C_\varepsilon h^{-\vartheta_\varepsilon  \lambda_{\min} (1- \varepsilon)}$. \\
We choose some $0<\beta < \min(1/3, \vartheta_\varepsilon  \lambda_{\min} (1- \varepsilon)) $, which ensures that 
$$d(\zeta_n(x), \mathcal{T}_+) \leq  C h^{\beta} d(\rho, \mathcal{T}_+) + Ch^{\delta_0}$$
Note that since we work with $\varepsilon$ small, it is harmless to assume that $\varepsilon <  \beta$. 
We treat the two points separately : 
\begin{itemize}
\item For this first point, we distinguish two cases : 
\paragraph{First case : $h^{\delta_1} \leq d(\rho, \mathcal{T}_+)  \leq h^{\delta - \beta}$}
In this context, one has $d(\zeta_n(x), \mathcal{T}_+) \leq C h^{\beta} h^{\delta-\beta}+ Ch^{\delta_0} \leq Ch^\delta$. As a consequence, $\hat{\varphi}_+(\zeta_n(x) ) \le Ch^{2\delta}$. We also have $\hat{\varphi}_+(\rho) \geq C^{-1} ( h^{2 \delta} + h^{2 \delta_1} ) \geq C^{-1} h^{2\delta_1}$ which gives
$$ R_+(x) \leq \frac{(M + C)h^{2 \delta}}{Mh^{2 \delta} + C^{-1} h^{2\delta_1}} \leq C h^{2 (\delta- \delta_1)}  =Ch^{2 \varepsilon}$$ 
\paragraph{Second case: $ d(\rho, \mathcal{T}_+)  \geq h^{\delta - \beta} $}. In this context, we have 
$d(\rho, \mathcal{T}_+)^2 \gg h^{2\delta}$ so that we can bound the denominator $Mh^{2\delta} + \hat{\varphi}_+(\rho)$ from below by $C^{-1} d(\rho, \mathcal{T}_+)^2$. Concerning the numerator, we have 
$$d(\zeta_n(x), \mathcal{T}_+) \leq C h^\beta d(\rho, \mathcal{T}_+) + Ch^{\delta_0} \leq  C h^\beta d(\rho, \mathcal{T}_+)$$
since $ h^\beta d(\rho, \mathcal{T}_+) \geq h^\beta h^{\delta- \beta} \geq h^\delta \gg h^{\delta_0}$. 
We deduce also that $\hat{\varphi}_+(\zeta_n(x) ) \leq C  h^{2\beta} d(\rho, \mathcal{T}_+)^2$. As a consequence, 
$$R_+(x) \leq \frac{Ch^{2 \beta} d(\rho, \mathcal{T}_+)^2}{C^{-1}   d(\rho, \mathcal{T}_+)^2} \leq Ch^{2\beta} \ll h^{2 \varepsilon}$$
\item We now assume that $d(\rho, \mathcal{T}_+) \leq h^{\delta_1}$. As in the first case above, we can bound the numerator by $Ch^{2 \delta}$. Concerning the denominator, we simply use the fact that $\hat{\varphi}_+ \geq 0$ to bound it from below by $Mh^{2\delta}$, and this gives, as expected
$$ R_+(x) \leq C$$ 
\end{itemize}

\end{proof}

Let's recap these two estimates and their implications concerning $d(x)$ (and recall that by definition, $\beta > \varepsilon$ and $J^u_\mathbf{q} \geq C^{-1} h^{-\beta}$) 
\begin{align*}
d(\rho, \mathcal{T}_-) \geq h^{\delta} \implies d(x) \leq (Ch^{2 \beta})^t \ll h^{2t \varepsilon} \, , \,  \forall x \in [- \alpha_n^{1/2} h^{\delta_0}, \alpha_n^{1/2} h^{\delta_0}] \\
d(\rho, \mathcal{T}_-) \leq h^{\delta} \text{ and } d(\rho, \mathcal{T}_+) \geq h^{\delta_1}  \implies d(x) \leq \left( Ch^{2 \varepsilon} \right)^t  \, , \,  \forall x \in [- \alpha_n^{1/2} h^{\delta_0}, \alpha_n^{1/2} h^{\delta_0}]
\end{align*}

As a consequence, the $L^2$ norm of $\chi v_n$ is very small when $\rho$ lies outside the neighborhood of $\mathcal{T}$ defined before Proposition \ref{Prop_Key} :   

\begin{equation}
\mathcal{T}_{\delta,\delta_1} = \left\{ \rho \, , \, d(\rho, \mathcal{T}_-) \leq h^{\delta}, d(\rho, \mathcal{T}_+) \leq h^{\delta_1} \right\}
\end{equation}

Indeed, we obviously have 
\begin{prop}\label{Prop_crucial_estimates_0}
For all $L>0$, there exists $t=t(\varepsilon,L)$ such that the following holds. 
Assume that $\rho \not \in \mathcal{T}_{\delta,\delta_1} $. Then, 
$$\int_{ |x| \leq \alpha_n^{1/2} h^{\delta_0}} d(x)^2 dx \leq  C h^L$$
\end{prop}

\subsubsection{ Crucial estimates in $\mathcal{T}_{\delta,\delta_1}$}
We now turn to the crucial estimate which helps to control the $L^2$ norm of $\chi v_n$ when $\rho \in \mathcal{T}_{\delta,\delta_1} $. 

\begin{prop}\label{Prop_crucial_estimates}
Assume that $\rho \in \mathcal{T}_{\delta,\delta_1} $. Then, 
$$\int_{ |x| \leq \alpha_n^{1/2} h^{\delta_0}} d(x)^2 dx \leq  C \left( J^u_\mathbf{q} \right)^{d_H + \varepsilon} h^{ (\delta_0- \delta)(d_H + \varepsilon)} h^\delta$$
\end{prop}

\begin{proof}

\textbf{Step 0 : A simple estimates for $d(x)$. }
First recall from Proposition \ref{Prop_estimates_R_+}, $\rho \in \mathcal{T}_{\delta,\delta_1}  \implies d(x) \leq C R_-(x)^t$. 
Moreover, the numerator in $R_-(x)$ is bounded by $Ch^{2\delta}$ and since $\hat{\varphi}_-(\zeta_n(x)) \geq Ch^{2\delta} + Cd(\zeta_n(x), \mathcal{T}_-)^2$, we find that 
$$ d(x) \leq C \left( 1 +\left(  \frac{d(\zeta_n(x), \mathcal{T}_-)}{h^\delta} \right)^2 \right)^{-t}$$

\textbf{Step 1 : The mass is supported in an $h^\delta$-neighborhood of $\mathcal{T}$. }
We use Lemma \ref{Lemma_local_frame} which asserts that there exists $\zeta_n^{\mathbf{\star}}$ such that \begin{equation}\label{ll_h_delta}
d(\zeta_n(x) , W_u(\zeta_n^{\mathbf{\star}}) ) \leq Ch^\beta d(\rho, \mathcal{T}_+) + Ch^{\delta_0} \ll h^\delta
\end{equation}
with $\beta$ defined in the proof of Proposition \ref{Prop_estimates_R_+}. 
  Recall that in the chart $\kappa_{q_n}$, $\kappa_{q_n}(\zeta_n(x) ) = \hat{\rho}_n +(x, \beta_n x)$. Moreover, if $\varepsilon_0$ is small enough, we may assume that $\kappa_{q_n} (W_u(\zeta_n^{\mathbf{\star}}) ) $ can be written as the graph of a function : 
$$\kappa_{q_n} (W_u(\zeta_n^{\mathbf{\star}}) )=  \{ (x, G_u(x)) , x \in I_u \}$$ where $I_u$ is a small interval of size $ \sim \varepsilon_0$ and $G_u$ a smooth function with bounded $C^\infty$ norms (with bounds depending only on $F$ and the charts). Since $d(\rho_n,W_u(\zeta_n^{\mathbf{\star}})  ) \ll h^\delta$,  up to translating, we may assume that $\hat{\rho}_n = (0 , \xi_n)$ and $|G_u(0) - \xi_n| \ll h^\delta$. In particular, if $h$ is small enough, we may assume that $[ -\alpha_n^{1/2} h^{\delta_0}, \alpha_n^{1/2} h^{\delta_0} ] \subset I_u$. Finally, if $\varepsilon_0$ is small enough, we can also assume that $|G_u^\prime (x)| \leq 1/4$ if $|x| \leq 2 \alpha_n^{1/2} h^{\delta_0} \ll 1$ (recall that the chart $\kappa_q$ is centered at a point $\rho_q$ such that $\kappa_q(E_u(\rho_q)) =\R \times \{ 0 \}$).  We now set 

$$ X(\mathcal{T} ) = \{ x \in [ -2\alpha_n^{1/2} h^{\delta_0}, 2\alpha_n^{1/2} h^{\delta_0} ], \kappa_{q_n}^{-1}(x,G_u(x)) \in \mathcal{T} \} $$ 
Let's cover $X(\mathcal{T} ) $ by $N$ intervals of size $2h^\delta$, centered at points $x_1, \dots, x_N \in X(\mathcal{T} )$. In virtue of Lemma \ref{Lemma_number_of_intervals}, we can choose $N$ such that 
$$ N \leq C \left( J_\mathbf{q}^u h^{\delta_0 - \delta} \right)^{d_H + \varepsilon}$$ 

Each interval around $x_i$ of size $O(h^\delta)$ supports a mass of order $O(h^\delta)$. Our aim in the following lines is to show that the weight of the integral supported at distance larger than $h^\delta$ of the $x_i$ is also $O(h^\delta)$, so that we will be able to estimate the whole integral by $N h^\delta$, which would conclude the proof. 
Let us consider $x \in [ -\alpha_n^{1/2} h^{\delta_0}, \alpha_n^{1/2} h^{\delta_0} ]$ and assume that for all $1 \leq i \leq N$, $|x - x_i | \geq 2 h^\delta$. Let us choose $i$ such that $|x-x_{i} | = \min_{1 \leq k \leq N } |x-x_k|$. We claim that there exists $\nu >0$, uniform with respect to $\rho$, $h$ and $ x \in [ -\alpha_n^{1/2} h^{\delta_0}, \alpha_n^{1/2} h^{\delta_0} ]$ such that 
\begin{equation}\label{claim_mass}
 d(\zeta_n(x), \mathcal{T}_-) \geq \nu |x-x_{i}| 
\end{equation}
Let's admit it for a while. For $i \in \{1, \dots, N \}$, let's note $$J_i = \{ x \in  [-\alpha_n^{1/2} h^{\delta_0},  \alpha_n^{1/2} h^{\delta_0}], |x-x_i| = \min_{1 \leq k \leq N } |x-x_k| \} $$
These intervals form a partition of $[-\alpha_n^{1/2} h^{\delta_0},  \alpha_n^{1/2} h^{\delta_0}]$. 
\begin{align*} 
\int_{ |x| \leq \alpha_n^{1/2} h^{\delta_0}} d(x)^2 dx &\leq \int_{ |x| \leq \alpha_n^{1/2} h^{\delta_0}}   C \left( 1 +\left(  \frac{d(\zeta_n(x), \mathcal{T}_-)}{h^\delta} \right)^2 \right)^{-2t}dx \\ 
&\leq C \sum_{i=1}^N \int_{J_i} \left( 1 +\left(  \frac{d(\zeta_n(x), \mathcal{T}_-)}{h^\delta} \right)^2 \right)^{-2t}dx \\
&\leq C \sum_{i=1}^N \left(  \int_{|x_i - x| \leq 2h^{\delta} } 1 dx +  \int_{|x_i - x| > 2h^{\delta} } \left( 1 +\left(\frac{\nu |x-x_i|}{h^\delta} \right)^2 \right)^{-2t}dx \right) \\
&\leq C \sum_{i=1}^N \left(4 h^\delta +\int_{|y| >2} \left( \frac{1}{1 + (\nu y)^2}\right)^{2t} h^\delta dy  \right) \\
&\leq C N h^\delta
\end{align*}
Here, $t$ is large enough and in particular, we may ensure that $t \geq 1$ so that the integral converges. 

\textbf{Step 2 : Proof of the claim (\ref{claim_mass}).} 
We argue by contradiction and assume that  $d(\zeta_n(x), \mathcal{T}_-) \leq \nu |x-x_{i}| $ for some sufficiently small $\nu$ (with conditions specified below). 
Since $\mathcal{T}_-$ is made of local stable leaves near $\mathcal{T}$ (and $\zeta_n(x)$ lies in a small neighborhood of $\mathcal{T}$), we may chose $\rho_- \in \mathcal{T}$ such that $d(\zeta_n(x), \mathcal{T}_-) = d(\zeta_n(x), W_s(\rho_-) )$. Let's still note $\rho_- \in \mathcal{T}$ the unique point of $W_u(\zeta_n^{\mathbf{\star}}) \cap W_s(\rho_-) $ and let's write $\kappa_{q_n}(\rho_-) = (x_-, G_u(x_-) ) $. Again if $\varepsilon_0$ is small enough, all the stable leaves in $\kappa_{q_n}$ can be written as graphs in the vertical variable : let us write 
$$\kappa_{q_n}( W_s( \rho_-) ) = \{ (H_s(\xi), \xi) , \xi \in I_s \}$$ where $I_s$ is a small interval of size $O( \varepsilon_0)$ and $H_s$ a smooth function with $C^\infty$ norms bounded by constants only depending on the dynamics and the chart.  Up to translating, we may assume that $(H_s(0),0)= \kappa_{q_n}(\rho_-) = (x_-,0)$. As for $G_u$, if $\varepsilon_0$ is small enough, we can assume that $|H_s^\prime(\xi)| \leq 1$ for all $\xi \in I_s$. Finally, let us note $\rho_{\min} = \kappa_{q_n}^{-1}(H_s(\xi_{\min}), \xi_{\min})$ a point in $W_s(\rho_-)$ such that 
$d(\zeta_n(x), \rho_{\min}) = d(\zeta_n(x), W_s(\rho_-))$ (see Figure \ref{Fig_graph_2}).
\begin{figure}
\includegraphics[scale=0.5]{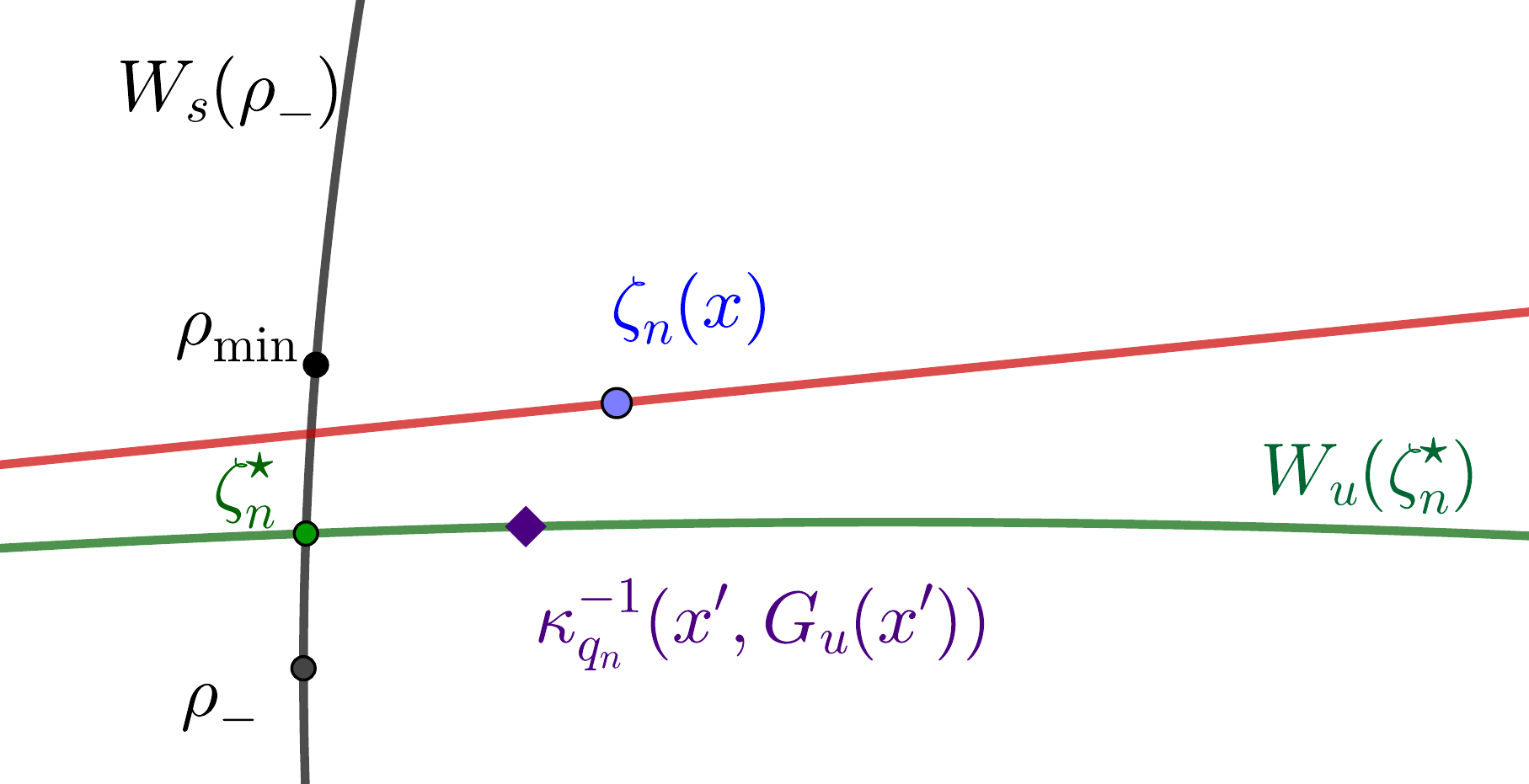}
\caption{The points appearing in the proof of the claim \ref{claim_mass}. The curve $\zeta_n(x)$ is in red. The image of the important point $(x^\prime, G_u(x^\prime))$ is a purple diamond. }
\label{Fig_graph_2}
\end{figure}

Since by (\ref{ll_h_delta}), $d(\zeta_n(x), W_u(\zeta_n^{\mathbf{\star}})) \ll h^\delta$, we can find $x^\prime \in I_u$ such that $||(x,\beta_n x) - (x^\prime, G_u(x^\prime) ) || \ll h^\delta$. This inequalitiy implies
\begin{align*}
|x-x^\prime| \ll h^\delta  \, ,  \\
|\beta_n x - G_u(x^\prime)  | \ll h^\delta  \, , \\
|G_u(x) - G_u(x^\prime) | \ll h^\delta  \, , \\
|G_u(x) - \beta_n x | \ll h^\delta.
\end{align*}
Since by assumption, $|x-x_i| \geq 2 h^\delta$, when $h$ is small enough, the inequality 
$|G_u(x) - \beta_n x | \leq \nu |x-x_i| $ holds. 
We have $$|| (H_s(\xi_{\min}), \xi_{\min})  - (x, \beta_n x) || \leq C d(\zeta_n(x), \rho_{\min})  \leq C d(\zeta_n(x), \mathcal{T}_-) \leq C \nu |x- x_i|$$
From this we deduce that 
\begin{align*}
|\xi_{\min}| &\leq |\xi_{\min} - G_u(x) | + |G_u(x) | =  |\xi_{\min} - G_u(x) | + |G_u(x) -G_u(x_-)|\\
&\leq |\xi_{\min} - \beta_n x | + |\beta_n x - G_u(x) | + \frac{1}{4} |x-x_-| \\
&\leq || (H(\xi_{\min}), \xi_{\min})  -(x, \beta_n x) || + \nu|x-x_i| + \frac{1}{4} |x-x_-| \\
&\leq C \nu |x-x_i| + \frac{1}{4} |x-x_-|
\end{align*}
Finally, we find that, 
\begin{align*}
|x_- - x| &\leq |x_- - H_s(\xi_{\min}) | + |H_s(\xi_{\min}) - x |\\
&\leq |H(0) - H(\xi_{\min} )| + || (H(\xi_{\min}) , \xi_{\min})- (x,\beta_nx ) ||   \\
&\leq   |\xi_{\min}| + C \nu |x-x_i| \text{  (recall that } |H^\prime| \leq 1) \\
& \leq \frac{1}{4}|x-x_-| + C \nu |x- x_i|
\end{align*}
From this, we deduce that 
\begin{equation}\label{equation_proof_claim}
|x-x_-| \leq \frac{4}{3} C \nu |x- x_i|
\end{equation}
A first consequence of this inequality is that if $\nu$ is small enough so that $\frac{4\nu C}{3} \leq \frac{1}{4}$, we have 
$$|x_-| \leq |x| + \frac{1}{4} |x-x_i| \leq \frac{5}{4}|x| + |x_i| \leq \frac{5}{4}\alpha_n^{1/2} h^{\delta_0} +\frac{2}{4} \alpha_n^{1/2} h^{\delta_0}  \leq 2 \alpha_n^{1/2} h^{\delta_0}$$

Since $\kappa_{q_n}^{-1}(x_-, G_u(x_-) )= \rho_- \in \mathcal{T}$, we deduce that $x_- \in X(\mathcal{T})$. In particular, there exists $j \in \{1 , \dots, N \}$ such that $|x_- - x_j| \leq h^\delta$. But then, we would have 
$$|x_i - x| \leq |x_j - x| \leq |x_- - x_j| + |x_- - x| \leq h^\delta + \frac{1}{4} |x-x_i| \leq \frac{1}{2} |x-x_i| + \frac{1}{4} |x-x_i|  < |x-x_i|$$
(recall that $|x-x_i| \geq 2 h^\delta$) . This gives the required contradiction and concludes the proof of the claim (\ref{claim_mass}).

\end{proof}

\subsection{End of the proof. } 
We can use Lemma \ref{Lemma_restriction_close_to_0}, Proposition \ref{Prop_crucial_estimates_0} and Proposition \ref{Prop_crucial_estimates} to conclude the proof of Proposition \ref{Prop_Key}. 
Indeed, since $||\widetilde{\Psi}_m||_\infty < + \infty$, we have 
\begin{align*}
 \left| \left|\left( 1 + \frac{x ^2}{\alpha_n h} \right)^{m/2} \chi v_n\right| \right|_{L^2}^2& \leq  \int_{|x| \leq \alpha_n^{1/2}h^{\delta_0}} d(x)^2 |\Lambda_{\alpha_n h} \widetilde{\Psi}_m(x)|^2 dx + \hinf \\
  &\leq C_m  (\alpha_n h)^{-1/2} \int_{|x| \leq \alpha_n^{1/2}h^{\delta_0}} d(x)^2 dx + \hinf \\
\end{align*}
It gives a bound $C_m h^{L}$ when $\rho \not \in \mathcal{T}_{\delta,\delta_1}$ (with $L$ as large as necessary by choosing $t$ large enough) and when $\rho \in \mathcal{T}_{\delta,\delta_1}$, we find that 
$$\left| \left|\left( 1 + \frac{x ^2}{\alpha_n h} \right)^{m/2} \chi v_n \right| \right|_{L^2}^2   \leq  C_m  \left( J^u_{\mathbf{q}}\right)^{d_H-1 + \varepsilon} h^{( \delta_0 - \delta)(d_H + \varepsilon) + \delta -1/2  } $$
When $m=0$, it gives a control of the leading term, since we have 
$$||u_n^0||_{L^2}^2 \leq  \Pi_{\alpha,n}(\rho)^2 ||\chi v_n||_{L^2}^2 $$
and since $\Pi_{\alpha,n}(\rho)^2 = O(h^{-L_2})$ for some $L_2>0$, so that for $\rho \not \in \mathcal{T}_{\delta,\delta_1}$ we can have $||u_n^0||_{L^2}^2 = O(h^{L})$ for any $L$ by choosing $t$ large enough.

 It controls the first term of the expansion given by Corollary \ref{Cor_expansion}. We recall that the number of terms in this expansion is controlled by a integer $N \in \N$. 
For the other terms in the expansion given by Corollary \ref{Cor_expansion}, as already explained with (\ref{other_terms_control}), they all have their $L^2$ norms controlled by some $$\varepsilon(h)  \Pi_{\alpha,n}(\rho) \left| \left|\left( 1 + \frac{x^2}{\alpha_n h} \right)^{m/2} \chi v_n\right| \right|_{L^2}$$ with $m \leq N$ and $\varepsilon(h) \to 0$ when $h \to 0$. Finally, we can choose $N= N(\varepsilon)$ such that the remainder has an $L^2$ norm $O(h^{2L})$. This concludes the proof of Proposition \ref{Prop_Key}, and eventually of Theorem \ref{Theorem_main}.

\section{Proof of the fractal Weyl upper bounds in obstacle scattering and scattering by a potential}
\subsection{Proof of Theorem \ref{THM_new}}
 \label{subsection_billiard_map}
Let us show how Theorem \ref{Theorem_main} implies Theorem \ref{THM_new}. Suppose that the obstacles $\mathcal{O}_j$  are strictly convex, have smooth boundary and satisfy Ikawa condition of no-eclipse. We will use the results of \cite{NSZ14} to apply Theorem \ref{Theorem_main} to the case where $F$ is the billiard map. To be precise,  let us introduce the following notations. 

For $j \in \{1, \dots, J \}$, let $B^* \partial \mathcal{O}_j$ be the co-ball bundle of $ \partial \mathcal{O}_j$, $S^*_{ \partial \mathcal{O}_j} $ be the restriction of $S^* \Omega$ to $ \partial \mathcal{O}_j$, $\pi_j : S^*_{ \partial \mathcal{O}_j} \to B^* \partial \mathcal{O}_j$ the natural projection and $\nu_j(x)$ be the outward normal vector at $x \in  \partial \mathcal{O}_j$ (see Figure \ref{figure_billiard_map}).

\begin{figure}
\begin{subfigure}{0.4 \textwidth}
\centering
\includegraphics[scale=0.28]{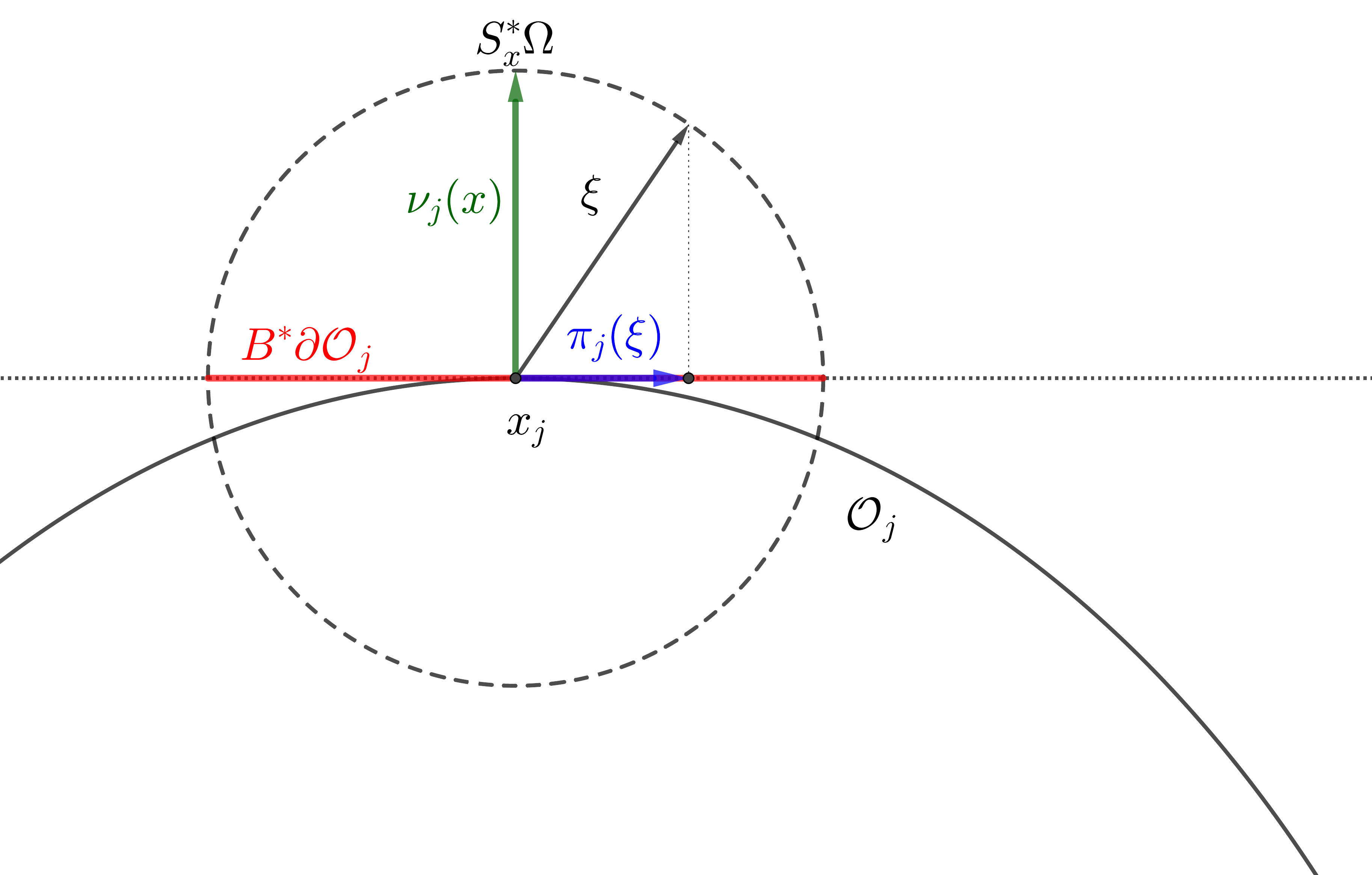}
\caption{The notations used to define the billiard map and the shadow map. }
\end{subfigure} \hspace{0.3cm}\vline  \hspace{0.3cm}
\begin{subfigure}{0.4 \textwidth} 
\centering
\includegraphics[scale=0.33]{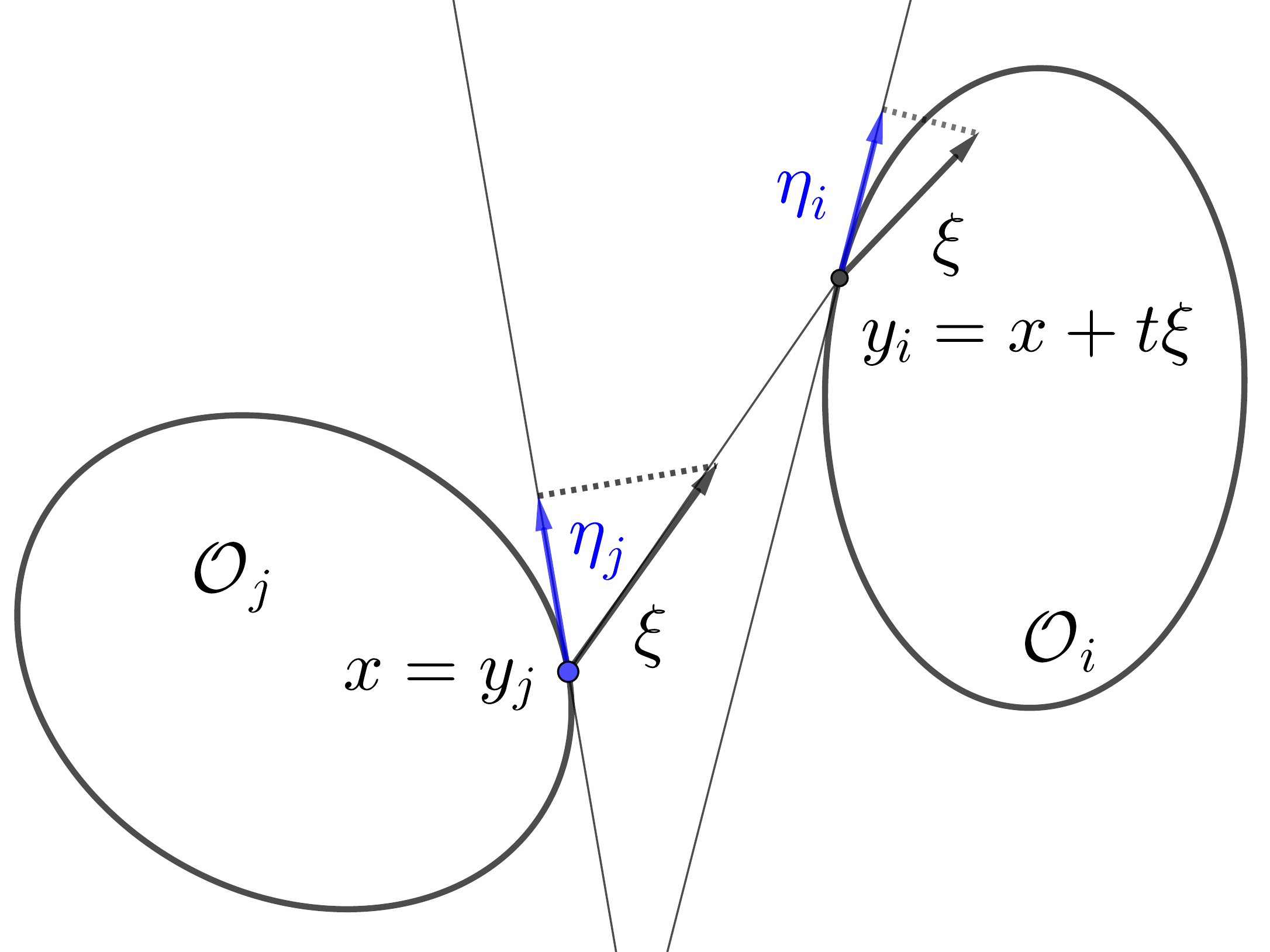}
\caption{The billiard map. $B^+_{ij}(y_j,\eta_j) = (y_i,\eta_i)$.  }
\end{subfigure}
\caption{Definition of the billiard map}
\label{figure_billiard_map}
\end{figure}
$\mathcal{B}$ is then the union of the maps $\mathcal{B}_{ij}$ corresponding to the reflection on two obstacles : for $(\rho_i, \rho_j) \in B^* \partial \mathcal{O}_i \times B^* \partial \mathcal{O}_j$ (with  $\rho_i=(y_i,\eta_i),  \rho_j=(y_j,\eta_j)$). 
 
 \begin{align*}
   &\rho_i= \mathcal{B}_{ij}(\rho_j) \iff  
   \exists t >0 \, , \,\xi \in \mathbb{S}^1\, , \, x \in \partial \mathcal{O}_j  \\
&\pi_j(x, \xi) = \rho_j \, , \, \pi_i(x + t \xi, \xi) = \rho_i \, , \, \nu_j(x) \cdot \xi >0 \, ,\, \nu_i(x+t \xi) \cdot \xi  <0 .
 \end{align*}
It is a standard fact in the study of chaotic billiards (see for instance \cite{Che}) that the billiard map is hyperbolic due to the strict convexity assumption. Ikawa's condition ensures that the restriction of the dynamical system to the trapped set has a symbolic representation (\cite{Mo91}). It is possible to restrict the study to a neighborhood of the trapped set. Since $\pi_y \left( \mathcal{T} \cap B^*\partial\mathcal{O}_j \right) \neq \partial \mathcal{O}_j$, it is possible to work with an interval $Y_j \subset \partial\mathcal{O}_j$ instead of the whole boundary. Moreover, it is known that $\mathcal{T}$ is compact and totally disconnected, so that the relation $\mathcal{B}$ satisfies the assumption of Theorem \ref{Theorem_main}. 

In \cite{NSZ14}, the author have shown that there exists a family $M_h(z)$ of open hyperbolic quantum maps associated with $\mathcal{B}$, depending holomorphically on $z \in \Omega(h) = ]-R, R[ + i ]- C\log 1/h, R[$ for some fixed $R >1$ and $C>1$, and such that for $h$ small enough and for $z \in \Omega(h)$, $\frac{1}{h} + z$ is a resonance if and only if  $ \det( 1- M_h(z) ) = 0$, and the multiplicity of the resonance coincides with the multiplicity of $z$ as a zero of $\det(1-M_h(z))$. 
The construction of this operator relies on the study of the operators $M_0(z) :  \cinf(\partial\mathcal{O} ) \to \cinf(\partial \mathcal{O})$ defined as follows : 
for $1 \leq j \leq J$, let $H_j(z) : \cinf(\partial\mathcal{O}_j ) \to \cinf (\R^2 \setminus \mathcal{O}_j ) $ be the resolvent of the problem 
$$ \left\{ \begin{array}{l}
(-h^2 \Delta - (1 + hz)^2) (H_j(z) v )  = 0 \\
H_j(z) v \text{ is outgoing} \\
H_j(z) v =v \text{ on } \partial \mathcal{O}_j
\end{array}\right. .$$
Let $\gamma_j$ be the restriction of a smooth function $u \in \cinf(\R^2)$ to $\cinf(\partial \mathcal{O}_j )$ and define $M_0(z)$ by  : 

$$M_0(z) =  \left\{ \begin{array}{l}
0 \text{ if } i=j \\
- \gamma_i H_j(z) \text{ otherwise.} 
\end{array}\right. $$
Using the analysis of Gérard (\cite{Ge88}, Appendix II) and restricting the study near the trapped set by the use of escape functions, the author transform $M_0$ into a Fourier integral operator associated with the billiard map (see Section 6 in \cite{NSZ14}). Moreover, by analyzing the formula given in \cite{Ge88}, Appendix II) we see that the amplitude of $M_h(z)$ is related, via the solutions of the eikonal equation, to the distance between two collisions. In particular, near the trapped set, it is given by 

\begin{equation}\label{amplitude_obstacle}
\alpha_h(z)(\rho) = \exp \left(-  t_{ret}(\rho) \im z \right)+O\left( h^{1-}S_{0^+} \right). 
\end{equation}
For $\rho \in \mathcal{T}$, $t_{ret}(\rho)$ is described as follows : assume that $\rho = (x,\xi)$ and $(y,\eta) =\mathcal{B}(x, \xi)$, then $t_{ret}(\rho) = |x-y|$. $t$ continues smoothly in a neighborhood of $\mathcal{T}$ and is called a \emph{return time function}. 

We can apply Theorem \ref{Theorem_main} to this family of open quantum maps and we find that, for any fixed $\varepsilon>0$ and for $r \gg 1$ (with $h=r^{-1}$, recalling that the resonances are given by $1/h + z$ where $z$ is a pole of $\det(1-M_h(z))$), the number $N(r,\gamma)$ of resonances, counted with multiplicity, in $[r,r+1]  - i[0, \gamma]$ satisfies
$$N(r,\gamma) \leq m_M \big \{  |\re z | <2,  \im z \geq - \gamma \} \big) \leq C_{\varepsilon, \gamma} r^{d_H - p(\gamma + \varepsilon)_+ +  \varepsilon}$$
Here, $p(\beta)$ is given by 
$$p(\beta) =  -\frac{1}{6\lambda_{\max}} P (2 \beta t_{ret} - \varphi_u). $$
Using the continuity of the pressure, we can choose $\varepsilon^\prime >0$ to ensure that $$P(2( \gamma + \varepsilon^\prime) t - \varphi_u) \geq P(2 \gamma t_{ret} - \varphi_u) + \varepsilon/2$$ and we may assume that $\varepsilon^\prime \leq \varepsilon/2$. Applying the above formula with $\varepsilon^\prime$ , we find that 
$$N(r,\gamma) \leq C_{\varepsilon, \gamma} r^{d_H - \sigma(\gamma) + \varepsilon}$$
with 
\begin{equation}\label{value_of_sigma}
\sigma(\gamma) =  \max\left(0,- \frac{1}{6 \lambda_{\max}} P (-\varphi_u + 2 \gamma t_{ret}) \right).
\end{equation}
To check that $\sigma$ satisfies the properties listed in Theorem \ref{Theorem_main}, we invoke the theory of Axiom A flows (\cite{BoRu}) : the map $s \mapsto  P (-\varphi_u +s t)$ is strictly increasing and has a unique root given by $\gamma_{cl}$. In particular, we deduce that $\sigma(\gamma) > 0 $ for $\gamma < \gamma_{cl}/2$ and $\sigma(\gamma) = 0$ for $\gamma \geq \gamma_{cl}/2$, as expected. Finally, since the bound $N(r,\gamma)=O(r^{d_H})$ holds for any $\gamma$, we can change $\sigma(\gamma) - \varepsilon$ into $(\sigma(\gamma)-\varepsilon)_+ = \max( \sigma(\gamma) - \varepsilon,0)$. This concludes the proof of Theorem \ref{THM_new}.

\subsection{Proof of Theorem \ref{Theorem_potential_scattering}}\label{subsection_potential_scattering}
Let us show how Theorem \ref{Theorem_main} implies Theorem \ref{Theorem_potential_scattering}. The ideas are the same as for the case of obstacle scattering and rely on the reduction performed in \cite{NSZ11}. 

We consider $V \in \cinfc(\R^2)$, $E_0>0$ and the semiclassical pseudodifferential operator $P_h = -h^2 \Delta + V - E_0$.
We note $p(x,\xi) = \xi^2 + V - E_0$ and we assume that
$$ dp \neq 0 \text{ on } p^{-1}(0).$$ 
Let's note $H_p$ the Hamiltonian vector field associated with $p$ and $\Phi_t = \exp(tH_p)$ the corresponding Hamiltonian flow. Let's note $K_0$ the trapped set of $\Phi_t$ at energy $0$ and we assume that $\Phi_t$ is hyperbolic on $K_0$ and $K_0$ is topologically one dimensional.  
More generally, we could work with more general Schrödinger operators in manifolds with Euclidean ends. We refer the reader to \cite{NSZ11} (Section 2.1) for more general assumptions. 

To apply Theorem \ref{Theorem_main}, we use the results of \cite{NSZ11} (Theorem 1 and 2). Under the assumptions above, there exists a smooth Poincaré hypersurface $\Sigma$ for the flow $\Phi_t$ on the energy shell $p^{-1}(0)$ near $K_0$. $\Sigma$ is made of several disjoint pieces $\Sigma_j$, $1 \leq j \leq J$. The reduced trapped set is now $\mathcal{T} \coloneqq K_0 \cap \Sigma$, and if we write $2d_H+1$ for the dimension of $K_0$, $\mathcal{T}$ has dimension 
$$ \dim \mathcal{T} = \dim K_0 - 1 = 2d_H . $$
The assumption that $\Sigma$ is a smooth Poincaré hypersurface ensures that there exists $\varepsilon_{\min}>0$ such that the map 
$$ (\rho, t) \in \Sigma \times ]-\varepsilon_{\min}, \varepsilon_{\min}[ \mapsto \Phi_t(\rho)$$ 
is a smooth diffeomorphism onto its image. We note $t_{ret}$ the return time function on $\Sigma$ : for $\rho \in \Sigma$, 
$$t_{ret}(\rho) = \inf \{ t >\varepsilon_{\min} , \Phi_t (\rho) \in \Sigma \} \in [\varepsilon_{\min}, + \infty]$$ 
$t_{ret} < + \infty$ in a neighborhood $U \subset \Sigma$ of $\mathcal{T}$. We then define the Poincaré return map $F$, which is an open hyperbolic map defined on an open subset of $\Sigma$ : 
$$ F: \rho \in \Sigma \mapsto \Phi_{t_{ret} }(\rho) \in \Sigma$$
In \cite{NSZ11}, the authors construct a family of finite-dimensional matrices $(\mathbf{M}(z;h))$ for $z \in \Omega(h) = ]-R,R[ + i ]-C \log 1/h , R[$ (with $R$ fixed but large) such that for $h$ small enough and for all $z \in \Omega(h)$, 
$$ \det(I - \mathbf{M}(z;h) ) = 0 \iff hz \text{ is a resonace of } P_h$$
The matrices $\mathbf{M}(z;h)$ satisfy uniformly for $z \in \Omega(h)$ and for $h$ small enough, 

\begin{equation}\label{def_matric_M}
 \mathbf{M}(z;h) = \Pi_h M(z;h) \Pi_h + O(h^L)
\end{equation}
where $L>0$ can be chosen as large as necessary , $\Pi_h$ is a finite rank projector and $M(z;h)$ is a family of open hyperbolic quantum maps associated with $F$ (in the sense of Definition \ref{def_FIO}). The amplitude of $M(z;h)$ satisfies 
$$ \alpha_h(z)(\rho) = \exp( - t_{ret}(\rho) \im z )   +O\left( h^{1-}S_{0^+} \right). $$
By their construction, $M(z;h)$ and $\Pi_h$ satisfy, for some $L>0$ as large as necessary, uniformly for $z \in \Omega(h)$ and for $h$ small enough, 
\begin{equation}\label{equation_M_Pi}
 \Pi_h M(z;h) \Pi_h = M(z;h) + O(h^L)
\end{equation}

We can apply Theorem \ref{Theorem_main} to the family $M(z;h)$ of open quantum maps and we find that, for any fixed $\varepsilon>0$ and $K>0$ (with $K < R)$ and for $h \ll 1$, the number $N_M(R, \gamma ;h)$ of zeros of $\det (\Id - M(z;h) ) $ in $\{ |\re z | < R , \im z \in [-\gamma,0]\}$ satisfies 
$$N_M(R,\gamma; h) \leq 
 C_{R,\varepsilon, \gamma} h^{-d_H + p(\gamma + \varepsilon)_+ - \varepsilon}$$
Here, $p(\beta)$ is given by 
$$p(\beta) = -\frac{1}{6\lambda_{\max}} P (2 \beta t_{ret} - \varphi_u) $$
where $\varphi_u$ is the unstable Jacobian associated with $F$. Here, it can also be obtained by differentiating the flow $\Phi_t$. In fact, by inspecting the proof of Theorem \ref{Theorem_main} and by using (\ref{def_matric_M}) and (\ref{equation_M_Pi}), we see that the same conclusion holds for $\mathbf{M}$ instead of $M$. Indeed, in the formula (\ref{key_equation}) in Proposition \ref{Prop_Key_2}, one can replace $M(z;h)$ by $\mathbf{M}(z;h)$ since $M(z;h)^{N(h)} = \mathbf{M}(z;h)^{N(h)} + O(|\log h | h^L)$ as soon as $N(h) = O(\log h)$. We now conclude as for the case of obstacle scattering in \ref{subsection_billiard_map} and find that 
$$ N (R,\gamma; h) \leq C_{R,\gamma,\varepsilon} h^{-d_H + \sigma(\gamma) - \varepsilon}$$
where 
\begin{equation}
\sigma(\gamma) =\max\left( 0,  - \frac{1}{6 \lambda_{\max}} P (-\varphi_u + 2 \gamma t_{ret})\right) 
\end{equation}

\appendix
\section{}

\subsection{Proofs of the missing Lemmas involving stationary phase expansions} \label{appendix_lemma_stationnary_phase}
In this appendix, we give the missing proofs of Lemmas \ref{Lemma_1_stationnary_phase}, \ref{Lemma_2_stationnary_phase} and \ref{Lemme_stationnary_phase_S_delta}. It relies on different uses of stationary phase theorems.

\subsubsection{Proof of of Lemma \ref{Lemma_1_stationnary_phase}}
To alleviate the notations, let's note $q(x,\eta) = \langle D^2 \psi(x_1, \xi_0)(x,\eta) , (x,\eta) \rangle$ and write it $q(x,\eta) = u x^2 + 2 v x \eta + w\eta^2$. The metaplectic operator $\mathcal{M}(d_{\rho_0} F)$ admits the kernel 
$$ k(x,y) \coloneqq \frac{|v|^{1/2}}{2\pi h} \int_{\R} e^{\frac{i}{h}\left( \frac{1}{2} q(x,\eta) - y \eta\right) } d\eta $$
and $\overline{k}(y,x)$ is the kernel of $\mathcal{M}(d_{\rho_0}F)^*$. 
We also note $$\mathcal{M}_bu(x) = 
\frac{1}{2\pi h} \int_{\R^2} e^{\frac{i}{h} \left(  \frac{1}{2} q(x,\eta)  - y\eta\right)  } \tilde{b}(x,\eta ) u(y) dy d\eta$$
We have 
\begin{align*}
 \left( \mathcal{M}(d_{\rho_0}F)^* \mathcal{M}_b u \right)(x) &= \frac{|v|}{(2\pi h)^2 } \int_{\R^4} e^{\frac{i}{h} \left( - \frac{1}{2} q(y,\eta) + x \eta+  \frac{1}{2} q(y,\xi)  - z \xi\right)  } \tilde{b}(y,\xi ) u(z) dy d\eta dz d\xi \\
 &= \frac{1}{2\pi h} \int_{\R^2}  u(z) e^{\frac{i}{h}(x-z) \xi} \underbrace{\left(\frac{|v|}{2 \pi h } \int_{\R^2}  e^{\frac{i}{h} \left(\frac{1}{2} q(y,\xi)- \frac{1}{2}q(y,\eta) + x (\eta - \xi)   \right)  } \tilde{b}(y,\xi) dy d\eta\right)}_{\check{b}(x,\xi)} dz d\xi \\
 &= \text{Op}_h^{R} (\check{b})u(x)\\
 &= \op(b) u (x) 
\end{align*}
where $\text{Op}_h^{R} $ denotes the right quantization, and by \cite{ZW} (Theorem 4.13), $b(x,\xi) = e^{-\frac{ih}{2} \langle D_x, D_\xi \rangle } \check{b}(x,\xi)$. 
Let's analyze $\check{b}$: 
\begin{align*}
\check{b}(x,\xi) &=  \frac{|v|}{2 \pi h } \int_{\R^2}  e^{\frac{i}{h} \left( \frac{1}{2} q(y,\xi)- \frac{1}{2} q(y,\eta+ \xi) + x \eta   \right)  } \tilde{b}(y,\xi) dy d\eta  \\
&=  \frac{|v|}{2 \pi h } \int_{\R^2}  e^{\frac{i}{h} \left( \frac{1}{2}w \eta^2 - w \eta \xi - vy\eta + x \eta   \right)  } \tilde{b}(y,\xi) dy d\eta  \\
&=  \frac{1}{2 \pi h } \int_{\R^2}  e^{\frac{i}{2h} w \eta^2}  e^{\frac{i}{h}(x- y)  \eta  } \tilde{b}(v^{-1}(y - w \xi),\xi) dy d\eta \text{ (\text{change of variable }} vy + w \xi \mapsto y) \\
&= e^{ \frac{ih}{2} w D_x^2 } \tilde{b}(v^{-1}(x-w\xi) , \xi)\\
\end{align*}
In particular, if $w=0$, we directly find that $\check{b}(x,\xi) = \tilde{b}(v^{-1}x, \xi)$. Otherwise, it is represented by the formula (\cite{ZW}, Theorem 4.8): 
$$\check{b} (x,\xi) = \frac{e^{i \frac{\pi}{4} \frac{w}{|w|}}}{2\pi h |w|} \int_\R e^{-\frac{ih}{2w} y^2} \tilde{b}(v^{-1}(y + x -w\xi), \xi ) dy$$
As a consequence, we see that $b$ is obtained from $\tilde{b}$ by composing 3 actions : the one of $e^{-\frac{ih}{2} \langle D_x, D_\xi \rangle }$, the change of variable $ (x,\xi) \mapsto (v^{-1}(x-w\xi) , \xi)$ and $e^{ \frac{ih}{2} w D_x^2 } $. The second one is obviously continuous from $S(\langle \rho \rangle^{3N})$ to $S(\langle \rho \rangle^{3N})$. 
We can now use \cite{ZW} Theorem 4.17 (or more precisely, the estimates given in the proof) : both the action of $e^{ \frac{ih}{2} w D_x^2 } $ and $e^{-\frac{ih}{2} \langle D_x, D_\xi \rangle }$ are continuous from $S(\langle\rho\rangle^{3N})$ to $S(\langle\rho\rangle^{3N})$, and more precisely, there exists a universal integer $M$ and universal constants $C_\alpha$ such that, for every $\alpha \in \N^2$, $(x,\xi) \in T^*\R$, 
$$ |\partial^\alpha (L \tilde{b}) (x,\xi)| \leq C_\alpha \sup_{|\beta| \leq |\alpha| + M} || \langle \rho \rangle^{-3N}\partial^\beta \tilde{b} ||  \langle \rho \rangle^{3N}$$
with $L$ being either  $e^{ \frac{ih}{2} w D_x^2 } $ or $e^{-\frac{ih}{2} \langle D_x, D_\xi \rangle }$. The same holds for the change of variable. 
This gives the required estimates for the symbol $b$ and concludes the proof of the Lemma. 

\begin{flushright}
\qed
\end{flushright}

\subsubsection{Proof of Lemma \ref{Lemma_2_stationnary_phase}} 

Fix $s \in [0,1]$ and recall that, with the notation $q$ introduced above$$
\tilde{R}_s u (x) =  \frac{1}{2\pi} \int_{\R^2} e^{i \left(  \frac{1}{2} q(x,\eta)+  sh^{1/2} r_3^\psi(x,\eta;h) - y\eta\right)  } b_N(x,\eta) u(y) dy d\eta$$
Let's introduce 
$$ R_s = \Lambda_h \tilde{R}_s \Lambda_h^*$$ and observe that the Schwartz kernel of $R_s$ is given by 
$$k_s(x,y) = \frac{1}{2\pi h} \int_\R  e^{\frac{i}{h} \left(  \frac{1}{2} q(x,\eta)+  s\rho_3^\psi(x,\eta;h) - y\eta\right)  } \tilde{b}_N(x,\eta) u(y)  d\eta$$
where 
$$\rho_3^\psi (x,\eta) = h^{3/2} r_3^\psi(h^{-1/2} x, h^{-1/2}\eta) = \psi( x_1 + x , \xi_0 + \eta) - \psi(x_1, \xi_0) - x \partial_x \psi(x_1, \xi_0)  -\eta \partial_\eta (x_1,\xi_0) - \frac{1}{2}q(x,\eta) $$
and $\tilde{b}_N(x,\eta) = b_N(h^{-1/2} x, h^{-1/2} \eta)$ which lies in $S_{0^+}(\langle \rho \rangle^{3N})$.
Let's note 
$\psi_s (x,\eta) = \frac{1}{2}q(x,\eta) + s \rho_3 ^\psi $ and remark that 
$$ \partial_{x \eta}^2 \psi_s = (1-s) \partial_{x \eta}^2 \psi(x_1, \xi_0) + s \partial_{ x \eta}^2 \psi(x_1 + x, \xi_0 + \eta)$$ 
Since $\partial_{ x \eta}\psi $ does not vanish on $\Omega_x \times \Omega_\eta$, it has constant sign and hence, $\partial_{x \eta}^2 \psi_s(x,\eta) \neq 0$ on $\Omega_x \times \Omega_\eta$. 
We now analyze the kernel $K_s$ of $R_s^* R_s$ and find that this kernel is 
\begin{align*}
K_s(x,y) &= \int_{\R} \overline{k}_s(z,x) k_s(z,y)dz\\
& = \int_{\R^3} \exp \left( \frac{i}{h}(\psi_s(z,\eta) - \psi_s(z,\xi) - y\eta + x\xi) \right)  \overline{ \tilde{b}_N(z,\xi)} \tilde{b}_N(z,\eta) d\eta d\xi d z  \\
& = \int_{\R} d\xi e^{ \frac{i}{h} (x-y)\xi}  \underbrace{\int_{\R^2}    \exp \left( \frac{i}{h}(\psi_s(z,\eta) - \psi_s(z,\xi) - y(\eta - \xi)) \right)  \overline{ \tilde{b}_N(z,\xi)} \tilde{b}_N(z,\eta) d\eta d z   }_{B_s(y,\xi)} 
\end{align*}
which is the kernel of $\text{Op}_h^R(B_s)$. To analyze $B_s$, we want to apply a stationary phase theorem and we need to know the stationary points in the variable $(z, \eta)$,  of the phase
$$\Phi_s(z,\eta, y ,\xi)= \psi_s(z,\eta) - \psi_s(z,\xi) - y(\eta - \xi)$$
We have 
$$ \partial_z \Phi_s (z,\eta, y, \xi) = \partial_x \psi_s(z,\eta) - \partial_x \psi_s(z, \eta) = \partial_{x \eta}^2 \psi_s(z, \eta(z,\xi, \eta) ) (\eta- \xi) $$ 
for some $\eta(z, \xi, \eta) \in [\eta, \xi]$. 
Hence,  since $\partial_{x \eta}^2 \psi$ does not vanish, $$\partial_z \Phi_s (z,\eta, y, \xi)=0 \iff \xi = \eta$$ We also have 
$$ \partial_\eta \Phi_s (z,\eta, y , \xi) = \partial_\eta \psi_s(z,\eta) - y$$ so that the equation $\partial_\eta \Phi_s (z,\xi, y , \xi)=0 $ has at most one solution, using again the fact that $\partial_{x \eta}^2 \psi_s$ does not vanish. When there is no stationary point, a non stationary phase argument gives that $|B_s(y,\xi)| \leq  \hinf \langle \rho \rangle^{6N}$. If there is a stationary point, it is given by a smooth function $z_s(y,\xi)$ locally around $(y,\xi)$ and a stationary phase argument shows that $|B_s(y,\xi)| \leq C_M  \langle \rho \rangle^{6N}$ where $C_M$ depends on the first $M$ semi-norms (for some universal integer $M$) of $\tilde{b}_N$. We can treat the derivatives of $B_s$ by differentiating under the integral and integration by part to obtain the same estimates for $\partial^\alpha B_s$, involving derivatives of $\tilde{b}_N$ up to order $|\alpha| +M$. This shows that $B_s \in S(\langle \rho \rangle^{3N})$. We conclude the proof by passing from $\text{Op}_h^R$ to $\op$ as in the proof of Lemma \ref{Lemma_1_stationnary_phase} and we come back to $h=1$ by standard scaling arguments.  
\begin{flushright}
\qed
\end{flushright}

\subsubsection{Proof of Lemma \ref{Lemme_stationnary_phase_S_delta}}\label{appendix_lemma_stationnary_phase_S_delta}

Let's write $u(x) = a(x) e^{ i \beta_n \frac{x^2}{2h}} $ with $a$ satisfying (\ref{classe_condition}).  
\begin{align*}
\op(m) u (x) &= \frac{1}{2\pi h} \int_{\R^2} m\left(\frac{x+y}{2}, \xi \right) e^{\frac{i}{h}(x-y)\xi}   a(y) e^{ i \beta_n \frac{y^2}{2h} } dy d\xi \\
&= \frac{1}{2\pi h} \int_{\R^2} m\left(x+ \frac{y}{2}, \beta_n x + \xi \right)  a(x +y ) e^{-\frac{i}{h}y(\xi + \beta_n x)}  e^{ i \beta_n \frac{(x+y)^2}{2h} } dy d\xi   \\
&=  e^{i \beta_n \frac{x^2}{2h}} \underbrace{\frac{1} {2\pi h} \int_{\R^2} m\left(x+ \frac{y}{2}, \beta_n x + \xi \right)  a(x +y ) e^{\frac{i}{2h} \left( \beta_n y^2- 2y\xi\right)}dy d\xi}_{B(x)}  \\
\end{align*}
To analyze $B(x)$ ,we invoke the stationary phase theorem in the quadratic case (see \cite{ZW}, Theorem 3.13) with the non singular symmetric matrix 
$Q_n = \left( \begin{matrix}
 \beta_n & -1 \\
 -1 & 0 
\end{matrix}\right) $ and we follow the proof of \cite{ZW}, Theorem 4.17. We fix a cut-off function $\chi \in \cinfc(\R^2)$ with $\supp \chi \subset B(0,1)$ and $\chi = 1$ in a neighborhood of $0$. We write (with $\chi_1 = \chi, \chi_2 = 1 - \chi$)
$$B(x)= B_1(x) + B_2(x) \quad ; \quad B_i (x)= \frac{1} {2\pi h} \int_{\R^2} \chi_i(y,\xi) m\left(x+ \frac{y}{2}, \beta_n x + \xi \right)  a(x +y ) e^{\frac{i}{2h} \left( \beta_n y^2- 2y\xi\right)}dy d\xi $$ 
We also set $v_i(x)=B_i(x) e^{i \beta_n \frac{x^2}{2h}}$.
By the stationary phase expansion, we can expand $B_1$ : 
for every $N \in \N$, 
$$ B_1(x) = \sum_{k=0}^{N-1} \frac{h^k}{k!} \left(\frac{(Q_n^{-1}D, D)}{2i} \right)^k c(x,0,0) + R_N(x) $$
$$  c(x,y,\xi)  =\chi(y,\xi) m\left(x+ \frac{y}{2}, \beta_n x + \xi \right)  a(x +y )  \quad ; \quad D = \left(\begin{matrix}
D_y \\
D_\xi
\end{matrix} \right)$$
$$ R_N(x) =O\left(h^N \sup_{y,\xi} \sup_{ k + l\leq 2N +2} |\partial_y^{k} \partial_\xi^{l} c(x,y,\xi)| \right)$$
We observe that : 
\begin{itemize}
\item The first term of the expansion of $B_1$ is given by $m(x, \beta_n x) a(x)$ ; 
\item $B_1$ is smooth since we can derive under the integral and obtain the same kind of expansion; 
\item The $k$-th term, that is $ c_k(x) = \frac{1}{k!} \left(\frac{(Q_n^{-1}D, D)}{2i} \right)^k c(x,0,0) $ is a sum of terms of the form $c_\alpha \partial^{\alpha} m(x, \beta_n x) a^{(l)}(x)$ with $ \alpha \in \N^2$, $l \in \N$, $|\alpha| + l \leq 2k$ and $c_\alpha \in \R$. The coefficients $c_\alpha$ of this sum depend on $Q_n$. Since $\beta_n = O(\varepsilon_0)$, these coefficients are bounded uniformly in $n$.  As a consequence, there exists $c_{k,p}= c_{k,p}(m)$ such that for $p \in \N$ with $k+p >0$, 
$$ q_p(c_k) \leq c_{k,p} h^{-2 k \delta} q_{2k+p}(u)$$ 
Hence we set $A_k u(x) = h^{2 k\delta } c_k(x)^{i \beta_n \frac{x^2}{2h}}$, which has the required form in virtue of the expression of $c_k(x)$. 
\end{itemize}
Concerning the remainder term, we have 
$$R_N(x) \leq C_N(m) h^N h^{-(2 N+2) \delta} \sup_{|y| \leq 1}  \left( 1 + \frac{(x+y)^2}{\alpha_n h} \right)^{-2}$$
It is not hard to see that $$
 \sup_{|y| \leq 1}  \left( 1 + \frac{(x+y)^2}{\alpha_n h} \right)^{-2} \leq C  \left( 1 + \frac{x^2}{\alpha_n h} \right)^{-2}$$
 We choose $M >0$ such that $M(1-2\delta) - 2 \delta>0$, so that $R_{N+M}(x) \leq C_N(m) h^{N(1-2\delta)} \sup_{|y| \leq 1}  \left( 1 + \frac{(x+y)^2}{\alpha_n h} \right)^{-2}$. By writing,  
$B_1 (x) = \sum_{k=0}^{N-1} c_k(x) + \sum_{k=N}^{N+M-1}c_k(x) + R_{N+M}(x)$
, we see that $$ q_0\left( v_1- \sum_{k=0}^{N-1} h^k A_k u \right) \leq C_N h^{N(1-2\delta)} q_{2N+M}(u)$$ 
By differentiating under the integral, we can show similarly that 
$$ q_j\left( v_1- \sum_{k=0}^{N-1} A_k u \right) \leq C_N h^{N(1-2\delta)} q_{j+2N+M}(u)$$
It remains to analyze $B_2$. Since there is no stationary point in the integral defining $B_2$, we do repeated integration by part using the differential operator $L(y,\xi) = \frac{(Q_n(y,\xi),D)}{|Q_n(y,\xi)|^2} $ which satisfies $L\left( e^{ \frac{i}{h}(Q_n(y,\xi), (y,\xi) )} \right) = e^{ \frac{i}{h}(Q_n(y,\xi), (y,\xi) )}$. Set $c_2(x,y,\xi)= (1- \chi(y,\xi) ) m(x+y/2, \beta_n + y/2) a(x+y)$. Since $|Q(y,\xi)| \geq c  (y^2 + \xi^2)^{1/2}$ on $\supp (1- \chi)$, we observe that for $M \in \N$. 
\begin{align*}
(L^*)^{2M} c_2(x,y,\xi) &\leq C_M (1 + y^2+\xi^2)^{-M/2} h^{2(1-\delta) M} q_{2M}(u) (\alpha_n h)^{1/4} \left( 1 + (x+y)^2/\alpha_n h \right)^{-2} \\ &\leq C_M (1 + y^2+\xi^2)^{-M/2} h^{2(1-\delta) M}q_{2M}(u)  (\alpha_n h)^{1/4} \left( 1 + x^2/\alpha_n h \right)^{-2}  \left( 1 + y^2/\alpha_n h \right)^{-2}\\
&\leq C_M (1 + y^2+\xi^2)^{-M/2} h^{2(1-\delta) M}q_{2M}(u) (\alpha_n h)^{1/4} \left( 1 + x^2/\alpha_n h \right)^{-2}  
\end{align*} 
Integrating over $\R^2$, we find that $|B_2(x)| \leq C_M h^{2(1- \delta )M} q_{2M}(u)(\alpha_n h)^{1/4} \left( 1 + x^2/\alpha_n h \right)^{-2}  $. In particular, with $M=N$, $q_0(v_2) \leq C_N h^{N(1-2\delta)} q_{2N}(u)$. Similarly, we can show that $q_j(v_2) \leq C_{j,N} h^{N(1-2\delta)} q_{j+2N}(u)$.
Since $\op(m) u = v_1 + v_2$, this concludes the proof of the Lemma \ref{Lemme_stationnary_phase_S_delta}. 
\begin{flushright}
\qed
\end{flushright}

\subsection{Formulas for approximation of exponential}\label{Appendix_exp}
We consider 
\begin{itemize}
\item a Hilbert space $H$ ($H = L^2(\R)$ for applications in this article) ; 
\item a bounded operator $A : H \to H$ ; 
 \item a parameter $h$; 
 \item a "class" $\mathcal{C}$ of elements of $H$, that is a subspace of $H$. 
\end{itemize}
We assume that for each $j \in \N$, there exists $A_j : \mathcal{C} \to \mathcal{C}$ such that, in some sense to be specified in applications, 
$A u \sim \sum_{j =0}^\infty h^j A_j u $.  More precisely, we assume that for all $N \in \N$ and all $u \in \mathcal{C}$, we can write 
$$ Au = \sum_{j=0}^{N-1} h^j A_j u + h^N R_N(u)$$
We are interested in understanding the action of the operator $e^{tA}$ on elements of $\mathcal{C}$. 
Recall that if $u_0 \in H$, $ t \mapsto e^{tA}u_0$ is the solution of the Cauchy problem 
$$\left\{ \begin{array}{l}
\frac{d}{dt} u(t) = A u(t) \\
u(0) = u_0 
\end{array}\right. $$

Moreover, we assume that $A_0$ extends to a bounded operator on $H$, so that $e^{tA_0}$ is a well-defined operator and we assume also that $e^{tA_0}(\mathcal{C}) \subset \mathcal{C}$ for all $t \in \R$. 
We introduce in this appendix formulas and notations to give an approximation of $e^{tA}u$. Of course, the interesting mathematical work lies in controlling the following terms and the accuracy of the expansion, which is done in applications.  Let us fix an integer $N \in \N$ and an initial state $u \in \mathcal{C}$. 

\subparagraph{Leading term. } For our leading term, we simply state 
$ u_0(t) = e^{tA_0} u$. Then, we set $R_0(t) = e^{tA} u - e^{tA_0} u$. 
We have $
\dot{R}_0(t) = A e^{tA} u - A_0 e^{tA_0} u$. Hence, we have
\begin{equation}\label{app_exp_reste_0}
\dot{R}_0(t) = A R_0(t) + \sum_{j=1}^{N-1}h^j A_j e^{tA_0} u  + \tilde{r}_{0,N}(t) \quad  ; \quad  \tilde{r}_{0,N}(t) = h^N R_N(u_0(t) ) 
\end{equation}

\subparagraph{First correction. } When $N = 1$, we stop. Otherwise, we can correct this first approximation by a term of order $h$. Of course, it is possible to write down directly a general formula for every $j$, but it seems to the author that the case $j=1$ helps to understand the general case. Let's try the Ansatz $u_1(t) = e^{tA_0} v_1(t)$ and set 
$$R_1(t) = e^{tA} u - e^{tA_0} \left( u + h v_1(t) \right)$$ 
Then we have, 
\begin{align*}
\dot{R}_1(t)  &= \dot{R}_0(t) - h e^{tA_0} \left( A_0 v_1(t) + v_1^\prime(t) \right) \\
&= A R_0(t) + \sum_{j=1}^{N-1} h^j A_j e^{tA_0} u  + \tilde{r}_{0,N}(t) - h A e^{tA_0} v_1(t) + h (A - A_0) e^{tA_0} v_1(t) - h e^{tA_0} v_1^\prime(t)\\
&=AR_1(t)+\sum_{j=1}^{N-1} h^j A_j e^{tA_0} u  + \tilde{r}_{0,N}(t) + h (A - A_0) e^{tA_0} v_1(t) - h e^{tA_0} v_1^\prime(t)
\end{align*}

To cancel the term of order $h$ in the sum, we set 
\begin{equation}
v_1(t) = \int_0^t e^{-s A_0} A_1 e^{s A_0} u ds
\end{equation}
To proceed with our expansion, we need to assume that $v_1(t) \in \mathcal{C}$ for all $t \in \R$. This will be the case in the applications, with precise control on $v_1(t)$. 
\subparagraph{Higher order terms.} 
For convenience, let's note $A_j(s) = e^{-s A_0} A_j e^{s A_0} $. We can construct by induction a family of functions $v_k(t)$ by setting $v_0(t) = u$ and for $1 \leq k \le N-1$, 
\begin{equation}\label{formula_v_k}
v_k(t) =  \sum_{l=0}^{k-1} \int_0^t A_{k-l}(s)  v_{l}(s) ds 
\end{equation}
For these formulas to hold, we assume this construction ensure that $v_k(t) \in \mathcal{C}$ for all $t \in \R$. It will be easily satisfied in applications. 
We also set 
$$\tilde{r}_{k,N}(t) = (A-A_0)e^{tA_0} v_k(t) - \sum_{j=1}^{N-k-1} h^j A_j e^{tA_0} v_k(t) $$
and $$R_k(t) = e^{tA}u - e^{tA_0} \sum_{l=0}^k h^l v_l(t) $$
$\tilde{r}_{k,N}(t) $ has to be seen as a term of order $h^{N-k}$. 
These formulas ensure that 
$$\dot{R}_k(t) = AR_k(t) + \sum_{j=k+1}^{N-1} h^j \left( \sum_{l=0}^k A_{k-l}e^{tA_0} v_l(t) \right) + \sum_{j=0}^k h^j \tilde{r}_{j,N}(t) $$
In particular, when $k=N-1$, 
\begin{equation}\label{formule_reste_exact}
\dot{R}_{N-1}(t) = AR_{N-1}(t) + \sum_{j=0}^{N-1} h^j \tilde{r}_{j,N}
\end{equation}

\newpage
\bibliographystyle{alpha}
\bibliography{biblio_these}

\end{document}